\documentclass[11pt,final]{book} 

\usepackage{ifdraft,ifluatex}

\ifoptionfinal{%
\usepackage[paperheight=24cm,paperwidth=17cm,%
            hmargin={1.8cm,2.3cm},vmargin={2.2cm,2.0cm},%
            headheight=14pt%
           ]{geometry}
}{%
\usepackage[a4paper,hmargin={2.4cm,2.4cm},vmargin={3.1cm,2.9cm},headheight=14pt]{geometry}
}

\ifoptiondraft{%
\hfuzz=2pt
\overfullrule=5pt
\PassOptionsToPackage{draft}{impnattypo}
}{}

\ifluatex
\else
\usepackage[utf8]{inputenc}
\fi

\usepackage[T1]{fontenc}
\usepackage{fourier}  

\usepackage[defaultlines=2,all]{nowidow}
\ifluatex

\else
\usepackage[hyphenation]{impnattypo}
\fi

\usepackage[font=small,labelfont=bf]{caption}
\usepackage{fancyhdr}
\usepackage{nextpage}
\usepackage{mathtools}
\usepackage{array}
\usepackage{eldering}

\usepackage{footnote}
\usepackage{etoolbox}
\preto\subequations{\ifhmode\unskip\fi}

\usepackage{makeidx}
\usepackage{ragged2e}
\usepackage[hangindent=20pt,columnsep=10pt,
            justific=RaggedRight,font=small]{idxlayout}
\makeindex
\let\oldindex\index
\renewcommand{\index}[1]{\oldindex{#1}}

\newcommand{\idx}[1]{#1\oldindex{#1}}
\newcommand{\idxb}[1]{#1\oldindex{#1}}
\newcommand{\idxalso}[2]{\oldindex{#1}\oldindex{#1|seealso{#2}}}

\usepackage[english,dutch]{babel}
\addto\captionsenglish{}
\addto\captionsenglish{\renewcommand{\bibname}{Bibliography}}

\usepackage{hyphenat}
\allowhyphens 

\fancypagestyle{plain}{%
  \fancyhf{}
  
  \fancyfoot[C]{\bfseries\thepage}
}

\fancypagestyle{main}{%
  \fancyhf{}

  \fancyhead[RE]{\sffamily\bfseries\nouppercase\leftmark}
  \fancyhead[LO]{\sffamily\bfseries\nouppercase\rightmark}
  \fancyhead[LE,RO]{\bfseries\thepage}
}

\fancypagestyle{appendix}{%
  \fancyhf{}

  \fancyhead[RE]{\sffamily\bfseries\nouppercase\leftmark}
  \fancyhead[LO]{\sffamily\bfseries\nouppercase\rightmark}
  \fancyhead[LE,RO]{\bfseries\thepage}
}

\let\cleardoublepage\clearemptydoublepage

\makeatletter
\renewcommand*{\@pnumwidth}{2.1em}
\renewcommand*{\@dotsep}{2.3}
\makeatother

\newcommand{\chapternonumber}[1]{%
  \chapter*{#1}
  \phantomsection
  \addcontentsline{toc}{chapter}{\protect\numberline{}#1}
  \setcounter{footnote}{0}
}

\usepackage{enumitem}
\setlist[enumerate]{leftmargin=4.5ex}
\setlist[itemize]{leftmargin=3ex}

\labelformat{enumi}{\textsl{#1}}

\def\figpath{}
\graphicspath{{\figpath}}

\DeclareFontShape{OMS}{cmsy}{m}{n}{<-> s * [0.920] cmsy9}{}
\DeclareMathAlphabet{\mathcal}{OMS}{cmsy}{m}{n}

\makeatletter
\def\resetMathstrut@{%
  \setbox\z@\hbox{%
    \mathchardef\@tempa\mathcode`\(\relax
    \def\@tempb##1"##2##3{\the\textfont"##3\char"}%
    \expandafter\@tempb\meaning\@tempa \relax
  }%
  \ht\Mathstrutbox@1.2\ht\z@ \dp\Mathstrutbox@1.2\dp\z@
}
\makeatother

\usepackage{accents}

\DeclareMathSymbol{\widehatsym}{\mathord}{largesymbols}{"62}
\DeclareMathSymbol{\widetildesym}{\mathord}{largesymbols}{"65}

\renewcommand{\bar}[1]{\accentset{\rule{.42em}{.6pt}}{#1}}

\usepackage{overpic}

\newcommand{\BC}{C_b}           
\newcommand{\BUC}{C_{b,u}}      

\newcommand{\Christoffel}{\Gamma} 
\newcommand{\partrans}{\Pi}     
\newcommand{\geodflow}{\Upsilon}
\newcommand{\contr}{q}          
\newcommand{\Xsize}{\beta}      
\newcommand{\Ysize}{\eta}       
\newcommand{\Vsize}{\zeta}      
\newcommand{\Ynonlin}{\tilde{f}}
\newcommand{\fracdiff}{r}       

\newcommand{\DF}{\smash{\underset{\widetilde{}}{\D}}}
\newcommand{\TF}{\smash{\underset{\widetilde{}}{\T}}}

\newcommand{\sx}{{\scriptscriptstyle X}}
\newcommand{\sy}{{\scriptscriptstyle Y}}
\newcommand{\sz}{{\scriptscriptstyle Z}}

\newcommand{\vx}{\tilde{v}_\sx} 

\newcommand{\BX}{\mathcal{B}^\rho_\Xsize(I;X)}
\newcommand{\BY}{B^\rho_\Ysize(I;Y)}
\newcommand{\BYc}{\bar{B}^\rho_\Ysize(I;Y)}

\newcommand{\dx}{\delta\!\!x}
\newcommand{\dy}{\delta\!\!y}

\newcommand{\rinj}[1]{r_\text{inj}(#1)}

\renewcommand{\th}{\ndash th }

\newcommand{\embedto}{\hookrightarrow}
\newcommand{\diffto}{\xrightarrow{\raisebox{-0.2 em}[0pt][0pt]{\smash{\ensuremath{\sim}}}}}

\let\epsilon\varepsilon

\let\phi\varphi

\DeclareMathOperator{\supp}{supp}
\DeclareMathOperator{\Hor}{Hor}
\DeclareMathOperator{\Ver}{Vert}

\newcommand{\thesistitle}%
{Persistence of noncompact\\ normally hyperbolic invariant manifolds\\
  in bounded geometry}

\newcommand{\nativetitle}%
{Persistentie van niet-compacte\\ normaalhyperbolische invariante
  vari\"eteiten\\ onder aanname van begrensde meetkunde}

\newcommand{\thesisauthor}{J.\ Eldering}

\title{\thesistitle}
\author{\thesisauthor}

\begin{document}

\selectlanguage{english}
\input{ushyphex}
\hyphenation{boun-da-ry boun-ded ma-ni-fold ma-ni-folds non-emp-ty Lip-schitz}
\hyphenation{as-t\'er-is-que co-do-main co-do-mains to-po-lo-gy}

\frontmatter
\pagestyle{plain}

\maketitle

\thispagestyle{empty}
\clearpage

\hypersetup{linkbordercolor={1 0.5 0.4},pdfborder={0 0 1}}

\chapternonumber{Preface}

In this thesis we prove persistence of normally hyperbolic invariant
manifolds. This result is well-known when the invariant manifold is
compact; we extend this to a setting where the invariant manifold as
well as the ambient space are allowed to be noncompact manifolds. The
ambient space is assumed to be a Riemannian manifold of bounded geometry.

Normally hyperbolic invariant manifolds (NHIMs) are a generalization
of hyperbolic fixed points. Many of the concepts, results and proofs
for hyperbolic fixed points carry over to NHIMs. Two important
properties that generalize to NHIMs are persistence of the invariant
manifold and existence of stable and unstable manifolds.

We shall focus on the first property. Persistence of a hyperbolic
fixed point follows as a straightforward application of the implicit
function theorem. For a NHIM the situation is significantly more
subtle, although the basic idea is the same.
In the case of a hyperbolic fixed point we only have stable and
unstable directions. When we consider a NHIM, there is a third
direction, tangent to the manifold itself. The dynamics in the
tangential directions is assumed to be dominated by the stable and
unstable directions in terms of the respective Lyapunov exponents.
Thus the dynamics on the invariant manifold is approximately neutral and the dynamics
in the normal directions is hyperbolic; hence the name \emph{normally
hyperbolic}. The system is called $\fracdiff$\ndash normally
hyperbolic if the \emph{spectral gap condition} holds that the
tangential dynamics is dominated by a factor $\fracdiff \ge 1$.
An $\fracdiff$\ndash NHIM persists under $C^1$ small perturbations of
the system. The persistent manifold will be $C^\fracdiff$ if the
system is, but it may not be more smooth, even if the system is
$C^\infty$ or analytic. This can also be formulated as follows: $\fracdiff$\ndash
normal hyperbolicity is an `open property' in the space of
$C^\fracdiff$ systems under the $C^1$ topology. The description above
shows that the spectral properties of NHIMs and center manifolds are similar.
The difference is that NHIMs are globally uniquely defined, while
center manifolds are not.

There are two basic methods of proof for hyperbolic fixed points and
center manifolds: Hadamard's graph
transform and Perron's variation of constants integral method. Both
can be extended to prove persistence of NHIMs, as well as existence of
its stable and unstable manifolds. We employ the Perron method.

Both methods of proof construct a contraction scheme to find the
persistent NHIM (and a similar contraction scheme can be
used to find its stable and unstable manifolds). \emph{Heuristically},
we can construct the implicit function $F(M,v) = \Phi^t(M) - M = 0$,
where $M$ is the NHIM and $\Phi^t$ is the flow of the vector field $v$
after some fixed time $t$. Normal
hyperbolicity of $M$ implies that $\D_1 F$ is invertible. Hence, there
is a function $\tilde{M} = G(\tilde{v})$ that maps perturbed vector
fields $\tilde{v}$ to persistent manifolds $\tilde{M}$, at least in a
neighborhood of $v$. This idea does not work directly for higher
derivatives. An inductive scheme can be set up that typically uses
some form of the fiber contraction theorem. This scheme will break
down after $\fracdiff$ iterations, hence the limited smoothness.
Example~\ref{exa:opt-smooth} shows that this is an intrinsic problem.

To tackle the noncompact case, we replace compactness by uniformity
conditions. These include uniform continuity and global boundedness of
the vector field and the invariant manifold and their derivatives up
to order $r$. We require additional uniformity conditions on the
ambient manifold, namely `bounded geometry'. This means that the
Riemannian curvature is globally bounded, and as a result we have a
uniform atlas which allows us to retain uniform estimates throughout
all constructions in the proof.

\hspace{0.5cm}\rule[0.4em]{\textwidth-1cm}{1pt}\hspace{-5cm}

This thesis is organized as follows. In the introduction, we give a
broad overview of the theory of NHIMs with references to more details
in the later chapters. We start by describing how NHIMs are related to
hyperbolic fixed points and center manifolds. Then we give some basic
examples and motivation for studying the noncompact case. We give a
brief overview of the history and literature and compare the two
methods of proof in the basic setting of a hyperbolic fixed point.
Then we continue to introduce the concept of bounded geometry and a
precise statement of the main result of this thesis and discuss its
relation to the literature. We describe a few extensions and details
of the results and conclude the chapter with notation used throughout
this thesis.

Chapter~\ref{chap:boundgeom} treats Riemannian manifolds of bounded
geometry. We first introduce the definition of bounded geometry
and some basic implications. We explicitly work out the relation
between curvature and holonomy in Section~\ref{sec:BG-curv-holonomy}.
This we use in Section~\ref{sec:NHIM-smoothness} to prove smoothness of the
persistent manifold. In the subsequent sections we develop the
theory required to prove persistence of noncompact NHIMs in general
ambient manifolds of bounded geometry. We extend results for
submanifolds to uniform versions in bounded geometry, to finally show
how to reduce the main theorem to a setting in a trivial bundle. A
number of these results are new and may be of independent interest, namely
the uniform tubular neighborhood theorem, the uniform smooth approximation
of a submanifold, and a uniform embedding into a trivial bundle.

In Chapter~\ref{chap:persist} we finally prove the main result in the
trivial bundle setting. We first state both this and the general
version of the main theorem and discuss these in full detail. We
include a precise comparison with results in the literature, followed
by an outline of the proof. Section~\ref{sec:cpt-unif} contains a
discussion of the differences to the compact case and
presents detailed examples to illustrate these. Then we start the actual
proof. We first prepare the system: we put it in a suitable form and
obtain estimates for the perturbed system. Then we prove that there
exists a unique persistent invariant manifold and that it is
Lipschitz. Secondly, we set up an elaborate scheme in
Section~\ref{sec:NHIM-smoothness} to prove that this
manifold is $C^\fracdiff$ smooth by induction over the smoothness
degree.

In Chapter~\ref{chap:extension-results} we discuss how the main result
can be extended in a number of different ways that may specifically be
useful for applications. We show how time and parameter dependence can
be added and we present a slightly more general
definition of overflow invariance that might be applicable to
systems that are not overflowing invariant under the standard
definition.

Finally, the appendices contain technical and reference material.
These are referenced from the main text where appropriate.
Appendix~\ref{chap:smoothflows} shows an important idea that permeates
this work: the implicit function theorem allows for explicit estimates
in terms of the input, hence it `preserves uniformity estimates'. This
can then directly be applied to dependence of a flow on the vector
field. In Appendix~\ref{chap:nemytskii}, the Nemytskii operator is
introduced as a technique to prove continuity of post-composition with
a function. This is an essential basic part in the smoothness proof,
together with the results on the exponential growth behavior of
higher derivatives of flows in Appendix~\ref{chap:exp-growth}. Here,
we also develop a framework to work with higher derivatives on
Riemannian manifolds. The last appendices include the fiber
contraction theorem of Hirsch and Pugh that is used in the smoothness
proof, Alekseev's nonlinear variation of constants integral defined on
manifolds, and a brief overview of those parts of Riemannian geometry
that we use.


\cleardoublepage

\vspace*{4cm}
\begin{center}
  \textit{In memory of Hans Duistermaat}
\end{center}

\begingroup
\hypersetup{pdfborder=0}
\let\cleardoublepage\clearpage
\cleartoevenpage[\pagestyle{plain}]
\cleardoublepage
\tableofcontents
\cleartooddpage[\pagestyle{empty}]
\endgroup

\mainmatter
\pagestyle{main}


\chapter{Introduction}\label{chap:introduction}

The basics of the theory of hyperbolic dynamics date back to the
beginning of the $20$th century, and the general formulation of the theory of
normally hyperbolic systems was stated around 1970. Since then,
many people have extended the theory, and even more people have
applied it to problems in all kinds of areas.

Normally hyperbolic invariant manifolds are important fundamental
objects in dynamical systems theory. They are useful in understanding
global structures and can also be used to simplify the description of
the dynamics in, for example, slow-fast or singularly perturbed systems.

In this thesis, we are specifically interested in noncompact normally
hyperbolic invariant manifolds. We extend classical results that
were previously only formulated for compact manifolds. However, in many applications the
manifold is not compact, so an extension of the theory to the general
noncompact case allows one to attack these problems in their natural
context. The main result of this thesis is an extension of the theorem
on persistence of normally hyperbolic invariant manifolds to a general
noncompact setting in Riemannian manifolds of bounded geometry type.



\section{Normally hyperbolic invariant manifolds}
\index{normally hyperbolic invariant manifold}

We should first point out that the theory of (normally) hyperbolic
systems can be applied to both discrete and continuous dynamical
\index{dynamical system!continuous}
\index{dynamical system!discrete}
systems. That is, if we have a dynamical system $(T,\,X,\,\Phi)$ with
$X$ a smooth manifold and $\Phi\colon T \times X \to X$ the evolution
function, then the system\footnote{%
  For simplicity of presentation we ignore the facts that $\Phi$ may
  have a smaller domain of definition, or that it is a semi-flow or
  semi-cascade, only defined on $T \ge 0$.%
} is called discrete if $T = \Z$ and continuous if $T = \R$. In the
discrete case, one typically has a diffeomorphism $\phi\colon X \to X$
and the full evolution function is defined as $\Phi(n,x) = \phi^n(x)$,
i.e.\ws iterated application of $\phi$. In the continuous case, the
map $\Phi$ is called a flow. It is generated by a vector field
$v \in \vf(X)$ and in that case the map $\Phi^t\colon X \to X$ is
again a diffeomorphism for any $t \in \R$.

The two cases can be related by fixing a $t \in \R$ in the continuous
case and then view $\phi = \Phi^t$ as generating a discrete system.
The statements of definitions and results are (almost) identical if
formulated in terms of the evolution function $\Phi$. The methods of
proof share this similarity and can be translated into each other. We
shall adopt the continuous formulation in this work, and refer to the
evolution parameter $t \in T = \R$ as time. Even though our system is
defined in terms of a vector field $v$, we call $x \in X$ a \idx{fixed
point} of the system when $\Phi^t(x) = x$ for any $t \in \R$. This is
equivalent to saying that $v(x) = 0$, i.e.\ws that it is a critical
point of $v$; we adhere to the former terminology to better preserve
the analogy with discrete systems.

Before we proceed to explaining normally hyperbolic invariant
manifolds, it should be pointed out that these are a generalization of
\idx{hyperbolic fixed point}s. Many of the characteristic properties
generalize as well, so we first sketch the basic picture for
hyperbolic fixed points. Let $x$ be a fixed point of a vector field,
$v(x) = 0$; it is called hyperbolic if the derivative $\D v(x)$ has no
eigenvalues with zero real part. This means that the eigenvalue
spectrum splits into parts left and right of the imaginary axis, that
is, the stable and unstable eigenvalues, but no neutral ones. The
corresponding stable and unstable eigenspaces $E_\pm$ are both
invariant under the linear flow of $\D v(x)$ and these spaces are
characterized by the fact that solution curves on them converge exponentially
fast towards the fixed point under forward or backward time
evolution respectively. It is a well-known result that there are
corresponding stable and unstable (local) manifolds, denoted $W_\text{loc}^S$
\index{stable (unstable) manifold}
\index{unstable manifold|see{stable manifold}}
and $W_\text{loc}^U$ respectively, which are the nonlinear versions
of these, see Figure~\ref{fig:hyperb-fp}. This situation can be
generalized to a normally hyperbolic invariant manifold by replacing
the single fixed point by a `fixed set of points', that is, a
manifold which is, as a whole, invariant.
\begin{figure}[htb]
  \centering
  \input{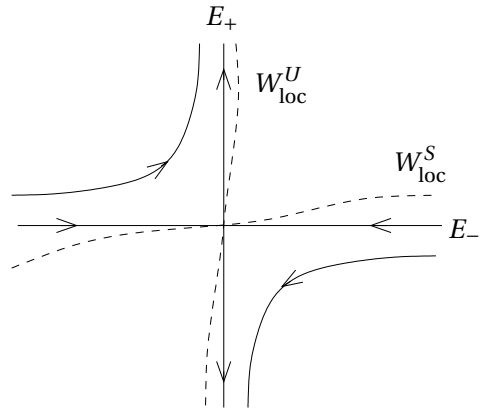}
  \caption{A hyperbolic fixed point with (un)stable manifolds
    $W_\text{loc}^S,W_\text{loc}^U$.}
  \label{fig:hyperb-fp}
\end{figure}

\idxalso{NHIM}{normally hyperbolic invariant manifold}
Let us start with a somewhat informal explanation of the concept of a
normally hyperbolic invariant manifold, which we shall from now on
often abbreviate as a \emph{NHIM}, as is common in the literature. If
we have a dynamical system $(T, Q,\Phi)$ with phase space $Q$ (which
we shall often refer to as the `\idx{ambient manifold}') and evolution map
$\Phi$, then a manifold $M \subset Q$ is called invariant under the
system if it is mapped to itself under evolution. In the continuous
case this means that $\Phi^t(M) = M$ for all times $t \in \R$, that
is, any point $x \in M$ stays in $M$, so its complete orbit is
contained in $M$.

\begin{figure}[tb]
  \centering
  \parbox{9.5cm}{%
    \centering
    \input{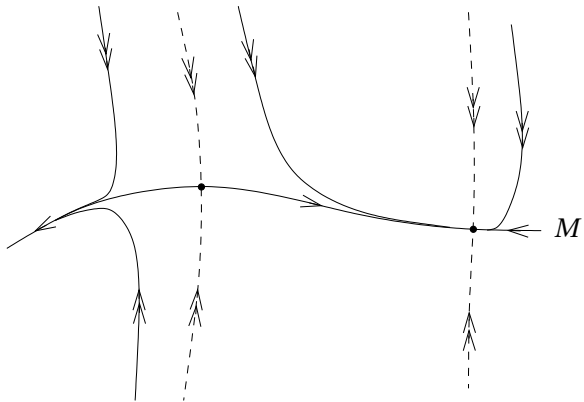}
    \caption{A normally hyperbolic invariant manifold. The single and
      double arrows indicate slow and fast flow respectively.}
    \label{fig:hyperb-NHIM}
  }
\end{figure}
An invariant manifold $M$ is then called normally hyperbolic if in
the normal directions, transverse to $M$, the linearization of the
flow $\Phi^t$ has a spectrum separate from the
imaginary axis again. Although the precise definition is a bit more
technical than in the case of a hyperbolic fixed point, the geometric
idea is the same. The normal directions must separate into directions
along which the linearized flow exponentially converges towards $M$
and directions along which it exponentially expands; no neutral
directions are allowed. Finally, the flow on $M$ itself
may expand or contract, but only at rates that are dominated by the
expansion and contraction in the normal directions.
Figure~\ref{fig:hyperb-NHIM} shows part of a normally hyperbolic
invariant manifold $M$ that has only stable normal directions. Note
that the dynamics on $M$ itself can be very complex; it can have fixed
points or even be chaotic. The only restriction is that the vertical
contraction rate is stronger than horizontal ones (and similarly for
expansion), as is indicated by the double and single arrows and
visible from the convergence of solution curves to the rightmost fixed
point on $M$.

\subsection{Persistence and (un)stable manifolds}
\index{normally hyperbolic invariant manifold!persistence}
\index{persistence|see{normally hyperbolic invariant manifold}}

There are two important properties that generalize from hyperbolic
fixed points to normally hyperbolic invariant manifolds. These are
persistence of the fixed point and the existence of stable and
unstable manifolds. The generalization of these properties is not a
trivial statement nor easily proven in the generalized case of NHIMs,
however.

Let us first focus on persistence. In case of a hyperbolic fixed
point, this is trivially stated and proven. If the fixed point $x$ is
hyperbolic, then it will persist as a nearby fixed point under small perturbations of the
vector field $v$ and stay hyperbolic. The proof is a direct
application of the implicit function theorem. If $\D v(x)$ has no
eigenvalues on the imaginary axis, then certainly it has no zero
eigenvalue, and therefore is a bijective linear map. So a
slightly perturbed vector field $\tilde{v}$ will again have a fixed
point $\tilde{x}$ nearby $x$ and the eigenvalues of
$\D \tilde{v}(\tilde{x})$ will be close to those of $\D v(x)$ if
$\tilde{v} - v$ is small in $C^1$\ndash norm. Hence the eigenvalues
are still separated by the imaginary axis. For a NHIM the situation is
similar but technically much more involved due to the fact that there
is no control on the behavior of solution curves in the
invariant manifold. A normally hyperbolic manifold $M$ does persist under $C^1$ small
perturbations and the perturbed manifold $\tilde{M}$ is again normally
hyperbolic and close to $M$ in a precise way. The most important
difference, however, is that $\tilde{M}$ generally has only limited
smoothness, even if $M$ and the system were smooth or analytic\footnote{%
  I do not know whether loss of smoothness is generic for NHIMs.
  See~\cite{Hasselblatt1994:reg-anosov-split-fol,%
    Hasselblatt1999:nonlipsch-anosov-fol} for the case of Anosov
  systems.%
}. This \idx{smoothness} is
dictated by the \emph{\idx{spectral gap condition}}, which is roughly the
ratio between the normal exponential expansion/contraction and the
exponential expansion/contraction tangential to $M$. This fact already
indicates that the proof of persistence of a NHIM cannot be a
straightforward application of the implicit function theorem.

The stable and unstable manifolds generalize as well. That is, a
normally hyperbolic invariant manifold $M$ has stable and unstable
manifolds $W^S(M)$ and $W^U(M)$ such that solution curves on these
converge exponentially fast towards $M$ in forward or backward time,
respectively. Their intersection is precisely $M$. But there is
actually more structure: these manifolds---we consider $W^S(M)$ but
everything is equivalent for $W^U(M)$---are fibrations of families of
stable and unstable fibers to each point $m \in M$,
\index{stable (unstable) fibration}
\begin{equation*}
  W^S(M) = \bigcup_{m \in M} W^S(m).
\end{equation*}
We should be a bit careful with this last statement, as points
$m \in M$ are generally not fixed points. These fibers $W^S(m)$ are
invariant in the sense that the flow commutes with the fiber
projection $\pi^S$:
\begin{equation*}
  \forall\; t \in \R,\, m \in M,\, x \in W^S(m)\colon
  \pi^S\circ\Phi^t(x) = \Phi^t\circ\pi^S(x).
\end{equation*}
In other words, each fiber is mapped into another single fiber under
the flow, namely the fiber over the flow-out of the base point $m$.
This important fact means that if we use the fibration for local
coordinates, then in these coordinates the horizontal, base flow
decouples from the vertical, fiber flow. This is sometimes also called
an \emph{isochronous} fibration~\cite{Guckenheimer1974:isochrons} as
all points in a fiber have the same
long-term behavior. Each single fiber is as
smooth as the system, but the dependence on the base point $m$, and
thus the smoothness of the fibrations as a whole, is generally not
better than continuous, see
Fenichel~\cite[Sec.~I.G]{Fenichel1973:asymstab1}. We do not
investigate these invariant fibrations in the present thesis, although
the mentioned results should hold for noncompact NHIMs as well.

\subsection{The relation to center manifolds}

Normally hyperbolic invariant manifolds bear a close resemblance to
center manifolds. Their spectral properties are roughly equivalent;
they differ in the fact that NHIMs have an intrinsically global
definition, while center manifolds are defined in local terms.

A \idx{center manifold} $W_\text{loc}^C(x)$ of a fixed point $x$ is a local
invariant manifold such that its tangent space at the fixed point is
the (generalized) eigenspace $E_0$ of the eigenvalues with real part
zero, that is,
\begin{equation}\label{eq:center-mfld}
  \T_x W_\text{loc}^C(x) = E_0.
\end{equation}
We can extend the definition of center manifold a bit by including all
eigenvalues $\lambda$ with real part bounded by
$\abs{\Re(\lambda)} \le \rho_0$. An associated generalized center
manifold consists of solutions that converge or diverge
from $x$ at an exponential rate bounded by $\rho_0$. Curves in the
strongly\footnote{%
  We remove the eigenvalues associated to $E_0$ from $E_\pm$ so that
  $E_-,\,E_0,\,E_+$ together disjointly span the total tangent space
  at $x$.%
} stable or unstable manifold converge or diverge at exponential rates
larger than $\pm \rho_0$, respectively. These conditions can directly
be compared to the description of NHIMs above, or
Definition~\ref{def:NHIM} (with $\rho_0 = \rho_M$).

If we take a look at Figure~\ref{fig:hyperb-NHIM} again, then we see
that both fixed points (indicated with a dot) on $M$ have
(generalized) center manifolds; $M$ itself is a center manifold for
these, but for the rightmost fixed point we can actually construct
the center manifold from any two solution curves converging to that
fixed point from the left and right. For example, the union of the two
curves drawn in the figure that converge to it could be taken as
alternative center manifold. This reflects the well-known fact that
center manifolds are generally not unique. This is the main difference
with the case of NHIMs: center manifolds are only defined in terms of growth rates
of solution curves locally with respect to one fixed point, while
NHIMs are globally invariant objects, where the spectral splitting
must hold everywhere along the invariant manifold. This difference is
effectively the reason that center manifolds are not uniquely defined,
while the perturbations of NHIMs are, see below. If we perturb the
system in Figure~\ref{fig:hyperb-NHIM} a bit, then the persistent NHIM
must everywhere be close to the original invariant manifold $M$. This
enforces uniqueness; in Figure~\ref{fig:hyperb-NHIM} this is clearly
visible: the alternative choice of center manifold to the rightmost
fixed point diverges far from $M$. See also the example in
Section~\ref{sec:spectral-gap}.

There is a subtle question of smoothness both for center manifolds and
NHIMs, related to the \emph{spectral gap condition}~\ref{eq:spectral-gap}.
Center manifolds are arbitrarily smooth in a sufficiently small
neighborhood of the fixed point $x$, but they are generally \emph{not}
$C^\infty$, even though they satisfy an infinite spectral gap. See Van
Strien's short note~\cite{VanStrien1979:centermfld-not-cinfty}. The
reason is that the size of the neighborhood may depend on the degree
of differentiability $C^k$. Persistent NHIMs generally have bounded
smoothness due a finite spectral gap; but even if they have an
infinite spectral gap, the smoothness of a persistent NHIM is
(generally) not $C^\infty$ for the same reasons. See
Remark~\ref{rem:gap-smoothness} and Example~\ref{exa:NHIM-non-Cinfty}.
\index{smoothness!non-$C^\infty$}

\section{Examples}

We present a few examples. The first detailed example serves to
show explicitly that smoothness of a persistent manifold depends
crucially on the spectral gap condition. The next examples motivate
the usefulness of a noncompact version of the theory of normal
hyperbolicity.

\subsection{The spectral gap condition}\label{sec:spectral-gap}

An invariant manifold is called an $\fracdiff$\ndash NHIM if the flow
contracts or expands at exponential rates along the normal directions,
and if these rates dominate any contraction or expansion
along tangential directions at least by a factor~$\fracdiff$. This
separation between growth rates along directions tangential and normal
to the NHIM is encoded in equations~\ref{eq:NHIM-rates}
and~\ref{eq:spectral-gap}.
\index{spectral gap condition}
\index{exponential growth rate}

Here we introduce a simple example
where the growth rates can be identified with eigenvalues $\lambda$ of
the linearization of the vector field at stationary points.
Furthermore, we consider the simplified case where only a stable
normal direction is present. That is, we consider a flow that
contracts in the normal direction at an exponential rate of at least
$\rho_\sy < 0$ and along the invariant manifold it contracts at most
at the rate $\rho_\sx$ with the simplified spectral gap condition
\begin{equation}\label{eq:spectral-gap-single}
  \rho_\sy < \fracdiff\,\rho_\sx
  \qquad\text{with}\quad \rho_\sx \le 0,\;\fracdiff \ge 1.
\end{equation}

The spectral gap
is fundamental to persistence of invariant manifolds: the compact
invariant manifolds that are persistent under any small perturbation
are precisely those that are normally hyperbolic\footnote{%
  The definition of normal hyperbolicity
  in~\cite{Mane1978:persistmflds} is a bit more general than the
  definition in this paper. That definition only requires a growth
  ratio $\fracdiff \ge 1$ along solution curves in the invariant
  manifold, and not as a ratio of global growth rates
  $\rho_\sx, \rho_\sy$, see also Remark~\ref{rem:NHIM-ratios}.%
}~\cite{Mane1978:persistmflds}. Ma\~n\'e only proved this inverse
implication for $1$\ndash normal hyperbolicity, the question is still
open for $\fracdiff$\ndash normal hyperbolicity with $\fracdiff > 1$.
A further property of normally hyperbolic invariant manifolds is that
the differentiability of
a slightly perturbed manifold depends not only on the smoothness of
the original manifold and the perturbed vector fields, but also on the
spectral gap. The spectral gap determines an upper bound
$1 \le \fracdiff < \infty$ on the smoothness of the perturbed system,
as $\fracdiff$ has to satisfy\footnote{%
  The case $\fracdiff = \infty$ would require $\rho_\sx > 0$; when
  $\rho_\sx = 0$, any finite order $\fracdiff$ can be obtained, but
  only for perturbations sufficiently small depending on $\fracdiff$.%
}~\ref{eq:spectral-gap}. This condition stems from the fact that when
the flow has exponential growth behavior $e^{\rho\,t}$, then higher
order derivatives will generally have growth behavior $e^{k\,\rho\,t}$
and the interval inclusion
$\intvCC{k\,\rho}{\rho} \subset \intvOO{\rho_\sy}{\rho_\sx}$ is required
to show existence and uniqueness of the $k$\th derivatives via a
contraction. The optimal differentiability degree $\fracdiff$ can be
extended to a real number by viewing $\alpha$\ndash \idxb{H\"older
continuity} as a fractional differentiability degree. That means that
the perturbed manifold can be shown to be $C^{k,\alpha}$ when
$\fracdiff = k + \alpha$ satisfies the spectral gap condition and the
system is $C^{k,\alpha}$ to start with.

The following example shows that this result is sharp. We construct a
very simple compact, normally hyperbolic invariant manifold, and then
show that an arbitrarily small perturbation yields a unique perturbed
$C^{k,\alpha}$ invariant manifold, where $\fracdiff = k + \alpha$ satisfies
$\rho_\sy = \fracdiff\,\rho_\sx$. This in fact precisely violates the
spectral gap condition, since that requires a strict inequality. The
example could be adapted to obtain a perturbed manifold with
smoothness no better than $C^{\fracdiff'}$ for some
$\fracdiff' < \fracdiff$, cf.\ws Example~\ref{exa:NHIM-non-Cinfty}.
A more qualitative exposition of this example
can also be found in~\cite[p.~198--200]{Fenichel1971:invarmflds}
and~\cite[p.~239,~251]{Hale1969:ODEs}.

\begin{example}[Optimal $C^{k,\alpha}$ \idx{smoothness} of persistent manifolds]
  \label{exa:opt-smooth}
Let the horizontal space $X = S^1$ be the circle and the vertical
space $Y = \R$. Take two points
$x_- = 0$ and $x_+ = \pi$ in $X$ and set the vector field $v$ to zero
at $(x_-,0), (x_+,0)$. We turn these stationary points into hyperbolic
fixed points, with $v$ linear in neighborhoods around them and
$\D v(x_-,0), \D v(x_+,0)$ having eigenvalues
$\lambda_- < 0 < \lambda_+$ along $X$, respectively, and one global
eigenvalue $\lambda_\sy < \lambda_-$ in the vertical direction along
$Y$, i.e.\ws $\dot{y} = v_y(x,y) = \lambda_\sy\,y$, see
Figure~\ref{fig:NHIM-opt-smoothness}. We extend the horizontal
component $v_x$ of the vector field $v$ to the whole space $X \times
Y$ in such a way that it is $C^\infty$, independent of $y$, and has no
critical points except for $x_-,\,x_+$. Hence, $M$ is an invariant
manifold for the flow $\Phi^t$ of~$v$.
\begin{figure}[b]
  \centering
  \input{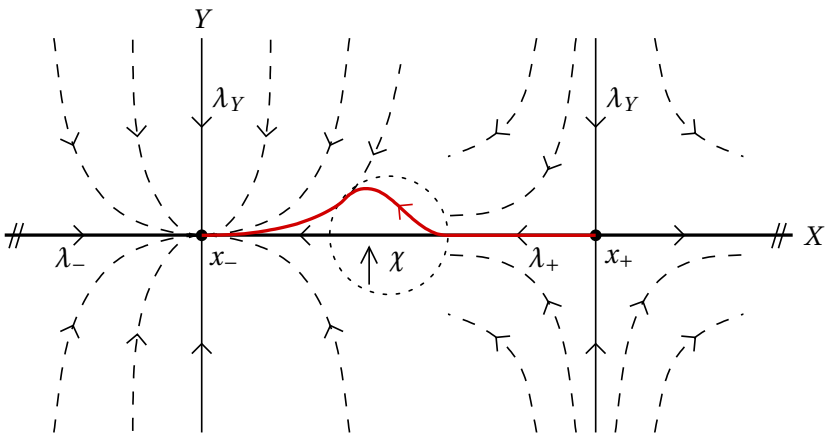}
  \caption{an example invariant manifold exhibiting $C^{k,\alpha}$
    smoothness under perturbation.}
  \label{fig:NHIM-opt-smoothness}
\end{figure}

First, we check that $M$ is normally hyperbolic. The long time
behavior of any point $m \in M$ is governed by its approach of the
stable fixed point $(x_-,0)$, except for $m = (x_+,0)$. For
$m = (x_+,0)$ we have $\D \Phi^t(m) = e^{\D v(m)\,t}$, hence
\begin{equation*}
  \begin{aligned}
    \D \Phi^t(m)|_{\T_m X} &= e^{\lambda_+\,t}, \\
    \D \Phi^t(m)|_{\T_m Y} &= e^{\lambda_\sy\,t}.
  \end{aligned}
\end{equation*}
More generally, consider a point $m \in M$ in the neighborhood of
either $(x_\pm,0)$ where $v$ is linear. Then $\D \Phi^t(m)$ is given
by
\begin{equation*}
  \D \Phi^t(m) =
  \begin{pmatrix}
    e^{\lambda_\pm\,t} & 0 \\
    0 & e^{\lambda_\sy\,t}
  \end{pmatrix}
\end{equation*}
for as long as $\Phi^t(m)$ stays in that neighborhood of $(x_\pm,0)$
where the vector field is linear. The transition time between these
two neighborhoods is finite as $v$ does not have zeros and the
transition map preserves vertical lines $\{ x \} \times Y$. The latter
fact is because $v_x$ is
independent of $y$, that is, we have also found the invariant, foliated
stable manifold of $M$. Gluing together these $\D \Phi^t$ maps on the
different domains, we see that the resulting tangent flow splits again
into independent horizontal and vertical parts, which can be estimated by
\begin{equation*}
  \begin{aligned}
    \forall t \le 0\colon \norm{\D \Phi^t(m)|_{\T_m X}} &\le C_\sx\,e^{\lambda_-\,t}, \\
    \forall t \ge 0\colon \norm{\D \Phi^t(m)|_{\T_m Y}} &\le C_\sy\,e^{\lambda_\sy\,t},
  \end{aligned}
\end{equation*}
where the constants $C_\sx,\,C_\sy$ are determined by the flow
$\Phi^t$ in the domain where $v$ is nonlinear. For any point $m$
close to $(x_-,0)$ this estimate is sharp, hence we expect maximal
smoothness $\fracdiff = \lambda_\sy / \lambda_-$ for a generic perturbation.

Next, we add a perturbation term $\epsilon\,\chi$ to the vector field
$v$, so we have a perturbed vector field
$\tilde{v} = v + \epsilon\,\chi$, where $\chi \in C^\infty_0$ is chosen with
support on a small ball intersecting $M$ away from the fixed points
and pointing upward. This will `lift' the invariant manifold as
indicated in Figure~\ref{fig:NHIM-opt-smoothness} for any
$\epsilon > 0$. Let $\tilde{M}$ denote this lifted manifold, that is,
$\tilde{M}$ is the image of the two heteroclinic solution curves that
run from $(x_+,0)$ to $(x_-,0)$ together with these fixed points. The
solution curve that runs to the left is lifted up from the $x$\ndash
axis after entering the region $\supp \chi$.

We first investigate two claims: that $\tilde{M}$ is invariant and that it is
the unique invariant manifold that is close to $M$. The invariance is
obvious; to the right of $x_+$ nothing has changed, so there $\tilde{M} = M$. To
the left of $x_+$ we follow the original unstable manifold, get
pushed up within the domain of support of $\chi$ and after leaving
that domain and entering the linear flow around $(x_-,0)$ we follow a
standard curve ending at $(x_-,0)$. This is a solution curve of
$\tilde{v} \in C^\infty$, hence invariant and even smooth. Now assume
there exists another invariant manifold $M'$ nearby and let
$(x',y') \in M' \setminus \tilde{M}$. The backward orbit
of the point $(x',y')$ must diverge to $\abs{y} \gg 1$. If $x' = x_-$,
then $y' \neq 0$ and this is clear. If $x' \neq x_-$, then the
backward orbit will end up at a point $(x,y)$ with $x$ close to $x_+$
and $y \neq 0$; since we are in the linear domain of $(x_+,0)$, this
orbit will then diverge (in reverse time) along the stable manifold
towards $\abs{y} \gg 1$. Hence, $M'$ is not close to $M$.

Next, we show that (for any $\epsilon > 0$) the perturbed manifold
$\tilde{M}$ is not more than $C^{k,\alpha}$ with $k+\alpha = \lambda_\sy / \lambda_-$,
even though the original and perturbed systems are $C^\infty$\ndash
smooth. To the left of $(x_-,0)$, $\tilde{M}$ is given by the graph of the zero
function from $X$ to $Y$ (as the continuation from $(x_+,0)$ to the right
along $X = S^1$). To the right of $(x_-,0)$, the solution curve is given by
$(x,y)(t) = (x_0\,e^{\lambda_-\,t}, y_0\,e^{\lambda_\sy\,t})$, hence
$y = C\,x^{\lambda_\sy/\lambda_-}$ where $C$ depends on $x_0, y_0$ only.
So we can write $\tilde{M}$ as the graph of the function
\begin{equation*}
  \tilde{h}\colon X \to Y\colon x \mapsto
  \begin{cases}
    0                            & \text{if}\;\; x \le 0, \\
    C\,x^{\lambda_\sy/\lambda_-} & \text{if}\;\; x > 0.
  \end{cases}
\end{equation*}
This function is exactly $C^{k,\alpha}$ for $k+\alpha = \fracdiff$ in
$x = 0$. Note that the loss of smoothness appears at a different place
than the perturbation of the vector field. The relevant fact is that
the different solution curves approaching the stable limit point have
finite differentiability with respect to each other, and this depends
on the horizontal and vertical rates of attraction at~$(x_-,0)$.
\end{example}

If we had assumed that $\rho_\sy = \rho_\sx$, that is, $\fracdiff=1$,
but with a non-strict inequality $\rho_\sy \le \fracdiff\,\rho_\sx$,
then normal hyperbolicity precisely fails and the invariant manifold
indeed need not persist. By the arguments above it can already be seen
that the persistent manifold can lose differentiability: when
$\fracdiff = 1$, the graph of the manifold will be given by
\begin{equation*}
  \tilde{h}(x) =
  \begin{cases}
    0    & \text{if}\;\; x \le 0, \\
    C\,x & \text{if}\;\; x > 0,
  \end{cases}
\end{equation*}
which is clearly non-differentiable at $x = x_- = 0$. We can extend
the example above to show that even more serious problems can occur.
\begin{example}[Non-persistence of non-NHIMs]
  \label{exa:break-nonNHIM}
  \index{persistence!non-}
  We consider Example~\ref{exa:opt-smooth} with $\rho_\sy = \rho_\sx$.
  If we perturb the system with a small circular vector field around
  $x_- = (0,0)$, then $\D v(0,0)$ will have two eigenvalues
  $\lambda_\sy \pm i\,\omega$ with $\lambda_\sy < 0$ and
  $\omega \in \R$ small. Thus, the solution curves that should make up
  the invariant manifold around $(0,0)$ will spiral in, which leads to
  the picture in Figure~\ref{fig:NHIM-breakdown}. Note that the curves
  wind around the origin infinitely often. At the origin this is not a
  manifold anymore, and cannot be described by a function
  $\tilde{h}\colon X \to Y$.
\end{example}
\begin{figure}[htb]
  \centering
  \input{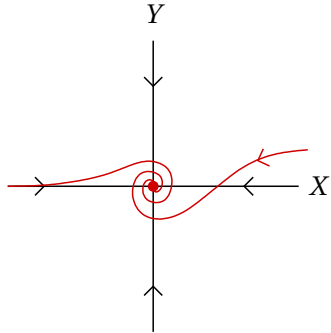}
  \caption{breakdown of a non-NHIM under a circular perturbation.}
  \label{fig:NHIM-breakdown}
\end{figure}

The idea to perturb around the stable fixed point $x_-$ also leads to
the following example.
\begin{example}[Non-$C^\infty$ persistence for $\fracdiff = \infty$ NHIMs]
  \label{exa:NHIM-non-Cinfty}
  \index{smoothness!non-$C^\infty$}
  We consider again Example~\ref{exa:opt-smooth}, but now with
  $\lambda_- = 0$. Then we have $\rho_\sx = 0$ and spectral gap
  $\fracdiff = \infty$. If we let $\lambda_- = \epsilon$ depend on the
  perturbation parameter $\epsilon > 0$, then this decreases the
  spectral gap condition\footnote{%
    The ratio $\fracdiff$ in the spectral gap is defined by a strict
    inequality, which we ignore here for simplicity of presentation.%
  } to a finite number $\fracdiff = \lambda_\sy/\lambda_-$. Even
  though $\fracdiff \to \infty$ as the perturbation size $\epsilon$
  goes to zero, we still have a finite spectral gap for any fixed
  perturbation. We conclude that the corresponding perturbed manifolds
  are not $C^\infty$, but have smoothness $C^\fracdiff$ where
  $\fracdiff$ can be made arbitrarily large by decreasing the
  perturbation size.
\end{example}

\subsection{Motivation for noncompact NHIMs}\label{sec:noncpt-motivation}
\index{noncompactness}

Most of the literature on normal hyperbolicity and its applications
treat compact NHIMs only. This excludes possibly interesting
applications. Settings where a noncompact, general geometric version
of normal hyperbolicity may be useful include chemical reaction
dynamics~\cite{Uzer2002:geom-reactdyn} and problems in classical and
celestial mechanics~\cite{Delshams2006:orbits-unbounded}.

We describe a two examples where noncompactness
naturally comes into play. The first example, a normally attracting
cylinder, is set in Euclidean space. This example could be complicated
a bit more by adding normal expanding directions to get a fully
normally hyperbolic system. Such situations show up in Hamiltonian or
reversible systems with invariant
tori~\cite{Broer2009:quasiper-stab-tori}. The second example is
set in ambient manifolds with nontrivial topology, thus motivating the
need for a theory of noncompact NHIMs in such a geometric setting.

Let us first treat a simple example.
\begin{example}[A normally attractive cylinder]
  \label{exa:cylinder}
  Let us consider the infinite cylinder $y^2 + z^2 = 1$ in $\R^3$.
  If we define a very simple dynamics by
  \begin{equation*}
    (\dot{x},\,\dot{r},\,\dot{\theta}) = (0,\,r(1-r),\,1)
  \end{equation*}
  in cylindrical coordinates, then the cylinder is normally attractive
  and the motion on the cylinder consists of only periodic orbits, see
  Figure~\ref{fig:NA-cylinder}.

  The dynamics on the cylinder is completely neutral, while it
  attracts in the normal direction with rate $-1$. Hence, there exists
  a unique persistent manifold diffeomorphic and close to the original
  cylinder. For any $k \ge 1$, the persistent manifold has $C^k$
  smoothness if the perturbation is chosen sufficiently small. The
  perturbed manifold must be uniformly close to the original cylinder;
  this rules out Example~\ref{exa:zero-inj-radius} of a cylinder with
  exponentially shrinking radius.

  The dynamics on the persistent manifold can be perturbed in
  arbitrary ways. It could slowly spiral towards $x$\ndash infinity,
  or develop attracting and repelling periodic orbits on the cylinder.
  If the cylinder were higher dimensional, it could even become
  chaotic.
\end{example}
\begin{figure}[htb]
  \centering
  \input{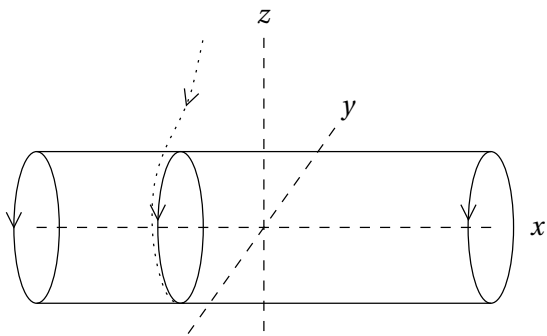}
  \caption{A normally attracting cylinder.}
  \label{fig:NA-cylinder}
\end{figure}

The second example actually motivated this work.
\index{nonholonomic system}
\begin{example}[Nonholonomic systems as singular perturbation limit]
  \label{exa:nonhol-reduction}
  Let a classical mechanical system be given by a smooth Riemannian
  manifold $(Q,g)$ as configuration space and a Lagrangian
  $L\colon \T Q \to \R$. The vector field $v$ on $\T Q$ is determined
  by the Lagrange equations of motion, given in local coordinates by
  \begin{equation}\label{eq:lagrange-EoM}
    \big[L\big]^i = \der{}{t}\pder{L}{\dot{x}_i} - \pder{L}{x_i} = 0.
  \end{equation}

  A nonholonomic constraint can be placed on such a system by
  specifying a distribution\footnote{%
    Here, a distribution is meant in the sense of differential
    geometry as a subbundle of the tangent bundle, not a generalized
    function (nor a probability distribution).%
  } $\mathcal{D} \subset \T Q$ and adding reaction forces to $[L]$
  according to the Lagrange--d'Alembert principle, that is, we require
  that a solution curve $\gamma$ satisfies
  \begin{equation}\label{eq:nonhol-EoM}
    \big[L\big](\gamma)(t) \in \mathcal{D}^0
    \quad\text{and}\quad
    \dot{\gamma}(t) \in \mathcal{D}
    \quad\text{for all $t \in \R$}
  \end{equation}
  where $\mathcal{D}^0 \subset \T^* Q$ denotes the annihilator of
  $\mathcal{D}$. This means that we
  restrict the velocities---but not the positions---of the system and
  adapt the vector field such that it preserves $\mathcal{D}$. Such
  constraints are called `nonholonomic' if the distribution
  $\mathcal{D}$ is not integrable. This means that some small
  positional changes can only be obtained through long orbits due
  to the constraints. The prototypical example is that parallel
  parking a car a small distance sideways requires repeated turning
  and moving forward and backward.

  As a concrete example of a nonholonomic system, let us consider a
  ball rolling on a flat surface. The possible positions of the ball
  are specified by $Q = SO(3) \times \R^2$, i.e.\ws orientation and
  position in the plane. If we enforce the constraint that the ball
  can only roll and not slip, then its linear velocity is determined
  by its angular velocity $\omega \in \mathfrak{so}(3)$, thus we have
  \begin{equation*}
    \mathfrak{so}(3) \times SO(3) \times \R^2
    \cong \mathcal{D} \subset \T\big(SO(3) \times \R^2\big).
  \end{equation*}

  The addition of the nonholonomic reaction forces specified by the
  Lagrange--d'Alembert principle can be argued for on physical
  grounds, and some experimental verification has been done by Lewis and
  Murray~\cite{Lewis1995:var-constrsys-th-exp} to check its
  correctness against the alternative vakonomic principle. Still, it
  would be nice to rigorously derive these forces from fundamental
  principles; this would complement~\cite{Rubin1957:motion-constr,%
    Takens1980:motion-constr,Kozlov1990:real-hol-constr} which
  showed this for holonomic constraints. The nonholonomically
  constrained system can be obtained from the unconstrained system by
  adding friction forces, see~\cite{Karapetian1981:realnonhol,%
    Brendelev1981:realnonhol,Kozlov1992:realconstr}. Heuristically,
  one could say that if a rolling ball feels a strong contact friction
  force, then if this force is taken to infinity, it suppresses all
  slipping. This can be viewed as a singular perturbation limit, where
  $\mathcal{D}$ precisely is the invariant manifold, and it is
  normally attracting due to the dissipative friction force.

  The cited works prove this result, but only asymptotically on finite
  time intervals. The extension of the theory of NHIMs to noncompact
  manifolds as developed in this thesis can be applied here. It allows
  one to improve upon this result and make it exact on infinite time
  intervals and general noncompact configuration spaces $Q$, as long
  as these satisfy the `bounded geometry' condition. One could think,
  for example, of a gently sloping surface and a ball that is not
  perfectly round, or even a time-dependent perturbation, as long as
  it is uniformly bounded in time.
\end{example}

\section{Historical overview}

As already mentioned, the theory of normally hyperbolic invariant
manifolds is a generalization of the theory of hyperbolic fixed points. The study
of these dates back to the beginning of the $20$th century, or even
the end of the $19$th century. From 1892 onwards, Poincar\'e published
his works ``Les m\'ethodes nouvelles de la m\'ecanique c\'eleste''%
~\cite{Poincare1892:mech-celeste1}, in
which he founded the theory of dynamical systems and famously studied
the three-body problem. This triggered further research in nonlinear
dynamical systems and persistence questions. Another important work
published in the same year is ``The general problem of the stability
of motion'' by Lyapunov; the original is in Russian, but translations
in French~\cite{Lyapunov1892:stab-motion-french1907} and
English~\cite{Lyapunov1892:stab-motion-english1992} are available. In
this work, he introduced the concept of characteristic numbers,
nowadays called `Lyapunov exponents', to study `conditional stability'
\idxalso{Lyapunov exponent}{exponential growth rate}
of nonlinear differential equations at a fixed point. Conditional
stability corresponds to the existence of stable (and unstable)
linearized directions and Lyapunov proves the existence of a stable manifold
by means of a series expansion under the assumption that the system is
analytic.

In the beginning of the $20$th century, the problem of stable
manifolds was studied, without assuming analyticity, by
Hadamard~\cite{Hadamard1901:iterasymp} and
Cotton~\cite{Cotton1911:asympeqdiff}. Both Frenchmen applied
different methods to obtain the stable and unstable manifolds of a
hyperbolic fixed point. Later, the German mathematician Perron
extended the ideas of Cotton to allow for generic complex eigenvalues,
possibly of higher multiplicity, as long as the real parts of the
eigenvalues are separated by zero (or even a number $r \neq 0$),
see~\cite{Perron1929:stabasympdiff,Perron1930:stabdiffeq}. Hadamard's
method is now named after him, and also known as the `graph transform'.
The other method was first formulated by Cotton, although the idea of
exponential growth of solution curves can be traced to Lyapunov. This
method is commonly referred to as the Perron or Lyapunov--Perron
\index{Perron method}\index{Lyapunov--Perron method|see{Perron method}}
method in the literature. This seems to pay too little credit to
Cotton, even though Perron himself~\cite{Perron1929:stabasympdiff}
does attribute the method to Cotton\footnote{%
  These facts were pointed out to me by Duistermaat.%
}.

From around 1960, renewed activity in the area of hyperbolic dynamics led to the
generalization of the theory of (un)stable manifolds for hyperbolic
fixed points to persistence and (un)stable fibrations for normally
hyperbolic invariant manifolds. Many authors have contributed to this
subject, culminating in the seventies in the works by
Fenichel~\cite{Fenichel1971:invarmflds} and Hirsch, Pugh, and
Shub~\cite{Hirsch1977:invarmflds}. These two works formulate the
theory slightly differently, but in broad generality and can be viewed
as the basic references nowadays; references to earlier works can be
found in both. Both Fenichel and Hirsch, Pugh, and Shub use Hadamard's
graph transform as their fundamental tool. In these works, compactness
of the invariant manifold is a basic assumption. Noncompact, immersed
manifolds are considered in~\cite[Section~6]{Hirsch1977:invarmflds},
albeit under the assumption that the immersion image is compact again.

The theory of normal hyperbolicity has seen some interesting
developments since these foundational works, and the applications
have slowly started to flourish, see~\cite{Wiggins1994:invarmflds} for
a list of subjects. A major development was the generalization to
semi-flows in Banach spaces. This situation can arise when one wants
to study partial differential equations as ordinary differential
equations on appropriate function spaces. This technique has been
applied to PDEs such as the Navier--Stokes or reaction-diffusion
equations.

\index{comparison!of results}
In his book on parabolic PDEs, Henry extended the Perron method to
apply to semi-flows with a NHIM given as the horizontal
submanifold\footnote{%
  Henry actually has reversed notation where the `vertical' manifold
  $Y \times \{0\}$ is the NHIM.} %
$X \times \{0\}$ in a product $X \times Y$ of Banach
spaces~\cite[Chap.~9]{Henry1981:geom-semilin-parab-PDEs}. Henry's idea is
to linearize only the normal directions, but keep the horizontal flow
along $M$ in its general, nonlinear form, while at the same time
splitting the Perron contraction map into a two-stage contraction
map on horizontal and vertical curves separately. Henry obtains
$C^{1,\alpha}$ smoothness only. In the series of
papers~\cite{Bates1998:invarmflds-semiflows,Bates1999:persist-overflow,Bates2008:approx-invarmfld},
Bates, Lu, and Zeng study more general NHIMs of semi-flows in
Banach spaces. They employ Hadamard's graph transform and allow
so-called `overflowing invariant manifolds', as in~\cite{Fenichel1971:invarmflds}.
They also allow the NHIM to be noncompact and an immersed instead of
an embedded submanifold. In~\cite{Bates1999:persist-overflow} the
unperturbed NHIM is assumed to be $C^2$ to obtain $C^1$ persistence
results, for the technical reason of constructing $C^1$ normal bundle
coordinates. In their later paper~\cite{Bates2008:approx-invarmfld},
this technicality is overcome\footnote{%
  Their Hypothesis (H2) that a certain approximate splitting
  like~\ref{eq:NHIM-split} ``does not twist too much'', can be
  obtained from uniform Lipschitz continuity of the tangent spaces of
  the invariant manifold. I am not sure if this is a significantly
  weaker hypothesis. See also the discussion in
  Remark~\ref{rem:loss-smoothness}.%
}, and existence of a NHIM is even proven when sufficiently close,
approximately normally hyperbolic invariant manifolds exist; the
persistence result is then obtained for compact NHIMs only, though.

Vanderbauwhede and Van
Gils~\cite{Vanderbauwhede1987:centermflds,Vanderbauwhede1989:centremflds-normal}
introduced the technique of considering a scale (family) of Banach spaces of
curves with exponential growth, and using the fiber contraction theorem (see
Appendix~\ref{chap:fibercontr}), proved smoothness of center
manifolds with the Perron method. Although not the same, center
manifolds have many properties in common with NHIMs and
Sakamoto~\cite{Sakamoto1990:invarmlfds-singpert} has built upon the
works of Henry and Vanderbauwhede and Van Gils to prove persistence
and $C^{k-1}$ smoothness for singularly perturbed systems in a
finite-dimensional $\R^m \times \R^n$ product space setting. The loss
of one degree of smoothness is again due to the construction of
normal bundle coordinates, although this fact is obscured by the
explicit $\R^m \times \R^n$ setting.

Singularly perturbed, or, slow-fast systems are another important
class of applications. These describe systems where the dynamics
is governed by multiple, separate time scales, or when a system can be
viewed as an approximation of an idealized, restricted system.
Singularly perturbed systems can be studied using the theory of normal
hyperbolicity by turning them into a regular perturbation problem via a
rescaling of time, see foundational work by
Fenichel~\cite{Fenichel1979:singpertODE} or the more introductory
expositions~\cite{Jones1995:geomsingpert,Kaper1999:intro-singpert,%
  Verhulst2005:methappl-singpert}.

\section{Comparison of methods}\label{sec:comparison}
\index{comparison!of methods}

There are two well-known methods for proving the existence and
smoothness of invariant manifolds in hyperbolic-type dynamical systems. The
Hadamard graph transform and the variation of constants method, also
known as the (Lyapunov--)Perron method. Variations of both have been
applied in many situations with some form of hyperbolic dynamics. This
ranges from the relatively simple problem of finding the stable and
unstable manifolds of a hyperbolic fixed point, to center manifolds,
partially hyperbolic systems, and normally hyperbolic systems. The
quote of Anosov~\cite[p.~23]{Anosov1969:geodflow} that ``every five
years or so, if not more often, someone `discovers' the theorem of
Hadamard and Perron, proving it either by Hadamard's method of proof
or by Perron's'' is nowadays probably familiar to many researchers in
these areas; it illustrates the pervasiveness of these methods.

In this section, I describe the ideas that are common to both
methods, as well as their differences. I hope to elucidate the merits
and weak points of both methods, especially when applied to normally
hyperbolic systems. Basically they seem to be able to produce the same
conclusions, but each method takes a different viewpoint to the
problem.

Let us first identify some basic common ideas. As a sample problem, we
consider finding the invariant unstable manifold $W^U$ of a
\index{stable (unstable) manifold}\idx{hyperbolic fixed point},
positioned at the origin of $\R^n$. The system is defined
by either a diffeomorphism $\Phi$ in the discrete case, or a flow
$\Phi^t$ in the continuous case. Both methods use the splitting of the
tangent space into stable and unstable directions:
\begin{equation*}
  \T_0 \R^n \cong \R^n = U \oplus S.
\end{equation*}
Let $(x^+, x^-)$ denote coordinates in $U \oplus S$ according to
projections $\pi^+,\,\pi^-$ from $\R^n$ onto the unstable and stable
directions $U$ and $S$, respectively. We shall use the notation
$\Phi_\pm = \pi^\pm \circ \Phi$.

\subsection{Hadamard's graph transform}

The \idx{graph transform} is due to Hadamard. His
paper~\cite{Hadamard1901:iterasymp} (in French, 4~pages) can be used
as a concise and basic introduction to the graph transform, applied to
the stable and unstable manifolds of a hyperbolic fixed point. He
does not prove smoothness or even continuity of these invariant
manifolds, although continuity could easily be concluded by
introducing the Banach space of bounded continuous functions with
supremum norm.

The basic idea of the graph transform is to view the unstable manifold
$W^U$ as the graph of a function $g\colon U \to S$. The graph, as a
set, is invariant under $\Phi$ (or e.g.\ws $\Phi^1$ in the continuous
case). The diffeomorphism $\Phi$ can also be interpreted as a map
acting on functions $g$ through its action on their graphs. This
induces a mapping
\begin{equation}\label{eq:graph-transform}
  T\colon g \mapsto \tilde{g}
  \qquad\text{implicitly defined by}\qquad
  \tilde{g}\big(\Phi_+(x,g(x))\big) = \Phi_-(x,g(x)).
\end{equation}
Thus, by definition, any point $(x,g(x))$ on the graph of $g$ gets
mapped to a point $(x',\tilde{g}(x'))$ on $\Graph(\tilde{g})$.
The map $T$ turns out to be well-defined and a contraction on
functions $U \to S$ that are sufficiently small in Lipschitz norm. The
graph of the unique fixed point $g^\star$ of $T$ must correspond to
the unstable manifold, that is, $W^U = \Graph(g^\star)$.

By considering the invariant sets, this method focuses on the geometry
of the problem. The method uses a diffeomorphism map $\Phi$; the
continuous case can be studied by considering the flow map $\Phi^t$
for a fixed time $t$.
The diffeomorphism can easily be studied locally in charts on a
manifold. Therefore this method lends itself well to the generalized
setting of normally hyperbolic invariant manifolds, where the
invariant manifold is intrinsically a global object. Even if this
global object is nontrivial, it can still be studied in local charts.

\subsection{Perron's variation of constants method}\label{sec:perron-method}
\index{Perron method}

This method is commonly referred to as the Perron or Lyapunov--Perron
method. Although in the literature this is attributed to
Perron~\cite{Perron1929:stabasympdiff}, he in turn cites
Cotton~\cite{Cotton1911:asympeqdiff} for the main idea.

This method focuses on the behavior of solution curves. The solutions
on the unstable manifold are precisely characterized by the fact that
they stay bounded under backward evolution. In the following, we
explain the Perron method for the continuous case\footnote{%
  Contrary to the graph transform (which is only intrinsically defined
  for mappings), the Perron method can be formulated both for flows
  and discrete mappings. For the discrete case, the integral must be
  replaced by a sum, the
  mapping $\Phi$ must be split into a linear and nonlinear part, and
  the linearized flow must be replaced by iterates of the linearized
  mapping. See for example~\cite{Aulbach2002:smooth-invarfib,
    Poetzsche2004:smooth-invarfib}.%
}. We adopt the notation from the graph transform setting. A
\idx{contraction operator} \idx{$T$} is constructed via a variation of
constants integral. The nonlinear part of the vector field is viewed
as a perturbation of the linear part. The integral equation is split
into the components along the stable and unstable directions. Then the
integration of the unstable component is switched from the interval
$\intvCC{0}{t}$ to $\intvCC{-\infty}{t}$, and only bounded functions
are considered. Writing the vector field
$v(x) = \D v(0)\cdot x + f(x)$ in linearized form with nonlinearity
$f$, this leads to the following contraction operator on curves
$x = (x^+,x^-) \in C^0(\intvCC{-\infty}{0};\R^n)$:
\begin{equation}\label{eq:perron-map}
  \begin{aligned}
  T\colon \big(x^+(t),x^-(t)\big) \mapsto
    \Big(x_0^+
      - &\int_t^0         \D\Phi_+^{t-\tau}(0)\;f_+\big(x^-(\tau),x^+(\tau)\big) \d\tau\;,\\
  &\!\!\!\int_{-\infty}^t \D\Phi_-^{t-\tau}(0)\;f_-\big(x^-(\tau),x^+(\tau)\big) \d\tau\,
    \Big).
  \end{aligned}
\end{equation}
This mapping $T$ is well-defined and a contraction on curves
$x \in C^0(\intvCC{-\infty}{0};\R^n)$ whose stable component $x^-$
is bounded and sufficiently small. Note that $T$ does not depend on
the stable component $x_0^-$ of the initial conditions anymore. The
fixed point of $T$ is a solution curve on $W^U$ with $x_0^+$ given as
a parameter. The unstable manifold is described, finally, by
evaluating the stable component at zero, leading to a graph
\begin{equation*}
  g\colon U \to S\colon x_0^+ \mapsto x^-(0).
\end{equation*}

First of all, it must be noted that this method requires $f$ to be
small in $C^1$\ndash norm. We can make $f$ small by restricting to a
sufficiently small neighborhood of the origin and cutting off $f$
outside of it. This cut-off does not influence the results: due to the
boundedness condition, curves $x$ stay in the neighborhood. The
method can be generalized to a separation of stable and unstable
spectra (i.e.\ws a dichotomy) away from the imaginary axis\footnote{%
  This is for the continuous case. The imaginary axis of the spectrum
  of a vector field corresponds (via the exponential map) to the unit
  circle for the spectrum of a diffeomorphism in the discrete case.%
}, and for example be applied to show existence of center
manifolds. In that case, uniqueness is lost as solutions will
generally run out of small neighborhoods. This makes the Perron method
not directly applicable to normally hyperbolic invariant manifolds.
The center direction corresponds to the invariant manifold, but
solution curves are global objects that cannot be treated locally.

The Perron method can be extended to overcome this problem.
Henry~\cite[Chap.~9]{Henry1981:geom-semilin-parab-PDEs}
linearizes the vector field only in the normal directions of the
invariant manifold. Henry uses a two-step contraction scheme, but this
can be reduced to a single contraction $T = T_- \circ T_+$ that is a
composition of two maps. The maps $T_\pm$ are essentially the
components of~\ref{eq:perron-map}. Still, the results obtained are not
quite as general as those obtained with the graph transform. For the
graph transform, the condition of normal hyperbolicity can be
formulated in terms of the ratio of the normal and tangential growth
rates of the flow along orbits, while for the Perron method it must be
formulated in terms of the ratio of global growth rates. This less
general assumption is required because the contraction
operator~\ref{eq:perron-map} is studied on spaces of solution curves
with a fixed exponential growth behavior, see
Definition~\ref{def:rho-norm}.

Explicit time dependence can be added to the Perron method with only
trivial modifications. This allows one to study hyperbolic fixed
points in non-autonomous systems\footnote{%
  The term `fixed point' in the context of a non-autonomous system is
  not definable in a coordinate-free way: any orbit of the
  system can be made into a fixed point under a suitable
  time-dependent coordinate transformation. However, there may be a
  preferred ``time-independent'' coordinate system. Moreover, the
  hyperbolicity of an orbit with respect an intrinsic metric is
  independent of a choice of coordinates.%
}. An application is the study of
invariant fibrations of, for example, normally hyperbolic invariant
manifolds. These have fibered stable and unstable manifolds. Points in
a single fiber are characterized by the unique orbit on the normally
hyperbolic invariant manifold they are exponentially attracted to
under forward or backward evolution, respectively. Finding these
fibers is turned into a non-autonomous hyperbolic fixed point problem
by following a point on the invariant manifold.

\subsection{Smoothness}\label{sec:comp-smooth}
\index{smoothness}

In the truly hyperbolic case---when the stable and unstable spectra
are separated by a neighborhood of the imaginary axis---the Perron
method allows for a direct proof of smoothness of the manifolds $W^U$
and $W^S$, see~\cite{Irwin1970:stabmfld-thm,Irwin1972:smooth-comp-map}
where this is formulated for discrete systems.
One first verifies that the contraction operator $T$ is as
smooth as the system, still acting on continuous curves $x$. Then, by
an implicit function theorem argument, the fixed point depends
smoothly on the (partial) initial value parameter $x_0^+$. To the
best of my knowledge, there is no similarly simple approach
for the graph transform. The contraction map acts directly on graphs
$g$, so to obtain smoothness, one must consider the maps
$g \in C^k(U;S)$. A direct estimate of contractivity in $C^k$\ndash
norm requires higher than $k$\th order Lipschitz estimates on the
system.

When the spectra are not separated by the imaginary axis---this occurs
for example in normally hyperbolic systems---things become more
complicated. The spectral gap condition defines an intrinsic upper
bound for the smoothness that one can generically expect for a system, as
was seen in Example~\ref{exa:opt-smooth}. Both methods apply induction
over the smoothness degree in their proof. Formal derivatives of the
contraction map $T$ are constructed. These are again contractions, but
now on higher derivatives of the fixed point mapping, while fixing the
derivatives below. Finally, the fiber contraction theorem (see
Appendix~\ref{chap:fibercontr}) can be used to conclude that these
higher order derivatives converge to a fixed point, jointly with all
lower orders.
\index{higher order derivative}

Explicit calculation of higher derivatives of $T$ is very tedious; one
should focus on their form as dictated by
Proposition~\ref{prop:compfunc-deriv}.
For the graph transform, the relevant terms that one obtains
from~\ref{eq:graph-transform} are, ignoring arguments,
\begin{equation*}
    \D^k \tilde{g} \cdot \big(\D_1\Phi_+ + \D_2\Phi_+\,\D g\big)^k + \ldots
  = \D_2\Phi_- \cdot \D^k g + \ldots
\end{equation*}
This leads to a contraction when
$\norm{\D_2\Phi_-}\!\cdot\!\norm{\D_1\Phi_+^{-1}}^k < 1$. The limit on
$k$ precisely corresponds to the spectral gap condition, at least when
we replace $\Phi$ by a sufficiently high iterate $\Phi^N$ of itself,
or in the continuous case, if we take the flow map $\Phi^t$ at a
sufficiently large time $t$.

For the Perron method, the essential form of the derivatives of $T$ is
\begin{equation}\label{eq:deriv-perron-map}
  \D^k T(x)\big(\dx_1,\ldots,\dx_k\big)(t)
  = \int \D\Phi^{t-\tau}(0)\cdot
         \D^k f(x(\tau))\big(\dx_1(\tau),\ldots,\dx_k(\tau)\big) \d\tau.
\end{equation}
The solution curve $x$ as well as its variations $\delta x_i$ are of
growth order $e^{\rho\,t}$, so the variation of $f$ in the integrand
is of growth order $e^{k\,\rho\,t}$, even if $\D^k f$ itself is
bounded. This means that $k$\th order variations must be considered in
spaces of growth order $e^{k\,\rho\,t}$ and $\D^k T$ is only
contractive on such spaces if both $\rho$ and $k\,\rho$ are contained
in the spectral gap.

\section{Bounded geometry}

The main results of this thesis are formulated in a geometric context
on differentiable manifolds. Already
in~\cite{Fenichel1971:invarmflds,Hirsch1977:invarmflds} the results
are formulated in such a context. This allows for more general
situations than choosing $\R^n$ as ambient space. In the compact case,
it does not require a change in the basic proofs (as can be seen from
the approach taken in~\cite{Fenichel1971:invarmflds}), but it does
bring in some additional formalism. It turns out that if one switches
to a noncompact setting in manifolds, then a fundamental new idea must
be added. First, a choice of Riemannian metric (or possibly a weaker
form: a Finsler structure) is required since not all metrics are
equivalent anymore on a noncompact manifold, see Example~\ref{exa:non-equiv-metrics}.
As an extension, Example~\ref{exa:pertsize-metric-dep} shows that one
cannot reduce the noncompact to a compact case by compactification.
\index{noncompactness} Secondly, the \idxb{ambient manifold} and
functions on it should satisfy uniformity criteria that can be
captured in terms of `\idx{bounded geometry}'\footnote{%
  We do not claim that bounded geometry is a necessary condition to
  generalize the theory of normal hyperbolicity to noncompact ambient
  spaces, only that it is sufficient. Section~\ref{sec:cpt-unif} does
  contain some examples, though, that indicate that some form of
  bounded geometry is necessary.%
}. For full details see Section~\ref{sec:cpt-unif} on compactness and
uniformity and Chapter~\ref{chap:boundgeom} on bounded geometry. Let
us just give a quick overview here.

A Riemannian manifold has bounded geometry, loosely speaking, if it is
globally, uniformly well-behaved. More precisely, its curvature must
be bounded and the injectivity radius must be bounded away from zero,
see Definition~\ref{def:bound-geom}.
Then there exists a preferred set of so-called normal coordinate charts
for which coordinate transition maps are uniformly continuous and
bounded, smooth functions. That is, in $k$\th order bounded geometry
we have a $C^k$ uniform atlas. As a consequence, uniformly continuous and
bounded submanifolds, vector fields, and other objects can be defined
and manipulated in a natural way in terms of these coordinates. Note
that $\R^n$ and compact manifolds have bounded geometry, see
Example~\ref{exa:BG-manifolds}. Together with
corollaries~\ref{cor:persistNHIMcpt} and~\ref{cor:persistNHIM-Rn} of
the main theorem, this shows that bounded geometry provides a natural
generalization to the known settings of compact and Euclidean spaces.

We use bounded geometry to obtain boundedness estimates on
holonomy, see Section~\ref{sec:BG-curv-holonomy}.
This is a fundamental ingredient in our proof of smoothness of
the perturbed manifold. Finally, we present more technical results in
bounded geometry: a uniform tubular neighborhood, uniform smoothing of submanifolds, and a
trivializing embedding of the normal bundle. We use these to reduce
the full problem of persistence of a normally hyperbolic submanifold
$M$ in an ambient manifold $Q$ to the trivialized situation
$X \times Y$, where $M$ is represented by the graph of a small
function $h\colon X \to Y$ and $Y$ is a vector space. Uniformity
permeates all these constructions in order to obtain uniform estimates
required for the persistence proof in the trivialized setting.

\section{Problem statement and results}\label{sec:prob-statement}

The main problem in this thesis is the persistence of
normally hyperbolic invariant manifolds under small perturbations of
the dynamical system. That is, given a flow $\Phi^t$ defined by some
vector field $v$ and a normally hyperbolic invariant submanifold $M$,
we want to show that for any vector field $\tilde{v}$ sufficiently
close to $v$, there exists a unique manifold $\tilde{M}$ close to $M$
that is invariant under the flow of $\tilde{v}$; moreover we'd like to
show that $\tilde{M}$ is normally hyperbolic again. To
make this statement precise, we need to define a lot of things: first
of all, we need to rigorously define normal hyperbolicity. Secondly,
the statements about vector fields and manifolds being `close' need to
be formalized and finally, we need to specify the ambient space $Q$ on
which the system is defined.

We start with a Riemannian manifold $(Q,g)$ as ambient space and a
submanifold $M$. For technical reasons this manifold is assumed to be
complete and of bounded geometry (or at least in a $\delta > 0$
neighborhood of $M$, since the whole analysis can be restricted to
such a neighborhood). Basically, these conditions impose uniformity of
the space, and fit in the principle of replacing compactness by
uniform estimates, see Section~\ref{sec:cpt-unif} and
Chapter~\ref{chap:boundgeom} for more details. Note that $Q = \R^n$
with the standard Euclidean metric is an easy (and typical) special
case.

Let $v \in \vf(Q)$ be a vector field on $Q$ with
$v \in \BUC^{k,\alpha}$, that is, $v$ up to its $k$\th derivative is
uniformly continuous and bounded, and $\alpha$\ndash H\"older
continuous if $\alpha \neq 0$. On $\R^n$ these statements make
immediate sense; on general manifolds $Q$, results from
Chapter~\ref{chap:boundgeom} are required, in particular
Definition~\ref{def:unif-bounded-map}, to make sense of uniform
boundedness and continuity by means of normal coordinates. Let
$\tilde{v}$ be another such vector field. The closeness of $v$ and
$\tilde{v}$ will be measured using supremum norms. The $C^1$\ndash
norm is required to be small for the persistence result. Thus, even
though we consider the space of $C^{k,\alpha}$ bounded vector fields,
we endow this space with a $C^1$ topology. See
Section~\ref{sec:induced-topo} for some more remarks on this topology
and a comparison with standard topologies on noncompact function
spaces. If we assume that $\tilde{v} - v$ is small in
$C^{k,\alpha}$\ndash norm as well, then $\tilde{M}$ will be
$C^{k,\alpha}$\ndash close\footnote{%
  We actually only obtain $C^k$ closeness for integer
  $k \le \fracdiff - 1$ where $\fracdiff$ is the ratio in the spectral
  gap condition~\ref{def:r-NHIM}. This is probably an artifact of the
  techniques we used, while $C^{k,\alpha}$ closeness with
  $k + \alpha = \fracdiff$ should be obtainable.%
} to $M$. These $C^1$ and $C^k$ norm requirements and results are
direct analogues of those in the implicit function theorem.

Finally, we define normal hyperbolicity of a submanifold $M$ with
respect to a continuous dynamical system $(\R,\,Q,\,\Phi)$. The flow
$\Phi^t$ should have a domain of definition containing at least a
neighborhood of the invariant manifold $M$. This definition is easily
adapted to the discrete case of a diffeomorphism $\Phi\colon Q \to Q$;
simply replace $t \in \R$ by $t \in \Z$ as iterated powers of $\Phi$.
\index{dynamical system!discrete}
\begin{definition}[Normally hyperbolic invariant manifold]
  \label{def:NHIM}
  \index{normally hyperbolic invariant manifold!definition}
  \index{normal hyperbolicity}
  \index{exponential growth rate}
  Let $(Q,g)$ be a smooth Riemannian manifold,
  $\Phi^t \in C^{\fracdiff \ge 1}$ a flow on $Q$, and let
  $M \in C^{\fracdiff \ge 1}$ be a submanifold of\/ $Q$. Then $M$ is
  called a normally hyperbolic invariant manifold of the system
  $(Q,\Phi^t)$ if all of the following conditions hold true:
  \begin{enumerate}
  \item $M$ is invariant, i.e.\ws $\forall\; t \in \R\colon \Phi^t(M) = M$;
  \item there exists a continuous splitting
    \begin{equation}\label{eq:NHIM-split}
      \T_M Q = \T M \oplus E^+ \oplus E^-
    \end{equation}
    of the tangent bundle\/ $\T Q$ over $M$ with globally bounded,
    continuous projections $\pi_M,\,\pi_+,\,\pi_-$ and this splitting
    is invariant under the tangent flow
    $\D\Phi^t = \D\Phi_M^t \oplus \D\Phi_+^t \oplus \D\Phi_-^t$;
  \item there exist real numbers
    $\rho_- < -\rho_M \le 0 \le \rho_M < \rho_+$ and $C_M,C_+,C_- > 0$ such that
    the following exponential growth conditions hold on the various
    subbundles:
    \begin{equation}\label{eq:NHIM-rates}
      \begin{alignedat}{2}
      &\forall\; t\in \R,\,(m,x) \in \T M&&\colon\quad
      \norm{\D\Phi_M^t(m)\,x} \le C_M\,e^{\rho_M\,\abs{t}}\,\norm{x},\\[5pt]
      &\forall\; t\le 0 ,\,(m,x) \in E^+ &&\colon\quad
      \norm{\D\Phi_+^t(m)\,x} \le C_+\,e^{\rho_+\,t}\,\norm{x},\\[5pt]
      &\forall\; t\ge 0 ,\,(m,x) \in E^- &&\colon\quad
      \norm{\D\Phi_-^t(m)\,x} \le C_-\,e^{\rho_-\,t}\,\norm{x}.
      \end{alignedat}
    \end{equation}
\end{enumerate}
\end{definition}
These exponential estimates imply that the tangent flow $\D\Phi^t$
must contract at a rate of at least $\rho_-$ along the stable
complementary bundle $E^-$, expand\footnote{%
  Note that expansion along $E^+$ could also be formulated as
  $\norm{\D\Phi^t(m)\,x} \ge C_+\,e^{\rho_+\,t}\,\norm{x}$ for $t \ge 0$
  and $(m,x) \in E^+$. This is equivalent to the condition as stated,
  which says that there is contraction for $t \le 0$, that is, in
  backward time. This latter formulation is preferable because it is
  the form required in estimates.} %
as $e^{\rho_+\,t}$ along the unstable bundle $E^+$, and may not
expand or contract at a rate faster than $\pm\,\rho_M$,
respectively, tangent along~$\T M$.

\begin{remark}\NoEndMark
  We added the condition that the projections $\pi_M,\,\pi_+,\,\pi_-$
  are globally bounded. This is a natural extension to the noncompact
  case, and is automatically satisfied in case $M$ is compact.
\end{remark}

\begin{remark}\label{rem:NHIM-ratios}
  This definition of normal hyperbolicity is not as general as could
  be. Fenichel~\cite[p.~200--204]{Fenichel1971:invarmflds} defines
  normal hyperbolicity in terms of `generalized Lyapunov type
  numbers'. It follows from his uniformity lemma that these are
  essentially exponentiated versions of our \idx{Lyapunov exponent}s $\rho$.
  \idxalso{$\rho$}{exponential growth rate}
  For example, his $\nu$ is equivalent to our $e^{-\rho_+}$. But
  Fenichel defines $\sigma$ in terms of the ratio $\rho_M/\rho_+$
  along orbits in $M$. His definition allows the expansion rate along
  $\T M$ to be large, for example, as long as the expansion rate along $E^+$ is
  large enough to keep the ratio $\sigma(m)$ bounded, \emph{along the
    orbit through $m$}. The definitions in~\cite{Hirsch1977:invarmflds,%
    Mane1978:persistmflds,Bates2008:approx-invarmfld} are equivalent in the compact context
  to the one in~\cite{Fenichel1971:invarmflds}. Ma\~n\'e's work shows
  that this definition is as general as possible, see below.
\end{remark}

When $M$ is compact, normal hyperbolicity is a sufficient condition
for the existence of a persistent manifold $\tilde{M}$ for a system
generated by $\tilde{v}$ if $\norm{\tilde{v} - v}_1$ is sufficiently
small. Conversely, Ma\~n\'e~\cite{Mane1978:persistmflds} has proved
that normal hyperbolicity (in the sense of e.g.\ws Fenichel's definition)
is also necessary: if a compact invariant manifold $M$
is persistent under any $C^1$ small perturbation, then $M$ is normally
hyperbolic (see also Example~\ref{exa:break-nonNHIM} and the clear
exposition in the introduction of~\cite{Fenichel1971:invarmflds}).
Definition~\ref{def:NHIM}, however, only guarantees $C^1$
smoothness for the perturbed manifold $\tilde{M}$. To obtain higher
order smoothness, a more stringent condition of $\fracdiff$\ndash
normal hyperbolicity must be satisfied.
\begin{definition}[$\fracdiff$\ndash normally hyperbolic invariant manifold]
  \label{def:r-NHIM}
  A manifold $M$ is called $\fracdiff$\ndash normally hyperbolic with
  $\fracdiff \ge 1$ a real number, if it satisfies $M \in C^\fracdiff$ and the
  conditions in Definition~\ref{def:NHIM}, but with the stronger
  inequalities
  \begin{equation}\label{eq:spectral-gap}
    \rho_- < -\fracdiff\,\rho_M \le 0 \le \fracdiff\,\rho_M < \rho_+.
  \end{equation}
\end{definition}
\idxalso{$r$}{$r$-normal hyperbolicity}%
\index{normal hyperbolicity@$r$-normal hyperbolicity}%
This means that the normal expansion and contraction must not just
dominate the tangential ones, but do so by a factor $\fracdiff$. For
$\fracdiff = 1$ we recover the original definition, while the
generalized inequality~\ref{eq:spectral-gap} is called the \idx{spectral gap
condition}. If $M$ is $\fracdiff$\ndash normally hyperbolic and
$v$ and the perturbation $\tilde{v}$ are $C^\fracdiff$ as well, then
the persistent manifold $\tilde{M}$ is $C^\fracdiff$ smooth again. The
example in Section~\ref{sec:spectral-gap} shows that
this spectral gap condition is sharp: even when everything is
$C^\infty$, the perturbed manifold $\tilde{M}$ in that example is only
$C^\fracdiff$ when no more than $\fracdiff$\ndash normal hyperbolicity
holds. Note that $\fracdiff$ can be interpreted as a `fractional
differentiability degree' when writing $\fracdiff = k + \alpha$ with
integer $k \ge 1$ the normal degree of differentiability and
$0 \le \alpha \le 1$ an additional \idxb{H\"older continuity} exponent.

\begin{remark}
  \label{rem:gap-smoothness}
  \index{smoothness!non-$C^\infty$}
  We explicitly exclude the case $\fracdiff = \infty$ from
  Definition~\ref{def:r-NHIM}, even though the spectral gap
  condition~\ref{eq:spectral-gap} could hold for $\fracdiff = \infty$,
  if $\rho_M = 0$. The reason is that one can generally not expect to
  obtain a persistent manifold $\tilde{M} \in C^\infty$ in this case.
  Even though for any order $\fracdiff < \infty$ there exist
  persistent manifolds $\tilde{M} \in C^\fracdiff$ for sufficiently
  small perturbations, the maximum perturbation size generally depends
  on $\fracdiff$ and may shrink to zero when $\fracdiff \to \infty$.
  See Example~\ref{exa:NHIM-non-Cinfty} and the example
  in~\cite{VanStrien1979:centermfld-not-cinfty} for the closely
  related case of center manifolds.

  On the other hand, it is shown in~\cite{Hirsch1977:invarmflds} that
  there is forced smoothness. If $M \in C^1$ is an $\fracdiff$\ndash
  NHIM, then $M$ must be $C^r$. We do not show that this also holds in
  our noncompact setting, but this is likely to be true.
\end{remark}

With these preliminary definitions in place, we are now ready state
the main theorem of this thesis; it is restated in Chapter~\ref{chap:persist}.
We should point out that $M$ is not required to be
an embedded submanifold; immersions are allowed as well, see
Section~\ref{sec:immersed}. For the details of the smoothness notation
$\BUC^{k,\alpha}$ on manifolds we refer to
definitions~\ref{def:unif-bounded-map} and~\ref{def:unif-imm-submfld}.
\declarerepeattheorem{persistNHIMgen}%
[Persistence of noncompact NHIMs in bounded geometry]{
  Let $k \ge 2,\, \alpha \in \intvCC{0}{1}$ and
  $\fracdiff = k+\alpha$. Let\/ $(Q,g)$ be a smooth Riemannian manifold
  of bounded geometry and $v \in \BUC^{k,\alpha}$ a vector field on
  $Q$. Let $M \in \BUC^{k,\alpha}$ be a connected, complete
  submanifold of\/ $Q$ that is $\fracdiff$\ndash normally hyperbolic
  for the flow defined by $v$, with empty unstable bundle, i.e.\ws
  $\rank(E^+) = 0$.

  Then for each sufficiently small $\Ysize > 0$ there exists a
  $\delta > 0$ such that for any vector
  field $\tilde{v} \in \BUC^{k,\alpha}$ with
  $\norm{\tilde{v} - v}_1 < \delta$, there is a unique submanifold
  $\tilde{M}$ in the $\Ysize$\ndash neighborhood of $M$, such that
  $\tilde{M}$ is diffeomorphic to $M$ and invariant under the flow
  defined by $\tilde{v}$. Moreover, $\tilde{M}$ is $\BUC^{k,\alpha}$
  and the distance between $\tilde{M}$ and $M$ can be made arbitrarily
  small in $C^{k-1}$\ndash norm by choosing
  $\norm{\tilde{v} - v}_{k-1}$ sufficiently small.
}
\vspace{0.3cm}
\reptheorempersistNHIMgen

This result generalizes the well-known results
in~\cite{Fenichel1971:invarmflds,Hirsch1977:invarmflds} to the case of
noncompact submanifolds of Riemannian manifolds. Again,
our definition of normal hyperbolicity is slightly less
general than the definitions used in these works. We also
assumed that only the stable bundle $E^-$ is present, see also
Section~\ref{sec:ext-full-NHIM}; note that we thus only have the spectral
gap condition $\rho_- < -\fracdiff\,\rho_M$ with $\rho_M \ge 0$.
See also the restatement of this theorem on
page~\pageref{thm:persistNHIMgen} and the list of
remarks~\ref{rem:persistNHIM} for more details.

We borrow the idea to generalize the Perron method to NHIMs from
Henry~\cite{Henry1981:geom-semilin-parab-PDEs}, and use the
techniques of Vanderbauwhede and Van Gils~\cite{Vanderbauwhede1987:centermflds}
(see~\cite{Vanderbauwhede1989:centremflds-normal} for a clear
presentation) for proving higher order smoothness. This is similar,
but developed independently from Sakamoto's
work~\cite{Sakamoto1990:invarmlfds-singpert} in which he used the same
ideas to study singular perturbation problems. We improve these
results in a couple of ways. First of all, we simplify the
basics of the proof by reducing the two-step contraction argument to a
single contraction mapping, still written as a composition of two
separate maps acting on horizontal curves in $M$ and vertical curves
in the normal bundle fiber,
respectively. More importantly, we remove the restriction of a trivial
product structure $X \times Y$. Thus, we neither require $M$ to have a
global chart in a Banach space $X$, so $M$ need not be topologically
trivial, nor do we require a global product, so the normal bundle of
$M$ need not be trivial either. On the other hand, the results by
Bates, Lu, and Zeng also allow $M$ to be a general submanifold, but
still assume the ambient space to be a Banach space. Our results are
for finite dimensional, but not necessarily linear, Riemannian ambient
spaces. In their paper~\cite{Bates2008:approx-invarmfld}, they only
require an approximate NHIM for finding a persistent
invariant manifold. We use this idea as well (see the setup of $h$
small in the formulation of Theorem~\ref{thm:persistNHIMtriv}), but
we do not expand this idea any further. Finally, this work was
initiated from the (unfortunately never published) preprint by
Duistermaat on stable manifolds~\cite{Duistermaat1976:stabmflds}.

\index{noncompactness}
It seems to be a well-known belief by many experts that the theory of
normal hyperbolicity can be extended to a general noncompact
setting~\cite[p.~165]{Delshams2006:orbits-unbounded}. The idea is to
replace compactness by uniform estimates. An important conclusion to
be drawn from the present work is that indeed this principle holds, but probably
in a more strict way than one would naively realize. Uniform estimates
are not only required for the vector field defining the system, but
for the underlying ambient space as well, in terms of bounded
geometry. This becomes clear only when one leaves the context of
Euclidean ambient spaces, which trivially have bounded geometry. On a
Riemannian manifold, already the very definition of uniform continuity
of a vector field $v$ and its derivatives requires some aspects of
bounded geometry. It should be noted though, that we do not prove that
bounded geometry is a strictly necessary condition for persistence of
NHIMs; nonetheless, the results do suggest that persistence of NHIMs
may break down in `unbounded geometry', see Section~\ref{sec:cpt-unif}.

In Section~\ref{sec:proof-outline} we present an outline of the proof
and how it is reduced to a more basic setting $M' \times Y$ of a
trivial normal bundle. Here $M'$ is a smoothed version of $M$ to
rectify an artificial loss of smoothness, as occurs e.g.\ws
in~\cite{Sakamoto1990:invarmlfds-singpert}. Below we present some
extensions to the main Theorem~\ref{thm:persistNHIMgen} above.

\subsection{Non-autonomous systems}\label{sec:non-autonomous}
\index{non-autonomous system}

Our main theorem can be trivially extended to the non\hyp{}autonomous,
time\hyp{}dependent case. First, extend the configuration space with
time $t$ as additional variable, i.e.\ws $\hat{Q} = Q \times \R$, and add the
equation $\dot{t} = 1$. If the original system was time\hyp{}independent,
then $\hat{M} = M \times \R$ is a NHIM for the extended system, and all
uniform assumptions still hold, since the flow along the time
direction is neutral and trivial. Note that this argument does not work
in the classical theory as $\hat{M}$ is not compact\footnote{%
  If the perturbation is time-dependent, but in an (almost) periodic
  way, then this can still be treated in the compact setting. One can
  extend the configuration space with the circle $S^1$ (or an
  $n$\ndash torus in the almost periodic case).%
}. Now we can make any
$C^1$ small perturbation, and obtain a persistent manifold
$\tilde{M}$ in the extended configuration space. The perturbation is
allowed to be generally time\hyp{}dependent, as long as it is uniform in
time, including derivatives. The resulting manifold $\tilde{M}$ will
still be invariant and close to the original $M$, although it will
depend on time. That is, if we assume local coordinates
$(x,y) \in \R^n \times \R^m$ for $Q$ such that
$M = \R^n \times \{ 0 \}$ locally, then we can write
$\tilde{M} = \Graph(h)$ for a function
\begin{equation*}
  h\colon \R^n \times \R \to \R^m,\quad y = h(x,t).
\end{equation*}
In other words, $\tilde{M}$ can be viewed as a graph over $M$ (i.e.\ws a
section of the normal bundle), but this graph now additionally depends
on time. The manifold $\tilde{M}$ itself is again normally hyperbolic
when viewed in the extended space $Q \times \R$, see also
Section~\ref{sec:ext-non-autonomous}.

Such time\hyp{}dependent invariant manifolds are called `integral
manifolds'. These have been studied as non\hyp{}autonomous
generalizations of stable and unstable manifolds of hyperbolic fixed
points~\cite{Palmer1975:lin-intgmfld}, but also as generalizations of
compact NHIMs~\cite{Hale1961:intgmflds-pertODE,Yi1993:gen-integrmfld}.
The theory of noncompact NHIMs allows one to treat all such integral
manifolds in the same way as the autonomous case. One can, for
example, also start with an integral manifold that is normally
hyperbolic: it will persist just as well.

\subsection{Immersed submanifolds}\label{sec:immersed}
\index{immersed submanifold}

In the main Theorem~\ref{thm:persistNHIMgen}, we intentionally do
not precisely state in what sense $M$ is a submanifold of $Q$. The
implicit assumption that $M$ is an embedded submanifold can be
weakened to $M$ being an immersion, see
also~\cite[Section~6]{Hirsch1977:invarmflds}
and~\cite{Bates1999:persist-overflow}. That is, $M$ can be viewed as an
abstract manifold together with an immersion map $\iota\colon M \to Q$
that need not be injective. This does not affect the theory as long as
$\iota$ is still locally injective: $\iota(M)$ including a
neighborhood modeled on its normal bundle $N$ can be pulled back via
the immersion $\iota$ to the abstract $M$. All local properties are
preserved, so we can study the system via this `covering'. We may not
always make a clear distinction between the abstract manifold $M$
and its immersed image $\iota(M) \subset Q$; the discussion below
shows that this distinction is not really necessary, as long as we do
not consider perturbations.

\begin{figure}[htb]
  \centering
  \parbox[b]{5.1cm}{%
    \centering
    \input{\figpath gen-immersion.pspdftex}
    \caption{An immersion with a transverse intersection.}
    \label{fig:gen-immersion}
  }
  \hspace{1cm}
  \parbox[b]{5.1cm}{%
    \centering
    \includegraphics{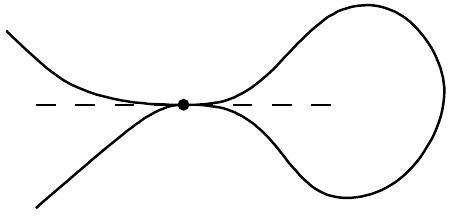}
    \caption{An allowed immersion with tangential intersection.}
    \label{fig:tang-immersion}
  }
\end{figure}
For a generic immersion one could expect a picture as in
Figure~\ref{fig:gen-immersion}, where the immersed manifold intersects
itself transversely. Such situations cannot occur if $M$ is a NHIM.
This follows from the exponential growth rates along tangent and
normal bundles of $M$. Let $m \in \iota(M)$ be an intersection point
of two preimages $m_1, m_2 \in M$. If the tangent spaces along $M$ at
$m_1$ and $m_2$ are embedded differently into $\T_m Q$, then one could find
$x \in \text{Im}\big(\D\iota(m_1)\big) \setminus \text{Im}\big(\D\iota(m_2)\big)$.
This would imply that $x$ has a component in $N_{m_2}$ and give
contradictory growth rates for $\D\Phi^t(m)\cdot x$ depending on
whether we view $m$ as image of $m_1$ or $m_2$, as the orbit of
$m \in \iota(M)$ is uniquely defined. Hence, at each point
$m \in \iota(M)$ the tangent spaces $\D\iota(m_i)$ of all preimages
$m_i \in \iota^{-1}(m)$ must coincide, see
Figure~\ref{fig:tang-immersion}. Stated more abstractly, $M$
must have contact of order one with itself. More generally it holds
that an immersed $k$\ndash NHIM has contact of order $k$ with
itself\footnote{%
  The order of contact is defined as the degree \emph{up to and
    including} which the Taylor expansions of the objects agree.%
}, see~\cite[p.~68]{Hirsch1977:invarmflds}.

Next, each maximal set of $\iota(M)$ with constant number of
preimages\footnote{%
  The number of preimages must be countable if $M$ is assumed to be
  second-countable.%
} $p \in \N \cup \{\infty\}$,
\begin{equation}\label{eq:immersion-invarsubset}
  M_p = \set[\big]{m \in \iota(M)}{\#\,\iota^{-1}(m) = p},
\end{equation}
is an invariant subset of $\iota(M)$. This is again due to uniqueness
of the flow. If an orbit would cross into a set of different preimage
number, then a least one of the `lifts' of this orbit from $\iota(M)$
to the `cover' $M$ would have to enter or leave $M$. This cannot
happen as $M$ itself is invariant. Hence, the conclusion is that
self-intersections of $\iota(M)$ must be invariant.

Immersed NHIMs may occur on themselves, or appear as a persistent
manifold under perturbation from an embedded manifold. An example of
an embedded noncompact NHIM that collapses under a small perturbation
into an immersed manifold can be found in Section~\ref{sec:cpt-unif}.
The same can happen with an immersed manifold with compact image. The
following example is taken from~\cite[p.~130]{Hirsch1977:invarmflds}
and shows that the injection map is relevant for how the NHIM
persists.
\begin{example}[Perturbation of a compact non-injectively immersed NHIM]
  \label{exa:cpt-immersed-NHIM}
  We consider on $\R^3$ the vector field
  \begin{equation*}
    \begin{aligned}
      \dot{x} &= \arctan(x^2) + \epsilon,\\
      \dot{y} &= y,\\
      \dot{z} &= -z
    \end{aligned}
  \end{equation*}
  and smoothly modify it outside the cylinder $y^2 + z^2 = 1$ such
  that it flows in the negative $x$\ndash direction and connects the
  basin of repulsion of the origin intersected with $x > 0$ to the
  basin of attraction intersected with $x < 0$. The perturbation
  parameter $\epsilon$ is initially set to zero.

  Note that the $x$\ndash axis is a NHIM (the arctangent is there to
  keep the vector field and tangential growth rate bounded). Due to
  the modification, the two loops in
  Figure~\ref{fig:NHIM-cpt-immersion} are also NHIMs of this system,
  both separately and their union. They start from the origin along
  the positive $x$\ndash axis, then diverge from it in opposite
  directions in the $xy$\ndash plane; once outside the cylinder
  $y^2 + z^2 = 1$ they start moving into the negative $x$ direction
  and finally return to the origin approximately along the $xz$\ndash
  plane.

  We can parametrize
  their joint image with an injection $\iota_1$ mapping
  $M = \{0,1\} \times S^1$ separately onto the two loops, but we can
  also parametrize with $\iota_2$ that maps $M = S^1$ onto the full
  figure eight image. If we perturb to $\epsilon > 0$, then $\iota_1$
  will result in Figure~\ref{fig:NHIM-cpt-immersion1} where the two
  loops are separated, while $\iota_2$ will result in
  Figure~\ref{fig:NHIM-cpt-immersion2} which has one loop, but the
  middle of the figure eight does not intersect anymore.
  Figure~\ref{fig:NHIM-cpt-immersion-sect} shows how the two orbits
  from the separate loops closely pass the $x$\ndash axis along
  hyperbolic trajectories. The single orbit of $\iota_2$ follows
  hyperbolic trajectories through the other two quadrants.
\end{example}

\begin{remark}
  Note that these different persistent NHIMs do not contradict the
  uniqueness property of persistence, since the (abstract) manifolds
  $M$ were different to begin with. Formulated differently, if we
  consider the universal cover of the tubular neighborhood of
  $\iota_1(M)$ (deduplicating the origin as image point), then
  Figure~\ref{fig:NHIM-cpt-immersion1} shows the unique invariant
  manifold that stays in this tubular neighborhood cover. We obtain a
  different persistent NHIM for any prescribed (possibly infinite)
  sequence of concatenating the two loops of the original figure eight
  into an immersion from $S^1$ (or $\R$ if the sequence is infinite).
\end{remark}

\begin{figure}[tbp]
  \newlength{\figwidth}
  \setlength{\figwidth}{6cm}
  \centering
  \parbox[b]{\figwidth}{%
    \centering
    \input{\figpath NHIM-cpt-immersion.pspdftex}
    \caption{a non-injectively immersed manifold with compact image.}
    \label{fig:NHIM-cpt-immersion}
  }
  \hspace{0.7cm}
  \parbox[b]{\figwidth}{%
    \centering
    \includegraphics[width=5.8cm]{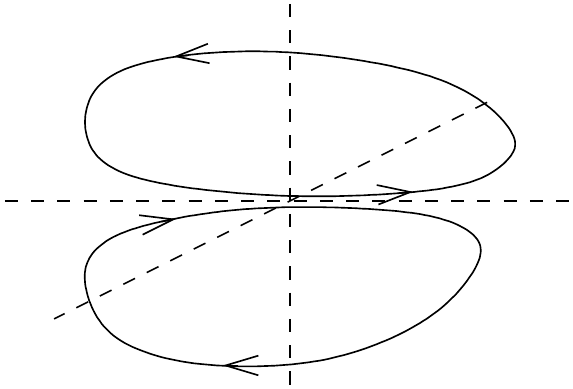}
    \caption{the persistent manifold of $\iota_1$ consisting of two
      separate loops.}
    \label{fig:NHIM-cpt-immersion1}
  }
  \\[0.5cm]
  \parbox[b]{\figwidth}{%
    \centering
    \includegraphics[width=5.8cm]{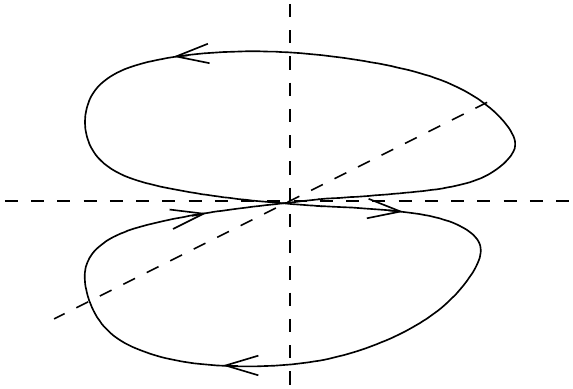}
    \caption{the persistent manifold of $\iota_2$ consisting of one
      figure eight loop without self-intersection.}
    \label{fig:NHIM-cpt-immersion2}
  }
  \hspace{0.7cm}
  \parbox[b]{\figwidth}{%
    \centering
    \input{\figpath NHIM-cpt-immersion-sect.pspdftex}
    \caption{projection onto the $yz$-plane showing the orbits
      of the persistent manifold $\iota_1$ while passing the origin.}
    \label{fig:NHIM-cpt-immersion-sect}
  }
\end{figure}

Finally, we present an example of an injectively immersed (but not
embedded) NHIM, see~\cite[p.~68]{Hirsch1977:invarmflds}. The mapping
below is known as Arnold's cat map.
\begin{example}[Injectively immersed dense line in the torus]
  \label{exa:immersed-dense-line}
  The matrix
  \begin{equation*}
  A = \left(\begin{array}{cc}
        2 & 1 \\
        1 & 1
      \end{array}\right)
  \end{equation*}
  acting on the two-torus $\mathbb{T}^2$ is an Anosov diffeomorphism.
  The line through $0$ with slope $\mfrac{1}{2}(1 - \sqrt{5})$ is
  densely immersed in the torus and it is a NHIM for this discrete
  system. If we take its suspension, then we have a flow with a NHIM
  that is densely immersed into the mapping torus
  $\big(\intvCC{0}{1} \times \mathbb{T}^2\big) /\!\sim$ with
  identification $(1,x) \sim (0,A\,x)$.
\end{example}

\subsection{Overflowing invariant manifolds}\label{sec:overflow-invar}
\index{overflow invariance}

In many applications of normally hyperbolic systems, the manifold $M$
has a boundary $\partial M$. A~typical reason is that the system
ceases to be normally hyperbolic across the
boundary. This happens, for example, when studying a singularly
perturbed, or slow-fast system and in the fast limit there are points
on $M$ with zero eigenvalues in the normal direction. At such points,
$M$ is not normally hyperbolic anymore, so one must restrict $M$ such
that these points are outside of $M$. Another, somewhat artificial but practical example
would be if the invariant manifold is noncompact and one would try to
use the classical theorems that are only applicable to compact
manifolds by cutting off $M$ to a compact manifold with boundary. One
can try to attack this latter case with our more general theory for
noncompact manifolds. The additional uniformity assumptions should be
checked then.

If $M$ is a manifold with boundary, some persistence results can still
be retained. This idea was introduced by
Fenichel~\cite{Fenichel1971:invarmflds} in studying so-called
overflowing invariant manifolds. These are normally hyperbolic
manifolds that are invariant under backward time flow, or in other
words, only under the forward flow, orbits can leave, i.e.\ws
`overflow' the manifold. The condition of overflowing invariant is
slightly stronger: the vector field must strictly point outward at the
boundary. This weakened version that the manifold is negatively
invariant does come at the additional cost that only stable normal
directions are allowed. The time-reversed situation of an inflowing
invariant manifold with only unstable normal directions is equivalent.
In Section~\ref{sec:ext-overflow-invar} we discuss how this idea can
be incorporated into the Perron method proof.

The attention of the reader is also drawn to the following remark made
in~\cite[p.~214]{Fenichel1971:invarmflds}. If an open submanifold
$N \subset M$ is overflowing invariant, and the spectral gap condition
is satisfied on $N$ with a higher ratio $\fracdiff_N$ than on the
whole of $M$, then the persistent manifold $\tilde{N}$ over $N$
retains $C^{\fracdiff_N}$ smoothness, even if smoothness of
$\tilde{M}$ will generally be lower.

\section{Induced topology}\label{sec:induced-topo}
\index{comparison!of topologies}

In this work the topologies for spaces of vector fields, submanifold
embeddings, et cetera, are (implicitly) defined by norms and distance
functions. The norms we use are uniform $C^k$\ndash norms for bounded
functions, and families with additional exponential growth rates.
Let us call the topologies induced by these norms $\BC^k$\ndash
topologies and consider how they compare to
two common topologies: the weak and strong Whitney topologies for maps
between manifolds, alternatively known as the compact-open and fine
topology, see~\cite{Hirsch1976:difftopo}.
\index{$C^k_b$!topology}

The weak topology has a subbasis generated by the set of functions $g$
that are close to some function $f$ in $C^k$\ndash norm on compact
subsets in local coordinate charts. This means that for example the
function family
\begin{equation*}
  f_\delta\colon \R \to \R\colon x \mapsto \delta\,\exp(x^2)
\end{equation*}
converges to zero for $\delta \to 0$ in this topology. On any compact
set $f_\delta$ will become arbitrarily small when
$\delta \to 0$ while it does not converge in uniform norm (nor with
additional exponential growth rate). Hence the weak topology is
weaker than our induced $\BC^k$\ndash topologies.

The strong topology has as basis all sets of functions $g$ that are
close to some function $f$ on a locally finite cover by compact sets
$K_i$, where $g$ must approximate $f$ in $C^k$\ndash norm on each
$K_i$ in local coordinates up to a given chart-dependent size
$\epsilon_i$. For any function without compact
support, a collection $\epsilon_i > 0$ can be found that converges
faster to zero on each larger $K_i$ than the function to zero when
$x \to \infty$. Hence the only sequences of functions $\R \to \R$
that converge to the zero function in the strong topology are
those with (eventually) compact support.
A~family $f_\delta$ of functions with noncompact
support cannot converge to the zero function, as can be seen by using
a diagonal argument. The family $f_\delta(x) = \delta\,\exp(-x^2)$,
for example, does not converge to the zero function in the strong
topology. Given a locally finite cover of $\R$ by compact sets $K_i$, we
choose $x_i \in K_i$ and corresponding $\epsilon_i = \exp(-x_i^2)/i$.
Then for any given $\delta > 0$, we will have
$\abs{f_\delta(x_i)} > \epsilon_i$ for some large $i$. On the other hand,
this family $f_\delta$ obviously converges under the uniform norm with
any exponential growth rate. Thus, the strong topology is stronger
than our induced $\BC^k$\ndash topologies, see also the remark
in~\cite[p.~43]{Golubitsky1973:stable-mappings} for noncompact
manifolds.

We conclude that the $\BC^k$\ndash topologies induced by our uniform norms are not
equivalent to either the weak or strong Whitney topology, because
the weak topology allows arbitrary behavior of functions outside
compact sets, while the strong topology completely restricts that
behavior. Our norms allow moderate variations at infinity. In
general, `moderate behavior' is not well-defined on a general
noncompact manifold, as it depends on the choice of charts. In the
setting of bounded geometry, though, the uniform, metric structure
makes this behavior unambiguous; we can restrict to normal coordinate
charts and consider `moderate behavior' with respect to these.
Note that these topologies are equivalent on compact domains.

\section{Notation}\label{sec:notation}

Here, we will establish some notation and conventions to be used
throughout this work. See the index for more specific symbols.
\begin{itemize}
\item
  The letters \idx{$I$} and \idx{$J$} will denote intervals in $\R$;
  $I$ will typically represent an interval that is unbounded on one
  side, while $J$ will be bounded.
\item
  $\epsilon, \delta > 0$ will denote (small) bounds for
  continuity-like estimates; $C > 0$ will denote arbitrary bounds.
  The specific meaning of these symbols will vary depending on
  context. $\epsilon_f(\delta)$ will denote a uniform continuity
  modulus of the  function $f$, that is,
  $\epsilon_f\colon \R_{\ge 0} \to \R_{\ge 0}$ satisfies
  \begin{equation}\label{eq:cont-modulus}
    d(f(x_2),f(x_1)) < \epsilon_f(d(x_2,x_1))
    \qquad\text{and}\qquad
    \lim_{\delta \to 0}\; \epsilon_f(\delta) = 0.
  \end{equation}
  Without subscript $f$ this will denote an arbitrary \idx{continuity
  modulus}. \idxalso{$\epsilon_f$}{continuity modulus}
\item
  The $\D$ denotes a total derivative, while $\D_i$ with index $i \in \N$
  denotes a partial derivative with respect to the $i$\th argument,
  or, when a subscript symbol is appended, say $\D_x$, then this
  denotes a partial derivative with respect to the argument commonly
  referred to by that symbol.
\item
  We use the following symbols to denote classes of function spaces:

  \begin{tabular}{l>{\raggedright}p{\textwidth-2.5cm}}
    $\BC$    & bounded, continuous functions;\tabularnewline
    $\BUC$   & bounded, uniformly continuous functions;\tabularnewline
    $C^k$    & $k$ times continuously differentiable functions;\tabularnewline
    $C^{k,\alpha}$ & $C^k$ functions with $\alpha$\ndash
    H\"older continuous $k$\th derivative. We will
    conventionally write $\fracdiff = k + \alpha \in \R_{\ge 1}$; the H\"older
    estimates are assumed to be uniform in $\BUC^{k,\alpha}$ spaces.\tabularnewline
    $\CLin$  & continuous, i.e.\ws bounded, (multi)linear operators;\tabularnewline
    $\vf$    & vector fields;\tabularnewline
    $\Gamma$ & sections of a fiber bundle.\tabularnewline
  \end{tabular}
  \index{$C^k_b$!function space}
  \index{$C^k_{b,u}$}
  \index{$X$@$\vf$}
  \index{$\Gamma$!fiber bundle section}

  Unless otherwise specified, $\BC^k$ and $\BC^{k,\alpha}$ spaces
  will be endowed with the canonical norms that turn these into Banach
  spaces, that is,
  \begin{equation}\label{eq:C-k-alpha-norm}
    \norm{f}_{k,\alpha}
    = \sum_{0 \le n \le k} \sup_x\; \norm{\D^n f(x)}
     +\sup_{x_2 \neq x_1} \frac{\norm{\D^k f(x_2) - \D^k f(x_1)}}{d(x_2,x_1)^\alpha}.
  \end{equation}

  We define the operator norm on a multilinear operator
  $A \in \CLin^k(V_1 \times \ldots \times V_k;W)$ as
  \begin{equation}\label{eq:operator-norm}
    \norm{A} = \sup_{\substack{v_i \in V_i\\ \norm{v_i}=1}} \norm{A(v_1,\ldots,v_k)}.
  \end{equation}
  This multilinear operator norm can be extended to sections $s$ of
  real-valued tensor bundles by taking the operator norm pointwise of
  $s(x)$ as a multilinear operator into $\R$.
\item
  On a Riemannian manifold, $\Christoffel$ will denote the Christoffel
  symbols, while $\partrans$ will be used for parallel transport along
  a curve given as argument, for example, $\partrans(\gamma|_a^b)$
  will denote parallel transport along the curve $\gamma$ restricted
  to the interval $\intvCC{a}{b}$. We shall denote induced parallel
  transport on products of the tangent bundle by
  $\partrans(\gamma|_a^b)^{\otimes k}$.
  \index{$\Gamma$!Christoffel symbols}
  \idxalso{$\Pi$}{parallel transport}
\item
  We shall often work with maps that are defined on the tangent space
  over a point $x \in M$ and denote this dependence on $x$ by a
  subscript, for example $h_x\colon \T_x M \to \T_x M$.
  If we want to refer to the whole family of such maps for all $x \in
  M$, then we denote this by
  \begin{equation*}
    h_\bullet\colon \T_\bullet M \to \T_\bullet M,
  \end{equation*}
  particularly if we want to stress that this family satisfies some
  properties uniformly in $x$.
  \index{${}_\bullet$}

\item
  We use the notation $B(x;\delta)$ not only to indicate open balls of
  radius $\delta$ around a single point $x$, but also $B(M;\delta)$ to
  indicate a (tubular) neighborhood of some set or submanifold $M$,
  that is,
  \begin{equation*}
    B(M;\delta) = \set{x}{d(M,x) < \delta}.
  \end{equation*}
\end{itemize}

The following definition of a scale of Banach spaces
(cf.~\cite{Vanderbauwhede1987:centermflds}) is
fundamental to the rest of this work.
\begin{definition}
  \label{def:rho-norm}
  \index{${}\norm{\slot}_\rho,\,d_\rho$}
  \index{$B^\rho$}
  Let $X$ be a normed linear space and $\mathcal{F} = C(I;X)$ the
  space of continuous functions from an interval $I \subset \R$ to
  $X$. We define a family of exponential growth norms with parameter
  $\rho \in \R$ by
  \begin{equation}\label{eq:rho-norm}
    \norm{f}_\rho = \sup_{t \in I}\; \norm{f(t)}\,e^{-\rho\,t}
    \qquad\text{for}\quad f \in \mathcal{F}.
  \end{equation}
  We define $B^\rho(I;X)$ to be the normed space consisting of all
  functions $f \in \mathcal{F}$ with $\norm{f}_\rho < \infty$. If $X$ is
  a Banach space, then $B^\rho(I;X)$ is a Banach space as well.
\end{definition}

\begin{remark}\label{rem:sign-rho}
  When the interval $I$ is bounded from below, then the embedding
  $B^{\rho_1}(I;X) \hookrightarrow B^{\rho_2}(I;X)$ is continuous for
  $\rho_1 \le \rho_2$. The time reversed version when $I$ is bounded
  above and $\rho_2 \le \rho_1$ holds, will frequently recur
  throughout this work. See also Remark~\ref{rem:nemytskii-time-rev}
  and the note on integrals of exponentials~\ref{eq:exp-int} below. In
  Chapter~\ref{chap:persist} we shall use $I = \R_{\le 0}$ and
  negative rates $\rho$, while in the appendices~\ref{chap:nemytskii}
  and~\ref{chap:exp-growth} we use (the somewhat more natural)
  $I = \R_{\ge 0}$; though $\rho$'s can take both signs there.
\end{remark}

The definition of an exponential growth norm can be generalized to
curves mapping into a metric space. Let $(X,d)$ be a metric space,
then analogously to~\ref{def:rho-norm}, we define a family of
exponential growth distance functions on $\mathcal{F}$ by
\begin{equation}\label{eq:rho-dist}
  d_\rho(f_1,f_2) = \sup_{t \in I}\; d\big(f_1(t),f_2(t)\big)\,e^{-\rho\,t}.
\end{equation}
Note that this distance function might be infinite for some
$x_1, x_2 \in \mathcal{F}$.

We will be working with exponential growth estimates of the form
$C\,e^{\rho\,t}$ throughout this paper. The pair of numbers
$C > 0, \rho \in \R$ that determine such a growth estimate will be
referred to as \emph{\idx{exponential growth numbers}}, and $\rho$
\index{$\rho$} as an \emph{\idx{exponential growth rate}}.

We will frequently encounter integrals over a time interval, where the
integrand obeys an exponential estimate. As long as the
interval $\intvCC{a}{b}$ is bounded in the direction of exponential
growth and $\rho \neq 0$, these can be estimated as
\begin{equation}\label{eq:exp-int}
      \int_a^b e^{\rho\,t} \d t
  \le \frac{1}{\abs{\rho}}\,\exp\big(\sup_{t \in \intvCC{a}{b}} \rho\,t\;\big).
\end{equation}

We also state here some basic facts about uniformly H\"older
continuous functions.
\begin{lemma}[Product rule for \idx{H\"older continuity}]
  \label{lem:holder-product}
  Let $f,\,g \in \BUC^\alpha$ be defined on spaces such that the
  product $f \cdot g$ is well-defined. Then also
  $f \cdot g \in \BUC^\alpha$.
\end{lemma}

\begin{proof}
  \index{$C_{f,\alpha}$}
  Let $\norm{f}_0,\,\norm{g}_0 \le M$ and let
  $C_{f,\alpha},\, C_{g,\alpha}$ be the respective H\"older
  coefficients of $f,\,g$. Then we have for all $x_1 \neq x_2$
  \begin{equation*}
    \begin{aligned}
         \norm{f(x_2)\,g(x_2) - f(x_1)\,g(x_1)}
    &\le \norm{f(x_2)}\,\norm{g(x_2) - g(x_1)}
        +\norm{f(x_2) - f(x_2)}\,\norm{g(x_2)}\\
    &\le M\,(C_{f,\alpha} + C_{g,\alpha})\,\norm{x - y}^\alpha,
    \end{aligned}
  \end{equation*}
  which exhibits the H\"older coefficient
  $M\,(C_{f,\alpha} + C_{g,\alpha})$ for the product, and $f \cdot g$
  is clearly bounded by $M^2$.
\end{proof}

\begin{lemma}
  \label{lem:holder-lower}
  Let $f \in \BUC^\alpha$. Then it also holds that
  $f \in \BUC^\beta\,$ for any\/ $0 < \beta < \alpha$.
\end{lemma}

\begin{proof}
  Let $M$ be the bound on $f$, and $C_\alpha$ its $\alpha$\ndash
  H\"older coefficient. For $\norm{x_2 - x_1} \le 1$ the estimate for
  $\beta$ follows automatically from that of $\alpha$. For
  $\norm{x_2 - x_1} > 1$ we use boundedness to obtain
  \begin{equation*}
    \norm{f(x_2) - f(x_1)} \le 2\,M \le 2\,M\,\norm{x_2 - x_1}^\beta.
  \end{equation*}
  Hence, $C_\beta = \max(C_\alpha,2\,M)$ suffices as $\beta$\ndash
  H\"older coefficient.
\end{proof}

\subsection*{Typographical conventions}

As usual we close proofs with the symbol $\;\proofSymbol$, while we
shall use $\,\remarkSymbol\,$ and $\,\exampleSymbol\,$ to denote the
end of (a series of) remarks or examples, respectively.



\chapter{Manifolds of bounded geometry}\label{chap:boundgeom}
\index{bounded geometry}

For noncompact normally hyperbolic systems, uniformity assumptions
that were implicit in the compact case must be made explicit. Not only
assumptions on the vector field, but on the underlying space as well.
For this we need the concept of bounded geometry;
Section~\ref{sec:cpt-unif} contains a discussion and examples for why
we require this concept.

The class of manifolds of bounded geometry allows
us to uniformly apply constructions that are well-known for compact
manifolds. We single out the atlas of normal coordinate charts and
derive from the very definition of bounded geometry that all constructions and estimates
are uniform over all such charts. For completeness, we present here
all results that we need later on. Some of these results are
already present in the literature: the construction of a uniformly
locally finite cover and a subordinate $C^k$ uniformly bounded
partition of unity, and bounded coordinate transformations can be
found
in~\cite{Shubin1992:spec-elloper-noncpt,Schick2001:mlfd-bndry-boundgeom},
for example, while~\cite{Roe1988:index-openmflds} includes the result
on finite coloring of the connectedness graph of a uniformly locally
finite cover. I have not been able to find in the literature the
results about the existence of a uniform tubular neighborhood, the
approximation of a submanifold by a smoothed manifold, and the
construction of a trivial bundle embedding. Submanifolds are allowed
to be non-injectively immersed.

This chapter is organized as follows. First, the material is presented
that is already required for the global coordinate setting of
Theorem~\ref{thm:persistNHIMtriv}. These include the basic definitions of
bounded geometry, related results on bounded coordinate transition
maps, uniform covers and partitions of unity, and an explicit relation
between holonomy and curvature. Then we
continue to work towards the final goal of this chapter: to reduce a
noncompact normally hyperbolic system from a setting in general
manifolds to a trivial bundle setting, in order to generalize the
persistence theorem to the former setting. To this end, we need some
more technical results: a uniform tubular neighborhood, smooth
approximation of a submanifold, and embedding into a trivial bundle.

This chapter relies heavily on some more advanced concepts from
differential and specifically Riemannian geometry. On the other hand,
the results are used as tools in solving a dynamical systems
problem. Appendix~\ref{chap:riemgeom} provides a quick review for
non-experts of the most relevant geometric concepts used here. It also
provides further references to the literature. We shall assume the
contents of this appendix known from here on.

I suggest the reader to at least take a glance at the first two
sections of this chapter to familiarize himself with the basic
definitions and results of bounded geometry, without the need to go
through the details of the proofs. Then, depending on his interest, he
can choose to delve into the more technical geometric details or skip
to Chapter~\ref{chap:persist} for the more analytical side of the
proof of Theorem~\ref{thm:persistNHIMtriv}, and possibly return later
to read how Theorem~\ref{thm:persistNHIMgen} is reduced to the
former.

\section{Bounded geometry}
\label{sec:BG-basic}

We follow the definition in~\cite{Eichhorn1991:mfld-metrics-noncpt} to
introduce bounded geometry. Recall that the \idx{injectivity radius}
$\rinj{x}$ at a point $x \in M$ is the maximum radius for which the
exponential map at $x$ is a diffeomorphism, see also
Appendix~\ref{chap:riemgeom}.
\begin{definition}[Bounded geometry]
  \label{def:bound-geom}
  \index{bounded geometry!definition for manifold}
  \idxalso{$r_\textrm{inj}$}{injectivity radius}
  We say that a complete, finite-dimensional Riemannian manifold
  $(M,g)$ has $k$\th order bounded geometry when the following
  conditions are satisfied:
  \begin{description}[labelindent=1ex,labelwidth=4ex,leftmargin=!]
  \item[(I)] the global injectivity radius $\rinj{M} = \inf\limits_{x \in M}\; \rinj{x}$
    is positive, $\rinj{M} > 0$;
  \item[(B$_k$)] the Riemannian \idx{curvature} $R$ and its covariant
    derivatives up to $k$\th order are uniformly bounded,
    \begin{equation*}
      \forall\; 0\le i \le k\colon \sup_{x \in M}\; \norm{\nabla^i R(x)} < \infty,
    \end{equation*}
    with operator norm of\/ $\nabla^i R(x)$ as an element of the
    tensor bundle over $x \in M$.
  \end{description}
\end{definition}

\begin{remark}
  \label{rem:indep-BG-conditions}
  The conditions (I) and (B$_k$) are independent. We present a simple
  example which exhibits zero infimum for the injectivity radius while
  all derivatives of the curvature are globally bounded. Indeed, let
  $M = \R \times S^1$ be a cylinder with metric
  $g = {\rm d} x^2 + e^{-2x}\,{\rm d}\theta^2$ in coordinates $(x,\theta)$, see
  also Figure\footnote{%
    This is a noncompact surface with constant negative curvature,
    hence it cannot be isometrically embedded into $\R^3$,
    see~\cite{Hilbert1901:surfaces-const-curv}. The embedding is
    nearly isometric for $x \gg 0$ though, so the figure
    is still a good representation there.%
  }~\ref{fig:exp-cylinder} on page~\pageref{fig:exp-cylinder}. The
  injectivity radius $\rinj{M}$ is zero since the cylinder
  circumference shrinks to zero with $x \to \infty$. Global
  boundedness of the curvature and all of its derivatives follows from
  a symmetry argument. The family
  \begin{equation*}
    \phi_{\xi,\alpha}\colon (x,\theta) \mapsto (x+\xi,e^\xi\,\theta + \alpha)
    \qquad\text{with}\quad \xi \in \R,\; \alpha \in \intvCO{0}{2\pi}
  \end{equation*}
  is a set of local isomorphisms that acts transitively on $M$. That
  is, for any two points $(x_1,\theta_1),(x_2,\theta_2) \in M$ there
  exist $\xi,\alpha$ and a neighborhood $U \ni (x_1,\theta_1)$ such
  that $\phi_{\xi,\alpha}\colon U \to \phi_{\xi,\alpha}(U)$ is an
  isomorphism and $\phi_{\xi,\alpha}(x_1,\theta_1) = (x_2,\theta_2)$.
  For any $(x,\theta) \in U$ and $v,w \in \T_{(x,\theta)} M$ we have
  \begin{equation*}
    \begin{aligned}
      (\phi_{\xi,\alpha}^* g)_{(x,\theta)}(v,w)
      &= g_{(x+\xi,e^\xi\,\theta+\alpha)}
         \big(\D \phi_{\xi,\alpha}(x,\theta)\,v,
              \D \phi_{\xi,\alpha}(x,\theta)\,w\big)\\
      &= {\rm d} x(v){\rm d} x(w) + e^{-2(x+\xi)}\,e^\xi{\rm d}\theta(v)\,e^\xi{\rm d}\theta(w)\\
      &= g_{(x,\theta)}(v,w)
    \end{aligned}
  \end{equation*}
  so $\phi_{\xi,\alpha}^* g = g$ on $U$. Since the curvature and its
  derivatives are locally determined, this implies that these are
  constant across $M$, hence uniformly bounded (actually all
  derivatives of $R$ vanish). Note that these local isometries do not
  imply a finite global injectivity radius since the size of the
  neighborhood $U$ does depend on the points
  $(x_1,\theta_1),(x_2,\theta_2) \in M$.
\end{remark}

\begin{example}[Manifolds of bounded geometry]
  \label{exa:BG-manifolds}
  The following are examples of manifolds with bounded geometry of any
  (i.e.\ws infinite) order.
  \begin{itemize}
  \item Euclidean space with the standard metric trivially has bounded
    geometry.
  \item A smooth, compact Riemannian manifold $M$ has bounded
    geometry as well; both the injectivity radius and the curvature
    including derivatives are continuous functions, so these attain
    their finite minimum and maxima, respectively, on $M$.
    If $M \in C^{k+2}$, then it has bounded geometry of order $k$.
  \item Noncompact, smooth Riemannian manifolds that possess a
    transitive group of isomorphisms (such as hyperbolic space) have
    bounded geometry since the finite injectivity radius and curvature
    estimates at any single point translate to a uniform estimate for
    all points under isomorphisms. Note that the example in
    Remark~\ref{rem:indep-BG-conditions} above shows that it is not
    sufficient to have local isometries.
  \end{itemize}

  More manifolds of bounded geometry can be constructed with these
  basic building blocks in the following ways.
  \begin{itemize}
  \item The product of a finite number of manifolds of bounded
    geometry again has bounded geometry, since the direct sum
    structure of the metric is inherited by the exponential map and
    curvature. We give an outline of the proof. In a product
    coordinate chart
    \begin{equation*}
      (\phi_1,\phi_2) \colon U_1 \times U_2 \to \R^{n_1} \times \R^{n_2}
    \end{equation*}
    with coordinates $(x_1,x_2)$, the metric has diagonal form
    \begin{equation*}
      g(x_1,x_2) = (g_1 \oplus g_2)(x_1,x_2) =
      \begin{pmatrix}
        g_1(x_1) & 0 \\
        0 & g_2(x_2)
      \end{pmatrix}.
    \end{equation*}
    The coordinate dependence on $x_1,\,x_2$ is non-mixed and this is
    preserved under taking derivatives and index contractions, so $R$
    will split into a direct sum of $R_1$ and $R_2$ again. This can be
    extended to derivatives of $R$.

    A geodesic in $M_1 \times M_2$ is precisely given by
    $\gamma = (\gamma_1,\gamma_2)$ where $\gamma_1,\,\gamma_2$ are
    geodesics parametrized with constant speed in $M_1,\,M_2$,
    respectively. This follows easily since minimization of length is
    equivalent to minimization of the energy functional
    \begin{equation*}
      2\,E(\gamma) = \int_a^b g(\dot{\gamma},\dot{\gamma}) \d t
                   = \int_a^b g_1(\dot{\gamma}_1,\dot{\gamma}_1)
                             +g_2(\dot{\gamma}_2,\dot{\gamma}_2) \d t
    \end{equation*}
    and this splits nicely into independent minimization problems for
    $\gamma_1$ and $\gamma_2$. With a little effort one sees that
    $\rinj{M} \ge \min\big(\rinj{M_1},\rinj{M_2}\big)$.
  \item If we take a finite connected sum of manifolds with bounded
    geometry such that the gluing modifications are
    smooth and contained in a compact set, then the resulting manifold
    has bounded geometry again.
  \item We can endow the tangent bundle $\T M$ of a Riemannian
    manifold $(M,g)$ with the natural Sasaki
    metric~\cite{Sasaki1958:diffgeom-tangent}. Let $x^i$ denote
    coordinates on an open neighborhood $U \subset M$. These
    coordinate functions can be pulled back to $\T U$ and the
    one-forms ${\rm d} x^i$ can be viewed as additional coordinates
    $v^i$ such that the $x^i,\,v^j$ together form a complete set of
    induced coordinates on $\T U$. With respect to these coordinates
    the Sasaki metric is given by
    \begin{equation}\label{eq:TM-metric}
      \hat{g}(x,v) = g_{ij}(x)\big({\rm d} x^i\,{\rm d} x^j + \D v^i\,\D v^j\big)
      \qquad\text{where}\quad
      \D v^i = {\rm d} v^i + \Christoffel^i_{jk} v^j\,{\rm d} x^k
    \end{equation}
    and $\Christoffel^i_{jk}$ denote the Christoffel symbols on $M$,
    while the ${\rm d} v^i$ are one-forms on the manifold $\T U$.

    Bounded geometry of $(M,g)$ is not inherited by $(\T M,\hat{g})$
    since the extended Riemannian curvature $\hat{R}$ contains
    unbounded terms $v$ when expressed in terms of $R$,
    see~\cite[Prop.~7.5]{Gudmundsson2002:geom-tangent}. These
    expressions do readily show that the restriction
    $\T^r M = \set{(x,v) \in \T M}{g_x(v,v) \le r^2}$ satisfies
    curvature bounds of order $k-1$ if $(M,g)$ has $k$\ndash bounded
    geometry. The geodesic flow equation is given in induced
    coordinates by~\cite[eq.~(7.7)]{Sasaki1958:diffgeom-tangent}. By
    application of Theorem~\ref{thm:unif-smooth-flow}, one can then
    show that the injectivity radius is bounded.

    Note that $\T^r M$ is a manifold with boundary, but this is not
    problematic in our setting as long as the invariant submanifold
    stays away from the boundary. Alternatively one could try to use
    results from~\cite{Schick2001:mlfd-bndry-boundgeom}.
  \end{itemize}
\end{example}

When we say that a manifold has bounded
geometry without specifying the order $k$, then it is assumed that the
order is infinite, $k = \infty$, or sufficiently large. When $k \ge 1$
we have the following result, see~\cite[Thm~2.4 and
Cor.~2.5]{Eichhorn1991:mfld-metrics-noncpt}. In case $k = \infty$ the
converse also holds~\cite[lem.~2.2]{Roe1988:index-openmflds}.
\begin{theorem}[Boundedness of the metric]
  \label{thm:bound-metric}
  Let $(M,g)$ be a Riemannian manifold of\/ $k$\ndash
  bounded geometry. Then there exists a $\delta > 0$ such that the
  metric up to its $k$\th order derivatives and the Christoffel
  symbols up to its $(k \tm 1)$\th order derivatives are bounded in
  normal coordinates of radius $\delta$ around each $x \in M$, and the
  bounds are uniform in $x$.
\end{theorem}
This basic fact can be used to make the properties of all kinds of
constructions uniform over a noncompact manifold.
Note that here and in the following, all uniformity estimates
are assumed globally valid, that is, independent of the point $x \in M$. To
stress this, we shall use notation $f_\bullet$, for example as in
Definition~\ref{def:unif-bounded-map}, to indicate that the family of
maps $\{f_x\}_{x \in M}$ satisfies continuity estimates independent
of~$x$.
\index{${}_\bullet$}

With Theorem~\ref{thm:bound-metric} at hand, we shall exclusively use
\idx{normal coordinates} for local coordinate calculations. To establish
notation, we say that\index{$exp$@$\exp$}
\begin{equation}\label{eq:norm-coords}
  \phi = \exp_x^{-1}\colon B(x;\delta) \subset M
                       \to B(0;\delta) \subset \T_x M
\end{equation}
is a normal coordinate chart at $x \in M$. The radius $\delta$ will
always be chosen smaller than the injectivity radius $\rinj{M}$, so
$\phi$ is a diffeomorphism. Each tangent space $\T_x M$ carries the
inner product $g_x$, hence is isometric to Euclidean space $\R^n$ (but
identification requires a choice of basis).

\begin{proposition}
  \label{prop:unif-coord-chart}
  Let $(M,g)$ be a Riemannian manifold of\/ $k \ge 1$ bounded geometry.
  For every $C > 1$ there exists a\/ $\delta > 0$ such that the normal
  coordinate charts $\phi_x$ in~\ref{eq:norm-coords} are defined on
  $B(x;\delta)$ for each $x \in M$ and the Euclidean distance $d_E$ on
  the normal coordinates is uniformly $C$\ndash equivalent to the
  metric distance $d$ induced by $M$, that is,
  \begin{equation*}
    \forall x_1,x_2 \in B(x;\delta)\colon\quad
       C^{-1}\,d(x_1,x_2) \le d_E(\phi_x(x_1),\phi_x(x_2)) \le C\,d(x_1,x_2).
  \end{equation*}
\end{proposition}

\begin{proof}
  Let $\delta < \frac{1}{2}\rinj{M}$ and $x \in M$. We consider a
  normal coordinate chart $\phi_x$ on $B(x;2\delta)$. According to
  Theorem~\ref{thm:bound-metric}, the metric $g$ and its derivatives
  are bounded in normal coordinates. We have $(\exp_x^* g)(0) = g_x$,
  the Euclidean inner product on $\T_x M$, while the total derivative
  $\D(\exp_x^* g)(\xi)$ is bounded on $B(0;2\delta) \ni \xi$, say by
  $\norm{\D(\exp_x^* g)(\xi)} \le C_1$, independent of $x \in M$. By
  the mean value theorem this induces the uniform bounds
  \begin{equation*}
    1 - 2\delta\,C_1 \le \norm{(\exp_x^* g)(\xi)} \le 1 + 2\delta\,C_1.
  \end{equation*}

  Let $x_1,x_2 \in B(x;\delta)$ and let $\gamma_E$ be the straight
  curve between $\phi_x(x_1)$ and $\phi_x(x_2)$ in $\T_x M$
  parametrized by arc length. This curve $\gamma_E$ attains the
  Euclidean distance $l_E(\gamma_E) = d_E(\phi_x(x_1),\phi_x(x_2))$.
  On the other hand, it gives an upper bound on the metric distance
  \begin{equation*}
    \begin{aligned}
      d(x_1,x_2) = \inf_\gamma\; l(\gamma)
      &\le \int_0^{l_E(\gamma_E)} \sqrt{(\exp_x^* g)_{\gamma_E(t)}
                          \big(\gamma_E'(t),\gamma_E'(t)\big)} \d t\\
      &\le \sqrt{1 + 2\delta\,C_1}\;d_E\big(\phi_x(x_1),\phi_x(x_2)\big).
    \end{aligned}
  \end{equation*}

  Let $\gamma$ be a geodesic minimizing the distance $d(x_1,x_2)$.
  Then $\gamma$ is contained in $B(x;2\delta)$: the distance from each
  $x_i$ to the boundary of $B(x;2\delta)$ is at least $\delta$, so if
  $\gamma$ would leave and reenter $B(x;2\delta)$ then its length
  would be at least $2\delta$. On the other hand, $x_1$ and $x_2$ can
  be connected via $x$ with a curve of length less than $2\delta$. Let
  us write $\eta = \phi_x\circ\gamma$ and assume that $\eta$ is
  parametrized by arc length with respect to the Euclidean metric
  $g_x$. Then we obtain an inverse estimate to the one above:
  \begin{equation*}
    \begin{aligned}
      d_E(\phi_x(x_1),\phi_x(x_2)) \le \int_0^{l_E(\eta)} 1 \d t
      &\le \int_0^{l_E(\eta)}
         (1 - 2\delta\,C_1)^{-\frac{1}{2}}\,\sqrt{(\exp_x^* g)_{\eta(t)}
                                    \big(\eta'(t),\eta'(t)\big)} \d t\\
      &= (1 - 2\delta\,C_1)^{-\frac{1}{2}}\,\int_0^{l_E(\eta)}
             \sqrt{g_{\gamma(t)}\big(\gamma'(t),\gamma'(t)\big)} \d t\\
      &\le (1 - 2\delta\,C_1)^{-\frac{1}{2}}\;d(x_1,x_2).
    \end{aligned}
  \end{equation*}

  Finally, we complete the proof by choosing $\delta > 0$ small enough
  that
  \begin{equation*}
    \max\big((1 + 2\delta\,C_1)^{ \frac{1}{2}},
             (1 - 2\delta\,C_1)^{-\frac{1}{2}}\big) \le C.
  \end{equation*}
\end{proof}

From here on, we shall frequently represent objects living in
$B(x;\delta) \subset M$ on normal coordinate neighborhoods
$B(0;\delta) \subset \T_x M$ via the normal coordinate chart $\phi_x$.
We will mostly use $B(x;\delta)$ to
clearly indicate the base point, or $B(0_x;\delta) \subset \T_x M$ to
stress the tangent space domain of the coordinates as well.
In spaces of bounded geometry, normal coordinate charts are the
natural charts to works in and coordinate transition maps are not just
smooth, but uniformly bounded, as stated in the following lemma.
\begin{lemma}[Boundedness of transition maps]
  \label{lem:bound-coordtrans}
  Let $(M,g)$ be a Riemannian manifold of\/ $k$\ndash bounded geometry
  with $k \ge 2$. There exists a $\delta$ with $0 < \delta < \rinj{M}$
  and constants $C,\,L > 0$ such that for all $x_1,x_2 \in M$ with
  $d(x_1,x_2) < \delta$ the following holds.
  \begin{enumerate}
  \item The \idx{coordinate transition map}
    \begin{equation}\label{eq:normal-coord-trans}
      \phi_{2,1} = \phi_2\circ\phi_1^{-1}\colon U \to \T_{x_2} M
      \qquad\text{with}\quad
      U = \phi_1(B(x_1;\delta) \cap B(x_2;\delta)) \subset \T_{x_1} M
    \end{equation}
    is $C^{k-1}$ bounded with $\norm{\phi_{2,1}}_{k-1} \le C$.

  \item Let\/ $\gamma_{2,1}\colon \intvCC{0}{1} \to B(x_1;\delta)$ be
    the unique shortest geodesic connecting $x_1$ and $x_2$ and let\/
    $\partrans(\gamma_{2,1})$ be the associated parallel transport.
    Then the map
    \begin{equation*}
      \phi_{2,1} - \partrans(\gamma_{2,1})\colon U \to \T_{x_2} M
    \end{equation*}
    has $C^{k-2}$\ndash norm bounded by the Lipschitz estimate
    \begin{equation}\label{eq:normal-coord-lipschitz}
      \norm{\phi_{2,1} - \partrans(\gamma_{2,1})}_{k-2} \le L\,d(x_1,x_2).
    \end{equation}
  \end{enumerate}
\end{lemma}

\begin{remark}\label{rem:coordtrans-loss-smooth}
  One degree of smoothness is lost because the exponential map is
  defined in terms of the geodesic flow. This flow in turn is defined
  in terms of the Christoffel symbols, which depend on derivatives of
  the metric, so these are only $C^{k-1}$ bounded. We lose another
  degree of smoothness in estimating
  $\phi_{2,1} - \partrans(\gamma_{2,1})$ since the Lipschitz estimate
  follows from a uniform bound on one higher derivative of these.
\end{remark}

We shall first compare both $\phi_{2,1}$ and $\partrans(\gamma_{2,1})$
to the identity in normal coordinates and finally conclude with the
triangle inequality that their difference must be small. We compare
$\phi_{2,1}$ to the parallel transport $\partrans(\gamma_{2,1})$ since
this is the most natural way to identify the tangent spaces
$\T_{x_1} M$ and $\T_{x_2} M$.

\begin{proof}
  Let $B(x_1;\delta),\,B(x_2;\delta)$ be two normal coordinate
  neighborhoods with nonempty intersection. The coordinate transition
  map $\phi_{2,1} = \phi_2\circ\phi_1^{-1} = \exp_{x_2}^{-1}\circ\exp_{x_1}$
  can be studied as the exponential map
  $\exp_{x_1}\colon \T_{x_1} M \to M$ in normal coordinates on
  $B(x_2;\delta)$, since $\phi_2 = \exp_{x_2}^{-1}$. From here on, we
  will implicitly be working in normal coordinates around $x_2$, using
  some choice of basis to isometrically identify $\T_{x_2} M \cong \R^n$.

  Let $x \in B(x_1;\delta) \cap B(x_2;\delta)$, hence
  $x_1 \in B(x_2;2\,\delta)$. We choose $\delta \le 1$, and small
  enough so that the results of Theorem~\ref{thm:bound-metric} and
  Proposition~\ref{prop:unif-coord-chart} (with $C = 2$) hold for
  $2\,\delta$. The exponential map is given by the time-one \idx{geodesic
  flow} projected on the base manifold. For the base point $x_2$, this
  is the identity map, while for the base point $x_1$ we will show
  that it is a small perturbation thereof. The geodesic flow on $\T M$
  is given in local coordinates by
  \begin{equation}\label{eq:geodflow-local}
    \begin{aligned}
      \dot{x}^i &= v^i,\\
      \dot{v}^i &= -\Christoffel^i_{jk}(x)\,v^j\,v^k,
    \end{aligned}
  \end{equation}
  \index{$\Gamma$!Christoffel symbols}%
  where $\Christoffel^i_{jk}$ denote the \idx{Christoffel symbols} with
  respect to the coordinates $x^i$ on $M$ and the $v^j$ are induced
  additional coordinates on $\T M$, see the explanation
  above~\ref{eq:TM-metric}. The Christoffel symbols are
  $C^{k-1}$ bounded due to Theorem~\ref{thm:bound-metric}.
  Let $\geodflow^t$ denote the geodesic flow of~\ref{eq:geodflow-local}
  on $\T M$ restricted to $B(x_2;2\,\delta)$. We denote by
  $(x(t),v(t))$ a solution curve of $\geodflow^t$. The geodesic flow
  preserves the length of tangent vectors with respect to the metric
  $g$, so we have $\norm{v(t)} \le 2\,\norm{v(0)} \le 2\,\delta$ with
  respect to the Euclidean distance in the normal coordinates. This
  implies that the vector field~\ref{eq:geodflow-local} is bounded in
  these induced coordinates. Hence, by Theorem~\ref{thm:unif-smooth-flow},
  $\geodflow^t \in \BC^{k-1}$ is bounded as well on the interval
  $\intvCC{0}{1}$. Moreover, $\D\geodflow^t \in \BC^{k-2}$ exhibits a
  Lipschitz estimate for the base point dependence
  $\norm{\phi_2(x_1)}_E$. By Proposition~\ref{prop:unif-coord-chart}
  the local Euclidean distance is equivalent to the distance on $M$,
  so $\norm{\phi_2(x_1)}_E \le 2\,d(x_1,x_2)$. These conclusions
  directly translate to $\exp_x(\slot) = \pi\circ\geodflow^1(x,\slot)$
  and we conclude that
  $\phi_{2,1} = \exp_{x_2}^{-1}\circ\exp_{x_1} \in \BC^{k-1}$ with
  bound $C > 0$ uniform in $x_1,x_2 \in M$ and
  $\norm{\phi_{2,1} - \Id}_{k-2} \le L'\,d(x_1,x_2)$ for some $L' > 0$.

  The parallel transport $\partrans(\gamma_{2,1})$ is given by
  integrating the pullback of the connection along $\gamma_{2,1}$.
  This yields a differential equation similar
  to~\ref{eq:geodflow-local} and similarly leads to $C^{k-1}$
  boundedness estimates in normal coordinates and Lipschitz estimates
  for the $C^{k-2}$\ndash norm. Thus, the difference
  $\phi_{2,1} - \partrans(\gamma_{2,1})$ is $C^{k-1}$ bounded, and
  has $C^{k-2}$\ndash norm that satisfies the Lipschitz
  estimate~\ref{eq:normal-coord-lipschitz} for some $L > 0$.
\end{proof}

\begin{definition}[$M$\ndash small coordinate radius]
  \label{def:M-small}
  \index{$\delta_M$}
  Let $(M,g)$ be a Riemannian manifold of bounded geometry. We define
  $\delta > 0$ to be \emph{$M$\ndash small} if
  Theorem~\ref{thm:bound-metric} and Lemma~\ref{lem:bound-coordtrans}
  hold on all normal coordinate charts of radius $\delta$.
\end{definition}
Note that such a $\delta > 0$ always exists. From now on, we shall
always assume to have selected such a $\delta$ for any given manifold
of bounded geometry and restrict its atlas to include these normal
coordinate charts only.

Lemma~\ref{lem:bound-coordtrans} shows that normal coordinate
transformations respect $C^k$ boundedness of functions in coordinate
representations. Thus, it is natural to consider manifolds of bounded
geometry as the class of $C^k$ bounded manifolds with respect to this
restricted atlas. This also makes the following definition natural.
\begin{definition}[$C^k$ bounded maps]
  \label{def:unif-bounded-map}
  \index{$C^k_b$!definition in bounded geometry}
  Let $X,Y$ be Riemannian manifolds of\/ $k \tp 1$\ndash bounded
  geometry and $f \in C^k(X;Y)$. We say that $f$ is of class $\BC^k$
  when there exist $X,Y$\ndash small $\delta_X,\,\delta_Y > 0$ such
  that for each $x \in X$ we have
  $f(B(x;\delta_X)) \subset B(f(x);\delta_Y)$ and the representation
  \begin{equation}\label{eq:func-coord-repn}
    \tilde{f}_x = \exp_{f(x)}^{-1} \circ f \circ \exp_x
    \colon B(0;\delta_X) \subset \T_x X \to \T_y Y
  \end{equation}
  in normal coordinates is of class $\BC^k$ and the
  associated $C^k$\ndash norms of\/ $\tilde{f}_\bullet$ are bounded
  uniformly in $x \in X$. We define the classes of\/ $\BUC^k(X;Y)$ and
  $\BUC^{k,\alpha}(X;Y)$ functions analogously when $X,Y$ are of
  $k \tp 2$\ndash bounded geometry.
\end{definition}

\begin{remark}\label{rem:unif-bounded-VF}\NoEndMark
  We shall say that a vector field $v \in \vf(X)$ is of class $\BC^k$,
  also denoted by $v \in \vf^k_b(X)$, when $v \in \BC^k$ with respect
  to coordinates on $\T M$ induced by normal coordinates on $M$. This
  is slightly different from normal coordinates on $\T M$ induced by
  the metric~\ref{eq:TM-metric}. Note that since $\norm{v} \le r$ is
  assumed bounded, we could restrict to the submanifold $\T^r M$ of
  bounded geometry and consider $v \in \BC^k(M;\T^r M)$, but this is
  less practical.
\end{remark}

\begin{remark}\NoEndMark
  The manifolds $X,Y$ need to have bounded geometry of one or two
  degrees higher than the smoothness of the maps to preserve
  boundedness and uniform continuity estimates under normal coordinate
  transformations. This shall from now on always be an implicit
  assumption.
\end{remark}

\begin{remark}[Locally/globally defined continuity modulus]
  \label{rem:loc-cont-modulus}
  \index{continuity modulus!in bounded geometry}
  The continuity modulus $\epsilon_f$ of a function
  $f \in \BUC^k(X;Y)$ is only defined on the interval
  $\intvCO{0}{\delta_X} \subset \R$. On the other hand,
  $\norm{\D^k f(x)}$ is globally well-defined in terms local charts
  and assumed to be bounded. We shall want to compare $\D^k f$ at
  points $x_1,\,x_2$ far apart. If we have isometric isomorphisms
  \begin{equation*}
    \phi\colon \T_{x_1} X    \diffto \T_{x_2} X
    \quad\text{and}\quad
    \psi\colon \T_{f(x_1)} Y \diffto \T_{f(x_2)} Y,
  \end{equation*}
  then this allows us to compare
  \begin{equation}\label{eq:global-cont-mod-estimate}
    \norm{\D^k f(x_2)\circ\phi^{\otimes k} - \psi\circ\D^k f(x_1)}
    \le \norm{\D^k f(x_2)} + \norm{\D^k f(x_1)}.
  \end{equation}
  Note that the right-hand expression does not depend on the
  choice\footnote{%
    In practice, we shall use isomorphisms defined by parallel
    transport on $X = Y$, cf.\ws Proposition~\ref{prop:unif-partrans}.
    This is a non-canonical choice, since it depends on the path
    connecting $x_1,\,x_2$. A canonical choice that depends
    continuously on $x_1,\,x_2$ cannot be made in general, since it
    would imply that the tangent bundle is trivializable.%
  } of isomorphisms.

  Thus, with such isomorphisms at hand, we can
  use~\ref{eq:global-cont-mod-estimate} to \emph{heuristically} extend
  the local to a global continuity modulus. That is, for nearby points
  $x_1,\,x_2$ we use an estimate in terms of local charts; if this is
  not possible, then the points must be separated by a distance larger
  than a $\delta$ as in Definition~\ref{def:M-small}. Since the
  functions we consider are globally bounded, we then use some
  (non-canonical) choice to identify the vector bundle fibers over
  $x_1,\,x_2$ that the function lives in and estimate by the
  right-hand side of~\ref{eq:global-cont-mod-estimate}. This estimate
  is crude but independent of the choice of identification and will
  always satisfy our needs. For example, if
  $f \in \BUC^{k,\alpha}(X;Y)$, with H\"older coefficient $C_\alpha$
  locally for $d(x_1,x_2) \le \delta$ then we have
  \begin{equation*}
    \norm{\D^k f(x_2) - \D^k f(x_1)} \le
    \begin{cases}
      C_\alpha\,d(x_1,x_2)^\alpha & \text{if}\; d(x_1,x_2) < \delta,\\
      \frac{2\,\norm{f}_k}{\delta^\alpha}\,
      d(x_1,x_2)^\alpha           & \text{else.}
    \end{cases}
  \end{equation*}
  This shows that we can heuristically consider
  $\max\big(C_\alpha,\frac{2\,\norm{f}_k}{\delta^\alpha}\big)$ as
  a global H\"older coefficient.
\end{remark}

The following proposition shows that we may measure continuity of the
derivatives of  a function $f$ using local parallel transport. With
the remark above we see how it can be extended to a global continuity
modulus if a (non-unique) choice is made for how to connect non-close
points $x_1,\,x_2$ by a path; this idea will be developed in
Section~\ref{sec:formal-tangent}.
\begin{proposition}[Equivalence of continuity moduli]
  \label{prop:unif-partrans}
  \index{continuity modulus!in bounded geometry}
  Let $X,\,Y$ be Riemannian manifolds of bounded geometry and
  $f \in \BC^k(X;Y)$. Then the following statements are equivalent:
  \begin{enumerate}
  \item $f \in \BUC^{k,\alpha}(X;Y)$ according to
    Definition~\ref{def:unif-bounded-map};
  \item we have the continuity estimate
  \begin{equation}\label{eq:unif-cont-partrans}
    \begin{aligned}
\fst\exists\;\epsilon_{f,\partrans} \in C^\alpha(\R_+;\R_+),\,\delta_0 > 0\colon
    \forall\;x_1,x_2 \in X,\,d(x_1,x_2) \le \delta_0\colon\\
\con\norm[\big]{\D^k \tilde{f}_{x_2}(0)\cdot\partrans(\gamma_{2,1})^{\otimes k}
               -\partrans(\eta_{2,1})\cdot\D^k \tilde{f}_{x_1}(0)}
    \le \epsilon_{f,\partrans}(d(x_1,x_2)),
    \end{aligned}
  \end{equation}
  where\/ $\partrans(\eta_{2,1})$ and\/
  $\partrans(\gamma_{2,1})^{\otimes k}$ denote parallel transport
  along the unique shortest geodesic between $f(x_1),\,f(x_2)$ and
  $x_1,\,x_2$, respectively, and $\epsilon_{f,\partrans}$ denotes a
  uniform or $\alpha$\ndash H\"older continuity modulus.
  \end{enumerate}
\end{proposition}

\begin{proof}
  We first prove the statement in case $Y$ is a normed linear space,
  hence no parallel transport term $\partrans(\eta_{2,1})$ appears.

  Let $\delta_0 \le \delta_X$ as in Definition~\ref{def:unif-bounded-map}
  (thus, in particular $\delta_0$ is $X$\ndash small), and let
  $d(x_1,x_2) \le \delta_0$. Then we have the Lipschitz estimate
  $\norm{\phi_{2,1} - \partrans(\gamma_{2,1})} \le L\,d(x_1,x_2)$
  while the normal coordinate representations~\ref{eq:func-coord-repn}
  of $f$ at $x_1,\,x_2$ are related by
  $\tilde{f}_{x_1} = \tilde{f}_{x_2} \circ \phi_{2,1}$. This leads to
  {\jot=5pt
  \begin{align*}
    \fst \norm[\big]{\D^k \tilde{f}_{x_2}(0)\cdot\partrans(\gamma_{2,1})^{\otimes k}
                    -\D^k \tilde{f}_{x_1}(0)}\\
    &=   \norm[\big]{\D^k \tilde{f}_{x_2}(0)\cdot\partrans(\gamma_{2,1})^{\otimes k}
                    -\D^k[\tilde{f}_{x_2}\circ\phi_{2,1}](0)}\\
    &\le \norm[\big]{\D^k \tilde{f}_{x_2}(0)\cdot\partrans(\gamma_{2,1})^{\otimes k}
                     -\D^k \tilde{f}_{x_2}(\phi_{2,1}(0))\cdot\big(\D\phi_{2,1}\big)^{\otimes k}}\\
    \con+\sum_{l=1}^{k-1} \norm{\D^l\tilde{f}_{x_2}(\phi_{2,1}(0))\cdot
                                P_{l,k}\big(\D^\bullet\phi_{2,1}(0)\big)}\displaybreak[1]\\
    &\le \norm[\big]{\D^k \tilde{f}_{x_2}(0)-\D^k \tilde{f}_{x_2}(\phi_2(x_1))}
        +\norm{\D^k \tilde{f}_{x_2}(\phi_2(x_1))}\,\norm{\partrans(\gamma_{2,1}) - \D\phi_{2,1}}^k\\
    \con+\sum_{l=1}^{k-1} \norm{\D^l\tilde{f}_{x_2}(\phi_2(x_1))}\,
                          \norm{P_{l,k}\big(\D^\bullet\phi_{2,1}(0)\big)}\\[-3pt]
    &\le \epsilon_f(d(x_2,x_1)) + \norm{f}_k\,\big(L\,d(x_1,x_2)\big)^k
        +\sum_{l=1}^{k-1} \norm{f}_l\,
                          \norm{P_{l,k}\big(\D^\bullet\phi_{2,1}(0)\big)},
  \end{align*}}
  where $\epsilon_f$ denotes the continuity modulus of $f$ and its
  derivatives according to Definition~\ref{def:unif-bounded-map}, and
  the $P_{l,k}$ denote $(l,k)$\ndash linear maps according to
  Proposition~\ref{prop:compfunc-deriv}. We used the fact that both
  $\partrans(\gamma_{2,1})$ and $\D\phi_{2,1}(0)$ act on the $k$\ndash
  tensor bundle as a $k$\ndash tuple of copies. By assumption
  $\norm{f}_k$ is bounded, and to estimate the $P_{l,k}$ terms, we
  note that $l < k$, so each of the $P_{l,k}$ contains at least a
  factor $\D^i\phi_{2,1}(0)$ with $i \ge 2$. Since $\phi_{2,1}$ is
  close to $\partrans(\gamma_{2,1})$ and
  $\D^i\partrans(\gamma_{2,1}) = 0$ for $i \ge 2$, it follows that
  \begin{equation*}
    \norm{P_{l,k}\big(\D^\bullet\phi_{2,1}(0)\big)}
    \le C\,L\,d(x_1,x_2)
  \end{equation*}
  for some constant $C$ independent of $x_1,x_2$. This shows that the
  continuity modulus $\epsilon_{f,\partrans}$
  of~\ref{eq:unif-cont-partrans} can be estimated by the continuity
  modulus $\epsilon_f$ plus additional Lipschitz terms. We can reverse
  the estimates above to arrive at the same conclusion when expressing
  $\epsilon_f$ in terms of $\epsilon_{f,\partrans}$. Hence, the
  continuity statements are equivalent for any $\alpha \le 1$.

  \enlargethispage*{0.3cm}
  If $Y$ is a Riemannian manifold of bounded geometry, we just apply
  the same estimates in the codomain. To this end, we must have
  $d(f(x_2),f(x_1)) \le \delta_Y$, so we choose $\delta_0$ small
  enough that
  \begin{equation*}
    d(f(x_2),f(x_1))
    \le \norm{\D f}\,d(x_2,x_1)
    \le \norm{f}_1\,\delta_0 \le \delta_Y
  \end{equation*}
  holds with $\delta_Y$ as in Definition~\ref{def:unif-bounded-map}.
\end{proof}

The definition of bounded geometry can be extended to vector bundles,
see also~\cite[p.~65]{Shubin1992:spec-elloper-noncpt}.
\begin{definition}[Vector bundle of bounded geometry]
  \label{def:BG-bundle}
  \index{bounded geometry!definition for vector bundle}
  Let $(M,g)$ be a manifold of bounded geometry and $\delta$ be
  $M$\ndash small as in Definition~\ref{def:M-small}. We say that a
  vector bundle $\pi\colon E \to M$ with fiber $F$ has $k$\th order
  bounded geometry when there exist preferred trivializations
  \begin{equation}\label{eq:BG-bundle-triv}
    \tau\colon \pi^{-1}\big(B(m;\delta)\big) \to B(m;\delta) \times F
    \qquad\text{for each } m \in M
  \end{equation}
  such that if we have a transition function
  $\phi_{2,1} = \tau_2\circ\tau_1^{-1}$ between two trivializations on
  $B(m_1;\delta)$ and $B(m_2;\delta)$, then the function
  $g\colon B(m_1;\delta) \cap B(m_2;\delta) \to \CLin(F)$ defined by
  $\phi_{2,1}(m,f) = g(m)\,f$ satisfies $g \in \BC^k$ independent of
  the points $m_1,m_2 \in M$.
\end{definition}

\begin{remark}
  Note that we could have replaced $B(m;\delta)$ by arbitrary
  (preferred) coordinate charts. The relevant property is that we
  express uniformity of the transition functions in terms of
  uniformity of the function $g$ with respect to the underlying
  coordinate charts of $M$, which are normal coordinates in our case.
\end{remark}

It follows from Lemma~\ref{lem:bound-coordtrans} that the tangent
bundle $\T M$ has bounded geometry of order $k-2$ if $(M,g)$ has
bounded geometry of order $k \ge 2$. One order of smoothness is lost
(beyond the one expected) as noted in
Remark~\ref{rem:coordtrans-loss-smooth}.

We introduce the concept of a uniformly locally finite cover of a
manifold of bounded geometry. This is a natural extension of a locally
finite cover. Uniformity means that we require a global bound $K$ on
the number of sets in the cover that intersect any small open ball.
\begin{lemma}[Uniformly locally finite cover]
  \label{lem:unif-loc-cover}
  \index{uniformly locally finite cover}
  Let $(M,g)$ be a Riemannian manifold of bounded geometry.

  Then for $\delta_2 > 0$ small enough and any\/
  $0 < \delta_1 \le \delta_2$, $M$ has a countable cover\/
  $\big\{B(x_i;\delta_1)\big\}_{i \ge 1}$ such that
  \begin{enumerate}
  \item $\forall\; i \neq j\colon d(x_i,x_j) \ge \delta_1$;
  \item there exists an explicit global bound \idx{$K$} such that for each
    $x \in M$ the ball $B(x;\delta_2)$ intersects at most $K$ of the
    $B(x_i;\delta_2)$.
  \end{enumerate}
\end{lemma}
Note that the second result implies both that the cover is locally
finite with fixed neighborhood size, and that each set in the cover
overlaps with at most $K$ others, cf. Lebesgue covering dimension.

\begin{proof}
  Using Proposition~\ref{prop:unif-coord-chart}, choose $\delta > 0$
  such that Euclidean distance in normal coordinates on each
  $B(x;\delta)$ is $C = 2$ equivalent to the metric distance and set
  $\delta_2 \le \delta/3$.

  Let $\{M_k\}_{k \in \N}$ be a compact exhaustion of $M$. Cover $M_k$
  with a sequence of balls $B(x_i;\delta_1)$, where
  $d(x_i,x_j) \ge \delta_1$. This sequence is finite, because an
  infinite sequence $\{x_i\}_{i \ge 0}$ must have an accumulation
  point in $M_k$, which contradicts $d(x_i,x_j) \ge \delta_1$.
  Choosing the first $x_i$'s in $M_{k+1}$ to coincide with those of
  $M_k$, it follows that the union of all balls
  $B(x_i,\delta_1)$ is a countable cover of $M$ such that
  $\forall\; i \neq j\colon d(x_i,x_j) \ge \delta_1$.

  Let $x \in M$ arbitrary. Any ball $B(x_i;\delta_2)$ that intersects
  $B(x;\delta_2)$ must be completely contained in $B(x;3\,\delta_2)$.
  Each of these balls has an exclusive subset $B(x_i;\delta_1/2)$, so
  in normal coordinates around $x$, each has an exclusive volume of at
  least $\textrm{Vol}\big(B(0;\delta_1/(2C))\big)$, while
  $B(x;\delta)$ has volume of at most
  $\textrm{Vol}\big(B(0;C\,3\,\delta_2)\big)$. With $n = \dim(M)$,
  this leads to the explicit upper bound
  \begin{equation}\label{eq:unif-cover-bound}
    K \le \frac{(3\,C\,\delta_2)^n}{(\delta_1/(2\,C))^n}
      =   \Big(24\,\frac{\delta_2}{\delta_1}\Big)^n.
  \end{equation}
  Thus, only finitely many can intersect $B(x;\delta_2)$. These
  estimates are uniform and do not depend on $x \in M$ so the bound
  $K$ is global.
\end{proof}

\begin{lemma}[Uniform partition of unity]
  \label{lem:part-unity}
  \index{partition of unity (uniform)}
  Let $M$ be a manifold with a uniformly locally finite cover with
  $\delta_1 < \delta_2$ and $\delta_2$ sufficiently small, as per
  Lemma~\ref{lem:unif-loc-cover}.

  Then there exists a partition of unity by functions
  $\chi_\bullet \in \BUC^k(B(x_i;\delta_2);\intvCC{0}{1})$ subordinate
  to this cover.
\end{lemma}
We shall also apply this lemma to submanifolds which have a uniformly
locally finite cover due to Corollary~\ref{cor:submfld-cover} on
page~\pageref{cor:submfld-cover}.

{\hfuzz=2.1pt
\begin{proof}
  Let $\delta_2$ be small enough that by Lemma~\ref{lem:bound-coordtrans}
  coordinate transition maps are $\BUC^k$.
  Define a standard radially symmetric smooth bump function
  $\phi \in C^\infty(\R^n;\intvCC{0}{1})$ that is identically one on
  $B(0;\delta_1)$ and has compact support in $B(0;\delta_2)$, hence
  $\phi \in \BUC^k$. We set $\phi_i = \phi \circ \exp_{x_i}^{-1}$ by
  isometric identification $\T_{x_i} M \cong \R^n$ and zero outside
  $B(x_i;\delta_2)$. We have $\phi_\bullet \in \BUC^k$ in any
  coordinate patch. Define in the usual way
  \begin{equation}\label{eq:part-unity-func}
    \chi_i = \phi_i \;\big/\; \sum_{n \ge 1} \phi_n.
  \end{equation}
  The sum is finite as at most $K$ of the $B(x_n;\delta_2)$ overlap
  any $B(x_i;\delta_2)$. The balls $B(x_i;\delta_1)$ already cover
  $M$, so the denominator is at least one, from which it follows that
  $\chi_\bullet \in \BUC^k$.
\end{proof}
}

\begin{corollary}
  \label{cor:part-squared-unity}
  Similar to a uniform partition of unity, we can construct a
  partition by functions $\chi_\bullet \in
  \BUC^k(B(x_i;\delta_2);\intvCC{0}{1})$ whose squares sum to one.
\end{corollary}

In the proof of Lemma~\ref{lem:part-unity} we simply
replace~\ref{eq:part-unity-func} by
\begin{equation}\label{eq:part-squared-unity-func}
  \chi_i = \phi_i \;\big/\; \sqrt{\,\rule[-1.1em]{0pt}{2.2em}\smash{\sum_{n \ge 1} \phi_n^2}}.
\end{equation}

\section{Curvature and holonomy}
\label{sec:BG-curv-holonomy}

To prove smoothness of the persistent manifold in
Section~\ref{sec:NHIM-smoothness}, we shall want to estimate the
holonomy along closed loops to be close to the identity, that is, if
$c$ is a closed loop, then we want $\partrans(c) - \Id$ to
be small. To this end, we relate the holonomy to the curvature and
finally obtain an estimate in terms of a global bound on the curvature
and the area of a surface enclosed by $c$.

The result that \idx{curvature} is the generator of \idx{holonomy}
dates back at least to
Ambrose and Singer~\cite{Ambrose1953:holonomy} who formulated this in
differential form in the 1950's; they cite an even older statement
(without proof) by \'Elie Cartan~\cite{Cartan1926:groupes-holonomie}.
More recent work by Reckziegel and
Wilhelmus~\cite{Reckziegel2006:curv-gen-hol} shows explicit integral
formulas for this relation, formulated on fiber bundles, a context
far more general than is required here. We shall present a formulation
for Riemannian manifolds $(M,g)$.

Let $\partrans$ denote the parallel transport functional, which takes
$C^1$ curves to orthogonal maps between the tangent spaces at their
endpoints, see~\ref{eq:partrans-RG}. If $c$ is a closed loop,
then $\partrans(c)$ is a linear endomorphism on $\T_{c(0)} M$ and we
can measure $\norm{\partrans(c) - \Id}$. Our goal is to bound
this quantity by the integral of the curvature form $R$ over a surface
with boundary precisely $c$. This result can be viewed as a
generalization of Stokes' theorem where the curvature is the exterior
derivative of the connection form $\omega$, while the connection on
the other hand generates parallel transport along the boundary of the
surface $A$ that the curvature is integrated over. Note though, that
we actually have $R = \d \omega + \omega \wedge \omega$, so there is
an additional term due to the noncommutativity of the connection form.

{
\newcommand{\Pf}{{}_f\partrans}
\newcommand{\ds}{\mfrac{\partial}{\partial s}}
\newcommand{\dt}{\mfrac{\partial}{\partial t}}

Let
\begin{equation}\label{eq:path-homotopy}
  \gamma\colon D = \intvCC{0}{\smash{\bar{t}}} \times \intvCC{0}{\smash{\bar{s}}} \to M
        \colon (t,s) \mapsto \gamma(t,s)
\end{equation}
parametrize the surface $A = \gamma(D) \subset M$. The idea is that
$\gamma$ is the homotopy of a (closed) curve $c$.
We shall only consider parallel transport along
horizontal or vertical lines in $D$; let us
denote by $\partrans^{s_2,s_1}_t$ \idx{parallel transport} along
$s \mapsto \gamma(s,t)$ with $s \in \intvCC{s_1}{s_2}$ and by
$\partrans^s_{t_2,t_1}$ parallel transport along
$t \mapsto \gamma(s,t)$ with $t \in \intvCC{t_1}{t_2}$.

We shall calculate the holonomy along $\partial A$ with respect to a
chosen frame on the pullback bundle $\gamma^*(\T M)$. The final result
will turn out to be independent of this choice, hence it is covariantly
defined. Let $f$ be an orthonormal frame on $\gamma^*(\T M)$, that is,
$f_{t,s}\colon \T_{\gamma(t,s)} M \to \R^n$ is an isometry of inner
product spaces. The
Levi-Civita connection $\nabla$ on $M$ can be pulled back to the
connection $\gamma^*(\nabla)$ on $\gamma^*(\T M)$ and it can be
expressed in terms of the connection form
$\omega \in \Omega^1\big(D;\text{End}(\R^n)\big)$ with respect to the
frame $f$. The curvature of $\gamma^*(\nabla)$ is equal to the
curvature $R$ of $\nabla$ pulled back to $D$, so we have
${\rm d}\omega + \omega \wedge \omega = \gamma^*(R)_f$, where the
subscript $f$ indicates that everything is expressed with respect to
the chosen frame. In the same notation, parallel transport along a
curve $s \mapsto c(s)$ satisfies
the linear, homogeneous differential equation\footnote{%
  If the frame $f$ is induced by local coordinates, then $\omega$ will
  precisely be given by the Christoffel symbols and we recover
  equation~\ref{eq:partrans-local-RG}.%
}
\index{parallel transport!differential equation}
\begin{equation}\label{eq:partrans-ODE-frame}
  \der{}{s}\Pf(c|_0^s) = -\omega(\dot{c}(s)) \circ \Pf(c|_0^s),
  \qquad \Pf(c|_0^0) = \Id_{\R^n},
\end{equation}
which has a unique solution $s \mapsto \Pf^{s,0} = \Pf(c|_0^s)$. This
can be viewed as time-dependent flow in $\text{End}(\R^n)$.

\begin{figure}[htb]
  \centering
  \input{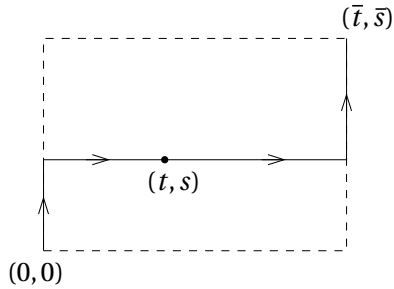}
  \caption{the path of the parallel transport term $P(s)$ in $D$.}
  \label{fig:partrans-P}
\end{figure}
Let us define the parallel transport term
\begin{equation}
  P(s) = \partrans^{\bar{s},s}_{\bar{t}}\circ
         \partrans^s_{0,\bar{t}}        \circ
         \partrans^{s,0}_0
         \colon \T_{\gamma(0,0)} M \to \T_{\gamma(\bar{t},\bar{s})} M,
\end{equation}
see Figure~\ref{fig:partrans-P}. The holonomy defect can be expressed
as
\begin{equation*}
    \Id - \partrans(\partial A)
  = \Id - \partrans(\partial D)
  = \Id - P(\bar{s})^{-1} \circ P(0)
  = P(\bar{s})^{-1} \circ \big(P(\bar{s}) - P(0)\big),
\end{equation*}
where $\partrans(\partial D)$ is defined using the pullback
connection. We use the fundamental theorem of calculus to write
\begin{equation}\label{eq:holonomy-int}
  P(\bar{s}) - P(0) = \int_0^{\bar{s}} \der{P(s)}{s} \d s.
\end{equation}

Expressing everything with respect to the frame $f$, we see that the
first and last factor of $P(s)$ are easily differentiated
using~\ref{eq:partrans-ODE-frame}:
\begin{equation}\label{eq:var-partrans-endpoints}
  \der{}{s}\Pf^{s,0}_0               = -\omega(\ds) \circ \Pf^{s,0}_0
  \quad\text{and}\quad
  \der{}{s}\Pf^{\bar{s},s}_{\bar{t}} = \Pf^{\bar{s},s}_{\bar{t}} \circ \omega(\ds).
\end{equation}
The middle term $\Pf^s_{\bar{t},0}$ can be differentiated by viewing
$s$ as parameter in the differential
equation~\ref{eq:partrans-ODE-frame}. Variation of constants yields
(see e.g.~\cite[App.~B]{Duistermaat2000:liegroups} for a proof of the
differentiable dependence of a flow on parameters)
{\hfuzz=31pt \jot=8pt
\begin{align*}
     \der{}{s}\Pf^s_{\bar{t},t}
  &= \int_0^{\bar{t}}
       \Pf^s_{\bar{t},t}\circ\der{}{s}\big[-\omega(\dt)\big]\circ\Pf^s_{t,0} \;\d t\\
  &= \int_0^{\bar{t}}
      -\Pf^s_{\bar{t},t}\circ\Big(
       {\rm d}\omega(\ds,\dt) + \der{}{t}\big[\omega(\ds)\big] + \omega(\lie[\big]{\ds}{\dt})
       \Big)\circ\Pf^s_{t,0} \;\d t\\
\intertext{%
  using standard rules for exterior derivatives. Next we note that
  $\lie[\big]{\ds}{\dt} = 0$, and integrate by parts the term
  $\der{}{t}\big[\omega(\ds)\big]$
}%
  &= \begin{aligned}[t]
   & \int_0^{\bar{t}} \begin{aligned}[t]
      &-\Pf^s_{\bar{t},t}\circ\Big(
         -\omega(\dt)\circ\omega(\ds)
         +{\rm d}\omega(\ds,\dt)
         +\omega(\ds)\circ\omega(\dt)
        \Big)\circ\Pf^s_{t,0} \;\d t \end{aligned}\\
   &-\Big[ \Pf^s_{\bar{t},t}\circ\omega(\dt)\circ\Pf^s_{t,0} \Big]_{t=0}^{\bar{t}}
     \end{aligned}\displaybreak[1]\\
  &= \int_0^{\bar{t}}
        \Pf^s_{\bar{t},t}\circ({\rm d}\omega + \omega\wedge\omega)\big(\dt,\ds\big)
        \circ\Pf^s_{t,0} \;\d t \;{}
    -\omega(\ds)\circ\Pf^s_{\bar{t},0} + \Pf^s_{\bar{t},0}\circ\omega(\ds).
     \eqnumber\label{eq:var-partrans}
\end{align*}}%
We see that this variation depends on the curvature form
$\gamma^*(R)_f = {\rm d}\omega + \omega\wedge\omega$ along the path and
two additional boundary terms. If we view $\gamma$ as a homotopy of
paths with homotopy parameter $s$ and we keep the path endpoints
$\gamma(0,s)$ and $\gamma(\bar{t},s)$ fixed for all
$s \in \intvCC{0}{\bar{s}}$, then these boundary terms vanish and the
result~\ref{eq:var-partrans} agrees with~\cite[Cor.~3]{Reckziegel2006:curv-gen-hol}.

Instead, we insert~\ref{eq:var-partrans-endpoints} and~\ref{eq:var-partrans}
into~\ref{eq:holonomy-int}. Then these boundary terms cancel against
the terms from~\ref{eq:var-partrans-endpoints} and we finally obtain
\begin{equation}\label{eq:holonomy-curv}
  \begin{aligned}
  P(\bar{s})_f - P(0)_f
  &= \int_0^{\bar{s}} \int_0^{\bar{t}}
       \Pf^{\bar{s},s}_{\bar{t}}\circ\Pf^s_{\bar{t},t}
       \circ\gamma^*(R)_f(\dt,\ds)\circ
       \Pf^s_{t,0}\circ\Pf^{s,0}_0 \d t \d s\\
  &= \Bigg(\int_D
       \partrans^{\bar{s},s}_{\bar{t}}\circ\partrans^s_{\bar{t},t}
       \circ\gamma^*(R)\circ
       \partrans^s_{t,0}\circ\partrans^{s,0}_0\Bigg)_f.
  \end{aligned}
\end{equation}
The integrand on the last line is a two-form on $D$ with
values in $\CLin(\T_{\gamma(0,0)}M;\T_{\gamma(\bar{t},\bar{s})} M)$.
This final expression is clearly independent of a choice of frame, so
we have recovered an explicit integral formula relating holonomy along
a null-homotopic loop to the curvature.
} 

We conclude from~\ref{eq:holonomy-curv} that if $c$ is a closed,
null-homotopic loop, and the curvature globally bounded, then
$\norm{\partrans(c) - \Id}$ can be estimated by $\norm{R}_{\sup}$
times the surface area of any null-homotopy $\gamma$ of $c$. Note that
we do not require $\gamma$ to be an embedding; the integral is
intrinsically defined on $D$ by pullback. Furthermore, $\gamma$ is
required to be $C^1$ only. This follows from the fact that both
sides of the equation are continuous with respect to $\gamma$ in
$C^1$\ndash norm; alternatively, an explicit calculation requires
that the mixed partial derivative
$\frac{\partial^2\,\gamma}{\partial s\,\partial t}$ is continuous to
perform integration by parts. Both lead to the to the following result.
\begin{lemma}[Exponential growth bound on holonomy]
  \label{lem:bound-holonomy}
  Let $(M,g)$ be a manifold of bounded geometry with normal coordinate
  radius $\delta$ that is $M$\ndash small as in Definition~\ref{def:M-small}. Fix
  $T > 0$ and $\rho > 0$ and let $x_1, x_2$ be two $C^1$ curves on $M$
  with derivatives bounded by $N$ such that
  $d_\rho(x_1,x_2)\,e^{\rho\,T} \le \delta < \rinj{M}$. Denote by
  $\gamma_t$ the unique shortest geodesic connecting $x_1(t)$ to
  $x_2(t)$ for any $t \in \intvCC{0}{T}$.

  If $\delta$ is sufficiently small, then the closed loop
  $\eta = x_2|_T^0 \circ \gamma_T \circ x_1|_0^T \circ \gamma_0^{-1}$
  satisfies the holonomy bound
  \begin{equation}\label{eq:bound-holonomy}
    \norm{\partrans(\eta) - \Id}
    \le \tilde{C}\,\norm{R}_0\,N\,d_\rho(x_1,x_2)\,\frac{e^{\rho\,T}}{\rho}
  \end{equation}
  where $\tilde{C}$ depends on the geometry of $M$ only.
\end{lemma}

\begin{proof}
  The two-parameter family $(s,t) \mapsto \gamma_t(s)$ defines a
  null-homotopy of the closed loop $\eta$. The map
  $s \mapsto \gamma_t(s)$ is defined through the exponential map as
  \begin{equation*}
    \gamma_t\colon \intvCC{0}{1} \to M
            \colon s \mapsto \exp_{x_1(t)}\big(s\,\exp_{x_1(t)}^{-1}(x_2(t))\big).
  \end{equation*}
  Since $\exp_x$ is a local diffeomorphism at
  least for $d(x_1(t),x_2(t)) < \delta\,e^{\rho\,t} < \rinj{M}$, that
  depends smoothly on $x$, it follows that $(s,t) \mapsto \gamma_t(s)$
  defines a homotopy between the curves $x_1,x_2$ restricted to the
  interval $\intvCC{0}{T}$. The map $\gamma_t(s)$ has continuous mixed
  derivatives with respect to $s,t$ (even though the double derivative
  with respect to $t$ does not exist since $x_1, x_2 \in C^1$ only),
  so integration by parts is allowed in~\ref{eq:var-partrans}.

  We estimate the surface area mapped by $\gamma_t(s)$. We use
  shorthand notation
  $\xi = s\,\exp_{x_1(t)}^{-1}(x_2(t)) \in \T_{x_1(t)} M$ and denote
  by $\D_x\exp_x$ the derivative of the exponential map with respect
  to the base point parameter $x$. Then
  \begin{align*}
    \der{}{s} \gamma_t(s)
    &= \D\exp_{x_1(t)}(\xi) \cdot \exp_{x_1(t)}^{-1}(x_2(t)),\\
    \der{}{t} \gamma_t(s)
    &= \D_x\exp_{x_1(t)}(\xi) \cdot \dot{x}_1(t)\\
    \con+\D\exp_{x_1(t)}(\xi) \cdot
       \big[s\,\D_x(\exp_{x_1(t)}^{-1})(x_2(t)) \cdot \dot{x}_1(t)
           +s\,\D   \exp_{x_1(t)}^{-1} (x_2(t)) \cdot \dot{x}_2(t)\big].
  \end{align*}
  Since $M$ has bounded geometry, $\D\exp_x$ and its inverse are
  bounded by Theorem~\ref{thm:bound-metric}, while $\D_x\exp_x$ and
  its inverse are bounded by Lemma~\ref{lem:bound-coordtrans}, say by
  $C > 1$. This leads to estimates
  \begin{align*}
    \norm[\Big]{\der{}{s} \gamma_t(s)}
    &\le C\,d(x_1(t),x_2(t)),\\
    \norm[\Big]{\der{}{t} \gamma_t(s)}
    &\le C\,\norm{\dot{x}_1(t)}
        +C\,s\big[C\,\norm{\dot{x}_1(t)} + C\,\norm{\dot{x}_2(t)}\big]
     \le 3\,C^2\,N,
  \end{align*}
  so the holonomy bound satisfies
  \begin{align*}
    \norm{\partrans(\eta) - \Id}
    &\le \norm{R}_0\,\int_0^1\int_0^T
                \norm[\Big]{\der{}{s} \gamma_t(s)}\,
                \norm[\Big]{\der{}{t} \gamma_t(s)} \d t \d s\\
    &\le \norm{R}_0\,\int_0^T 3\,C^3\,N\,d_\rho(x_1,x_2)\,e^{\rho\,t} \d t\\
    &\le 3\,C^3\,\norm{R}_0\,N\,d_\rho(x_1,x_2)\,\frac{e^{\rho\,T}}{\rho}.
  \end{align*}
\end{proof}

\begin{remark}
  It should be possible to obtain $\tilde{C} = 1$ if the curves $x_i$
  are generated by a flow $\Phi$ and we choose as homotopy
  $(s,t) \mapsto \Phi^t(\gamma(s))$, where $\gamma$ is the geodesic
  connecting $x_1(0)$ and $x_2(0)$. In our applications, though, the
  curves $x_1, x_2$ need not be solutions to exactly the same flow,
  while the current result is sufficient for our purposes.
\end{remark}

\section{Submanifolds and tubular neighborhoods}
\label{sec:BG-submanifolds}

From this section on, we shall prove results that---although they may
be of interest independently within bounded geometry---are building up
towards the final section of this chapter, where we prove how to
reduce Theorem~\ref{thm:persistNHIMgen} on persistence in
general manifolds of bounded geometry to the setting of a trivial
bundle. These results form the more technical part of this chapter and
are not required elsewhere.

In the following, we assume that $(Q,g)$ is an ambient manifold
that has bounded geometry of large or infinite order and
$M \in C^k$ will denote a submanifold of $Q$. Only a finite
order $l > k$ of bounded geometry is required of $(Q,g)$, but for
simplicity we shall assume $l = \infty$.
Recovering the explicit additional order $l \tm k$ would
amount to tediously tracking the details throughout all the proofs; it
should be sufficient if $l$ is larger than $k$ by some number between
$2$ and $10$.

Let $\iota\colon M \to Q$ be a $C^1$ immersion. With abuse of notation
we denote by $\T_x M = \text{Im}(\D\iota(x))$ and
$N_x = \text{Im}(\D\iota(x))^\perp$ the tangent and normal spaces of
$M$ with respect to the immersion. Note that even if $\iota$ is not
injective, the original point $x \in M$ uniquely selects the tangent
and normal spaces in $\T_{\iota(x)} Q$.
\begin{definition}[Uniformly immersed submanifold]
  \label{def:unif-imm-submfld}
  \index{uniformly immersed submanifold}
  \index{$M_{x,\delta}$}
  \index{$h_x$}
  Let $\iota\colon M \to Q$ be a $C^{k \ge 1}$ immersion of $M$ into
  the Riemannian manifold $(Q,g)$ of bounded geometry.
  Denote by $M_{x,\delta}$ the image under $\iota$ of the connected component of $x$
  in $\iota^{-1}\big(B(\iota(x);\delta) \cap \iota(M)\big)$. We define $M$ to
  be a $\BUC^k$ immersed submanifold when there exists a $\delta > 0$
  such that for all $x \in M$, the connected component $M_{x,\delta}$
  is represented in normal coordinates on $B(\iota(x);\delta) \subset Q$ by
  the graph of a function $h_x\colon \T_x M \to N_x$ and the family of
  functions $h_\bullet \in \BUC^k(\T_\bullet M;N_\bullet)$ has uniform
  continuity and boundedness estimates independent of $x$. We define
  $\BC^{k \ge 1}$ immersions in a similar way.
\end{definition}

\begin{figure}[b]
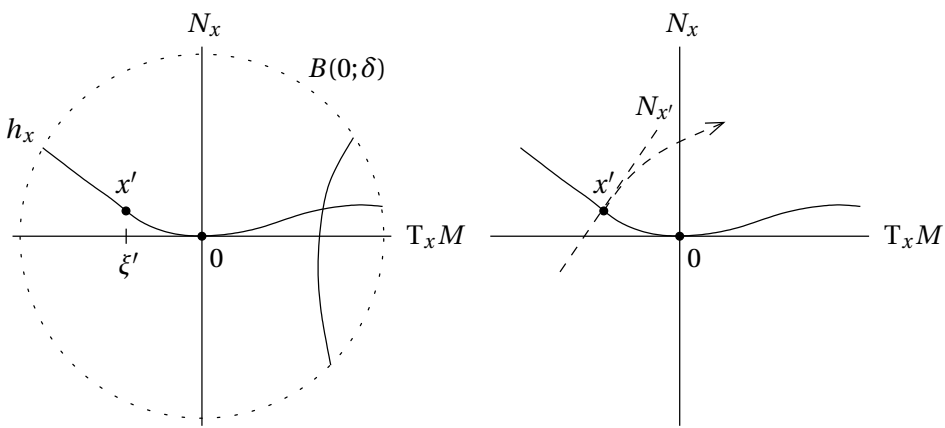

  \centering
  \hspace{-0.7cm}
  \input{\figpath normcoord-submfld.pspdftex}\hspace{1cm}
  \input{\figpath normcoord-normgeod.pspdftex}
  \caption{an immersed submanifold represented by the graph of $h_x$
    in normal coordinates. In the left figure, another part of $M$
    intersects transversely on the right; the right figure contains an
    orbit of the geodesic flow along a normal vector at $x' \cong
    \iota(x')$.}
  \label{fig:normcoord-submfld}
\end{figure}

\begin{remark}\label{rem:embedding}\NoEndMark
  By taking the connected component $M_{x,\delta}$ in $M$, we allow for
  immersed submanifolds that intersect, or nearly intersect themselves. See
  Figure~\ref{fig:normcoord-submfld} on the left: $M_{x,\delta}$ is described by
  the graph of $h_x$, while on the right side, a different part of $M$
  embeds into this same neighborhood $B(\iota(x);\delta)$. See
  Figure~\ref{fig:selfintersect} on page~\pageref{fig:selfintersect}
  for an example of a nearly self-intersecting submanifold. If we want
  to rule out such cases, we can assume that $M_{x,\delta}$ is the
  unique component of $M \cap B(\iota(x);\delta)$.
  This will turn $M$ into an embedded submanifold, but more strongly,
  the nearly self-intersecting case is also ruled out. We
  will refer to this as a \emph{\idx{uniformly embedded submanifold}}.
\end{remark}

\begin{remark}\label{rem:plaque}\NoEndMark
  The sets $M_{x,\delta}$ play a similar role as `plaques'
  in~\cite[p.~72--73]{Hirsch1977:invarmflds}.
\end{remark}

\begin{remark}\label{rem:imm-cont2bound}
  In case $k = 1$, boundedness is automatically implied by uniform
  continuity. This follows from the representation in normal
  coordinates. We have $\D h_x(0) = 0$, so by uniform continuity there
  exists a $\delta > 0$ such that $\norm{\D h_x(\xi)} < \epsilon = 1$
  when $\norm{\xi} < \delta$, hence $\D h_x$ is bounded. Put another
  way, there is no intrinsic measure for the `size of the derivative
  or tangent' of a submanifold.
\end{remark}

Note that the function $h_x$ is only defined on that part of the
domain $B(0;\delta) \subset \T_x M$ where its graph is
contained in $B(0;\delta) \subset \T_{\iota(x)} Q$, as can be seen in
Figure~\ref{fig:normcoord-submfld}. In the splitting
$\T_{\iota(x)} Q = \T_x M \oplus N_x$, we denote with $p_1, p_2$
orthogonal projections onto the $\T_x M$ and $N_x$ subspaces,
respectively.

From now on we shall continually assume that
$M \in \BUC^{k \ge 1}$ is a uniformly immersed submanifold of $Q$.
We will often identify $M$ with its image $\iota(M) \subset Q$, as
well as identify points $x \in M$ with $\iota(x)$,
keeping in mind the definition of $M_{x,\delta}$ to track local
injectivity. Furthermore, denote by $d_M$ the distance on $M$
induced by the pulled back Riemannian metric $\iota^*(g)$. This
distance function measures whether points are close when viewed along
the domain of the immersion, disallowing `shortcuts' through $Q$. It
also distinguishes different points with the same immersion image.
Note that it is different from the distance $d$ on $Q$ pulled back to
$M$. This we denote by $d_Q = \iota^*(d)$ but it is not a
distance on $M$ when $\iota$ is not injective. Still, we have the
following local result, which will be useful for later estimates.
\begin{lemma}[Local equivalence of distance]
  \label{lem:loc-equiv-Mdist}
  Let $M \in \BUC^1$ be a uniformly immersed submanifold of the bounded
  geometry manifold $(Q,g)$. Then $d_Q$ and $d_M$ are locally
  equivalent in the following sense:
  \begin{enumerate}
  \item $\forall x_1,x_2 \in M\colon d_Q(x_1,x_2) \le d_M(x_1,x_2)$;
  \item for any $C' > 1$ there exists a $\delta > 0$ such that for all
    $d_M(x_1,x_2) < \delta$, we have the local converse
    $d_M(x_1,x_2) \le C'\,d_Q(x_1,x_2)$.
  \end{enumerate}
\end{lemma}

\begin{proof}
  The first assertion follows directly from the fact that any path in
  $M$ induces a path of equal length in $Q$ via the immersion $\iota$.

  For the second part, we first note that if $\delta$ is small enough
  and $d_M(x_1,x_2) < \delta$, then we must have
  $x_2 \in M_{x_1,\delta}$. If this would not be the case, then any
  path $\gamma$ connecting $x_1,x_2$ through $M$ cannot be contained
  in $M_{x_1,\delta}$. But this implies that the path runs out of
  $B(x_1;\delta)$, so its length is greater than $\delta$. This
  contradicts the assumption that $d_M(x_1,x_2) < \delta$. Hence,
  $x_2$ can be represented as a point on the graph of $h_{x_1}$ in
  $B(x_1;\delta)$.

  Let $C > 1,\,\epsilon > 0$ be constants to be fixed later and let
  $\delta$ be small enough such that the metric coefficients are
  bounded by $C$ in normal coordinate charts, that
  Proposition~\ref{prop:unif-coord-chart} holds with $C$, and we have
  $\norm{h_\bullet}_1 \le \epsilon$ as in
  Remark~\ref{rem:imm-cont2bound}. We consider the normal coordinate
  chart on $B(x_1;\delta)$ and construct a path in $M$ to find an
  upper bound for $d_M(x_1,x_2)$. Let $x_2 = (\xi,h_{x_1}(\xi))$ and
  define $\gamma(t) = (t\,\xi,h_{x_1}(t\,\xi))$ for
  $t \in \intvCC{0}{1}$. We estimate the length of $\gamma$ as
  \begin{equation*}
    l(\gamma) \le \int_0^1 \sqrt{\norm{g}}\,\sqrt{1+\norm{h_\bullet}_1^2}\,\norm{\xi} \d t
              \le \sqrt{C}\,\sqrt{1 + \epsilon^2}\,\norm{\xi},
  \end{equation*}
  while the Euclidean norm can be estimated by the distance in $Q$ as
  \begin{equation*}
    \norm{\xi} \le \norm{(\xi,h_{x_1}(\xi))} \le C\,d(x_1,x_2).
  \end{equation*}
  We conclude that
  \begin{equation*}
    d_M(x_1,x_2) \le l(\gamma) \le C^{3/2}\,\sqrt{1+\epsilon^2}\,d_Q(x_1,x_2)
  \end{equation*}
  and for any $C' > 1$ we can find $C > 1,\,\epsilon > 0$ such that
  $C^{3/2}\,\sqrt{1+\epsilon^2} < C'$.
\end{proof}

A uniform submanifold of a bounded geometry manifold can be shown to
possess a uniformly locally finite cover as a corollary of
Lemma~\ref{lem:unif-loc-cover}, without the need to show that the
submanifold itself has bounded geometry. As a consequence, it also has
(square-sum) partitions of unity.
\begin{corollary}[Uniform cover of a submanifold]
  \label{cor:submfld-cover}
  \index{uniformly locally finite cover!of a submanifold}
  Let $M \in \BUC^1$ be a uniformly immersed submanifold of the bounded
  geometry manifold $(Q,g)$.

  Then for $\delta_2 > 0$ small enough and any
  $\delta_1 \in \intvOC{0}{\delta_2}$, $M$ has a uniformly locally finite
  cover by balls of radius $\delta_2$ in terms of the distance $d_Q$,
  such that the balls of radius $\delta_1$ already cover $M$. That is,
  there exist $\{x_i\}_{i \ge 1}$ such that\/
  $\bigcup_{i \ge 1} M_{x_i,\delta_1}$ covers $M$ with a uniform bound
  $K$ on the maximum number of sets $M_{x_i,\delta_2}$ covering any
  set $M_{x,\delta_2}$ with $x \in M$.
\end{corollary}

\begin{proof}
  The proof follows the ideas of Lemma~\ref{lem:unif-loc-cover}. As an
  additional requirement, let $\delta > 0$ be sufficiently small such
  that each $M_{x,\delta}$ is represented in normal coordinates by the
  graph of $h_x$. Under this assumption, the open sets $M_{x,\delta}$
  are induced by $d_Q$ and correspond to the connected component of
  $x$ of the preimage of $B(\iota(x);\delta)$. Consequently, we can
  locally push the argument to $\iota(M) \subset Q$ to conclude that
  there is an upper bound $K$ on the number of sets $M_{x_i,\delta_2}$
  that intersect any set $M_{x,\delta_2}$.
\end{proof}

Even though we do not require submanifolds to have bounded geometry
for the results in this section, the lemma below will be needed in
the final reduction to a trivial bundle. The essential idea of the
proof is to use Gau\ss' second fundamental form to relate curvature of
the submanifold to second derivatives of its immersion map.
\begin{lemma}[Submanifold of bounded geometry]
  \label{lem:boundgeom-submfld}
  \index{bounded geometry!submanifold of}
  Let $M \in \BC^{k \ge 2}$ be a uniformly immersed submanifold of the bounded
  geometry manifold $(Q,g)$. Then $(M,\iota^*(g))$ is a Riemannian
  manifold with bounded geometry of order $k - 2$.
\end{lemma}

\begin{remark}
  We lose two orders of smoothness in the bounded geometry
  definition. This is due to bounded geometry being defined in terms
  of the curvature, which depends on second order derivatives of the
  metric, and in this case also on second order derivatives of the
  embedding through Gau\ss' second fundamental form.
\end{remark}

\begin{proof}
  Let $\delta$ be sufficiently small such that for each $x \in M$ we
  have the representation $M_{x,\delta} = \Graph(h_x)$ with
  $\norm{\D h_x} \le 1$.

  The Riemann curvature tensor $R^M$ of $M$ can be expressed as a sum
  of the curvature $R$ on $Q$ and the second fundamental form of the
  (local) embedding, see
  e.g.~\cite[Thm~3.6.2]{Jost2008:riem-geom-anal}:
  \begin{equation}\label{eq:submfld-curvature}
    \begin{aligned}
    g(R^M(X,Y)\,Z,W) = g(R(X,Y)\,Z,W) &+ g(S(Y,Z),S(X,W))\\
                                      &- g(S(Y,W),S(X,Z)),
    \end{aligned}
  \end{equation}
  where
  \begin{equation}\label{eq:sec-fund-form}
    S\colon \T M \times \T M \to N
     \colon X, Y \mapsto (\nabla_X Y)^\perp
   \end{equation}
   is the second fundamental form, and it is indeed pointwise defined.
   In normal coordinates we find
   \begin{equation}\label{eq:sec-fund-form-coord}
     S_x(X,Y) = \D^2 h_x(0)(X,Y).
   \end{equation}
   Since $h \in \BC^k$ and $g,\,g^{-1} \in \BC^k$ as well, it
   follows that $S \in \BC^{k-2}$ and by~\ref{eq:submfld-curvature}
   then that $R^M \in \BC^{k-2}$, so condition (B$_{k-2}$) of
   Definition~\ref{def:bound-geom} is satisfied.

   Condition (I) on the injectivity radius follows from an implicit
   function argument applied to the geodesic flow using
   Theorem~\ref{thm:unif-smooth-flow}. We consider local coordinates
   around $x \in M$ by projecting the representation
   $M \cap B(x;\delta)$ onto $\T_x M$ in normal coordinates in $Q$.
   That is, we have the coordinate chart map
   \begin{equation*}
     \kappa_x\colon B(0;\delta/2) \subset \T_x M \to M
             \colon \xi \mapsto \exp_x(\xi,h_x(\xi))
   \end{equation*}
   and the corresponding embedding into normal coordinates
   $\T_x M \embedto \T_x Q\colon \xi \mapsto (\xi,h_x(\xi))$ of $Q$. We
   calculate explicit estimates for the exponential map $\exp^M_x$
   using Christoffel symbols of the connection $\nabla^M$ on $M$ in
   the coordinates in chart $\kappa_x$.

   Let $X,Y$ be vector fields on $M$. Their representation in
   $\kappa_x$ is mapped to normal coordinates $B(x;\delta)$ on $Q$ as
   \begin{equation*}
     X(\xi) \mapsto \tilde{X}(\xi) = \big(\Id, \D h_x(\xi)\big)^T \cdot X(\xi).
   \end{equation*}
   Hence, from the covariant derivative on $M$ in normal coordinates
   $B(x;\delta) \subset Q$ we can recover the Christoffel symbols in
   local coordinates $\kappa_x$ as
   \begin{equation*}
     \nabla^M_X Y = p_1 \circ \Big[ X^i(\xi) \pder{}{\xi^i} Y(\xi)
     + \Christoffel(\xi,h_x(\xi))\big(\tilde{X}(\xi),\tilde{Y}(\xi)\big)\Big],
   \end{equation*}
   where the first term has reduced to derivatives with respect to
   $\xi \in \T_x M \subset \T_x Q$ only, and
   $\Christoffel\colon B(x;\delta) \to \CLin^2(\T_x Q;\T_x Q)$ are the
   Christoffel symbols in normal coordinates at $x \in Q$. Thus, the
   Christoffel symbols
   \begin{equation}\label{eq:christoffel-submfld}
     \Christoffel^M(\xi)(X,Y) =
     p_1\circ \Christoffel(\xi,h_x(\xi))\big(\tilde{X},\tilde{Y}\big)
   \end{equation}
   of $M$ in $\kappa_x$ coordinates are uniformly bounded on
   sufficiently small balls $B(0;\delta') \subset \T_x M$. The
   Euclidean geodesic flow at time one defines the (trivial) Euclidean
   exponential map, which is an isomorphic diffeomorphism (with
   infinite injectivity radius actually). Since we study a small
   perturbation of this flow in local coordinates, given by the
   additional term~\ref{eq:christoffel-submfld}, and the perturbation
   is at least $\BC^{k-1}$ and $C^1$ small, the perturbed geodesic
   flow of $M$ can be made close enough that $\exp^M_x$ is still a
   diffeomorphism on $B(0;\delta')$ for some $\delta' > 0$. Hence,
   $\rinj{x} \ge \delta'$, but these estimates depend only on the
   perturbation size, so they hold uniformly for all $x \in M$.
\end{proof}

To obtain the final result of this section, the tubular neighborhood
theorem~\ref{thm:tubular-neighbhd}, we first need to work out some
details on local coordinates. If $M$ is a submanifold of $Q$,
it is natural to consider a specific
splitting on the normal coordinates at points $x \in M$, namely
$\T_x Q = \T_x M \oplus N_x$, where $N$ is the normal bundle over $M$.
We shall require bounds, not just on coordinate transformations, but
more specifically bounds on how well this splitting is preserved. The
lemmas are formulated in a more general context of splittings of
tangent spaces at any two nearby points, while the results for
coordinates along $M$ follow as an easy corollary.
\begin{lemma}[Coordinate transformations of splittings]
  \label{lem:split-coordtrans}
  Let $(Q,g)$ be a smooth Riemannian manifold of bounded geometry, let
  $C$ be sufficiently large and let $\delta,\zeta > 0$ be sufficiently small. Let
  $x_1, x_2 \in Q$ and let $\T_{x_i} Q = H_i \oplus V_i,\,i=1,2$ be
  splittings along `horizontal' and `vertical' perpendicular subspaces with
  $\dim(H_1) = \dim(H_2)$. Assume that $d(x_1,x_2) < \delta$ and that,
  for $i \neq j$, $H_i$ is represented in tangent normal coordinates
  at $x_j$ by the graph of\/ $L_i \in \CLin(H_j;V_j)$ with
  $\norm{L_i} \le \zeta$.

  Then the coordinate
  transformation $\phi_{2,1}$ in Lemma~\ref{lem:bound-coordtrans} is
  of the form
  \begin{equation}\label{eq:split-coordtrans}
    \phi_{2,1} = O_H \oplus O_V + \tilde{\phi}_{2,1}
    \quad\text{with}\quad
    \norm{\tilde{\phi}_{2,1}}_k \le C\big(\zeta+d(x_1,x_2)\big),
  \end{equation}
  where $O_H,O_V$ are orthogonal transformations between the $H_i$ and
  $V_i$ with $i=1,2$, respectively.
\end{lemma}

We first prove the following result and use it to prove
Lemma~\ref{lem:split-coordtrans}.
\begin{lemma}[Approximation of orthogonal maps]
  \label{lem:orthog-approx}
  Let $V$ be a finite-dimensional inner product space and define the
  map
  \begin{equation}\label{eq:map-ker-OV}
    f\colon \CLin(V) \to \text{Sym}(V)
    \colon A \mapsto A^T A - \Id.
  \end{equation}
  There exists an $\epsilon > 0$ and a tubular neighborhood $B(O(V);\eta)
  \subset GL(V)$ with fiber projection $\pi$, such that on
  $\set{A \in GL(V)}{\norm{f(A)} < \epsilon, \norm{A} \le 2}$, the
  map $\phi\colon A \mapsto (\pi(A),f(A))$ is a smooth diffeomorphism.
  As a direct corollary, if\/ $\norm{f(A)} < \epsilon$ and
  $\norm{A} \le 2$ then $U = \pi(A) \in O(V)$ is an orthogonal
  approximation of $A$ in the sense that
  $\norm{U - A} \le \norm{f(A)}$.
\end{lemma}

\begin{proof}
  The map $f$ is smooth and invariant under the left action of the
  orthogonal maps $O(V)$, while $O(V) = \ker(f)$. Since $GL(V)$ is a
  Lie group, we have the
  canonical trivialization $\T GL(V) = GL(V) \times gl(V)$ by left
  multiplication. The similar trivialization
  $\T O(V) = O(V) \times o(V)$ can be viewed as a subbundle of
  \begin{equation*}
    O(V) \times o(V) \oplus \text{Sym}(V) = \T GL(V)|_{O(V)},
  \end{equation*}
  where $o(V)$ is identified with the skew-symmetric linear maps. We
  restrict the exponential map $\exp\colon \T GL(V) \to GL(V)$ to
  $O(V) \times \text{Sym}(V)$. At $\Id \in O(V)$ this restriction has
  bijective derivative, hence it is a local diffeomorphism. Since
  $\exp$ is $O(V)$\ndash invariant, it defines a diffeomorphism onto a
  tubular neighborhood $B(O(V);\eta) \subset GL(V)$ of $O(V)$ of size
  $\eta > 0$ and a corresponding smooth fiber projection map
  $\pi\colon B(O(V);\eta) \to O(V)$.

  Now $\D f(\Id)\colon a \mapsto a^T + a$ has image precisely
  $\text{Sym}(V)$. Thus, if we restrict $f$ to the fiber over
  $\Id \in O(V)$ in the tubular neighborhood, then
  $\D f(\Id)|_{\text{Sym}(V)} = 2$ and $f$ is a diffeomorphism with
  $\norm{\D f^{-1}} \le 1$ in some neighborhood of
  $0 \in \pi^{-1}(\Id)$; if necessary, we reduce $\eta > 0$ for
  $\norm{\D f^{-1}} \le 1$ to hold on $B(O(V);\eta) \cap \pi^{-1}(\Id)$.
  By $O(V)$ invariance of $f$, this holds
  globally on all (fibers) of the tubular neighborhood. Since, $\D\pi$
  and $\D f$ have complementary image at $O(V)$, $\phi = (\pi,f)$ is a
  diffeomorphism on $B(O(V);\eta)$.

  The set $\overline{B(0;2) \setminus B(O(V);\eta)} \subset \CLin(V)$ is
  compact, so $\norm{f(\slot)}$ attains its nonzero minimum on it. Let
  $\epsilon$ be smaller than this minimum. Then, if
  $\norm{f(A)} < \epsilon$, we must have $A \in B(O(V);\eta)$ and hence
  $A = \exp(U,a)$ for a unique $(U,a) \in O(V) \times \text{Sym}(V)$.
  By $O(V)$\ndash invariance, we can assume w.l.o.g.\ws that $U = \Id$
  and use the mean value theorem to estimate
  \begin{equation*}
    \norm{A - \Id}
    \le \norm{\D f^{-1}}\,\norm{f(A) - f(\Id)}
    \le \norm{f(A)} < \epsilon.
  \end{equation*}
  In other words, when $A$ is sufficiently close to being orthogonal,
  measured according to $f$, then it is close to an orthogonal map $U$
  in operator norm.
\end{proof}

\begin{proof}[Proof of Lemma~\ref{lem:split-coordtrans}]
  Extending the results of Lemma~\ref{lem:bound-coordtrans}, let
  \begin{equation*}
    O = \partrans(\gamma_{2,1})\colon \T_{x_1} Q \to \T_{x_2} Q
  \end{equation*}
  denote the orthogonal linear map induced by parallel transport. We
  decompose $\phi_{2,1} = O + \hat{\phi}_{2,1}$, where
  $\hat{\phi}_{2,1}$ can be made arbitrarily small. Moreover, we write
  \begin{equation*}
    O = \begin{pmatrix} A & B \\ C & D \end{pmatrix}
    \in \CLin(H_1 \oplus V_1;H_2 \oplus V_2),
  \end{equation*}
  with the idea that $B,C$ should be small and $A,D$ should
  approximate orthogonal maps $O_H,O_V$, respectively.
  Orthogonality of $O$ implies
  \begin{equation*}
    \Id = O^T O =
    \begin{pmatrix} A^T & C^T \\
                    B^T & D^T
    \end{pmatrix} \cdot
    \begin{pmatrix} A & B \\
                    C & D
    \end{pmatrix} =
    \begin{pmatrix} A^T A + C^T C & A^T B + C^T D \\
                    B^T A + D^T C & B^T B + D^T D
    \end{pmatrix}.
  \end{equation*}
  For the operator norm we have
  $\norm{A},\norm{B},\norm{C},\norm{D} \le \norm{O} = 1$, so if we
  assume for the moment that $B,C$ can be made sufficiently small,
  then, by writing $A^T A - \Id = -C^T C$ and $D^T D - \Id = -B^T B$,
  Lemma~\ref{lem:orthog-approx} implies that we can find
  $O_H,O_V$ such that
  \begin{equation}\label{eq:orthog-approx-err}
    \begin{aligned}
         \norm{O - O_H \oplus O_V}
    &\le \norm[\Big]{O - \begin{pmatrix} A & 0 \\ 0 & D \end{pmatrix}}
        +\norm[\Big]{\begin{pmatrix} A   & 0 \\ 0 & D \end{pmatrix} -
                     \begin{pmatrix} O_H & 0 \\ 0 & O_V \end{pmatrix}}\\
    &\le \norm{B} + \norm{C} + \norm{A - O_H} + \norm{D - O_V}\\
    &\le \norm{B} + \norm{C} + \norm{C^T C} + \norm{B^T B}.
    \end{aligned}
  \end{equation}

  In normal coordinates around $x_2$ we have $H_1 = \Graph(L)$,
  so $C\colon H_1 \to V_2$ is represented by $L$ in these coordinates.
  The metric $g$ is close to the identity in these coordinates, so
  $C \cong L$ can be assumed bounded by $4\,\zeta \le 1$, as measured
  in the metric on $Q$. The same argument can be made for $B^T$ by
  considering $\phi_{1,2} = \phi_{2,1}^{-1}$, since
  \begin{equation*}
    O^{-1} = O^T = \begin{pmatrix} A^T & C^T \\ B^T & D^T \end{pmatrix}.
  \end{equation*}
  We conclude that both $\norm{B},\norm{C} \le 4\,\zeta$ when $\delta$
  is chosen small, hence $O$ can be approximated by $O_H \oplus O_V$,
  and the error from~\ref{eq:orthog-approx-err} can be absorbed into
  $\tilde{\phi}_{2,1}$:
  \begin{equation*}
    \tilde{\phi}_{2,1} = \hat{\phi}_{2,1} + (O - O_H \oplus O_V).
  \end{equation*}
  The errors introduced in $\tilde{\phi}_{2,1}$ from
  lemmas~\ref{lem:bound-coordtrans} and~\ref{lem:orthog-approx} are
  Lipschitz small in terms of $d(x_1,x_2)$ and $\zeta$, respectively,
  so these add up to the estimate in~\ref{eq:split-coordtrans}.
\end{proof}

\begin{corollary}
  \label{cor:split-coordtrans}
  Let $M \in \BUC^{k \ge 1}$ be a uniformly immersed submanifold of a
  smooth Riemannian manifold $(Q,g)$ of bounded geometry. Let
  $x_1, x_2 \in M$ and let\/
  $\T_{x_i} Q = \T_{x_i} M \oplus N_{x_i},\,i=1,2$, be the respective
  splittings in horizontal and vertical directions. Then the results
  of Lemma~\ref{lem:split-coordtrans} hold for $d_M(x_1,x_2) <\delta$.
  If moreover $M \in \BC^2$, then we have a Lipschitz estimate
  $\norm{\D\tilde{\phi}_{2,1}(0)} \le C\,d(x_1,x_2)$.
\end{corollary}

\begin{proof}
  This follows immediately from the local representation
  $M_{x_2,\delta} = \Graph(h_2)$ as $\T_{x_1} M$ is represented
  in tangent normal coordinates at $x_2$ by $L = \D h_2(\xi)$, where
  $x_1 = (\xi,h_2(\xi))$. And $ \D h_2(\xi)$ becomes small when
  $\delta$ is small. The same holds with $x_1, x_2$ interchanged.

  If $M \in \BC^2$, then we can estimate
  $\norm{\D h_\bullet(\xi)} \le \norm{\D^2 h_\bullet}\,\norm{\xi} \le C\,d(x_1,x_2)$.
  Hence, the Lipschitz result in Lemma~\ref{lem:split-coordtrans}
  transforms into a Lipschitz estimate in $d(x_1,x_2)$ only.
\end{proof}

Below we define when a mapping is approximately isometric, see for
example also~\cite[p.~505]{Attie1994:quasi-iso-class-boundgeom}. The
Lyapunov exponents of a dynamical system are preserved under these
quasi-isometries since the exponential growth dominates any
bounded factors when measuring sizes. This property is required when
we transfer a noncompact normally hyperbolic system to a different
space and want normal hyperbolicity to be preserved.
\begin{definition}[\idx{quasi-isometry}]
  \label{def:quasi-isometry}
  Let $M,N$ be manifolds with distance metrics $d_M, d_N$ and let
  $\phi\colon M \to N$ be a diffeomorphism. we call $\phi$ a $C$\ndash
  quasi-isometry with $C > 1$, if
  \begin{equation}\label{eq:equiv-metrics}
    \forall x,y \in M\colon
    C^{-1}\,d_M(x,y) \le d_N(\phi(x),\phi(y)) \le C\,d_M(x,y).
  \end{equation}
  We simply call $\phi$ a quasi-isometry if there exists an
  unspecified $C > 1$.
\end{definition}

We conclude this section with a version of the tubular neighborhood
theorem that is appropriate in the bounded geometry setting.
\begin{theorem}[Uniform tubular neighborhood]
  \label{thm:tubular-neighbhd}
  \index{uniform tubular neighborhood}
  \index{$\eta$ (tubular neighborhood size)}
  Let $M \in \BC^{k \ge 2}$ be a uniformly immersed submanifold of the bounded
  geometry manifold $(Q,g)$. Then for $\eta > 0$ sufficiently small
  (but depending explicitly on $M$ and $Q$), the $\eta$\ndash sized tubular
  neighborhood $B(M;\eta) = \set{y \in Q}{d(y,M) \le \eta}$ can be
  represented on the $\eta$\ndash sized normal bundle $N_{\le\eta}$ of\/
  $M$ by a diffeomorphism $\phi$, locally on each
  $N_{\le\eta}|_{M_{x,\delta}}$ and we have
  $\phi,\phi^{-1} \in \BUC^{k-1}$ (hence $\phi$ is a quasi-isometry).

  When moreover $M$ is uniformly embedded, i.e.\ws
  $M_{x,\delta} = M \cap B(x;\delta)$ for each $x \in M$ as in
  Remark~\ref{rem:embedding}, then $\phi$ is a global diffeomorphism.
\end{theorem}

In case $M$ is compact, the standard proof uses the fact that the
exponential map has bijective differential at the zero section, and then by
compactness it must be a diffeomorphism on a uniform neighborhood $N_{\le\eta}$ of the
zero section. Here, to get a uniform neighborhood $N_{\le\eta}$ on
which $\phi = \exp|_{N_{\le\eta}}$ is a diffeomorphism, we require
bounds on second order derivatives (that is curvature, cf.\ws
Lemma~\ref{lem:boundgeom-submfld}) of $M$ so that it has curvature radius
$r$ bounded from below, hence cut locus points can only occur at least at
distance $r$ away from $M$, making $\phi$ injective for $\eta < r$.
See Figure~\ref{fig:normcoord-submfld} for a representation of a
submanifold $M$ in normal coordinates around $x$ and a ray of the
normal bundle at a nearby point $x'$.

Note that for an immersed submanifold, we define the normal bundle as
\index{$N$ (normal bundle)}
\begin{equation}\label{eq:normal-bundle-imm}
  N = \set[\big]{ (x,\nu) \in M \times \T Q }%
                { \iota(x) = \pi(\nu),\,
                  \nu \perp \text{Im}\big(\D \iota(x)\big) }.
\end{equation}
This can again be viewed as immersed into $\T Q$.

\begin{proof}
  We set $\phi = \exp|_N$ and in the following we will implicitly
  apply Theorem~\ref{thm:bound-metric} and
  Proposition~\ref{prop:unif-coord-chart} to choose $0 < \delta \le 1$
  small enough such that the metric $g$ up to its second order
  derivatives is bounded, as well as that the Christoffel symbols are
  bounded. Also, we choose $\delta$ such that $M$ is uniformly locally
  representable by graphs according to Definition~\ref{def:unif-imm-submfld}.
  We will in sequence prove local and global injectivity, surjectivity
  of $\phi$ and finally that $\phi,\phi^{-1} \in \BUC^{k-1}$.

  We claim that for some $\eta > 0$, $\phi$ is locally injective on
  $N_{\le\eta}$, the normal bundle restricted to size $\eta$. Let
  $\nu,\, \nu' \in N$ be such that $\phi(\nu) = \phi(\nu')$ and denote
  by $x = \pi(\nu),\, x' = \pi(\nu')$ their base points in $M$. We
  consider normal coordinates at $x$, hence we have
  $\nu = (0,\sigma_2)$ for some $\sigma_2 \in N_x$, while $\nu'$ is
  given by $(\sigma_1',\sigma_2') \in \T_x M \oplus N_x$. From
  Corollary~\ref{cor:split-coordtrans} it follows for small $\delta$
  that $\nu'$ is nearly mapped onto $N_x$ in normal coordinates at
  $x$. Since $M \in \BC^2$, the deviation from mapping onto $N_x$ is
  Lipschitz small in $d(x',x) \cong \norm{\xi}$, so we have
  \begin{equation}\label{eq:norm-comp-ineq}
    \norm{\sigma_1'} \le C\,\norm{\xi}\,\norm{\sigma_2'}.
  \end{equation}

  Now, $\phi(\nu') = \phi(\nu)$ can only hold if the respective
  horizontal coordinates along $\T_x M$ are equal. By definition of
  normal coordinates around $x$, we have $\exp(\nu) = (0,\sigma_2)$.
  Therefore it is sufficient to prove that some $\eta > 0$ exists as a
  lower bound for
  \begin{equation*}
    \set{\norm{\nu'} \in \R}{\phi(\nu')_1 = 0, \pi(\nu') \neq x}.
  \end{equation*}
  We view the exponential map as the time-one geodesic flow, which is
  given in local coordinates by~\ref{eq:geodflow-local}.
  The geodesic flow along $\nu' = (\sigma_1',\sigma_2')$ starting at
  $(\xi',h(\xi'))$ is a small perturbation of the flow along
  $(0,\sigma_2')$ starting at $(0,0)$. The latter has solution curve
  $t \mapsto (0,t\,\sigma_2') \in \T_x M \oplus N_x$.

  By Theorem~\ref{thm:bound-metric} we arrange for
  $\norm{g - \Id},\norm{\Christoffel} \le 2$ and
  $\norm{\D\Christoffel} \le C$ in local coordinates. We have
  estimates
  \begin{gather*}
        \norm{\Christoffel(x'(t)) - \Christoffel(x(t))}
    \le \norm{\D\Christoffel}\,\norm{x'(t) - x(t)}
    \le C\,\norm{x'(t) - x(t)},\\
        \norm{\sigma'(t)}
    \le \sqrt{g(\sigma'(t),\sigma'(t))}
    =   \sqrt{g(\sigma'(0),\sigma'(0))}
    \le \sqrt{2}\,\norm{\sigma'(0)}.
  \end{gather*}
  With these, we obtain the Gronwall-like estimates
  \begin{equation*}
    \begin{aligned}
         \der{}{t} \norm{x'(t) - (0,t\,\sigma_2')}
    &\le \norm{\sigma'(t) - (0,\sigma_2')},\\
         \der{}{t} \norm{\sigma'(t) - (0,\sigma_2')}
    &\le \begin{aligned}[t]
       & \norm{\D\Christoffel}\,\norm{x'(t) - (0,t\,\sigma_2')}\,\norm{\sigma'(t)}^2\\
       &+\norm{\Christoffel}\,\big(\norm{\sigma'(t)} + \norm{(0,\sigma_2')}\big)\,
         \norm{\sigma'(t) - (0,\sigma_2')}
         \end{aligned}\\
    &\le \big(4\,C\,\eta^2+4\,\sqrt{2}\,\eta\big)\,
         \norm{\sigma'(t) - (0,\sigma_2')},
    \end{aligned}
  \end{equation*}
  for which Gronwall's inequality yields
  \begin{equation*}
        \norm{\sigma'(t) - (0,\sigma_2')}
    \le \norm{(\sigma_1'(0),0)}\,e^{\tilde{C}\,\eta\,t}.
  \end{equation*}
  Now, if $\eta$ is chosen sufficiently small, then
  using~\ref{eq:norm-comp-ineq}, we have for all $0 \le t \le 1$ that
  \begin{equation*}
        \norm{x_1'(t) - x_1(0)}
    \le \int_0^t C\,\norm{\xi'}\,\eta\,e^{\tilde{C}\,\eta\,t} \d\tau
    \le \frac{1}{2}\,\norm{\xi'}.
  \end{equation*}
  This shows that there exists an explicit $\eta > 0$ such that $\phi$
  is injective on $N_{\le\eta}$ restricted to a neighborhood
  $M_{x,\delta}$, and by construction this $\eta$ is uniform over $M$.
  Later modifications to choose $\eta$ smaller will only depend on the
  global geometry of $Q$, but not on any details of $M$.

  If moreover $M_{x,\delta} = M \cap B(x;\delta)$ is the unique
  connected component of $M$ in each normal coordinate chart, then
  $\phi$ is injective globally on $N_{\le\eta}$. This follows easily
  by taking $\eta < \frac{\delta}{2}$. Then, any $\nu'$ and $\nu$ that
  have the same image, must have base points $x',x$ separated by a
  distance less than $\delta$, as $\phi = \exp$ will only map onto
  points at most $\eta$ away from the base point. Therefore, $x'$ must
  lie in $B(x;\delta)$ and in $M$, hence on
  $\Graph(h) = M_{x,\delta}$. This case was already treated.

  Finally, we will show that $\phi$ is surjective onto $B(M;\eta)$
  when $\eta < \rinj{Q}$. Take $y \in B(M;\eta)$,
  then $M \cap B(y;\rinj{Q})$ contains a nonempty compact set, so
  there exists an $x \in M$ such that $d(y,M) = d(y,x)$. This distance
  must be realized by a (unique) geodesic $\gamma$. We will derive a
  contradiction if $\gamma'(0) \not\in N_x$, by showing that then the
  minimum distance is not attained at $x$. Let $(\xi_1, \xi_2) \in
  \T_x M \oplus N_x$ be the normalized tangent vector of $\gamma'(0)$.
  By assumption we have $\xi_1 \neq 0$, so $\norm{\xi_2} < 1$. We
  parametrize
  \begin{equation*}
    \gamma\colon \intvCC{0}{d(y,x)} \to \T_x Q
          \colon t \mapsto t\,(\xi_1,\xi_2)
  \end{equation*}
  by arc length in normal coordinates, thus $d(x,\gamma(t)) = t$.
  Consider the Euclidean distance in normal coordinates at $x$ of
  $\gamma(t)$ to its vertical projection onto $M = \Graph(h)$.
  This shows that
  \begin{equation*}
    d_E(\gamma(t),M) \le \norm{t\,\xi_2 - h(t\,\xi_1)}
                     \le t\,\norm{\xi_2} + o(t\,\norm{\xi_1})
  \end{equation*}
  as $\D h(0) = 0$ and $h \in \BC^2$. The Euclidean distance is
  $C$\ndash equivalent to the $g$\ndash induced distance, so we have
  \begin{equation*}
        \lim_{t \downarrow 0} \frac{d(\gamma(t),M)}{d(x,\gamma(t))}
    \le \lim_{t \downarrow 0} C\,\norm{\xi_2} + o(t)/t = C\,\norm{\xi_2}.
  \end{equation*}
  By assumption $\norm{\xi_2} < 1$, so we can restrict to a small
  enough neighborhood $B(x;\delta)$ such that
  $1 < C < \norm{\xi_2}^{-1}$ and conclude that
  $d(\gamma(t),M) < d(x,\gamma(t))$ for some $t > 0$, which shows that
  a shorter (broken) geodesic from $y$ to $M$ exists. This completes
  the contradiction and proves that $\phi$ is surjective.

  Finally, $\phi \in \BUC^{k-1}$ follows directly from the fact that
  it is the restriction of the exponential map to
  $N_{\le\eta} \in C^{k-1}$ and in induced normal coordinate charts,
  we have $\exp \in \BUC^{k-1}$. For $\phi^{-1}$ we use a formula and
  arguments similar to~\ref{eq:impl-func-der}, showing that if
  $\D \phi^{-1}$ is uniformly bounded, then $\phi^{-1} \in \BUC^{k-1}$
  holds as well. Now $\D\phi = \Id$ in induced normal coordinates at
  $M$, so by uniform continuity, there exists some $\eta > 0$ such
  that $\D\phi$ stays away from non-invertibility on $N_{\le\eta}$,
  hence $\D\phi^{-1}$ stays bounded. This automatically implies that
  $\phi$ is a quasi-isometry with
  $C = \max(\norm{\D\phi},\norm{\D\phi^{-1}})$.
\end{proof}

\section{Smoothing of submanifolds}
\label{sec:BG-smooth-submflds}

\index{normal hyperbolicity@$r$-normal hyperbolicity}
It is well-known, at least in the compact case, that $\fracdiff$\ndash
normal hyperbolicity is a persistent property under $C^1$ small
perturbations of class $C^\fracdiff$, that is, the persisting manifold
is again $\fracdiff$\ndash normally hyperbolic and specifically $C^\fracdiff$,
see~\cite[Thm~4.1]{Hirsch1977:invarmflds}
or~\cite[Thm~2]{Fenichel1971:invarmflds}. In other words, $\fracdiff$\ndash
normal hyperbolicity is an `open property' in the space of $C^\fracdiff$
systems with $C^1$ topology. Therefore, it is natural to only
assume that the original manifold is $C^\fracdiff$, but not smoother.
Even if we start out with an $\fracdiff$\ndash NHIM $M \in C^\infty$,
then after a perturbation we will generally only have a manifold
$M_\epsilon \in C^\fracdiff$. We could, however, also have tried to
obtain this manifold $M_\epsilon$ by first perturbing $M$ to an
intermediate manifold $M_{\epsilon/2}$ and then perturb that manifold
to $M_\epsilon$. When applying a persistence theorem in the second
step, we can only assume the initial manifold to be $C^\fracdiff$.

This restricted $C^\fracdiff$ smoothness assumption forces us
to be careful about the precise smoothness of each and every object.
For example, a vector field on a $C^\fracdiff$ manifold can only be $C^{r-1}$.
This could probably be overcome by considering discrete-time mappings
instead of flows, but we need other smoothness improvements as well.
For example, we want to model the persisting manifold as a section
of the normal bundle $N \in C^{\fracdiff-1}$ of the original manifold,
which is not smooth enough. So here we need a smoothing argument
as well, cf.~\cite[p.~205]{Fenichel1971:invarmflds}.
\index{smoothness!loss of}

We solve these problems by constructing an approximate, smoothed manifold
$M_\sigma \in C^\infty$. This allows $M$ to be modeled as a small
section of the normal bundle $N_\sigma$ of $M_\sigma$, so the system
in a neighborhood of $M$ can be transferred to $N_\sigma$ while
preserving smoothness and normal hyperbolicity properties. With this
construction we need not worry about smoothness in the proof, while
the conclusions are preserved up to $C^\fracdiff$ smoothness. Uniform
estimates must be preserved though, so standard methods for
constructing $M_\sigma$ and $N_\sigma$ do not readily apply or need
a careful analysis.

\index{$\sigma$ (approximation parameter)}
First, we construct $\sigma > 0$ close approximations
$M_\sigma \in C^\infty$ to $M$ by globalizing a local chart
construction of smoothing by convolution with a mollifier. Then we use
the fact that $M_\sigma$ has uniformly bounded `second-order
derivatives' to show that for a sufficiently small $\sigma > 0$,
$M_\sigma$ has a normal bundle diffeomorphic to a
neighborhood $B(M;\delta) \subset Q$ of uniform size.

\index{$\nu$ (mollifier size)}
We recall some standard techniques on $\R^n$, see for
example~\cite[p.~25]{Hormander2003:anal-lin-PDoper}. Let
$\phi_\nu \in C^\infty_0(\R^n;\R_{\ge 0})$ be a \idx{mollifier function},
with support in $B(0;\nu)$ and integral normalized to one for any
$\nu > 0$. We also define a generic cut-off function
$\chi_{\alpha,\beta} \in C^\infty(\R;\intvCC{0}{1})$ such that
\begin{equation}\label{eq:cutoff}
  \chi_{\alpha,\beta}(x) = \begin{cases}
    1 & \text{if~} x \le \alpha,\\
    0 & \text{if~} x \ge \beta.
    \end{cases}
\end{equation}
Note that $\phi_\nu,\,\chi_{\alpha,\beta} \in \BUC^k$ for any
$k \ge 0$, as they are constant outside compact sets.

\begin{lemma}[Smoothing by convolution]
  \label{lem:conv-smoothing}
  \index{convolution smoothing}
  Let $r,\,\delta\!r > 0$,
  $f \in \BUC^{k \ge 0}\big(B(0;r+2\,\delta\!r) \subset \R^m;\R^n\big)$,
  and fix $l > k$ and $\epsilon > 0$. If the mollifier support radius
  $\nu > 0$ is chosen sufficiently small, then $f$ can be approximated
  by a function $\tilde{f}$ such that
  \begin{enumerate}
  \item $\tilde{f} = f$ outside $B(0;r+\delta\!r)$;
  \item $\tilde{f} \in \BUC^l \cap C^\infty$ on $B(0;r)$ and wherever
    $f \in \BUC^l \cap C^\infty$;
  \item $\norm{\tilde{f} - f}_k \le \epsilon$;
  \item $\norm{\tilde{f}}_l \le C(\nu,l)\,\norm{f}_0$ on $B(0;r)$, for
    some\/ $C(\nu,l) > 0$.
  \end{enumerate}
  Note that\/ $C(\nu,l)$ may grow unboundedly as $\nu \to 0$ or
  $l \to \infty$.
\end{lemma}

\begin{proof}
  A function that is $\BC^{l+1}$ is automatically $\BUC^l$, that is,
  uniformly continuous up to one degree less, so we only need to prove
  $\tilde{f} \in \BC^l \cap C^\infty$ for $l$ shifted by one.

  We construct $\tilde{f}$ by a combination of convolution and
  cut-off. Let $\hat\chi(x) = \chi_{r,r+\delta\!r}(\norm{x})$ for
  $x \in \R^m$ and define
  \begin{equation}\label{eq:f-smoothed}
    \tilde{f}(x) = (1-\hat\chi(x)) f(x)
                    + \hat\chi(x) \int_{\R^m} f(x-y)\,\phi_\nu(y) \d y.
  \end{equation}
  When $\nu < \delta\!r/2$, this $\tilde{f}$ is smooth on $B(0;r)$ and
  equal to $f$ outside $B(0;r+\delta\!r)$.

  The convolution approximates $f$ in $C^k$\ndash norm, as for any
  $0 \le j \le k$ and $x \in B(0;r+\delta\!r)$
  \begin{equation*}
    \norm{\D^j (\phi_\nu * f)(x) - \D^j f(x)}
    \le \int_{B(0;\nu)} \norm{\D^j f(x-y) - \D^j f(x)}\,\phi_\nu(y) \d y
    \le \epsilon_{\D^j f}(\nu),
  \end{equation*}
  so by uniform continuity of $f$ up to $k$\th derivatives,
  $\nu$ can be chosen small enough such that
  $\norm{(\phi_\nu * f) - f}_k \le \epsilon$ on $B(0;r+\delta\!r)$.
  The map $x \mapsto \hat\chi(x)$ is $\BUC^k$, so we can estimate for
  $j \le k$
  \begin{equation*}
    \norm{\D^j \tilde{f}(x) - \D^j f(x)}
    \le \sum_{i=0}^j \binom{j}{i}  \norm{\D^i \hat\chi(x)} \cdot
                                   \norm{\D^{j-i}\big(\phi_\nu * f - f\big)(x)}
    \le C_j\,\epsilon(\nu).
  \end{equation*}
  Hence, we can construct $\tilde{f}$ close enough to $f$ in $C^k$\ndash
  norm by choosing $\nu$ small enough.

  Uniform continuity of $\tilde{f}$ follows from uniform continuity of
  $\phi_\nu * f$ as $\hat\chi \in \BUC^k$ on its compact
  support. We find for $0 \le j \le k$
  \begin{equation*}
    \begin{aligned}
\fst     \norm{\D^j(\phi_\nu * f)(x_2) - \D^j(\phi_\nu * f)(x_1)}\\
    &\le \int_{B(0;\nu)} \norm{\D^j f(x_2 - y) - \D^j f(x_1 - y)}\,\phi(y) \d y
     \le \epsilon_{\D^j f}(\norm{x_2 - x_1}).
    \end{aligned}
  \end{equation*}

  To estimate bounds for higher derivatives of $\tilde{f}$ within
  $B(0;r)$, we note that $\hat\chi = 1$ and let the derivatives act on
  $\phi_\nu$ in the convolution: these are bounded on the compact
  domain of support, but bounds will depend on the size $\nu$ and
  degree $l$, while $\norm{f}_0$ can be factored out.
\end{proof}

The smoothing technique in Lemma~\ref{lem:conv-smoothing} is
formulated for Euclidean space. To
adapt it to manifolds in a uniform setting, we need to have uniformly
sized coordinate charts, as well as uniform behavior of the function
under these smoothing operations. We cannot simply use local
coordinates and a partition of unity, because the images on different
charts cannot be glued together on the target manifold. Instead, we
will apply this smoothing operation sequentially on each coordinate
chart in a cover. We require a cover that is locally finite with a
global upper bound $K$ on the number of charts covering a point, so that
each point undergoes only a bounded number of smoothing operations and
hence the final smoothed manifold $M_\sigma$ differs by a controllable
amount from the original $M$.

When the graph representation of $M$ in one chart is modified, we need
control on how much the graph is modified in overlapping charts. To
this end, we extend Lemma~\ref{lem:split-coordtrans} and
Corollary~\ref{cor:split-coordtrans}.
\begin{lemma}[Graph difference under coordinate transformations]
  \label{lem:graphdiff-coordtrans}
  \index{coordinate transition map!graph change under}
  Let $(Q,g)$ be a smooth Riemannian manifold of bounded geometry. Let
  $x_1, x_2 \in Q$ and let\/ $\T_{x_i} Q = H_i \oplus V_i,\,i=1,2$ be
  splittings along horizontal and vertical perpendicular
  subspaces with $\dim(H_1) = \dim(H_2)$. Assume that
  $d(x_1,x_2) < \delta$ and that, for $i \neq j$, $H_i$ is represented
  in normal coordinates at $x_j$ by the graph of\/
  $L_i \in \CLin(H_j;V_j)$ with $\norm{L_i} \le \zeta$.

  Let $f_1, g_1 \in \BUC^{k \ge 1}\big(B(0;\delta) \subset H_1;V_1\big)$
  with $\norm{f_1}_1,\norm{g_1}_1 \le \epsilon$ and\/
  $\norm{f_1}_k,\norm{g_1}_k \le C$.

  When $\delta,\zeta,\epsilon > 0$ are sufficiently small, then there
  exists a constant $\tilde{C}$ such that the graphs of $f_1, g_1$ are
  (partially) represented by functions $f_2, g_2 \in \BUC^k(H_2;V_2)$
  and $\norm{f_2 - g_2}_k \le \tilde{C}\,\norm{f_1 - g_1}_k$. This
  result is uniform for all $x_1,x_2 \in Q$.
\end{lemma}

\begin{remark}
  The functions $f_i,g_i$ may only be defined on parts of
  $B(x_i;\delta)$; all claims should thus be read as only for those
  points where the respective functions are defined.
\end{remark}

\begin{proof}
  Let $(\xi_i,\eta_i) \in H_i \oplus V_i,\,i=1,2$ denote
  normal coordinates, decomposed in the split directions
  at $x_i \in Q$. By Lemma~\ref{lem:split-coordtrans}, transformations
  between these coordinates are of the form~\ref{eq:split-coordtrans},
  where $\norm{\tilde{\phi}_{2,1}}_{k+1}$ can be made uniformly small
  as $\delta,\zeta \to 0$.

  We aim to apply the implicit function theorem to find a function
  $f_2$ on $B(x_2;\delta)$ whose graph corresponds to that of a
  function $f_1$ on $B(x_1;\delta)$. We define
  \begin{gather*}
    X = H_1 \times V_2,            \quad
    Y = H_2 \times \BUC^k(H_1;V_1),\quad
    Z = H_2 \times V_2,            \quad\text{and}\\
    F\colon X \times Y \to Z
     \colon (\xi_1,\eta_2),(\xi_2,f_1) \mapsto \phi_{2,1}(\xi_1,f_1(\xi_1)) - (\xi_2,\eta_2).
    \eqnumber\label{eq:graphtrans-implfunc}
  \end{gather*}
  Note that $X$ and $Z$ are isomorphic vector spaces, so we can apply
  the implicit function theorem with $Y$ as parameter space. Moreover,
  if we have two functions $f_1, f_2$ whose graphs represent the same
  manifold on the intersection $B(x_1;\delta) \cap B(x_2;\delta)$,
  then we have
  \begin{equation*}
    F\big(p_1\circ\phi_{2,1}^{-1}\big(\xi_2,f_2(\xi_2)\big),f_2(\xi_2),\xi_2,f_1\big) = 0
  \end{equation*}
  for all $\xi_2 \in H_2$ where this is defined, so the
  implicit function
  \begin{equation*}
    G(\xi_2,f_1) = \big(p_1\circ\phi_{2,1}^{-1}(\xi_2,f_2(\xi_2)),f_2(\xi_2)\big)
  \end{equation*}
  encodes the representation $f_2$. We verify the conditions of the
  implicit function theorem:
  \begin{equation*}
    \D_1 F\big((\xi_1,\eta_2),(\xi_2,f_1)\big) =
    \begin{pmatrix}
        O_H + p_1\cdot\D\tilde{\phi}_{2,1}\cdot(\Id + \D f_1) & 0   \\
        O_V\cdot\D f_1                                        & \Id
    \end{pmatrix}
  \end{equation*}
  is unitary when $\tilde{\phi}_{2,1} = f_1 = 0$. When
  $\delta,\,\epsilon$ are sufficiently small, then these functions are
  still small enough such that $\D_1 F$ is invertible with uniformly
  bounded inverse, using Lemma~\ref{lem:invert-lin}. Furthermore,
  $F \in \BUC^k$, as the dependence on $\xi_1,\eta_2,\xi_2$ is clearly
  $\BUC^k$, while the omega
  Lemma~\cite[p.~101]{Abraham1988:mflds-tensor-appl2nd} guarantees
  joint $\BUC^k$\ndash dependence on $f_1$ as well. Note that
  compactness of the domain of $f_1$ is not required, as we assume
  these functions to be uniformly bounded and thus have compact image.

  The implicit function theorem has a corresponding formulation as a
  uniform contraction principle. The latter formulation shows that the
  implicit function $G$ must be unique, while existence holds if
  $\tilde{\phi}_{2,1}, f_1, \xi_2$ are sufficiently close to zero, due
  to a priori estimates.
  We apply Corollary~\ref{cor:unif-impl-func} as an extension of the
  implicit function theorem to conclude that $G \in \BUC^k$. This
  means that $f_2 = p_2 \circ G(\slot,f_1) \in \BUC^k$ on suitable
  neighborhoods. Using formula~\ref{eq:impl-func-der} for $\D G$, we
  can moreover conclude that $f_2$ depends Lipschitz on $f_1$. This
  follows from explicit control on the boundedness and continuity
  estimates, while variation with respect to
  $f_1$ only introduces additional $k \tp 1$\ndash order derivatives
  of $\tilde{\phi}_{2,1}$, which can be assumed uniformly bounded.
  Hence, there exists some constant $\tilde{C}$ such that
  $\norm{g_2 - f_2}_k \le \tilde{C}\,\norm{g_1 - f_1}_k$ and all
  estimates are uniform.
\end{proof}

\begin{corollary}[Graph size under coordinate transformations]
  \label{cor:graphsize-coordtrans}
  \index{coordinate transition map!graph change under}
  Under the assumptions of Lemma~\ref{lem:graphdiff-coordtrans}, there
  exist constants $A,B$ such that we have the estimate
  \begin{equation}\label{eq:graphsize}
    \norm{f_2}_k \le A\,\norm{f_1}_k + B
  \end{equation}
  on amplification of the size of a graph under coordinate
  transformations.
\end{corollary}

\begin{proof}
  We choose $g_1 = 0$ in Lemma~\ref{lem:graphdiff-coordtrans} and set
  $A = \tilde{C}$. There exists a uniform bound $B$ such that
  $\norm{g_2}_k \le B$, and when $\zeta,\epsilon$ are sufficiently
  small, then for $\delta' = \frac{9}{10}\,\delta$ we have
  $B(0;\delta') \subset \text{Dom}(g_2)$. Hence, we easily deduce
  \begin{equation*}
    \norm{f_2}_k
    \le \norm{f_2 - g_2}_k + \norm{g_2}_k
    \le \tilde{C}\,\norm{f_1}_k + B
  \end{equation*}
  for all $x \in B(0;\delta')$ where $f_1, f_2$ are defined.
\end{proof}

\begin{theorem}[Uniform smooth approximation of a submanifold]
  \label{thm:unif-smooth-submfld}
  Let $M \in \BUC^{k \ge 1}$ be a uniformly immersed submanifold of a
  smooth Riemannian manifold $(Q,g)$ of bounded geometry.

  Then for each $\sigma > 0$ and integer $l \ge k$, there exists a
  uniformly immersed submanifold $M_\sigma \in \BUC^l \cap C^\infty$ and
  $\delta > 0$ such that $\norm{M_\sigma - M}_k \le \sigma$ with
  respect to normal coordinate charts of radius $\delta$ along both
  $M$ and $M_\sigma$. If\/ $M$ is a uniformly embedded submanifold, then
  so is $M_\sigma$.
\end{theorem}

The proof relies on finding a (uniformly locally finite) cover of $M$
and then in each chart make smooth the graph representation $h_\bullet$.
All estimates are uniform, independent of the point $x \in M$, hence
so is the final result. Smoothing is done sequentially in each chart,
so we must be careful to check how smoothing in one chart influences
the graph representation in other charts. This makes the technical
estimates quite involved, but the basic idea is that we have uniform
control on the size of changes in each $h_\bullet$ by the convolution
kernel parameter $\nu$ in Lemma~\ref{lem:conv-smoothing}, as well as
the size of this change in other charts.

\begin{proof}
  This proof contains a lot of interdependent size estimation parameters.
  Giving explicit choices and dependencies would clutter the proof
  needlessly, so we make a few remarks on beforehand. Any $\delta$'s
  denote sizes of normal coordinate balls and $\epsilon$'s are used
  for sizes of (changes in) functions in these coordinates. The
  parameter $\nu$ from Lemma~\ref{lem:conv-smoothing} depends on most
  of the foregoing, while only the $C^l$ bound (but not the
  $C^k$ bound) of $M_\sigma$ depends on the choice of $\nu$.
  The various $\epsilon$'s will be fixed later, and depend on $\sigma$
  and global properties of $M$ and $(Q,g)$, but not on $\delta$'s.
  Also note that everything is independent of points $x \in M,Q$.

  We fix $2\,\delta_1 = \delta_2 = \frac{1}{2}\,\delta_3$,
  $\epsilon_\infty = 2\,\epsilon_0$, and $C_\infty = C_0 + 1$ and
  choose $\delta_3,\,\zeta,\,\epsilon_\infty,\,\epsilon_\phi > 0$
  sufficiently small such that all the following statements hold true.
  \begin{enumerate}
  \item By Proposition~\ref{prop:unif-coord-chart} and
    Lemma~\ref{lem:loc-equiv-Mdist}, distances and metrics are
    $C = 2$ equivalent on balls $B(x;\delta_3)$, and $d_Q,\,d_M$ are
    locally equivalent, all up to order $l+1$.
  \item By assumption of $M \in \BUC^k$, we have for each $x \in M$
    the representation $M_{x,\delta_3} = \text{graph}(h_x)$ with
    $\norm{h_\bullet}_1 \le \epsilon_0$ and
    $\norm{h_\bullet}_k \le C_0$.
  \item By Corollary~\ref{cor:submfld-cover}, there exists a
    uniformly locally finite cover
    $\bigcup_{i \ge 1} M_{x_i,\delta_2}$ of $M$ with all $x_i$
    separated by at least $\delta_1$, the balls $M_{x_i,\delta_1}$
    already covering $M$, and bound $K$ on the maximum number of
    $M_{x_i,\delta_2}$ intersecting any $M_{x,\delta_2}$.
    Formula~\ref{eq:unif-cover-bound} shows that $K$ depends on the
    ratio $\delta_2/\delta_1$, but does not increase when $\delta_3 \to 0$.
  \item By Lemma~\ref{lem:split-coordtrans}, all coordinate transition
    maps $\phi_{2,1}$ between any $x_1,x_2 \in Q,\,d(x_1,x_2) < \delta_3$
    are $C^l$ bounded. When the graph representations
    $H_i = \Graph(L_i)$ are bounded by $\zeta > 0$, then these
    are of the form $\phi_{2,1} = O_H \oplus O_V +
    \tilde{\phi}_{2,1}$,
    with $\norm{\tilde{\phi}_{2,1}}_{l+1} \le \epsilon_\phi$. And by
    Corollary~\ref{cor:split-coordtrans}, this holds for the
    coordinate transformations between the $x_i$ chosen for the cover
    of $M$.
  \item By Lemma~\ref{lem:graphdiff-coordtrans}, there exists a
    constant $\tilde{C}$ that estimates the graph change under
    coordinate transformations from the previous point, when
    $\norm{f}_1,\norm{g}_1\le \epsilon_\infty$ and
    $\norm{f}_k,\norm{g}_k \le C_\infty$, while
    Corollary~\ref{cor:graphsize-coordtrans} holds on balls of size
    $\delta' = \frac{9}{10}\,\delta_3 > \delta_2$.
  \item If $M$ was a uniformly embedded submanifold, then let
    $M_{x,\delta_3}$ be the unique connected component of $M$ in
    $B(x;\delta_3)$ for each $x \in M$.
  \end{enumerate}

  Let $M^j$ denote a modification of $M$ after applying smoothing
  operations in the first $j$ charts, and let $h_i^j$ denote the graph
  representation of $M^j$ in chart $i$. So, initially we have
  $M^0 = M$ and $h_i^0 = h_i$. Note that the sequence
  $\{h_i^j\}_{j \ge 1}$ is constant after some finite index $j(i)$,
  since it is changed at most $K$ times by smoothing in overlapping
  charts. Thus, the final graphs are given by
  $h_i^\infty = h_i^{j(i)}$.

  Initially, we have $\norm{h_\bullet}_1 \le \epsilon_0$ and
  $\norm{h_\bullet}_k \le C_0$, and we assume that throughout the
  sequential smoothings it holds for all $i,j$ that
  $\norm{h_i^j}_1 \le \epsilon_\infty$ and
  $\norm{h_i^j}_k \le C_\infty$, and therefore the final $h_i^\infty$
  satisfy these estimates as well. Let $\epsilon_0$ be sufficiently
  small, such that by a mean value theorem estimate
  \begin{equation*}
        \norm{h_\bullet(\xi)}
    =   \norm{h_\bullet(\xi) - h_\bullet(0)}
    \le \norm{\D h_\bullet}\,\norm{\xi}
    \le \epsilon_0\,\norm{\xi},
  \end{equation*}
  we have $B(0;\delta_2) \subset \text{Dom}(h_\bullet)$.

  We apply the convolution smoothing Lemma~\ref{lem:conv-smoothing}
  with some choice $r > \delta_1$ and $r+2\,\delta r < \delta_2$ to
  sequentially make smooth $h_i^{i-1}$ in the coordinate chart
  $B(x_i;\delta_3)$ to obtain $h_i^i \in \BUC^l \cap C^\infty$ on
  $B(0;r)$. The $h_\bullet^j$ representations of the
  $M^j_{x_\bullet,\delta_3}$ overlap, so we must be careful that (at
  most) $K$ repeated smoothing operations keep the $h_\bullet^j$
  within the bounds required to apply this lemma, while at the same
  time we must ensure that each point on the sequence of manifolds
  $M^j$ is smoothed to $\BUC^l$ at some stage, even though $M^j$
  changes to $M_\sigma = M^\infty$ throughout the sequential
  smoothings.

  Let us first show that each point is smoothed. We can keep track of
  each original point $x \in M$ as a sequence of points $x^j \in M^j$
  throughout the smoothings, and once $M^j$ is smoothed around $x^j$,
  then the convolution lemma guarantees that smoothness is preserved
  around the sequence $x^j$ under further smoothing in other charts.
  Let $\Phi\colon M \diffto M_\sigma$ denote the diffeomorphism that
  assigns to $x \in M$ the final point $x^\infty \in M_\sigma$.
  Each point $x \in M$ is element of a graph $h_i$ in at least one
  ball $B(x_i;\delta_1)$, so if the corresponding sequence of points
  $x^j \in M^j$ moves less than $r - \delta_1$, then it is smoothed in
  $B(x_i;r)$. Therefore, we choose $\nu$ in
  Lemma~\ref{lem:conv-smoothing} small enough, such that
  \begin{equation*}
    \norm{\tilde{f} - f}_k \le \epsilon(\nu) < \frac{r - \delta_1}{2\,K}.
  \end{equation*}
  The factor $2\,K$ accounts for at most $K$ charts in which $x^j$ is
  moved and $C = 2$ to correct for equivalence of distance in charts.
  Hence, the manifold is smoothed to $\BUC^l \cap C^\infty$ at each
  point.

  Next, we show that each $h_i^j$ is defined at least on
  $B(0;r+2\,\delta r)$ and satisfies the bounds $\epsilon_\infty$ and
  $C_\infty$. Initially, we have $\norm{h_i}_1 \le \epsilon_0$,
  $\norm{h_i}_k \le C_0$, and $\Graph(h_i) = M_{x_i,\delta_3}$
  is well-defined in $B(x_i;\delta_3)$. So for
  $\epsilon_0 \le \frac{1}{10}$, say, we must have
  $\norm{h_i}_0 \le \frac{1}{10}\,\delta_3$ and
  $\text{Dom}(h_i) \supset B(0;\frac{11}{10}\delta_2)$. The only
  reason that the domain of some $h_i^j$ decreases is if either the
  graph moves outside of $B(0;\delta_3)$ or the modified manifold
  cannot be represented by a graph anymore. The latter cannot occur if
  $\D h_i^j$ stays bounded, while the former can be controlled by
  bounding $\norm{h_i^j}_0$. Both can be controlled by estimating the
  $C^1$ changes $h_i^j - h_i^{j-1}$. First, in coordinate chart
  $i = j$ we can directly use the convolution smoothing
  Lemma~\ref{lem:conv-smoothing} to conclude that
  $\norm{h_i^j - h_i^{j-1}}_1 \le \epsilon(\nu)$. In any other chart
  $i \neq j$, this change is amplified by a bounded factor
  $\tilde{C}$, as per Lemma~\ref{lem:graphdiff-coordtrans}, so we have
  \begin{equation*}
    \norm{h_i^j - h_i^{j-1}}_1 \le \tilde{C}\,\epsilon(\nu).
  \end{equation*}
  When we choose $\nu$ small enough that
  \begin{equation*}
    K\,\tilde{C}\,\epsilon(\nu) <
    \min\big(\epsilon_0\,,\mfrac{1}{10}\,\delta_3\,,1\big)
  \end{equation*}
  holds, then this leads to
  \begin{align*}
    \norm{h_i^j}_0 &\le
    \norm{h_i}_0 + K\,\tilde{C}\,\epsilon(\nu)
    \le \frac{2}{10}\,\delta_3,\\
    \norm{h_i^j}_1 &\le
    \epsilon_0 + \epsilon_0 = \epsilon_\infty,\\
    \norm{h_i^j}_k &\le
    \norm{h_i}_k + K\,\tilde{C}\,\epsilon(\nu)
    \le C_0 + 1 = C_\infty.
  \end{align*}
  This shows that indeed the assumed bounds $\epsilon_\infty$ and
  $C_\infty$ hold, and that $h_i^j$ is defined at least on the ball
  $B(0;r+2\,\delta r)$.

  The sequential smoothings create and preserve $\BUC^l \cap C^\infty$
  smoothness, while every `point' $x^j \in M^j$ is touched by these
  operations. Moreover, Lemma~\ref{lem:conv-smoothing} and
  Corollary~\ref{cor:graphsize-coordtrans} together guarantee that the
  smoothing in each chart keeps
  \begin{equation*}
    \norm{h_i^j}_l \le A\,\norm{h_j^{j-1}}_0 + B
                   \le A\,C(\nu,l)\,\norm{h_j^{j-1}}_0 + B
                   \le A\,C(\nu,l)\,\epsilon_\infty +B
  \end{equation*}
  bounded with a uniform estimate, at least on charts
  $B(x_i;\delta_2)$.

  Finally, we want to estimate the sizes and distance between the
  graphs $h_\bullet,\,h_\bullet^\infty$ in split coordinate charts of
  radius $\delta = \delta_2 - r$ along either $M$ or $M_\sigma$. If
  $x$ is a point either on $M$ or $M_\sigma$, then it is contained in
  at least one ball $B(x_i;r)$ and
  $B(x;\delta) \subset B(x_i;\delta_2)$. If we also set
  $\epsilon_\infty \le \zeta$ and consider the coordinate
  transformation $\phi$ from normal coordinates at $x_i$ to $x$, then
  Lemma~\ref{lem:graphdiff-coordtrans} and
  Corollary~\ref{cor:graphsize-coordtrans} hold and can be used to
  estimate
  \begin{equation*}
    \norm{h_x^\infty}_l \le A\,\norm{h_i^\infty}_l + B \qquad\text{and}\qquad
    \norm{h_x^\infty - h_x}_k
    \le \tilde{C}\,\norm{h_i^\infty - h_i}_k
    \le \tilde{C}^2\,K\,\epsilon(\nu)
  \end{equation*}
  for all points in the domains of $h_x$ and $h_x^\infty - h_x$ within
  $B(0;\delta)$. So if we set
  $\tilde{C}^2\,K\,\epsilon(\nu) < \sigma$, then $M_\sigma$ is
  $C^k$ close to $M$ in normal coordinate charts of radius
  $\delta$ along either $M$ or $M_\sigma$, while at the same time
  $M_\sigma \in \BUC^l \cap C^\infty$.

  If $M$ is a uniformly embedded submanifold, then $\delta_3$ was
  chosen small enough such that $M_{x,\delta_3}$ is the unique
  connected component of $M$ in $B(x;\delta_3)$ for any $x \in M$. We
  now show that the same holds for $M_\sigma$ with balls of radius
  $\delta$. Let $\tilde{x} \in M_\sigma$ be arbitrary and
  $x = \Phi^{-1}(\tilde{x}) \in M$. We take
  $\tilde{y} \in M_\sigma \cap B(\tilde{x};\delta)$ and want to prove
  that $\tilde{y} \in (M_\sigma)_{\tilde{x},\delta}$. By the uniform
  estimates made before, both $M_{x,\delta_2}$ and
  $(M_\sigma)_{\tilde{x},\delta_2}$ can be represented by graphs
  $h_x, h_x^\infty$ respectively in coordinates
  $B(0;\delta_2) \subset \T_x M \oplus N_x$. We have
  $y = \Phi^{-1}(\tilde{y}) \in B\big(\tilde{x};\delta+(r-\delta_1)\big)
  \subset B\big(x;\delta+2(r-\delta_1)\big)$,
  so $y \in B(x;\delta_2) \cap M = M_{x,\delta_2} = \Graph(h_x)$.

  By the construction of $M_\sigma$
  we have $x \in B(x_i;r)$ in some chart $i$, but also
  $(M_\sigma)_{x,\delta_2} = \Graph(h_x)$. Let
  $y \in B(x;\delta) \cap M_\sigma$; we want to prove that
  $y \in (M_\sigma)_{x,\delta} = \Graph(h_x^\infty)$. Since
  $y \in B(x;\delta)$, then for its original it must hold that
  \begin{equation*}
    y^0 \in B\big(x  ;  \delta+(r-\delta_1)\big) \subset
            B\big(x_i;r+\delta+(r-\delta_1)\big),
  \end{equation*}
  hence $y^0 \in M_{x_i,r+\delta+(r-\delta_1)} = \Graph(h_x)$.
  Following the change of $M$ to $M_\sigma$ in coordinates around $x$,
  we see that $y \in \Graph(h_x^\infty)$ must hold.
\end{proof}

\section{Embedding into a trivial bundle}
\label{sec:BG-triv-bundle}

\newcommand{\barN}{{\mkern 4mu\overline{\mkern-4mu N\mkern-2mu}\mkern 2mu}}

Let $\pi\colon N \to M$ be the normal bundle over $M$ immersed in
$(Q,g)$, a Riemannian manifold of bounded geometry. We are going to
construct a trivial bundle $\barN$ over $M$ that contains $N$ and
preserves uniform properties. As a second step, we extend a normally
hyperbolic vector field to this trivial bundle setting. This procedure
is also alluded to
in~\cite[p.~333--334]{Sakamoto1994:smoothlin-invarmfld}, but
especially in the case of bounded geometry requires a more careful
inspection.

\begin{theorem}[Uniform embedding of a normal bundle in a \idx{trivial bundle}]
  \label{thm:normalbundle-triv-emb}
  Let $M \in \BUC^{k \ge 1}$ be a uniformly immersed submanifold of the bounded
  geometry manifold $(Q,g)$. Then there exists an embedding
  $\lambda\colon N \embedto \barN$ of the (nontrivial) normal bundle
  $\pi\colon N \to M$ into a larger, trivial vector bundle
  $\barN = M \times \R^{\bar{n}}$. The embedding map
  $\lambda \in \BUC^{k-1}$ is a quasi-isometry when restricted to
  $N_{\le\eta}$ for any $\eta > 0$ and the splitting
  $\barN = \lambda(N) \oplus N^\perp$ is $\BUC^k$, where $N^\perp$ is chosen
  perpendicular to $\lambda(N)$ according to the standard Euclidean metric on\,
  $\R^{\bar{n}}$.
\end{theorem}

Note that $\lambda$ only has smoothness $C^{k-1}$ since
$N \in C^{k-1}$ is the normal bundle of $M \in C^k$. The image bundle
$\lambda(N) \subset \barN$ has smoothness $C^k$, though, since its
construction only involves the immersion $M \to Q$. This increase of
smoothness is possible because we do not view $\lambda(N)$ as a normal
bundle with respect to the differentiable structure of $Q$ anymore.
This can be compared to the remark
in~\cite[p.~205]{Fenichel1971:invarmflds} and the reference
to~\cite[Lem.~23]{Whitney1936:diffmflds} therein.

The idea of the proof is to use normal coordinate charts of $Q$
covering $M$ to construct local
trivialization maps of $N$. In such charts $B(x;\delta)$, we have,
from Definition~\ref{def:unif-imm-submfld},
$M_{x,\delta} = \exp_x\big(\Graph(h_x)\big) \subset M$ and
trivialization maps $\tau_x$ for the vector bundle trivialization
diagram
\begin{equation}\label{eq:bundle-triv}
  \begin{aligned}
    \xymatrix@R=4em@C=1em{
      *+[l]{N \supset N|_{M_{x,\delta}}} \ar[d]_{\pi} \ar[r]^-{\tau_x} &
      *+[r]{M_{x,\delta} \times N_x}     \ar[dl]^{p_1} \\
      *+[l]{M \supset M_{x,\delta}}
    }
  \end{aligned}
\end{equation}
Then we take a uniformly locally finite cover of $M$ by sets
$M_{x_i,\delta}$. The trivializations on each $M_{x_i,\delta}$ induce
a spanning set of sections, i.e.\ws a frame. Using the uniformity of the
cover, we can globally glue these frames together to
obtain $\rank(\barN) = \bar{n} = (K \tp 1)\,\rank(N)$. Here, $K$
is the maximum number of overlapping charts in the cover. This
identifies $\lambda(N)$ as the subbundle of $\barN = M \times \R^{\bar{n}}$
spanned by these glued frames.

\begin{proof}
  Let $\delta$ be $Q$\ndash small as in Definition~\ref{def:M-small},
  as well as sufficiently small such that $M$ is given as the graph of
  $h_\bullet\colon \T_\bullet M \to N_\bullet$ in normal coordinates
  as in Definition~\ref{def:unif-imm-submfld}. For any $x \in M$ we
  have a trivialization map
  \begin{equation}\label{eq:loc-triv-N}
    \tau_x = (\exp_x,p_2)\circ\D\exp_x^{-1}
    \colon N|_{M_{x,\delta}} \diffto M_{x,\delta} \times N_x,
  \end{equation}
  where we canonically identified $\T(\T_x Q) \cong (\T_x Q)^2$ and
  apply $\exp_x$ only on the base $\T_x Q$ and $p_2$ on the fibers of
  $\T(\T_x Q)$. In a normal coordinate representation (see
  page~\pageref{fig:normcoord-submfld},
  Figure~\ref{fig:normcoord-submfld} on the right) this just means
  that we project the normal fiber $N_{x'}$ at any point
  $x' = \exp_x(\xi,h_x(\xi)) \in M_{x,\delta}$ onto $N_x$. By
  Corollary~\ref{cor:split-coordtrans} this projection is
  approximately orthogonal and bounded away from non-invertibility for
  small $\delta$, hence $\tau_\bullet \in \BUC^{k-1}$ and it is a
  quasi-isometry, but only on a finitely sized neighborhood
  $N_{\le\eta}|_{M_{x,\delta}}$ since it acts linearly on the fibers
  of $N|_{M_{x,\delta}}$. We then choose a uniformly locally finite
  cover $\bigcup_{i \ge 1} M_{x_i,\delta}$ of $M$ such that the sets
  $M_{x_i,\delta/2}$ already cover $M$.

  Next, we prove the existence of a finite set of $\BUC^{k-1}$
  sections that everywhere span $N$. Let $G = (V,E)$ be the (possibly
  infinite) graph whose vertices are sets in the cover,
  $V = \{M_{x_i,\delta}\}_{i \ge 1}$, and edges are added between
  overlapping sets, i.e.
  $E = \set{(A,B) \in V^2}{A \cap B \neq \emptyset}$. Each set in
  the cover overlaps at most $K$ other sets, so the maximal
  degree of $G$ is bounded by $K$. Therefore, we can `color' the
  vertices of $G$ with numbers $\{0,\ldots,K\}$ such that no two
  connected vertices have the same number. Sequentially for each
  $i \ge 1$, set the number of vertex $i$ to one of the numbers
  $\{0,\ldots,K\}$ that is not already taken by its neighbors. We thus
  obtain a map $c\colon \N \to \{0,\ldots,K\}$ such that each preimage
  $c^{-1}(k)$ labels a collection of mutually disjoint sets of the
  cover.

  Let $n$ denote the rank of $N$ and let
  $\barN = M \times \R^{\bar{n}}$ be a trivial bundle with rank
  $\bar{n} = n\,(K+1)$. On each $M_{x_i,\delta}$ we have an orthogonal
  frame $e_i$ of $n$ sections that span $N$, induced by the local
  trivialization, while on $\barN$ we have the global orthogonal frame
  $\bar{e}$ of standard unit sections. The latter can also be viewed
  as a $(K \tp 1)$\ndash tuple of $n$\ndash frames
  $\{\bar{e}_k\}_{0 \le k \le K}$ on $\R^{\bar{n}} = (\R^n)^{K+1}$.
  Since all spaces have (the standard Euclidean) inner products, the
  dual frames can be canonically identified as the inverse
  $e_i^* = e_i^{-1}\colon N_{x_i} \to \R^N$ and a projection
  $\bar{e}_k^* = p_k\colon \R^{\bar{n}} \to \R^n$ onto the $k$\th
  $n$\ndash tuple of all $\bar{n}$ coordinates respectively.
  Let the functions $\chi_i$ be a square-sum partition of unity
  subordinate to the cover according to
  Corollary~\ref{cor:part-squared-unity} and define the embedding
  \begin{equation}\label{eq:embed-triv-bundle}
    \lambda\colon N_{\le\eta} \embedto \barN_{\le\eta}
           \colon (m,\nu) \mapsto
           \sum_{i \ge 1} \chi_i(m)\,\bar{e}_{c(i)}\,e_i^*\,\tau_{x_i}(m,\nu).
  \end{equation}
  This mapping is $\BUC^{k-1}$ as a composition of such maps (and can
  be extended, albeit non-boundedly so, to a map $N \embedto \barN$).
  The $\tau_{x_i}$ are quasi-isometries, while
  $\bar{e}_{c(i)}\,e_i^*$ is isometric on each $M_{x_i,\delta}$. Each
  frame $\bar{e}_{c(i)}$ is orthogonal to the frame of any overlapping
  set $M_{x_{j},\delta}$ since $c(j) \neq c(i)$ and the $\chi_i$
  squared sum to one. Thus $\sum_i \chi_i\,\bar{e}_{c(i)}\,e_i^*$ is
  an isometry, and so $\lambda$ is a quasi-isometry.

  Let $p$ and $p^\perp = \Id - p$ denote the projections from
  $\barN$ onto $N$ and $N^\perp$, respectively. One can verify that
  \begin{equation}\label{eq:proj-N}
    p\, = \sum_{i,j \ge 1}
          \chi_i\,\chi_j\,\bar{e}_{c(i)}\,\bar{e}_{c(j)}^*
  \end{equation}
  is the projection onto $N$ by noting that $\lambda(N)$ equals the
  image of $\sum_{i \ge 1} \chi_i\,\bar{e}_{c(i)}$, while the
  identities
  \begin{equation*}
    \bar{e}_k^*\,\bar{e}_l = \delta_{kl}
    \qquad\text{and}\qquad
    \sum_{i \ge 1} \chi_i^2 = 1
  \end{equation*}
  can be used to show that $p^2 = p$. Formula~\ref{eq:proj-N} shows
  that both $p,p^\perp \in \BUC^k$, hence the splitting
  $\barN = \lambda(N) \oplus N^\perp$ is $\BUC^k$.
\end{proof}

Next, we must extend a vector field $v$ on $N$ to the larger bundle
$\barN$. There are additional directions along the fibers of
$N^\perp$ and the extended vector field $\bar{v}$ must be such that it
is normally hyperbolic in these directions as well. On the other hand,
the uniform boundedness of $v$ must be preserved. We do not assume
here that $M$ is the exact invariant manifold, since these results
shall be applied after application of
Theorem~\ref{thm:unif-smooth-submfld}, which has smoothed and slightly
altered $M$ such that it is not the original NHIM anymore. The
extension $\bar{v}$ will keep $N$ invariant and is identical to $v$ on
$N$, so in the end, we can conclude that the perturbed manifold is
contained in $N$ and restrict to the original setting again.

\begin{lemma}[Normally hyperbolic extension of a vector field]
  \label{lem:triv-bundle-vf-ext}
  Let $\lambda\colon N \embedto \barN = M \times \R^{\bar{n}}$ be a
  trivializing embedding of vector bundles as in
  Theorem~\ref{thm:normalbundle-triv-emb} and let
  $v \in \BUC^{l,\alpha}$ be a vector field on $N$ with
  $l+\alpha \le k-2$. Let $\barN|_{\le\eta}$ be the restriction to
  some radius $\eta > 0$. Then $v$ can be extended to a vector field
  $\bar{v}$ on $\barN$, such that $\bar{v}$ is $\BUC^{l,\alpha}$ on
  $\barN_{\le\eta}$, leaves $N$ invariant, and contracts at a given
  exponential rate $\rho < 0$ along the fiber direction of\/ $N^\perp$
  towards $N \subset \barN$.
\end{lemma}

To extend $v$ to a vector field $\bar{v}$ on $\barN$ with the
required properties, we must do two things. First of all, $v$ must be
extended from $N$ through $\lambda$ to the whole of
$\lambda(N) \oplus N^\perp$ and secondly, a
normal component along the fibers of $N^\perp$ must be added to make
$\bar{v}$ contracting, thus normally hyperbolic in that direction. The
idea can be expressed in local coordinates
$(m,y,z) \in \lambda(N) \oplus N^\perp$ as
\begin{equation*}
  \bar{v}(m,y,z)
  = \hat{v}(m,y,z) + v^\perp(m,y,z)
  = \big( v(m,y) , \rho\,z \big),
\end{equation*}
where $\hat{v}(m,y,z) = v(m,y)$ points `horizontally' along $\lambda(N)$ and
$v^\perp(m,y,z) = \rho\,z$ is the `vertical' component along the $N^\perp$
fibers. By construction, the latter has the required contraction
property in the $N^\perp$ direction, while it preserves
$\lambda(N) \oplus \{ 0 \}$ as an invariant manifold. We shall make this
intuitive idea rigorous by introducing an appropriate bounded
connection to lift $v$ to $\hat{v}$ for $z \neq 0$; the second term
$v^\perp$ is canonically defined.

\begin{proof}
  The embedding map $\lambda$ is a quasi-isometry and of class $\BUC^{k-1}$, hence
  the pushforward $\lambda_*(v) = \D\lambda \cdot v\circ\lambda^{-1}$
  is a $\BUC^{l,\alpha}$ vector field on $\lambda(N) \subset \barN$.
  From now on we identify $N$ with $\lambda(N)$ as well as $v$ with
  its pushforward.

  Let $g$ be the standard Euclidean metric on $\barN$ and $\nabla$
  the compatible, trivial, flat connection. The restricted metric
  $g^\perp = g|_{N^\perp}$ is preserved by the connection
  $\nabla^\perp = p^\perp\cdot\nabla$.
  We create the pullback bundle
  \begin{equation*}
    E = \pi_N^*(N^\perp)
      = \set[\big]{(y,z) \in N \times N^\perp}{\pi_N(y) = \pi_{N^\perp}(z)}
    \cong N \oplus N^\perp = \barN.
  \end{equation*}
  Note that we identify this pullback of $N^\perp$ along the
  projection $\pi_N\colon N \to M$ with the vector bundle
  $p\colon \barN \to N$, that is, we view $N$ as the base manifold
  and $\barN$ as bundle over $N$ with fibers
  $\pi_N^{-1}(y) = N^\perp_{\pi_N(y)}$. We naturally endow $E$ with
  the pullback connection $\hat{\nabla} = \pi_N^*(\nabla^\perp)$. With
  this connection, $v$ can be lifted to a unique vector field
  $\hat{v} \in \vf(\barN)$ that is horizontal along $N$ on
  $E \cong \barN$, and thus the flow of $\hat{v}$ preserves the norm
  along the fibers of $E$, that is, $\hat{v} \cdot \norm{\slot}_E = 0$.
  More heuristically, we can say that the pullback $\pi_N\colon N \to
  M$ introduces a trivial additional base coordinate $y \in N_m$ to
  the bundle $\pi_{N^\perp}\colon N^\perp \to M$.

  To prove that $\hat{v} \in \BUC^{l,\alpha}$, we first recover an
  explicit representation of the Christoffel symbols of $\hat{\nabla}$
  in terms of trivial coordinates on
  $E \cong \barN = M \times \R^{\bar{n}}$, and then a representation
  for the lifted vector field $\hat{v}$.
  Let $s \in \Gamma(\barN),\, s \equiv s_0 \in \R^{\bar{n}}$ be a
  constant section. Let $m \in B(m_0;\delta)$ denote normal
  coordinates in $M$. Define
  $s^\perp = p^\perp \cdot s \in \Gamma(N^\perp)$ and
  $\hat{s} = \pi_N^*(s^\perp) \in \Gamma(E)$. Let $\hat{X} \in \T N$ a
  tangent vector in the base of $E$ and
  $X = \D\pi_N(\hat{X}) \in \T M$. Then we find for the covariant
  derivative $\hat{\nabla}$ on $E$
  \begin{equation*}
    \hat{\nabla}_{\hat{X}}\, \hat{s}
    = \pi_N^*\big(p^\perp \cdot \nabla_X (p^\perp \cdot s)\big)
    = \pi_N^*\big(p^\perp \cdot X^i \Big[
         p^\perp\,\pder{s}{m^i} + \pder{p^\perp}{m^i}\,s \Big]\big).
  \end{equation*}
  We read off that the Christoffel symbols are given by
  $p^\perp\,\pder{p^\perp}{m^i}\,(\D\pi_N)^i$, so they are
  $\BUC^{k-1}$. The horizontal lift
  \begin{equation}\label{eq:ext-v-horizontal}
    \hat{v}(m,y,z) = v(m,y) - p^\perp \cdot (\pder{p^\perp}{m^i}\,z)(\D\pi_N\,v(m,y))^i,
  \end{equation}
  then, is also $\BUC^{l,\alpha}$ since all functions involved are at
  least $\BUC^{l,\alpha}$ in these coordinates, and $z$ is bounded on
  $\barN_{\le\eta}$.

  We define $v^\perp$ as the Euler vector field along the fibers of
  $\barN \cong E$, taking values in $\Ver(E)$. Each fiber
  $E_{(m,y)} = N^\perp_m$ is a linear space, so the tangent space at
  any point is canonically identified with the fiber itself, which
  allows us to canonically define
  \begin{equation}\label{eq:ext-v-vertical}
    v^\perp\colon \barN \to \T \barN
           \colon (m,y,z) \mapsto \rho\,z \in \T_z N^\perp_m = \Ver(E)_{(m,y,z)}.
  \end{equation}
  This vector field leaves $N \subset \barN$ invariant, while
  generating a flow that attracts towards $N$ at the exponential rate
  $\rho < 0$. It is clear that $v^\perp \in \BUC^k$ for any $k$ when
  $z$ is bounded on $\barN_{\le\eta}$.

  We conclude that the vector field $\bar{v} = \hat{v} + v^\perp$
  indeed leaves $N$ invariant. Since $\hat{v}$ is neutral in the fiber
  directions of $E$ and $v^\perp$ contracting at rate $\rho$, it
  follows that $\bar{v}$ is contracting with rate $\rho$ as well. The
  combined vector field $\bar{v}$ is defined in terms of $v$ and other
  functions that are all at least $\BUC^{l,\alpha}$, hence
  $\bar{v} \in \BUC^{l,\alpha}$.
\end{proof}

\section{Reduction of a NHIM to a trivial bundle}
\label{sec:BG-reduction}

Having set up the theory of bounded geometry spaces, we are finally in
a position to reduce a general normally hyperbolic system to the
setting of a trivial bundle. That setting is required to apply our
basic persistence Theorem~\ref{thm:persistNHIMtriv} for NHIMs. Let
$M \in \BUC^{k,\alpha}$ with $k \ge 2$ and $0 \le \alpha \le 1$ be a
uniformly immersed or embedded submanifold in $(Q,g)$ of bounded
geometry and furthermore assume that $M$ is an $\fracdiff$\ndash NHIM with
$\fracdiff = k+\alpha$ for the vector field $v \in \BUC^{k,\alpha}$ on $Q$.

\begin{remark}
  The bounded smoothness requirement $k \ge 2$ is dictated by
  Theorem~\ref{thm:tubular-neighbhd}. It is not present in the compact
  case where the normal bundle can be ``jiggled
  slightly''~\cite[Prop.~2]{Fenichel1971:invarmflds} to make it
  sufficiently smooth to model a $C^k$ flow for $k \ge 1$.
  Hypotheses~2 and~3 in~\cite[p.~987]{Bates1999:persist-overflow}
  require similar conditions in the noncompact setting in Banach
  spaces, see also the discussion after Corollary~\ref{cor:persistNHIM-Rn}
  and Remark~\ref{rem:loss-smoothness}.
  I have not investigated in detail whether the $\BC^2$
  requirement is necessary, or if $M \in \BUC^1$ could be modeled as
  a sufficiently small section of the normal bundle of a smooth
  approximate manifold.
\end{remark}

We reduce this system to a trivial bundle in the following steps:
\begin{enumerate}
\item approximate $M$ by a smoothed manifold $M_\sigma$;
\item construct a tubular neighborhood of $M$ in the normal
  bundle $N$ of $M_\sigma$;
\item embed $N$ into a trivial bundle
  $\barN = M_\sigma \times \R^{\bar{n}}$, and construct an extended,
  normally hyperbolic vector field $\bar{v}$;
\item after application of the basic persistence theorem in the
  enlarged bundle, push the results to the original setting and
  conclude that $M$ persists.
\end{enumerate}

\begin{proof}[Proof of Theorem~\ref{thm:persistNHIMgen}]
Assume the hypotheses of the theorem. First,
Theorem~\ref{thm:unif-smooth-submfld} gives a smooth, $C^k$\ndash close approximation
$M_\sigma \in \BUC^l$ of $M$, where the choice $l = k + 10$
suffices and $\sigma > 0$ will be fixed later. The $C^{k,\alpha}$
bounds of $M_\sigma$ are uniformly close to those of $M$ for all
$\sigma$ small. Then, Theorem~\ref{thm:tubular-neighbhd} says that
there exists a tubular neighborhood
$\phi\colon N_{\le\eta} \diffto B(M_\sigma;\eta)$ where the size
$\eta > 0$ depends only on the $C^2$ bounds of $M_\sigma$. These
bounds are of the same order as those of $M$, independent of $\sigma$.
Hence, we can choose $\sigma$ so small that
$\norm{M_\sigma - M}_k \le \sigma \le \frac{\eta}{2}$ and the
neighborhood $B(M;\eta/2)$ is fully within the tubular neighborhood
$N_{\le\eta}$ of $M_\sigma$. The map $\phi$ is a $C^{l-1}$ bounded
(local) diffeomorphism and a quasi-isometry, so by the
reasoning\footnote{%
  This is similar to Fenichel's argument
  in~\cite[p.~203]{Fenichel1971:invarmflds} that normal hyperbolicity
  is independent of a choice of metric when $M$ is compact.%
} before Definition~\ref{def:quasi-isometry}, a pullback by $\phi$
does not change the normal hyperbolicity growth rates of the vector
field $v$. The bounded continuous splitting~\ref{eq:NHIM-split} of
$\T_M Q$ is also preserved.

Next, as a result of Theorem~\ref{thm:normalbundle-triv-emb}, $N$ is
embedded into the trivial bundle
\begin{equation*}
  \lambda\colon N \embedto \barN = M_\sigma \times \R^{\bar{n}} = N \oplus N^\perp.
\end{equation*}
The embedding is a quasi-isometry and so preserves
hyperbolicity properties of $v$. The extended vector field
$\bar{v} \in \BUC^{k,\alpha}$ on $\barN_{\le\eta}$ is constructed in
Lemma~\ref{lem:triv-bundle-vf-ext} as a lift of $v$, so the flow
preserves $N_{\le\eta}$ and intertwines with the projection onto $N$,
while in the perpendicular direction along the fibers of $N^\perp$ it
has the same normal hyperbolicity properties as $v$. Boundedness of
the invariant splitting $\T_M Q = \T M \oplus N^+ \oplus N^-$ is also
preserved under these quasi-isometries. The additional
directions along $N^\perp$ are stable and invariant, and have bounded
projections by construction. Thus, $M$ is an $\fracdiff$\ndash NHIM for
$\bar{v}$ as well.

The invariant manifold $M$ is given by the graph of a section
$h \in \Gamma(N) \subset \Gamma(\barN)$, and from
Theorem~\ref{thm:unif-smooth-submfld} it follows that
$\norm{h}_k \le \sigma$ while $h \in \BUC^{k,\alpha}$.
By Lemma~\ref{lem:boundgeom-submfld}, $X = M_\sigma$ has bounded
geometry of order $l - 2 = k + 8$, which is sufficiently smooth for
the conditions of Theorem~\ref{thm:persistNHIMtriv}, while
$Y = \R^{\bar{n}}$ is clearly a Banach space. Hence, we are in the
trivial bundle setting, and all conditions are satisfied.

A small perturbation of $v$ in the original setting in $Q$ corresponds
to a small perturbation of $\bar{v} \in \vf(\barN)$, while
$N_{\le\eta}$ is preserved under the flow by construction. Therefore, after
application of Theorem~\ref{thm:persistNHIMtriv} we recover a unique
persistent invariant manifold $\tilde{M}$ and by construction
$\tilde{M} = \Graph(\tilde{h}) \subset N$, so we can restrict
the system to $N$. Then it can be transferred back to $Q$ under the
quasi-isometries of the embedding
$\lambda\colon N \embedto \barN$ and the tubular neighborhood map
$\phi\colon N_{\le\eta} \to B(M;\eta) \subset Q$.

All size estimates can be transferred between the settings (with
bounded factors) due to the near isometry and uniform $C^{k,\alpha}$
boundedness of $\phi$ and $\lambda$. We conclude that $\tilde{M}$ is a
$\BUC^{k,\alpha}$ submanifold of $Q$ for appropriate estimates
$\sigma, \epsilon$, and $\delta$, where $\sigma$ must be chosen
sufficiently small as well to make $\norm{\tilde{M} - M}_{k-1}$
small.


This completes the reduction from the general setting in bounded
geometry to that of a trivial bundle and proves
Theorem~\ref{thm:persistNHIMgen}.
\end{proof}


\chapter{Persistence of noncompact NHIMs}\label{chap:persist}

This chapter contains the main proof of persistence of noncompact
normally hyperbolic invariant manifolds, formulated in
Theorem~\ref{thm:persistNHIMtriv}. This theorem is formulated
in a specific setting: we assume that the invariant manifold $M$
is (nearly) the zero section of a trivial vector bundle. This is a
slightly more general formulation than
in~\cite{Henry1981:geom-semilin-parab-PDEs,Sakamoto1990:invarmlfds-singpert}.
There, it is assumed that in a product $X \times Y$ of Euclidean (or
Banach) spaces, the invariant manifold $M$ is given as the graph of a function
$h\colon X \to Y$. We shall also assume that $Y$ is a vector space,
but we let $X$ instead be a Riemannian manifold that is
finite-dimensional and has bounded geometry.
In Chapter~\ref{chap:boundgeom} on bounded geometry, we extended
the result obtained here to a setting where $M$ is assumed to
be a general submanifold of a Riemannian manifold $(Q,g)$ that is
again finite-dimensional and of bounded geometry. We assume the basic
statements from Section~\ref{sec:BG-basic} to be known.

This chapter is organized as follows. First we state the two main
theorems; both the general version with $M$ a submanifold of $Q$ and
the $X \times Y$ trivial bundle version to be proved in this chapter.
We provide detailed remarks on these theorems and compare them to the
literature. Then we present an outline of the proof of
Theorem~\ref{thm:persistNHIMtriv}.
Section~\ref{sec:cpt-unif} presents some thoughts on replacing the
classical compactness by uniformity conditions, and presents examples
that indicate the necessity of various assumptions we impose.

In Section~\ref{sec:preparation} we transform the (still somewhat
geometrical) formulation of Theorem~\ref{thm:persistNHIMtriv} into a
more explicit setup suitable for analysis. In the subsequent section, we
prove (with \emph{relatively} little work) the existence and uniqueness of
the persistent manifold $\tilde{M} = \Graph(\tilde{h})$. It
automatically follows that $\tilde{h}$ is bounded and uniformly
Lipschitz.

The last sections are devoted to the tougher job of proving
$C^{k,\alpha}$ smoothness, exhausting the spectral gap. A formal
scheme is set up, and we work out the details for $C^1$ smoothness.
Higher, $C^k$ smoothness follows along the same lines by induction.
The addition of H\"older continuity to obtain $C^{k,\alpha}$
smoothness is included as a natural extension to (uniform) continuity
that slightly simplifies the spectral gap estimates. See the proof
outline and the introduction of Section~\ref{sec:NHIM-smoothness} for
more details.

\section{Statement of the main theorems}\label{sec:main-thms}

The main theorem on persistence was already formulated in the
introduction. We state it again to directly compare it to the
trivialized bundle version of Theorem~\ref{thm:persistNHIMtriv}.
The main theorem is reduced to this trivialized version in
Section~\ref{sec:BG-reduction}; in this chapter we shall prove the
latter version. Then we formulate corollaries of these theorems, both
to present simpler versions and to compare our result to well-known
results from the literature.

\index{normally hyperbolic invariant manifold!persistence}
\maintheorempersistNHIMgen

\index{trivial bundle}
\index{$\eta$ (tubular neighborhood size)}
\index{$r$}
\index{$v_\sigma,\,h_\sigma$}
\index{$\sigma$ (approximation parameter)}
\begin{theorem}[Persistence of noncompact NHIMs in a trivial bundle]
  \label{thm:persistNHIMtriv}
  Let $k \ge 2$, $\alpha \in \intvCC{0}{1}$ and $\fracdiff = k+\alpha$.
  Let $(X,g)$ be a smooth, complete, connected Riemannian manifold of
  bounded geometry and\/ $Y$ a Banach space. Let
  $v_\sigma \in \BUC^{k,\alpha}$ be a family of vector fields defined
  on a uniformly sized neighborhood of the zero-section in
  $X \times Y$ with family parameter $\sigma \in \intvOC{0}{\sigma_0}$.
  Let the submanifold $M_\sigma = \Graph(h_\sigma)$ be given as
  the graph of a function $h_\sigma \in \BUC^{k,\alpha}(X;Y)$ and let
  $M_\sigma$ be an $\fracdiff$\ndash NHIM with $\rank(E^+) = 0$
  for the flow defined by $v_\sigma$ where all estimates are uniform
  in $\sigma$ and additionally $\norm{h_\sigma}_2 \le \sigma$ holds.

  Then for each sufficiently small $\Ysize > 0$ there exist
  $\sigma_1,\,\delta > 0$ such that for any $\sigma \in \intvOC{0}{\sigma_1}$
  and any vector field $\tilde{v} \in \BUC^{k,\alpha}$ with
  $\norm{\tilde{v} - v_\sigma}_1 < \delta$, there is a unique
  submanifold $\tilde{M} = \Graph(\tilde{h})$,
  $\tilde{h}\colon X \to Y$, $\norm{\tilde{h}}_0 \le \Ysize$ such that
  $\tilde{M}$ is invariant under the flow defined by $\tilde{v}$.
  Moreover, $\tilde{h} \in \BUC^{k,\alpha}$ and
  $\norm{\tilde{h}}_{k-1}$ can be made arbitrary small by
  choosing $\norm{h}_k$ and $\norm{\tilde{v} - v_\sigma}_{k-1}$
  sufficiently small.
\end{theorem}

\begin{remark}\label{rem:persistNHIM}
Let us make some remarks on these theorems.
\begin{enumerate}
\item The spectral gap condition contained in
  Definition~\ref{def:r-NHIM} of $\fracdiff$\ndash normal
  hyperbolicity is essential to the proof. The $C^{k,\alpha}$\ndash
  smoothness result is optimal, see Section~\ref{sec:spectral-gap}.

\item In Theorem~\ref{thm:persistNHIMgen}, both $M$ and $\tilde{M}$
  are assumed to be (non-injectively) immersed according to
  Definition~\ref{def:unif-imm-submfld}. If $M$ is assumed uniformly
  embedded according to Remark~\ref{rem:embedding}, then $\tilde{M}$
  will be uniformly embedded again when $\delta$ is sufficiently
  small.

\item The additional family parameter $\sigma$ in
  Theorem~\ref{thm:persistNHIMtriv} is required to reduce
  Theorem~\ref{thm:persistNHIMgen} to this case. If the unperturbed
  manifold is given as $M = X \times \{ 0 \}$, i.e.\ws the zero
  section, then the family $v_\sigma$ can simply be taken constant and
  all $\sigma$ dependence can be dropped from the formulation.

\item We only obtain a $C^{k-1}$\ndash norm estimate for the
  perturbation distance of $\tilde{M}$ away from $M$, even though
  $\tilde{M} \in C^{k,\alpha}$ is preserved. This is due to a
  linearization along $Y$ and the smoothing convolution used to
  restore $C^{k,\alpha}$ smoothness after linearization. See
  Section~\ref{sec:preparation}, in particular
  remarks~\ref{rem:loss-smoothness} and~\ref{rem:Ck-small-optimal},
  for more details. I fully expect it to hold that $\tilde{M}$ and $M$
  are $C^{k,\alpha}$ close when $\norm{\tilde{v} - v}_{k,\alpha}$ is
  small.\index{smoothness!loss of}

\item\label{enum:persist-k-two}
  The minimum smoothness requirement $k \ge 2$ is a stronger
  assumption than $k \ge 1$ in the well-known compact case. This seems
  to be intrinsic to the noncompact case. If the spectral gap
  condition only holds for some $1 \le \fracdiff < 2$, then we can
  still obtain a perturbed manifold $\tilde{M}$. This manifold
  $\tilde{M}$ will generally not have better than $C^\fracdiff$
  smoothness, though.

\item We allow both values $\alpha = 0$ and $\alpha = 1$, where
  $\alpha = 0$ is considered an empty condition (besides the
  boundedness and uniform continuity). Thus, if $r = k+1$ satisfies
  the spectral gap condition~\ref{eq:spectral-gap}, then we can choose
  both $C^{k,1}$ or $C^{k+1,0}$ as resulting smoothness for
  $\tilde{M}$, if $M$ had the same smoothness. Thus, if $M$ was
  sufficiently smooth, then the choice $\tilde{M} \in C^{k+1,0}$
  yields the best result. Note, though, that by Rademacher's theorem,
  Lipschitz functions are differentiable almost everywhere, so the
  difference is not that big.

  Finally, it should also be noted that
  the spectral gap condition is a strict inequality on $r$, so if we
  can choose $r = k$ integer, then we can also find an $\alpha > 0$
  such that $r' = k + \alpha$ satisfies the spectral gap as well. This
  shows that in this context $\BUC^k$ `integer' smoothness really is a
  special case of $\BUC^{k,\alpha}$ `fractional' smoothness.

\item Both Riemannian manifolds $Q$ in
  Theorem~\ref{thm:persistNHIMgen} and $X$ in
  Theorem~\ref{thm:persistNHIMtriv} are assumed to be
  finite-dimensional; multiple results on bounded geometry crucially
  depend on this fact. On the other hand, we allow $Y$ to be an
  infinite-dimensional Banach space simply because everything
  naturally generalizes to that setting. Note that we do not allow
  semi-flows as
  in~\cite{Henry1981:geom-semilin-parab-PDEs,Bates2008:approx-invarmfld},
  so the case that $Y$ is infinite-dimensional may not be that useful.

\item\label{enum:persist-gen-NHIM}
  These results are weaker than those in the well-known compact
  case in a few aspects. First of all, we use a stricter notion of
  normal hyperbolicity, see Remark~\ref{rem:NHIM-ratios}. This seems
  to be a fundamental restriction of the Perron method; the more
  general definition of normal hyperbolicity is successfully applied
  to noncompact manifolds in~\cite{Bates2008:approx-invarmfld}. Secondly, we
  only include the stable normal bundle $E^-$. Adding the unstable
  bundle $E^+$ as well should be possible, see
  Section~\ref{sec:ext-full-NHIM} for more details.

\item\label{enum:persist-split}
  While we do prove that the NHIM persists into a new invariant
  manifold $\tilde{M}$, we do not prove that $\tilde{M}$ is again
  normally hyperbolic. I fully expect this to be true though: the
  perturbed flow satisfies slightly perturbed exponential growth
  conditions and the spectral gap is an open condition, so should be
  preserved under sufficiently small perturbations. The difficulty
  lies in proving that $\tilde{M}$ again has a continuous invariant
  splitting~\ref{eq:NHIM-split} with bounded projections. This is one
  possible reason for breakdown of normal
  hyperbolicity~\cite{Haro2006:hyperb-breakdown}.
\end{enumerate}
\end{remark}

These two theorems reduce to the corollaries formulated below, when
$M$ is compact or when $Q = \R^{m+n}$ with standard Euclidean metric
and $M = \R^m \times \{ 0 \}$. The statements then significantly
reduce in complexity, and are comparable to well-known results.

Firstly, the case that $M$ is compact. Then we can take a (pre)compact
neighborhood $B(M;\epsilon) \subset Q$ of $M$ and thus conclude that
bounded geometry holds on $\overline{B(M;\epsilon)}$, ignoring
irrelevant boundary problems. Any
$C^{k,\alpha}$ function on $\overline{B(M;\epsilon)}$ is automatically
$\BUC^{k,\alpha}$, so Theorem~\ref{thm:persistNHIMgen} reduces to the
following corollary. For simplicity we leave out $\alpha$\ndash
H\"older continuity and the $C^k$ distance estimate between $M$ and
$\tilde{M}$.
\begin{corollary}[Persistence of compact NHIMs]
  \label{cor:persistNHIMcpt}
  Let $k \in \Z_{\ge 2}$, let\/ $(Q,g)$ be a smooth Riemannian manifold
  and $v \in C^k$ be a vector field on $Q$. Let $M \in C^k$ be a
  connected, compact submanifold of\/ $Q$ that is $k$\ndash normally
  hyperbolic for the flow defined by $v$, with empty unstable bundle,
  i.e.\ws $\rank(E^+) = 0$.

  Then for each sufficiently small $\Ysize > 0$ there exists a
  $\delta > 0$ such that for any vector field $\tilde{v} \in C^k$ with
  $\norm{\tilde{v} - v}_1 < \delta$, there is a unique submanifold
  $\tilde{M}$ in the $\Ysize$\ndash neighborhood of\/ $M$, such that
  $\tilde{M}$ is diffeomorphic to $M$ and invariant under the flow
  defined by $\tilde{v}$. Moreover, $\tilde{M}$ is $C^k$.
\end{corollary}
\index{comparison!of results}%
This corollary closely
resembles~\cite[Thm.~1]{Fenichel1971:invarmflds} in the absence of a
boundary $\partial M$ (a boundary is allowed when $M$ is overflowing invariant, see also
sections~\ref{sec:overflow-invar} and~\ref{sec:ext-overflow-invar}).
Note that our definition of normal hyperbolicity is less general (see
Remark~\ref{rem:NHIM-ratios}), and that we exclude unstable normal
directions and the case\footnote{%
  This was for technical reasons in the noncompact setting, see
  Remark~\ref{rem:persistNHIM},~\ref{enum:persist-k-two}, and could
  be repaired in the compact setting.%
} $k = 1$, while we do allow $M$ to be an immersed submanifold. The
persistence result of Hirsch, Pugh, and
Shub~\cite[Thm.~4.1~(f)]{Hirsch1977:invarmflds} is similar to that in
Fenichel's work; it additionally includes H\"older smoothness and
allows immersed submanifolds as well.

Secondly, the case of $M = \R^m \times \{ 0 \} \subset Q = \R^{m+n}$.
Again, $\R^{m+n}$ has bounded geometry with one trivial, global chart.
Thus, any object is $\BUC^k$ if it can (locally) be described by
$\BUC^k$ functions, but with common global bound and continuity
modulus. Then Theorem~\ref{thm:persistNHIMtriv} reduces to the
following corollary. We again suppress H\"older continuity and drop
the $\sigma$ parameter dependence, which was only relevant for the
reduction of Theorem~\ref{thm:persistNHIMgen}.
\begin{corollary}[Persistence of a trivial NHIM in Euclidean space]
  \label{cor:persistNHIM-Rn}
  Let $k \in \Z_{\ge 2}$. Let $v \in \BUC^k$ be a vector field on\,
  $\R^{m+n}$ and let $M = \R^m \times \{ 0 \}$ be a $k$\ndash NHIM
  for the flow defined by $v$, with empty unstable normal bundle.

  Then for each sufficiently small $\Ysize > 0$ there exists a
  $\delta > 0$ such that for any vector field $\tilde{v} \in \BUC^k$
  with $\norm{\tilde{v} - v}_1 < \delta$, there is a unique
  submanifold $\tilde{M} = \Graph(\tilde{h})$,
  $\tilde{h}\colon \R^m \to \R^n$, $\norm{\tilde{h}}_0 \le \Ysize$
  such that $\tilde{M}$ is invariant under the flow defined by
  $\tilde{v}$. Moreover, $\tilde{h} \in \BUC^k$ and
  $\norm{\tilde{h}}_{k-1}$ can be made arbitrary small by choosing
  $\norm{\tilde{v} - v}_k$ sufficiently small.
\end{corollary}
This theorem can be compared, for example,
to~\cite[Thm.~2.1]{Sakamoto1990:invarmlfds-singpert}. Sakamoto's
theorem is specifically targeted to singular perturbation problems.
His conditions are more specific and concrete: the invariant manifold
$M$ is assumed to consist of stationary points and normal
hyperbolicity is formulated in terms of the eigenvalues of normal
derivatives of the vector field at $M$. He starts with an invariant
manifold that is the graph of a nonzero function $h \in \BC^k$; this
he reduces to the zero graph case $M = \R^m \times \{ 0 \}$, while he
incurs a loss of one degree of smoothness, obtaining a $\BC^{k-1}$
persistent manifold and he requires $k \ge 3$,
see~\cite[p.~50]{Sakamoto1990:invarmlfds-singpert}. He does allow both
stable and unstable normal bundles.

In their series of papers~\cite{Bates1998:invarmflds-semiflows,%
  Bates1999:persist-overflow,Bates2008:approx-invarmfld}, Bates,
Lu, and Zeng obtained multiple results on noncompact NHIMs, including
a persistence result similar to mine. These results are in some senses
complementary, however. Most importantly, they work in Banach spaces
with semi-flows, which adds some nontrivial problems. On the other
hand, my setting allows the ambient space to be a manifold, albeit
finite-dimensional. They use the graph transform instead of the Perron
method. This allows for the more general definition of relative normal
hyperbolicity as in Remark~\ref{rem:NHIM-ratios}. They include both
stable and unstable normal directions, while they do not prove
H\"older regularity. Finally, in~\cite{Bates2008:approx-invarmfld} the
interesting idea is developed to start with an approximate NHIM only.

If we ignore these differences, then their results fit in between the formulations of
Theorem~\ref{thm:persistNHIMgen} and Corollary~\ref{cor:persistNHIM-Rn}.
Their invariant manifold $M$ is immersed in a Banach space, but not
necessarily described by the graph of a function $h$. Their
hypothesis~\cite[p.~363]{Bates2008:approx-invarmfld} that the
splitting does not twist too much is a bounded Lipschitz condition on
the (approximate) splitting of the (un)stable and tangent bundles over
$M$. This condition is similar, but slightly weaker than our condition
$M \in \BUC^2$, see also remarks~\ref{rem:persistNHIM},%
~\ref{enum:persist-k-two} and~\ref{rem:loss-smoothness}.

Although the results of Bates, Lu, and Zeng are more general and
complete in many aspects, I think that these cannot easily be
generalized to prove a version of Theorem~\ref{thm:persistNHIMgen},
set in an ambient manifold of bounded geometry. One could hope to use
the Nash embedding theorem to obtain the ambient manifold $(Q,g)$ as
an isometrically embedded
subspace of some $\R^n$. Then the dynamical system must be extended
from $Q$ to $\R^n$, such that $M$ is still normally hyperbolic as a
submanifold of $\R^n$; this procedure can be compared to the reduction
in Section~\ref{sec:BG-reduction}. The problem that arises is that the
Nash embedding theorem provides no control on the extrinsic curvature
of the embedding\footnote{%
  This can be seen from the result that the Nash embedding can be
  obtained into an arbitrarily small ball. As an explicit example,
  take $Q = \R$ with standard metric and embed it into $\R^2$ via the
  map $r(\theta) = \arctan(\theta)/\pi + \frac{1}{2}$ in polar
  coordinates. Since the integral of $r(\theta)$ diverges both when
  $\theta \to \pm\infty$, we obtain (after arc length reparametrization)
  an isometric embedding of $\R$ `curled up' into $B(0;1)$, while the
  extrinsic curvature grows unbounded for $\theta \to -\infty$.%
}, so $M$ need not be $\BC^2$ in $\R^n$. It might be possible to work around
this by proving a `bounded geometry version' of the Nash embedding
theorem.

We should also mention the paper~\cite{Jones1999:persist-invmlfds-PDE}
by Jones and Shkoller. They generalize Fenichel's results on
persistence of overflowing invariant manifolds to semi-flows on
infinite-dimensional Riemannian manifolds. They do assume the
invariant manifold itself to be compact.

\section{Outline of the proof}\label{sec:proof-outline}

The proof of Theorem~\ref{thm:persistNHIMtriv} is lengthy and
involves a lot of details. We therefore first present an overview of
the separate steps involved in the proof.

First, in Section~\ref{sec:preparation} we bring the system into a
form that is suitable for application of further analytical
techniques. That is, we decompose the vector fields along $X$ and $Y$
directions and linearize the vertical direction, leading to equations
\begin{equation*}
  \begin{aligned}
    \dot{x} &= v_\sx(x,y),\\
    \dot{y} &= v_\sy(x,y) = A(x)\,y + f(x,y)
  \end{aligned}
\end{equation*}
with a $C^1$\ndash small term $f$. This is a generalization
of the classical Perron method for hyperbolic fixed points (see
Section~\ref{sec:perron-method} for a quick overview) to NHIMs, first
presented by Henry~\cite[Chap.~9]{Henry1981:geom-semilin-parab-PDEs}.
In the case of a hyperbolic fixed point, we could fully linearize the
system; here, we can only linearize the normal directions, while we
keep the full nonlinear form in the directions along $X$. These cannot
be linearized because we have no control to localize the dynamics in
the directions along $X$.

In the theorem, the invariant manifold $M$ is given as a (small) graph
$h\colon X \to Y$. A coordinate change to represent the invariant manifold as
$M = X \times \{ 0 \}$ would (re)introduce a loss of smoothness that
we carefully worked around in Chapter~\ref{chap:boundgeom} by means of
a uniformly smoothed submanifold. The graph $h$ can be chosen
arbitrarily close to the zero section, and together with the small
perturbation $\tilde{v} - v$, this influences the exponential growth
rates~\ref{eq:NHIM-rates} only slightly. We recover
equations~\ref{eq:ODE-XY-mod} for the perturbed system that satisfy
slightly perturbed exponential estimates~\ref{eq:exp-est-XY-mod}, even
when we decouple the equations for $x$ and $y$ by inserting curves
$y(t)$ and $x(t)$, respectively, that are `close' to solution curves
of the original system. We directly include the perturbation
$\tilde{v} - v$ into the horizontal component of the vector field;
for the vertical component we include the perturbation in the
nonlinear term $\Ynonlin$. This gives rise to a nonlinear, horizontal
flow $\Phi_y(t,t_0,x_0)$ and a linear, vertical flow
$\Psi_x(t,t_0)\,y_0$ that depend on a curve in the other space and
satisfy estimates
\begin{equation*}
  \begin{alignedat}{2}
    &\forall\;t \le t_0\colon&\quad
    \norm{\D\Phi_y(t,t_0,x_0)} &\le C_\sx\,e^{\rho_\sx\,(t-t_0)},\\
    &\forall\;t \ge t_0\colon&
    \norm{\Psi_x(t,t_0)}       &\le C_\sy\,e^{\rho_\sy\,(t-t_0)},
  \end{alignedat}
\end{equation*}
where $\rho_\sx,\,\rho_\sy$ are close to the original exponential
rates $-\rho_M,\,\rho_-$.

The next step in Section~\ref{sec:NHIM-lipschitz} is to define a pair
of maps~\ref{eq:Txy} between curves in $X$ and $Y$ in terms of these
flows and the decomposed vector fields~\ref{eq:ODE-XY-mod}. The
composition $T = T_\sy \circ (T_\sx, \text{pr}_1)$ of these maps will
be a contraction on bounded curves in $Y$ depending on a parameter
$x_0 \in X$, but we measure these curves with $\norm{\slot}_\rho$
norms for some exponent $\rho$ with $\rho_\sy < \rho < \rho_\sx$.
Lemma~\ref{lem:perron-equiv} shows that the fixed points
of $T$ are precisely the vertical parts
$y(t)$ of solution curves of the perturbed vector field $\tilde{v}$
that stay in the tubular neighborhood $\norm{y} \le \Ysize$ of $M$ and
have initial value $x_0$ for their horizontal part $x(t)$. The maps
$T_\sx$ and $T_\sy$ generalize the center-unstable and stable
components, respectively, of the Perron integral in the fixed point
case, cf.\ws Section~\ref{sec:perron-method}.
The nonlinear flow along the invariant manifold is used
in $T_\sx$, but now depends on the vertical component $y(t)$ too,
while in the vertical, normal directions we use a variation of
constants integral to separate the nonlinear terms from the linearized
flow, just as in the classical Perron method.

This setup leads to a fixed point map $\Theta\colon X \to \BY$
that maps an initial value $x_0 \in X$ to the unique
bounded curve in $Y$ that corresponds to a full solution curve $(x,y)$
such that $x(0) = x_0$. If we now evaluate the
vertical solution curve $y = \Theta(x_0)$ at $t = 0$, then we obtain
the vertical component $y_0 = y(0) \in Y$ of the initial value
corresponding to $x_0 \in X$. All these solution curves stay close to
$M$ and form an invariant manifold, so the graph of
\begin{equation*}
  \tilde{h}\colon X \to Y
           \colon x_0 \mapsto \Theta(x_0)(0)
\end{equation*}
must describe the unique perturbed invariant manifold $\tilde{M}$.
Application of the contraction principle immediately implies that
$\tilde{h}$, and therefore $\tilde{M}$, is Lipschitz continuous.

In Section~\ref{sec:NHIM-smoothness} we continue to prove that
$\tilde{M}$ is $C^{k,\alpha}$. We start that section with a
more detailed overview of this smoothness part of the proof and in
Section~\ref{sec:scheme-deriv} we present a scheme to obtain the first
derivative in a number of steps. Higher smoothness then follows along
the same lines, just with more complex expressions, see
Section~\ref{sec:higher-derivs}. Let us focus here on the basic ideas.

Smoothness of $M$ follows directly from smoothness of $\Theta$. We
study the derivatives of $\Theta$ by formal differentiation of the
fixed point equation; this leads to Equation~\ref{eq:der-fixedpoint}.
But let us consider for a moment a simpler heuristic formulation,
similar to equation~\ref{eq:deriv-perron-map} for the hyperbolic fixed
point in Section~\ref{sec:comp-smooth}. Then the derivatives of the
Perron fixed point map are
\begin{equation*}
  \D^k T(y)\big(\dy_1,\ldots,\dy_k\big)(t)
  = \int_{-\infty}^t \Psi(t,\tau)\cdot
      \D^k \Ynonlin(y(\tau))\big(\dy_1(\tau),\ldots,\dy_k(\tau)\big) \d\tau.
\end{equation*}
Even if $\D^k \Ynonlin$ is bounded, it acts as a multilinear map on a
$k$\ndash tuple of variations $\dy_i$, each having exponential growth
of order $\rho$, so the result has exponential growth of order
$k\,\rho$. This is canceled by the exponential growth of
$\Psi(t,\tau)$ if $\rho_\sy < k\,\rho$. Then $\D^k T$ can be viewed as
a contraction on $\D^k \Theta$, but only when $\D^k \Theta$ is viewed
as a map into $B^{k\,\rho}(I;Y)$. To obtain continuity of the maps
$\D^k T$, we have to add another arbitrarily small term $\mu < 0$ to
the exponent, i.e.\ws $k\,\rho + \mu$; in case of $\alpha$\ndash
H\"older continuity we need $\mu = \alpha\,\rho$. These key facts show
how the spectral gap condition limits smoothness; see
Section~\ref{sec:spectral-gap} for a detailed discussion and an
example that shows that our smoothness result is in fact sharp.

The technique of using a scale of Banach spaces as developed by
Vanderbauwhede and Van~Gils~\cite{Vanderbauwhede1987:centermflds}, and
the fiber contraction theorem of Hirsch and
Pugh~\cite{Hirsch1970:stabmflds-hyperb} can be applied, and we obtain
each $\D^k \Theta$ as a fixed point in the appropriate space. The
final conclusion $\tilde{h} \in C^k$ follows by evaluating
$\D^k \Theta$ at $t = 0$.

We have to be very careful however: higher derivatives of maps between
manifolds are difficult to define (at least in a practical way), so we
develop some theory to describe higher derivatives using normal
coordinates in Appendix~\ref{chap:exp-growth} and generalize results
from $\R^n$ to this setting. Secondly, the derivatives of
$T_\sx,\,T_\sy$ only exist as `formal derivatives' on `formal tangent
bundles'. We endow these formal tangent bundles with a topology
induced by parallel transport. This allows us to study continuity of
the formal derivatives $\DF T_\sx,\,\DF T_\sy$ at the cost of
introducing additional holonomy terms. Finally, we do obtain the
$\D^k \Theta$ maps as derivatives of $\Theta$.

\section{Compactness and uniformity}\label{sec:cpt-unif}
\index{noncompactness}

The classical results on normally hyperbolic invariant
manifolds~\cite{Fenichel1971:invarmflds,Hirsch1977:invarmflds} assume
the invariant manifold to be compact. This is used to obtain
uniform boundedness and continuity of the vector field and other
objects. Here, instead, we assume these objects to have the required
uniformity directly, replacing the compactness requirement. In this
section, we expose some of the issues that need to be dealt with and
we present accompanying examples. We focus here on those issues that
are not (clearly) present in the literature and only show up when
considering general manifolds. See Section~\ref{sec:noncpt-motivation}
for motivation and examples of noncompact NHIMs.

The primary requirement in the noncompact case---well-known to
experts in the field---is that the vector field defining the system
must be uniformly bounded, including all its spatial derivatives up to
the order of the smoothness result requested.
Secondly, the vector field should be uniformly continuous, and
uniformly $\alpha$\ndash H\"older continuous when $\alpha \neq 0$.
The case $\alpha = 0$ really is a special case, whose proof
needs more care. H\"older continuity provides an explicit continuity
estimate, which is tailored to the problem; `plain' uniform continuity
does not provide this, forcing us to use an (arbitrarily) small amount
of the spectral gap to compensate.

In the compact case, Fenichel~\cite[p.~200]{Fenichel1971:invarmflds}
argues that persistence of the invariant manifold should be
independent of the choice of a Riemannian metric. Indeed, he proves
that the exponential growth rates are independent of such a choice, as
all metrics are equivalent. In a noncompact setting, however,
non-equivalent metrics do exist and we do expect persistence to depend
on the choice of metric, since it determines which perturbations are
globally $C^1$ small. Moreover, we make a technical uniformity
assumption of \emph{bounded geometry} (see
Chapter~\ref{chap:boundgeom}, also for example spaces of bounded
geometry) on both the underlying space and
the invariant manifold. These assumptions are automatically satisfied
in the compact case. It is not clear though to what extend they are
essential in the noncompact case.

The remainder of this section is devoted to examples that show
multiple aspects that should be treated carefully in the noncompact
setting, while being trivially fulfilled in the compact case.
Some interesting examples can also be found in the early
work~\cite{Hoppensteadt1966:singpert-inftime} by Hoppensteadt. He
presents counterexamples to uniform stability of solutions in a
time\hyp{}dependent singular perturbation setting when the stability
criteria do not have sufficient uniformity.

\subsection{Non-equivalent metrics}

As a simple example of two metrics leading to different results in the
noncompact case, let us consider the following.
\begin{example}[Non-equivalent metrics]
  \label{exa:non-equiv-metrics}
Let $X = \R^2,\, Y = \R$ with on the one
hand the usual Euclidean metric $g_e$ and on the other hand a metric
$g_s$ induced by a diffeomorphism similar to stereographic projection
from the sphere (with North Pole removed) onto $X = \R^2$.

Let the vector field be given by
\begin{equation}\label{eq:v-nonequiv-metric}
  v(x,y) = \big(\arctan(\norm{x})\,\frac{x}{\norm{x}},\;\lambda\,y\big)
\end{equation}
with $\lambda < 0$. This makes the vertical, $y$ direction uniformly
attracting with exponent $\lambda$, while in the plane
$X \times \{0\}$ the origin is an expanding fixed point and the
exponential growth rate is everywhere non-negative. Thus,
$M = X \times \{0\}$ is an $\fracdiff$\ndash NHIM for arbitrarily
large $\fracdiff$ and $v \in \BUC^k$ for any $k \ge 1$ on a tubular
neighborhood of $M$ of size $\Ysize = 1$, say. See
Figure~\ref{fig:nonequiv-metric} on the left side.

In polar coordinates $(s,\theta)$ on $X$ we have
\begin{equation*}
  v(s,\theta,y) = \big(\arctan(s),\,0,\,\lambda\,y\big).
\end{equation*}
Now, instead of the usual stereographic projection map
$(\phi,\theta) \mapsto (s = \tan(\phi/2),\theta)$ with $\phi = 0$ at
the South Pole, we take
\begin{equation}\label{eq:diffeo-nonequiv-metric}
  f\colon \intvCO{0}{\pi} \to \R_{\ge 0}
   \colon \phi \mapsto -\log(1 - \phi/\pi)
\end{equation}
and the corresponding diffeomorphism $\Phi$ that acts trivially along
the directions of $\theta$ and $y$ coordinates.
This diffeomorphism induces a metric $g_s$ on $X$ by pushforward of
the standard metric on the sphere. The system with metric $g_s$ is
most easily studied by pullback to the sphere with North Pole removed,
see Figure~\ref{fig:nonequiv-metric} on the right side.
This is an equivalent formulation since $\Phi$ is an isometry by
construction. The vector field is then represented on $S^2 \times \R$
by
\begin{equation}\label{eq:v-pullback-sphere}
  (\Phi^* v)(\phi,\theta,y) =
  \big(\pi(1-\phi/\pi)\arctan\big(-\log(1-\phi/\pi)\big),\;0,\;\lambda\,y\big).
\end{equation}
This shows that the vector field is still $\BC^1$, although not
$\BUC^1$ anymore in this metric. More importantly, the system can be
extended to include the North Pole as an attracting fixed point. The
rate of attraction along the perpendicular $y$ direction has not
changed from $\lambda$, but along the horizontal directions of the
sphere, the attraction rate now is
\begin{equation}\label{eq:exprate-stereo-metric}
  \lim_{\phi \to \pi} \D_1(\Phi^* v)(\phi,0,0) = -\frac{\pi}{2}.
\end{equation}
\begin{figure}[tb]
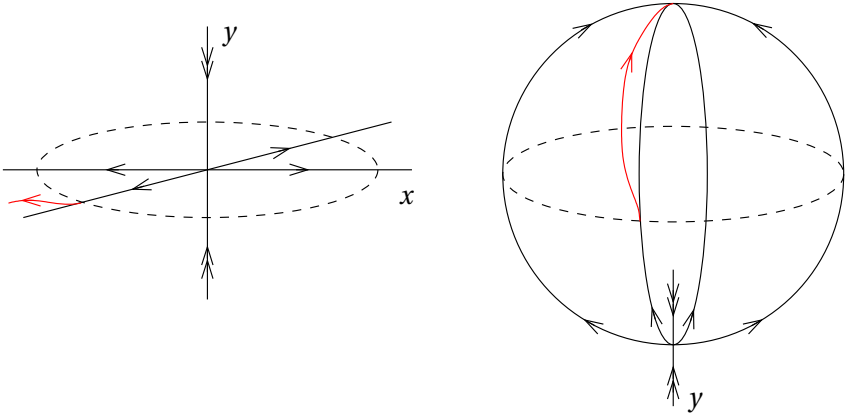

  \centering
  \raisebox{1.5cm}{%
  \input{\figpath nonequiv-metric-plane.pspdftex}\hspace{1cm}
  }
  \input{\figpath nonequiv-metric-sphere.pspdftex}
  \caption{a normally hyperbolic flow with respect to Euclidean and
    sphere induced metrics.}
  \label{fig:nonequiv-metric}
\end{figure}

In both metrics the normal exponential attraction rate is
$\rho_\sy = \lambda < 0$ since the linear flow on $Y$ is decoupled from
$X$. The exponential growth rate on $X$ does depend on the choice of
metric. For the Euclidean metric $g_e$, we have $\rho_\sx = 0$. This
follows from an analysis of the radial component
$\dot{s} = \arctan(s)$ of the system. At $s = 0$ this has unstable
exponent $1$, while away from $s = 0$ the tangent flow is uniformly
bounded away from both zero and infinity, so there the Lyapunov
exponent is zero. With respect to the metric $g_s$, it follows
from~\ref{eq:exprate-stereo-metric} that the Lyapunov exponent is
$\rho_\sx = -\frac{\pi}{2}$ for solutions approaching planar infinity.
Thus, we see that if $\lambda \ge -\frac{\pi}{2}$, then the system is
not normally hyperbolic with respect to the metric $g_s$, while
$X \times \{ 0 \}$ is an $\fracdiff$\ndash NHIM for any
$\fracdiff \ge 1$ under the metric~$g_e$. On the other hand, if
$\lambda < -\frac{\pi}{2}$, then the system is still only
$\fracdiff$\ndash normally hyperbolic with respect to~$g_s$ for
$\fracdiff < \pi/(2\lambda)$. Again, we can construct an explicit
perturbation similar to that in Example~\ref{exa:opt-smooth}. If we add a
small vertical perturbation $\epsilon\,\chi$ with support away from
the poles (compare with Figure~\ref{fig:NHIM-opt-smoothness}),
then along meridians passing through $\supp \chi$, the invariant
manifold is lifted and approaches the North Pole $\phi = \pi$
approximately along a graph $y = C\,(\pi-\phi)^{-2\lambda/\pi}$, while on
meridians not passing through $\supp \chi$, the invariant manifold
stays at $y = 0$. See the perturbed flow lines in Figure~\ref{fig:nonequiv-metric}.
This results in unbounded $C^\fracdiff$ derivatives
of the perturbed invariant manifold at the North Pole for
$\fracdiff > -2\lambda/\pi$.
\end{example}

We conclude that in the noncompact setting, normal hyperbolicity
explicitly depends on the choice of metric since metrics need not be
equivalent. Moreover, the allowed size of the perturbations depends on
the metric.
\begin{example}[Perturbation sizes depend on the choice of metric]
  \label{exa:pertsize-metric-dep}
We extend Example~\ref{exa:non-equiv-metrics} above. Let
$\lambda < -\frac{\pi}{2}$ so that the system is normally hyperbolic
with respect to both metrics, and set
\begin{equation}\label{eq:pert-nonequiv-metric}
  w(s,\theta,y) = \big(0,\;0,\;\arctan(s)\,\sin(\theta)\big).
\end{equation}
Then the vector field $v+\epsilon\,w$ is a $C^1$ small perturbation
of~\ref{eq:v-nonequiv-metric} with respect to~$g_e$ and perturbs the
original manifold $M$ smoothly to a manifold $\tilde{M}$ that has
height converging to $y = -\epsilon\,\sin(2\,\theta)\,\pi/(2\lambda)$
along radials when $s \to \infty$. The pullback of $\tilde{M}$ to the
sphere, however, has a discontinuity at the North Pole, since
$s \to \infty$ corresponds to $\phi \to \pi$, and so the North Pole is
approached at different constant heights along these radials. This
apparent contradiction that $\Phi^*(\tilde{M})$ is not a $C^1$ small
perturbation with respect to~$g_s$ stems from the fact that $w$ is not
$C^1$ small with respect to this metric. The vector field $w$ has
unbounded derivatives since $g_s$ `squeezes' distances when
approaching planar infinity, that is, the North Pole.
\end{example}

Thus, non-equivalent metrics also lead to different classes of $C^1$
small perturbations under which the invariant manifold persists.
Moreover we see that one cannot simply get rid of noncompactness by a
compactification argument. The metric $g_s$ is induced by a one-point
compactification to the sphere, but leads to different normal
hyperbolicity properties than the noncompact case with metric~$g_e$.
Any other choice of the diffeomorphism $\Phi$ would lead to the same
problems, since pullback of the metric~$g_e$ must introduce a
singularity at the North Pole. This cannot be equivalent to a metric
that extends regularly there.

\subsection{Non-persistence of embedded NHIMs}

\index{persistence!non-}
Let us give another example which shows that an embedded invariant
manifold need not persist. This example clarifies the remarks already
made in the introduction in Section~\ref{sec:immersed}: noncompact
embedded NHIMs can perturb into immersed manifolds. On the one hand,
this example shows that it is natural to consider immersed NHIMs. It
also shows that a noncompact NHIM must have a uniformly sized tubular
neighborhood that does not self-intersect, in order to guarantee
perturbation as an embedded manifold. Further details can be found in
Section~\ref{sec:BG-submanifolds} where the concept of a \idx{uniformly
embedded submanifold} is defined.

\begin{figure}[htb]
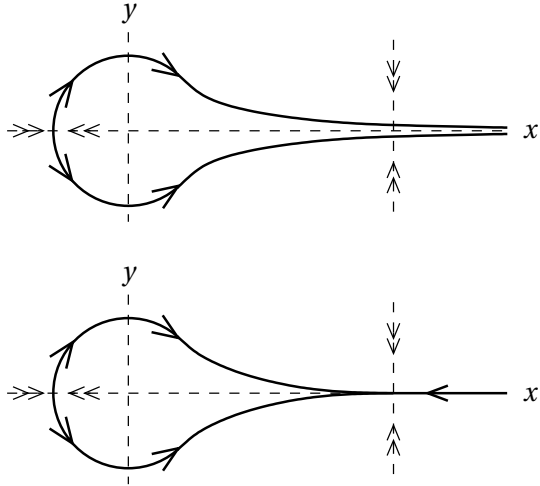

  \centering
  \input{\figpath selfintersect-orig.pspdftex}\\[10pt]
  \input{\figpath selfintersect-mod.pspdftex}
  \caption{collapse of a nearly self-intersecting invariant manifold.}
  \label{fig:selfintersect}
\end{figure}
In the example presented here, the unperturbed manifold is normally
hyperbolic but noncompact and `touches' itself in the limit to infinity, see
Figure~\ref{fig:selfintersect}, the top image. In this case, we can
find arbitrarily small perturbations that will let the two persisting
branches collapse into one at a finite point.

\begin{example}[A non-uniformly embedded NHIM]
  \label{exa:nonunif-embed-NHIM}
Let $(x,y) \in \R^2 = Q$. For $x \le 0$ we define the vector field $v$
of the system in polar coordinates, and for $x \ge \mfrac{1}{2}$ in
Cartesian coordinates as
\begin{equation}\label{eq:v-self-intersect}
  \begin{alignedat}{2}
    (\dot{r},\dot{\theta}) &= v(r,\theta) = ( 1-r \,,\, -\sin(\theta/2) )\qquad
    &&\text{if}\quad x = r\,\cos(\theta) \le 0,\\
    (\dot{x},\dot{y})      &= v(x,y) = ( e^{-x} \,,\, -y )
    &&\text{if}\quad x \ge \mfrac{1}{2}.\\
  \end{alignedat}
\end{equation}
We glue these vector fields together in a smooth way somewhere between
$x = 0$ and $x = \mfrac{1}{2}$. Then the manifold as shown at the top
in Figure~\ref{fig:selfintersect} is a NHIM. The flow attracts
uniformly in the normal direction with rate $-1$ (except that the rate
may deviate slightly around the glued area), while along the
manifold, the flow has an expanding fixed point at $(-1,0)$ and the
contraction in the direction of $x \to \infty$ is weaker than
exponential. Explicitly solving the flow for $x \ge \mfrac{1}{2}$
yields
\begin{equation*}
  \Phi^t(x,y) = ( \log(e^x + t) \,,\, y\,e^{-t} ),
\end{equation*}
which exhibits the rates of contraction in the normal and tangential
directions by considering either projection in
\begin{equation*}
  \lim_{t \to \infty}\; \frac{1}{t}\,\log\big(\pi_{x,y} \circ \D\Phi^t(x,y)\big).
\end{equation*}

Let us now introduce the very simple perturbation vector field
$w(x,y) = ( -\epsilon \,,\, 0 )$ for $x \ge 1$ and smoothly cut off to
zero left of $x = 1$. When this perturbation is added to the vector
field $v$, the vertical line $x = -\log(\epsilon)$ becomes a stable,
invariant set, see Figure~\ref{fig:selfintersect} the bottom image.
The upper and lower branch of the original
NHIM will both converge to the newly created fixed point
$( -\log(\epsilon) \,,\, 0)$. On the right side of this point the
manifold is given by the single line $y = 0$.

Each branch separately persists as a $C^\infty$ manifold, as could
(naively) be expected. The problem is that we have no control on the
distance between the two branches, so for any $\epsilon > 0$, these
branches will collapse at some point where the persisting object
ceases to be an embedded manifold. As already remarked, there are two
ways to address this issue. One can abandon the implicit assumption
that the NHIM is an embedded submanifold and replace this by immersed
submanifolds; this idea was introduced already
in~\cite{Hirsch1977:invarmflds}. If one insists on having
embedded submanifolds, even under perturbations, then one must
eliminate the possibility of these `collapses' occurring. A sufficient
condition is the existence of a uniformly sized tubular neighborhood
of the invariant manifold that does not intersect itself. Global
control on the perturbation distance of the invariant manifold will
imply that the perturbed manifold stays inside this tubular
neighborhood and thus will not self-intersect.
\end{example}

\subsection{Non-uniform geometry of the ambient space}
\index{ambient manifold}

The previous examples were set in Euclidean space. The next two
examples show that additional uniformity conditions must be imposed on
a nontrivial ambient space. It is not
enough to assume uniform continuity and boundedness for the dynamical
system. The first example is an extension to the previous one and
shows that the ambient space must have a uniformly finite injectivity
radius. The second example indicates that even if the ambient space
has finite injectivity radius and trivial topology, persistence might
be lost due to non-bounded curvature of the ambient space.

\begin{example}[Zero \idx{injectivity radius}]
  \label{exa:zero-inj-radius}
We construct as ambient space $Q$ a cylinder whose
radius shrinks exponentially. That is, we take $Q = \R \times S^1$
with metric $g(x,\theta) = {\rm d}x^2 + e^{-2x} {\rm d}\theta^2$. See
Figure~\ref{fig:exp-cylinder} for an impression, but note that the
metric induced by the embedding in $\R^3$ is not (and cannot be made)
the same as $g$. The vector field
\begin{equation}\label{eq:v-exp-cylinder}
  v(x,\theta) = ( 1,\, 0 )
\end{equation}
generates a simple flow along the cylinder, and \emph{each} solution
curve is a NHIM purely due to the fact that all curves flow into an
exponentially shrinking tube, while there is no contraction along the
curve.
\begin{figure}[htb]
  \centering
  \includegraphics[width=8cm]{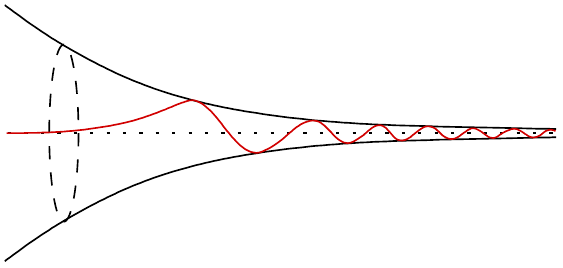}
  \vspace{-0.3cm}
  \caption{a non-uniform cylinder with a winding curve.}
  \label{fig:exp-cylinder}
\end{figure}

Let us consider the invariant manifold $M = \{\,\theta = 0\,\}$. We
add a perturbation to the vector field that is given by
$\dot{\theta} = \epsilon$ for $x \ge 0$ and is smoothly cut off to
zero left of $x = 0$. This perturbation is smooth and $C^1$ small with
respect to the metric\footnote{%
  Measuring the $C^1$ size with respect to $g$ requires taking
  covariant derivatives and may introduce results not directly
  apparent in coordinates $(x,\theta)$. A perturbation
  term $\dot{\theta} = \epsilon\,\exp(x)$ would still be globally
  small in this metric.%
} $g$. When the original curve $M$ enters the region $x \ge 0$, it is
modified to a curve $\tilde{M}$ that starts winding around the
cylinder, as indicated in Figure~\ref{fig:exp-cylinder}.

This clearly cannot be represented in a tubular neighborhood of $M$ in
$Q$ since the curve would leave the neighborhood `above' and reenter
`from below'. On the other hand, the normal bundle of $M$ can be
viewed as a covering of $Q$, and on that covering, $\tilde{M}$ is
represented by the function $\theta(x) = \epsilon\,x$, which is still
a bounded graph with norm $\norm{\theta(x)} = \epsilon\,x\,e^{-x}$,
but which winds around since $\theta \in \intvCO{0}{2\,\pi}$. Thus, a
globally finite injectivity radius seems a necessary requirement if we
want the perturbed manifold to be represented in a diffeomorphic
tubular neighborhood of $M$.
\end{example}

The second example indicates that a finite injectivity radius is not
enough; unbounded curvature of the ambient manifold might lead to loss
of persistence of the NHIM. It should be pointed out that this example
satisfies all properties of normal hyperbolicity with uniform
estimates up to $C^1$ smoothness, except that the vector field
has no uniformly continuous derivative. I~have not been able to add
this final property to create a complete counterexample where
persistence fails in the absence of the curvature property of bounded
geometry only.

\begin{example}[Unbounded curvature]
  \label{exa:unbounded-curv}
Let $Q = \R^3$ with metric
\begin{equation}\label{eq:metric-exp-geom}
  g(x,y,z) = \d x^2 + \exp\big(-2\,\abs{x}\arctan(x\,z)\big)\,{\rm d} y^2
                    + \exp\big(-2\,\abs{x}\big)\,{\rm d} z^2,
\end{equation}
but with component functions symmetrically smoothed around $x = 0$.
This Riemannian manifold is invariant under translations in $y$ and
has a mirror symmetry involution in any plane of fixed $y$. Hence,
each submanifold $\{\, y = y_0 \,\}$ is geodesically invariant.

Let the vector field be
\begin{equation}\label{eq:v-exp-geom}
  v(x,y,z) = \begin{cases}
    \;\big(x,\, 0,\, -\arctan(z)   \big) & \text{for}\quad \abs{x} \le 1,\\
    \;\big(\text{sign}(x),\, 0,\, 0\big) & \text{for}\quad \abs{x} \ge 1,\\
  \end{cases}
\end{equation}
and smoothly glued together in a neighborhood of the boundary $\abs{x} = 1$. Thus, the
whole system is invariant under translations in $y$, and within any
plane $\{\, y = y_0 \,\}$, the point $(x,z) = (0,0)$ is a hyperbolic
fixed point with eigenvalues $1$ and $-1$ in the $x$ and $z$
direction, respectively. The system also has a mirror symmetry around
$x = 0$; from now on we only consider $x \ge 0$.
\begin{figure}[htb]
  \centering
  \input{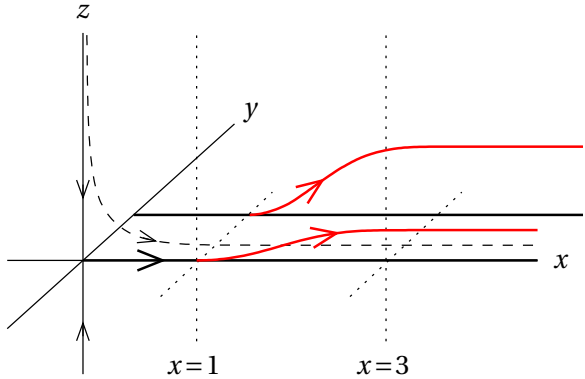}
  \caption{perturbation of a normally hyperbolic system in unbounded geometry.}
  \label{fig:exp-geom-R3}
\end{figure}

The plane $M = \{\, z = 0 \,\}$ is a NHIM; it is clearly invariant
under the flow, and similar to the exponentially shrinking cylinder,
the metric contracts in the $z$ direction along solution curves
$x \to \infty$, while no contraction occurs along the manifold. On
$\T M$ the metric reduces to $g|_{\T M} = {\rm d} x^2 + {\rm d} y^2$ while the
flow is linear in time. On a neighborhood of the $y$\ndash axis,
finally, normal hyperbolicity follows from the attraction along the
$z$ directions due to the term $-\arctan(z)$ in~\ref{eq:v-exp-geom}.

The vector field $v$ and its covariant derivative are uniformly
bounded with respect to the metric. For $x \le 2$, this follows from
the fact that $g$ including its inverse and derivatives, as well as
$v$ and its derivatives are bounded. For $x \ge 2$, explicit
calculations in local coordinates show that $\norm{v} = 1$, while for
$x \ge 1$ we have
\begin{equation*}
  \nabla\,v =
    \begin{pmatrix}
      0 & 0 & 0\\
      0 & -\frac{x\,z}{1 + (x\,z)^2} - \arctan(x\,z) & 0\\
      0 & 0 & -1
    \end{pmatrix},
\end{equation*}
expressed in an orthonormal frame, which is bounded as well. The
second covariant derivative $\nabla^2 v$ is unbounded, though. This
indicates that $\nabla v$ is probably not uniformly continuous for a
reasonable definition of uniform continuity, cf.\ws
Definition~\ref{def:unif-bounded-map}, although I have not completely
investigated this question.

The Ricci scalar curvature of $Q$ is unbounded and on
$M = \{\, z = 0 \,\}$ it is given by
\begin{equation*}
  S = -2\,\big(1 + x^4\,e^{2 x}\big).
\end{equation*}
Clearly, this implies that the Riemannian curvature is unbounded too.

\begin{remark}
  In hindsight, it should probably not come as a complete surprise
  that $\nabla^2 v$ is unbounded. The Riemannian curvature is
  unbounded, and since it is the generator of holonomy (see
  Section~\ref{sec:BG-curv-holonomy}), it can thus generally be
  expected that holonomies along infinitesimal loops act as an
  unbounded family of operators on $v$. These are expressed in local
  coordinates by second covariant derivatives of $v$:
  \begin{equation*}
    R(\partial_i,\partial_j)\,v
    = \nabla_{\partial_i} \nabla_{\partial_j} v
     -\nabla_{\partial_j} \nabla_{\partial_i} v.
  \end{equation*}
\end{remark}

We proceed with checking that $\exp\colon N \to Q$ has finite
injectivity radius on the normal bundle\footnote{
  We should actually show that the injectivity radius of $Q$ is
  finite, i.e.\ws $\rinj{Q} > 0$, at least in a neighborhood of $M$. I
  have not been able to do this. Finite injectivity radius of the
  normal bundle does allow us to construct a tubular neighborhood to
  model persistent manifolds close to $M$, though.%
} $N$ of $M$.
Then all assumptions for persistence are fulfilled, except for bounded
curvature (and uniform continuity of $\nabla v$). As each submanifold
$\{\, y = y_0 \,\}$ is invariant, we can restrict our investigation to
$y = 0$, such that $x$ denotes the coordinate along the base manifold
of the normal bundle; let $t$ denote the normalized coordinate in the
vertical direction. The exponential map of $(x,t)$ is generated by the
geodesic flow as follows: start at $x$ with vertical unit vector and
then follow a geodesic for time $t$. For $x = 2$, say, this flow is
well-defined and stays inside the region $x \ge 1$ for some bounded
time $\abs{t} \le r$. The diffeomorphism group
\begin{equation*}
  \phi(x,z) = (x+\xi,z\,e^\xi) \qquad\text{with}\quad \xi \in \R
\end{equation*}
translates along $x$ while simultaneously scaling $z$, see also
Remark~\ref{rem:indep-BG-conditions}.
In the region $x \ge 1$ this is an isometry, so the exponential
mapping defined for $x = 2,\, \abs{t} \le r$ can be isometrically
mapped onto the whole region $x \ge 2,\, \abs{t} \le r$. For $x$ on
the compact interval $\intvCC{0}{2}$ the exponential map must have a
finite injectivity radius too, so there exists a global $r > 0$ such
that $\exp\colon N_{\le r} \to Q$ is diffeomorphic onto its image.

Now we add a perturbation in a similar spirit to that in
Section~\ref{sec:spectral-gap}: we lift $M$ by a local, vertical
perturbation of the vector field, varying along~$y$. In a neighborhood
of the plane $x = 2$ we add a small vertical component
\begin{equation*}
  \dot{z} = \epsilon\,(2-\cos(y))\,\exp\Big(\frac{-1}{1-(x-2)^2}\Big)
  \qquad\text{if}\quad \abs{x-2} \le 1.
\end{equation*}
In the region $x \le 1$ the flow is unmodified, so there the perturbed
manifold $\tilde{M}$ must coincide with the original $M$; otherwise it
would not stay in a bounded neighborhood of $M$ under the backward
flow. Around $x = 2$, the flow lifts $\tilde{M}$ to at least a height
\begin{equation*}
  z \ge \epsilon\,\int_{1}^3 \exp\Big(\frac{-1}{1-(x-2)^2}\Big) \d x
    \ge \epsilon/4
\end{equation*}
and the height $z$ depends on $y$, see
Figure~\ref{fig:exp-geom-xslice}. Then in the region $x \ge 3$ the
manifold $\tilde{M}$ continues along $\dot{x} = 1$ at the same $y,z$
coordinates. Now we have $\epsilon/4 \le z \le 6\,\epsilon$, so for
all small $\epsilon > 0$ the $y$ component of the metric along the
flow on the invariant manifold eventually shrinks at an exponential
rate that is stronger than in the $z$ 
direction, while $\tilde{M}$ has variable height $z(y)$ independent of
$x$. Hence the Lipschitz norm of $\tilde{M}$ can be estimated by
measuring $z'(y)$ with respect to the metric $g$ along the manifold.
But $z'(y)$ is nonzero and constant along $x$ in coordinates, while
horizontal distances along $y$ shrink faster than vertical distances.
This means that the Lipschitz norm of $\tilde{M}$ grows unbounded for
$x \to \infty$.
Moreover, the normal exponential growth rate does not dominate the
tangential rate anymore, so the perturbed manifold is not normally
hyperbolic anymore.
\end{example}
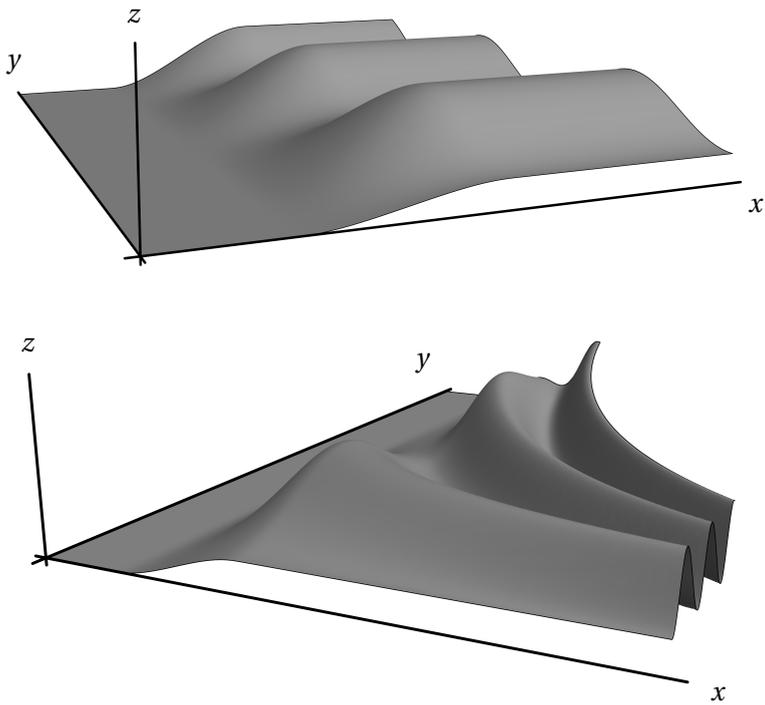
\begin{figure}[htb]
  \centering
  \begin{overpic}[width=10cm]{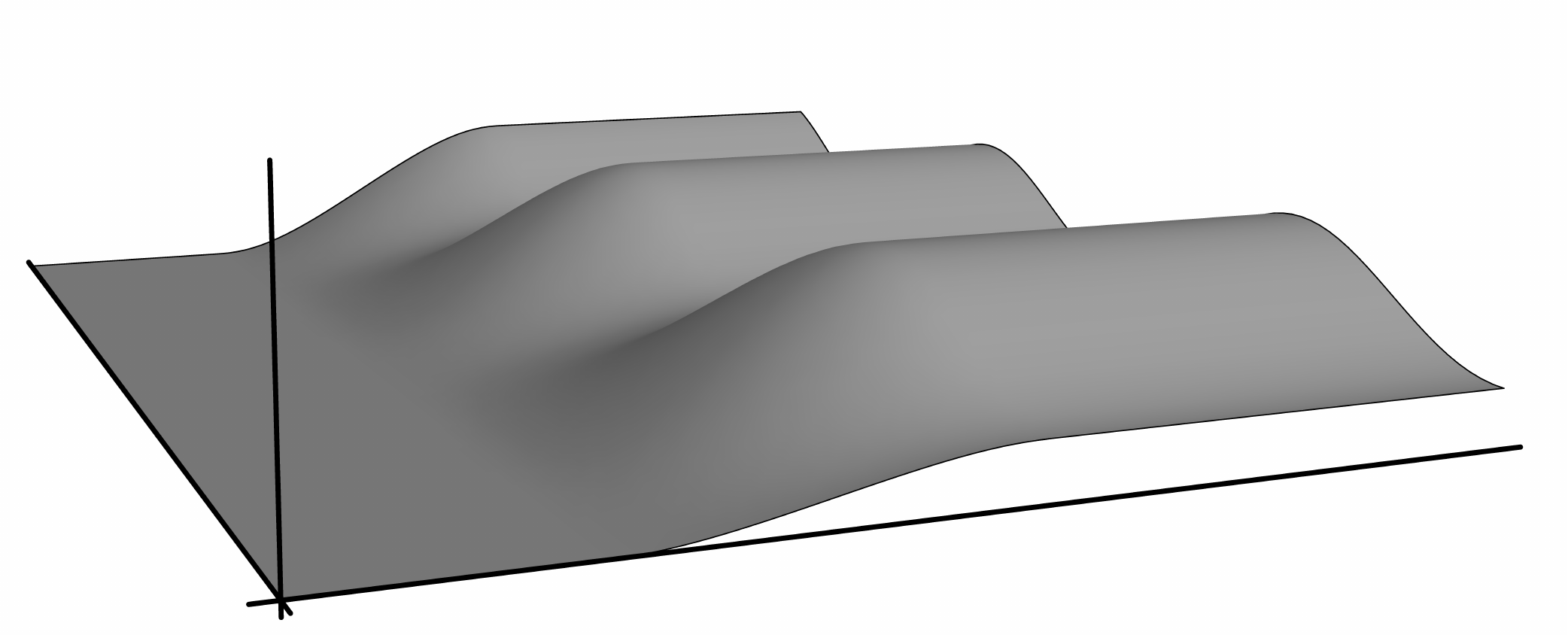}
    \put(98,08){$x$}
    \put(00,27){$y$}
    \put(16,33){$z$}
  \end{overpic}
  \begin{overpic}[width=10cm,clip,trim=0cm 0cm 3cm 0cm]{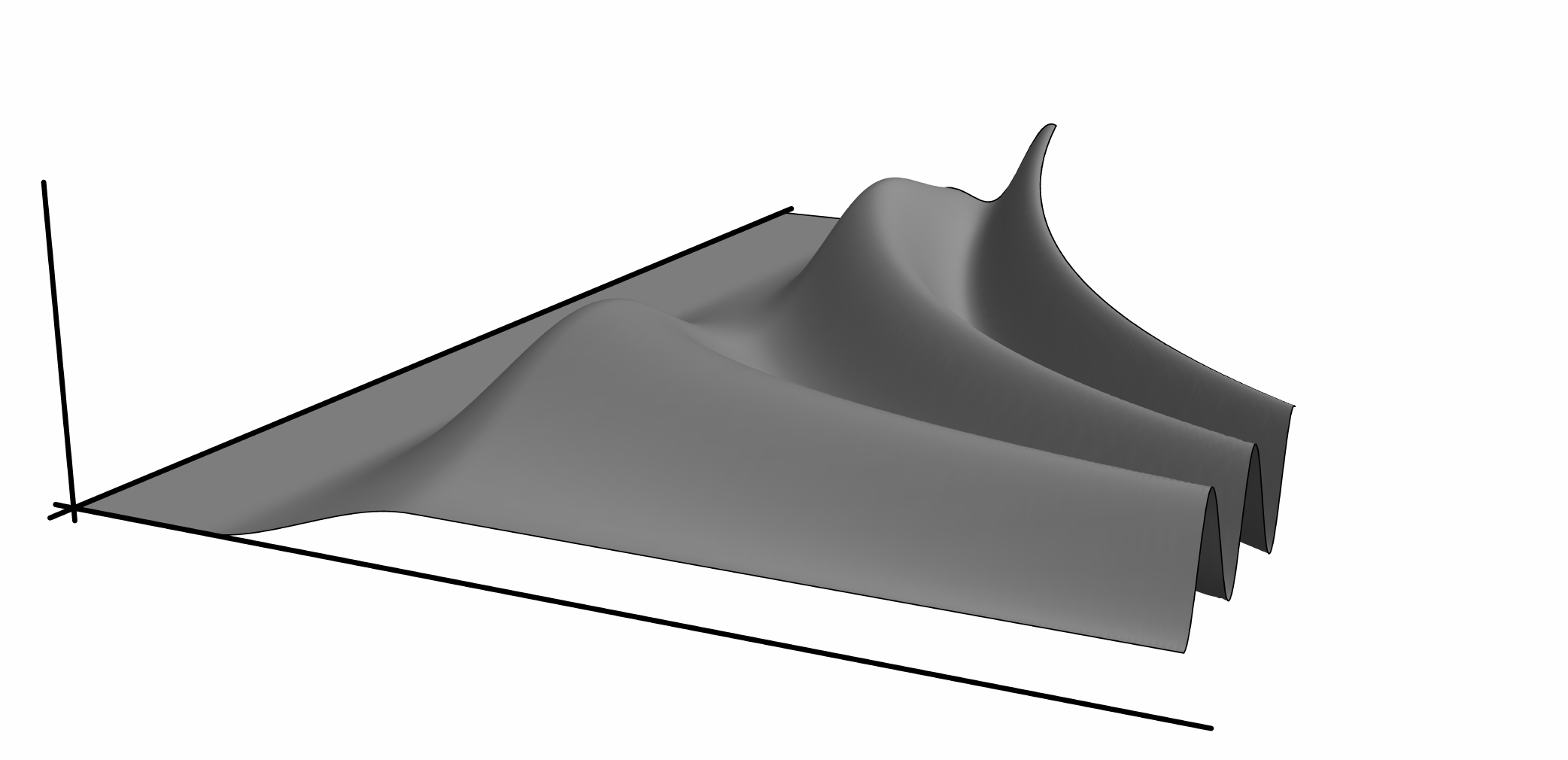}
    \put(93,00){$x$}
    \put(54,44){$y$}
    \put(02,46){$z$}
  \end{overpic}
  \caption{the graph of the perturbed manifold with respect to the
    Euclidean metric (top) and an approximate image of the same graph
    with respect to the metric $g$ (bottom).}
  \label{fig:exp-geom-xslice}
\end{figure}


\clearpage

\section{Preparation of the system}\label{sec:preparation}

As a first step towards proving Theorem~\ref{thm:persistNHIMtriv} we
shall bring the system in a form suitable to apply analytical tools to
it. Let
\begin{equation}\label{eq:ODE-XY-orig}
  \begin{aligned}
    \dot{x} &= v_\sx(x,y) \in \T_x X,\\
    \dot{y} &= v_\sy(x,y) \in Y,
  \end{aligned}
\end{equation}
be the decomposition of the vector field $v$ along $X$ and $Y$. The
invariant manifold is given as the graph $M = \{ y = h(x) \}$. We
dropped the explicit dependence on $\sigma$ from the notation. The
full flow of $v$ will be denoted by $\Upsilon^t$, while $\Phi,\Psi$
are reserved for flows defined in terms of the horizontal and vertical
components of $v$, respectively. To shorten notation we write
$g(x) = (x,h(x))$. We shall always assume that $\norm{y} \le \Ysize$.

Our goal is to establish a linearized form
\begin{equation}\label{eq:ODE-Y-lin}
  v_\sy(x,y) = A(x)\,y + f(x,y)
\end{equation}
for the vertical part of~\ref{eq:ODE-XY-orig} such that $f$
is small and $A,f \in \BUC^{k,\alpha}$, while the flows
$\Phi^t$ and $\Psi^t$ generated by
\begin{equation}\label{eq:ODE-XY-restr}
  \begin{aligned}
    \dot{x} &= v_\sx(g(x)),\\
    \dot{y} &= A(x(t))\,y \qquad
    \text{with $g(x(t))$ a solution curve on $M$},
  \end{aligned}
\end{equation}
should satisfy exponential growth estimates~\ref{eq:NHIM-rates} as in
Definition~\ref{def:NHIM} of normal hyperbolicity with exponents
$\rho_\sx,\,\rho_\sy$ close to the original $-\rho_M,\,\rho_-$,
respectively; the corresponding constants $\tilde{C}_M,\,\tilde{C}_-$
may differ arbitrarily from the original $C_M,\,C_-$.

\begin{figure}[tbh]
  \centering
  \input{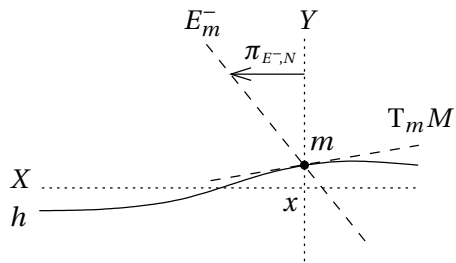}
  \caption{The splitting $\T_m (X \times Y) = \T_m M \oplus E_m^-$
    with $m = g(x)$.}
  \label{fig:triv-bundle-split}
\end{figure}
We first identify the invariant splitting and associated flows on
$\T_M (X \times Y)$ to be able to relate these exponential growth
rates, see also Figure~\ref{fig:triv-bundle-split}. By definition of
normal hyperbolicity (without an unstable bundle) we have
\begin{equation*}
  \T_M (X \times Y) = \T M \oplus E^-,
  \quad
  \Id = \pi_{\T M} + \pi_{E^-},
  \quad
  \D\Upsilon^t = \D\Upsilon_M^t \oplus \D\Upsilon_-^t
\end{equation*}
with associated exponential growth rates~\ref{eq:NHIM-rates}. On the
other hand we have the splitting
\begin{equation*}
  \T (X \times Y) = \pi_\sy^*(\T X) \oplus \pi_\sx^*(\T Y)
  \cong \T X \times (Y \times Y)
\end{equation*}
that is naturally induced by the trivial bundle structure. The
identification $\D g = \Id_{\T X} + \D h\colon \T X \to \T M$ is
bounded linear with bounded inverse, so the associated vector field
$g^*(v) = v_\sx \circ g$ on $X$ generates a flow $\Phi^t$ such that
$\D\Phi^t$ has the same exponential growth rate as $\D\Upsilon_M^t$,
up to a bounded factor $\norm{\D g^{-1}}\cdot\norm{\D g}$ due to the
norms on the different tangent spaces. Recall that $v_\sigma$ depends
on a parameter $\sigma \in \intvOC{0}{\sigma_0}$. We choose the bound
$\sigma_1$ small enough such that for all $\sigma \le \sigma_1$ we have
\begin{equation*}
  \forall\;t \le 0\colon \norm{\D\Phi^t} \le \tilde{C}_M\,e^{-\rho_M\,t}
  \quad\text{with}\quad \tilde{C}_M = 2\,C_M.
\end{equation*}

Let $N = \pi_\sx^*(\T Y)|_M$ denote the vertical bundle over $M$,
whose fibers can be canonically identified with $Y$. Just as above, we
want to project the flow $\D\Upsilon_-^t$ onto $N$ while preserving
the exponential growth rate. The projection $\pi_{E^-}$ along $\T M$
is uniformly bounded for all $\sigma$. This means that the angle
between $\T M$ and $E^-$ is bounded away from zero. Since $\T M$ can
be chosen arbitrarily close to the horizontal $\T X$ by choosing
$\sigma$ sufficiently small, it follows that the projection
$\D\pi_\sy|_{E^-}\colon E^- \to N$ and its inverse $\pi_{E^-\!\!,N}$
are bounded for all $\sigma \le \sigma_1$ when $\sigma_1$ is
sufficiently small, see also Figure~\ref{fig:triv-bundle-split}. To
this end, let $(0,\phi) \in \T_{g(x)} (X\times Y)$ and consider
the identity
\begin{equation*}
  \phi = \D\pi_\sy\cdot(0,\phi)
       = \D\pi_\sy\cdot(\pi_{\T M} + \pi_{E^-})\cdot(0,\phi).
\end{equation*}
We have $\pi_{\T M}\cdot(0,\phi) \in \T M$ so
$\pi_{\T M}\cdot(0,\phi) = (\xi,\D h(x)\,\xi)$ where
$\xi = \D\pi_\sx \cdot \pi_{\T M}\cdot(0,\phi) \in \T_x X$
Now we have estimates
\begin{equation*}
  \begin{aligned}
    \norm{\xi}
    &=  \norm{\D\pi_\sx\,\pi_{\T M}\,(0,\phi)}
    \le \norm{\pi_{\T M}}\norm{\phi},\\
    \norm{\pi_{\T M}\cdot(0,\phi)}
    &=  \norm{\D h(x)\,\xi}
    \le \sigma\,\norm{\pi_{\T M}}\norm{\phi},\\
    \norm{\D\pi_\sy\cdot\pi_{E^-}\cdot(0,\phi)}
    &=  \norm{\phi - \D\pi_\sy\cdot\pi_{\T M}\cdot(0,\phi)}
    \ge (1 - \sigma\,\norm{\pi_{\T M}})\norm{\phi}
\end{aligned}
\end{equation*}
from which it follows that $\D\pi_\sy|_{E^-}$ has an inverse
$\pi_{E^-\!\!,N}\colon N \to E^-$ for which we have the bound
$\norm{\pi_{E^-\!\!,N}} \le (1 - \sigma_1\,\norm{\pi_{\T M}})^{-1}\,\norm{\pi_{E^-}} \le 2\,\norm{\pi_{E^-}}$
if we choose $\sigma_1 \le \frac{1}{2\norm{\pi_{\T M}}}$.

Consider the flow
\begin{equation}\label{eq:ODE-Y-flow}
  \hat{\Psi}^t = \D\pi_\sy \circ \D\Upsilon^t \circ \pi_{E^-\!\!,N}\colon N \to N,
\end{equation}
generated by $\D v_\sy \circ \pi_{E^-\!\!,N}$ along solution curves
$g(x(t))$. Both $\D\pi_\sy|_{E^-}$ and $\pi_{E^-\!\!,N}$ are uniformly
bounded, so the exponential estimates of $\D\Upsilon_-^t$ carry over
to $\hat{\Psi}^t$ up to a constant factor:
\begin{equation}\label{eq:ODE-Y-est}
  \forall\;t \ge 0\colon \norm{\hat{\Psi}^t} \le \tilde{C}_-\,e^{\rho_-\,t}
  \quad\text{with}\quad \tilde{C}_- = 2\,\norm{\pi_{E^-}}\,C_-.
\end{equation}

We have thus constructed flows $\Phi^t$ and $\hat{\Psi}^t$ on $X$ and
$N$, respectively, that are generated by
\begin{equation*}
  v_\sx \circ g \quad\text{and}\quad
  \hat{A}(x) = \D v_\sy(g(x)) \cdot \pi_{E^-\!\!,N}(g(x))
\end{equation*}
with a solution curve $x(t)$ of the vector field $v_\sx \circ g$
inserted. These flows are of the form~\ref{eq:ODE-XY-restr} and
satisfy exponential estimates~\ref{eq:NHIM-rates} inherited from the
invariant bundle splitting.

The vector field $v_\sx \circ g$ already has sufficient
smoothness\footnote{%
  It may seem impossible to define a $C^k$ vector field $v$ on the
  tangent bundle $\T M$ of a $C^k$ manifold $M$ since $\T M \in C^{k-1}$.
  See~\cite[App.~1]{Palis1977:topequiv-normhyp}
  or~\cite[p.~398]{Palis1983:stab-gradfields} for a method to endow an
  invariant submanifold $M \in C^k$ with a compatible topology that
  makes $v|_M \in C^k$. We effectively used this in our definition of
  $v_\sx \circ g \in C^k$.%
}, but $\hat{A}$ is not smooth enough since the projection
$\pi_{E^-\!\!,N}$ is only continuous. We construct
$A \in \BUC^{k,\alpha}$ as a smoothed approximation of
$\D_y v_\sy \circ g$. This term is $C^0$\ndash close to $\hat{A}$,
since
\begin{equation}\label{eq:Ahat-DyVy-est}
      \norm{\hat{A} - \D_y v_\sy \circ g}_0
  =   \norm{(\D_x v_\sy \circ g) \cdot \D\pi_\sx \cdot \pi_{E^-\!\!,N}}_0
  \le \norm{\D_x v_\sy \circ g}_0\,\norm{\pi_{E^-\!\!,N}}_0
\end{equation}
because $\D\pi_\sy \cdot \pi_{E^-\!\!,N} = \Id_N$ and
$\norm{\D_x v_\sy \circ g}$ is small. Lemma~\ref{lem:pert-lin} will
imply that the flow $\Psi^t$ of this approximation has exponential
growth estimates close to those of $\hat{\Psi}^t$. The following lemma
will be used to obtain \idx{$A$} from $\D_y v_\sy \circ g$. We apply it with
$l = k-1$ to obtain $A \in \BC^k(X;\CLin(Y))$ such that
$\norm{A - \D_y v_\sy\circ g}_{k-1} \le \epsilon(\nu)$. This lemma is
a (strongly) simplified version of
Theorem~\ref{thm:unif-smooth-submfld}; the notation of $l,\,k$ is
reversed to match the context here.
\begin{lemma}[Uniform smoothing of a vector bundle section]
  \label{lem:VB-section-smoothing}
  Let\/ $(X,g)$ be a Riemannian manifold of bounded geometry and\/ $V$ a
  Banach space. Let $f \in \BUC^l(X;V)$ be a section of the trivial
  vector bundle $\pi\colon X \times V \to X$.

  Then for any $k > l$ and $\epsilon > 0$ there exists a smoothed
  function $\tilde{f} \in \BUC^k(X;V)$ such that
  $\norm{\tilde{f} - f}_l \le \epsilon$. (The bounds on higher than
  $l$\th order derivatives will generally depend on $\epsilon$.)
\end{lemma}

\begin{proof}
  We apply convolution smoothing of Lemma~\ref{lem:conv-smoothing} in
  each chart of a cover of $X$ and glue these together.

  Let $0 < \delta_1 < \delta_2 < \delta_3$ and let
  $\{B(x_i;\delta_2)\}_{i \ge 1}$ be a uniformly locally finite cover
  of $X$ obtained from Lemma~\ref{lem:unif-loc-cover}, such that the
  $\delta_1$\ndash sized sets already cover $X$, and the
  $\delta_3$\ndash sized sets still have normal coordinate charts.
  Lemma~\ref{lem:part-unity} yields a uniform partition of unity
  $\sum_{i \ge 1} \chi_i$ subordinate to this cover.

  In each chart $B(x_i;\delta_2)$ we apply
  Lemma~\ref{lem:conv-smoothing} to $f$ with $r = \delta_2$ and
  $2\,\delta\!r \le \delta_3 - \delta_2$. We obtain
  $\tilde{f}_i \in \BUC^k$ on each chart with uniformly bounded
  $C^k$\ndash norms and $\norm{\tilde{f}_i - f}_l$ can be made as
  small as required by choosing the parameter $\nu$ small. We glue
  these together to one function
  \begin{equation*}
    \tilde{f} = \sum_{i \ge 1} \chi_i\,\tilde{f}_i
  \end{equation*}
  defined globally on $X$ with the functions $\chi_i \in \BUC^k$.
  Together with the uniform bound on the number of charts in the cover
  that intersect any one point, this guarantees that $\tilde{f}$
  satisfies estimates equivalent to those of the $\tilde{f}_i$. Note
  that $\norm{\tilde{f}}_l$ does not depend on the smoothing parameter
  $\nu$, but the higher derivative norms do.
\end{proof}

\begin{remark}[On loss of smoothness]
\label{rem:loss-smoothness}
\index{smoothness!loss of}
We must carefully construct the system~\ref{eq:ODE-Y-lin} in order not
to lose one degree of smoothness, while at the same time retaining
exponential growth rates and proximity estimates.

The invariant complementary bundle $E^-$ is only continuous, while the
normal bundle of $M$ is only $C^{k-1}$, even if disguised in
coordinate expressions. We use the linearization at $y = 0$, but not
directly, since $\D_y v_y(\slot,0) \in \BUC^{k-1,\alpha}$ artificially
decreases the smoothness as well. The loss of smoothness
in~\cite{Sakamoto1990:invarmlfds-singpert} occurs for these reasons.
Note that even though we retain $C^{k,\alpha}$ smoothness by a
convolution smoothing, this does not preserve higher than $C^{k-1}$
bounds. This seems to be an artifact of the proof, inherent to the
partial linearization along $Y$.

In the proof of Theorem~\ref{thm:persistNHIMgen} we construct a
smoother, approximate manifold $M_\sigma$ exactly to circumvent these
problems. In the trivial bundle setting of
Theorem~\ref{thm:persistNHIMtriv} then, we must be careful not to
pick a representation that reintroduces this loss of smoothness. On
the other hand, we do not seem to obtain optimal results in the sense
that we require $\norm{h}_2$ small, while the classical results in the
compact case only require $M = \Graph(h) \in C^1$. Similarly,
$h \in \BC^k$ with $k \ge 3$ is assumed
in~\cite[p.~50]{Sakamoto1990:invarmlfds-singpert}, while hypothesis~H2
in~\cite[p.~987]{Bates1999:persist-overflow} is imposed to bound
`twisting' of the invariant manifold. This requirement seems closely
related to our condition on $h$, and is necessary for the same reason
as in our Theorem~\ref{thm:persistNHIMgen}: to construct a tubular
neighborhood of uniform size. I do not know whether these stronger
assumptions can be weakened or removed.
\end{remark}
\clearpage

\section{Growth estimates for the perturbed system}
\label{sec:growth-perturbed}

We shall finally put all the ingredients together to obtain
exponential growth estimates for perturbed flows contained in the
tubular neighborhood $\norm{y} \le \Ysize$ of $X \times Y$.
We write the perturbed vector field $\tilde{v}$ on $X \times Y$ as
\begin{equation}\label{eq:ODE-XY-mod}
  \begin{aligned}
    \dot{x} &= \tilde{v}_\sx(x,y),\\
    \dot{y} &= \tilde{v}_\sy(x,y) = A(x)\,y + \Ynonlin(x,y),\\
\text{where}\quad
     \Ynonlin(x,y) &= f(x,y) + \big(\tilde{v}_\sy(x,y) - v_\sy(x,y)\big).
  \end{aligned}
\end{equation}
\oldindex{$v_x$@$\vx$}
\oldindex{$f$@$\Ynonlin$}

Let us assume that the conditions of Theorem~\ref{thm:persistNHIMtriv} hold true.
First, if $\rho_M = 0$, then for any fixed $\fracdiff$ we can always
slightly increase\footnote{%
  We have $\rho_M \ge 0$ from~\ref{eq:NHIM-rates}. Note that we are
  interested in $-\rho_M$ for the stable side of the spectrum. In the
  rest of this chapter, all exponential rates will be negative.%
} to $\rho_M > 0$, such that the growth
rates~\ref{eq:NHIM-rates} and spectral gap condition
$\rho_- < -\fracdiff\,\rho_M$ still hold true; this way, we get
rid of degenerate exponentials in integrals. We have some `spectral
space' $\Delta\rho = \fracdiff\,\rho_M - \rho_- > 0$ that we use to
define modified \idx{exponential growth numbers}
\index{$\rho_\sx,\,\rho_\sy$}
\begin{equation}\label{eq:exp-est-XY-mod}
  \begin{alignedat}{2}
    \rho_\sx &= -\rho_M - \frac{\Delta\rho}{4},&\qquad
    C_\sx    &= 2\,\tilde{C}_M, \\
    \rho_\sy &=  \rho_- + \frac{\Delta\rho}{4},&
    C_\sy    &=    \tilde{C}_-.
  \end{alignedat}
\end{equation}
This allows us to get all perturbed flows within these slightly
modified growth rates, while we reserve another $\Delta\rho/2$
spectral space for later use, such as proving (higher order)
differentiability. Note that both $\rho_\sy,\,\rho_\sx$ are negative
since we focus on the stable normal bundle.

We first fix some notation to be used throughout the proof:
\begin{itemize}
\item $C_v$ denotes the global $C^{k,\alpha}$ bound on $v$ and
  $\tilde{v}$.
\item $\epsilon(\nu)$ denotes the perturbation size in
  Lemma~\ref{lem:VB-section-smoothing} depending on the smoothing
  convolution parameter $\nu$ from Lemma~\ref{lem:conv-smoothing},
  while $C_v(\nu)$ denotes the $C^{k,\alpha}$ bound on $A,\,\Ynonlin$,
  which may grow due to smoothing when $\nu \to 0$. We also have
  $\norm{A}_{k-1},\norm{\Ynonlin}_{k-1} \le C_v$.
  \index{$\nu$ (mollifier size)}
\item $\Vsize$ denotes a small bound both on the derivative of
  $\Ynonlin$ and on perturbations of the horizontal vector field on
  $X$, that is, we impose bounds
  \begin{equation*}
    \sup_{\substack{x \in X\\\norm{y} \le \Ysize}}\;
      \norm{\D\Ynonlin(x,y)} \le \Vsize
    \qquad\text{and}\qquad
    \sup_{\substack{x \in X\\\norm{y} \le \Ysize}}\;
      \norm{\vx(x,y) - v_\sx(x,h(x))} \le \Vsize,
  \end{equation*}
  and the size of $\Vsize$ will be controlled by
  $\delta,\,\sigma_1,\,\nu$, and $\Ysize$.
  \index{$\zeta$ (vector field change)}
\end{itemize}
Let us point out here that multiple parameters must be chosen small,
some dependent on other small parameters. The following graph shows
all dependencies; an arrow indicates that the choice of a parameter
influences the choice of the object pointed to.
\begin{equation}\label{eq:small-param-dep}
  \begin{aligned}
  \xymatrix@R=3em@C=4em{
      \delta
    & \sigma_1
    & \nu
    \\
      \Ysize           \ar[u]  \ar[ur]\ar[urr]!<0em,.2em> \ar@{<-}[dr]!<-1em,.5em>
    & \Vsize     \ar[l]\ar[ul] \ar[u] \ar[ur]             \ar@{<-}[d]
    & (\Xsize,T) \ar[l]\ar[ull]!<0em,.2em>\ar[ul]\ar[u]   \ar@{<-}[dl]!<1em,.5em>
    \\
    & *-={\text{other (small) constants and bounds}}
    &
  }
  \end{aligned}\\[4pt]
\end{equation}
The constants and bounds include~\ref{eq:exp-est-XY-mod} and $C_v$,
and are all fixed. Note that there are no circular dependencies, so we
are free to choose any of these parameters smaller if necessary
without the risk of having unsatisfiable constraints.

\enlargethispage{0.1cm}
By invariance of $M = \{ y = h(x) \}$ we have
\begin{equation}\label{eq:Vy-invariance}
  v_\sy(x,h(x)) = \der{y}{t} = \D h(x)\cdot v_\sx(x,h(x)).
\end{equation}
This can be used to estimate $\norm{v_\sy(x,h(x))} \le \sigma_1\,C_v$
and derived estimates, such as (taking the derivative with respect to
$x$)
\begin{equation}\label{eq:DxVy-bound}
  \begin{aligned}
       \norm{\D_x v_\sy\circ g}
  &\le \norm{\D_y v_\sy}\norm{\D h}
      +\norm{\D^2 h}\norm{v_\sy\circ g}\\
\con  +\norm{\D h}\big(\norm{\D_x v_\sx\circ g}
                      +\norm{\D_y v_\sx\circ g}\norm{\D h}\big)
  \le 4\,C_v\,\sigma_1
  \end{aligned}
\end{equation}
where $\sigma_1 \le 1$ has been assumed. Together with previous
estimates, this leads to
\begin{align*}
   \norm{f(x,0)}
  &=   \norm{v_\sy(x,0) - A(x)\cdot 0}\\
  &\le \norm{\D h(x)}\norm{v_\sx \circ g}_0
   \le \sigma_1\,C_v,
   \displaybreak[1]\\
   \norm{\D_x f(x,y)}
  &=   \norm{\D_x v_\sy(x,y) - \D A(x)\,y}\\
  &\le \begin{aligned}[t]
     & \norm{\D_x v_\sy(x,y) - \D_x v_\sy(x,h(x))}
      +\norm{\D_x v_\sy \circ g}_0\\
     &+\Big(\norm{\D(A - \D_y v_\sy \circ g)}_0
           +\norm{\D_x \D_y v_\sy}_0\Big)\,\norm{y}
       \end{aligned}\\
  &\le \epsilon_{\D_x v_\sy}(\Ysize + \sigma_1) + 4\,C_v\,\sigma_1
      + \big(\epsilon(\nu)+C_v\big)\,\Ysize,
   \displaybreak[1]\\
   \norm{\D_y f(x,y)}
  &=   \norm{\D_y v_\sy(x,y) - A(x)}\\
  &\le \norm{\D_y v_\sy(x,y) - \D_y v_\sy(x,h(x))}
      +\norm{\D_y v_\sy\circ g - A}_0\\
  &\le \epsilon_{\D_y v_\sy}(\Ysize + \sigma_1) + \epsilon(\nu),
   \displaybreak[1]\\
   \norm{\vx(x,y) - v_\sx(x,h(x))}
  &\le \norm{\vx(x,y) - v_\sx(x,y)} + \norm{v_\sx(x,y) - v_\sx(x,h(x))}\\
  &\le \delta + \epsilon_{v_\sx}(\Ysize + \sigma_1).
\end{align*}
Hence, $\norm{f(\slot,0)}_0$ can be made small independently of $\Ysize$,
while $\norm{f}_1,\,\norm{\vx(\slot,y) - v_\sx\circ g}_0 \le\Vsize$
can be obtained for any $\Vsize > 0$ depending on
$\delta,\,\Ysize,\,\sigma_1,\,\epsilon(\nu)$, and the continuity
moduli $\epsilon_{\D v_\sy},\,\epsilon_{v_\sx}$. So if we set
\begin{equation}\label{eq:Vsize}
  \Vsize = 5\,C_v\,\sigma_1 + 2\,\epsilon_{\D v_\sy}(\Ysize+\sigma_1)
          + \big(\epsilon(\nu) + C_v\big)\,\Ysize + \epsilon(\nu)
          + \epsilon_{v_\sx}(\Ysize+\sigma_1) + \delta,
\end{equation}
then $\norm{f}_1,\,\norm{\vx(\slot,y) - v_\sx\circ g}_0 \le \Vsize$
hold and $\Vsize$ is small when $\delta,\,\sigma_1,\,\nu,\,\Ysize$ are.

We need the following result to control the $C^{k-1}$ distance of the
perturbed manifold $\tilde{M}$ to $M$.
\begin{proposition}
  \label{prop:Ck-small-pert}
  For any $\epsilon > 0$, the nonlinearity $\Ynonlin$ and its partial
  derivatives with respect to $x \in X$ can be bounded as
  \begin{equation}\label{eq:Ynonlin-Ck-small}
    \forall\; 0 \le i \le k-1\colon \norm{\D_x^i \Ynonlin} \le \epsilon
  \end{equation}
  by choosing $\Ysize$, $\nu$, $\sigma_1$, $\norm{h}_k$,
  and $\norm{\tilde{v}-v}_{k-1}$ small enough.
\end{proposition}

The idea of the proof is the following. If $M$ is described exactly by
$h(x) \equiv 0$, then by invariance we have $v_\sy(x,0) \equiv 0$
(cf.~\ref{eq:Vy-invariance}), hence $\D_x^i v_\sy(x,0) \equiv 0$ as
well. We adapt the proof to incorporate small perturbations introduced
by the nonzero function $h$ and the convolution smoothing of~$A$.

\begin{proof}
  Note that $\Ynonlin$ is defined by~\ref{eq:ODE-XY-mod}
  and~\ref{eq:ODE-Y-lin} as
  \begin{equation}\label{eq:Ck-small-Ynonlin-def}
    \Ynonlin(x,y) = v_\sy(x,y) - A(x) \cdot y
                   +\big[\tilde{v}_\sy(x,y) - v_\sy(x,y)\big],
  \end{equation}
  where $A$ is defined as a convolution smoothing of
  $\D_y v_\sy \circ g$ such that
  $\norm{A - \D_y v_\sy \circ g}_{k-1} \le \epsilon(\nu)$. The term in
  brackets obviously becomes small when $\norm{\tilde{v}-v}_{k-1}$
  does. For the second term note that $\norm{y} \le \Ysize$, while
  $\norm{A}_{k-1}$ is bounded close to
  $\norm{\D_y v_\sy \circ g}_{k-1}$, which in turn can be estimated by
  $\norm{\D_y v_\sy}_{k-1} \le C_v$ and $\norm{h}_{k-1}$ after
  application of Proposition~\ref{prop:compfunc-deriv}.

  For the first term in~\ref{eq:Ck-small-Ynonlin-def} we use the
  continuity modulus of $\D_x^i v_\sy$ to estimate
  \begin{equation*}
    \norm{\D_x^i v_\sy(x,y)}
    \le \norm{\D_x^i v_\sy(x,h(x))}
        +\epsilon_{\D_x^i v_\sy}\big(\norm{y - h(x)}\big),
  \end{equation*}
  while $\norm{y - h(x)} \le \Ysize + \sigma_1$. We
  insert~\ref{eq:Vy-invariance} and apply
  Proposition~\ref{prop:compfunc-deriv} another time to obtain
  \begin{equation*}
      \D_x^i v_\sy(x,h(x))
    = \D_x^i\big[\D h(x)\cdot v_\sx(x,h(x))\big]
     -\!\!\sum_{\substack{l \ge 0, m \ge 1\\l+m \le i}}\!
         \D_x^l \D_y^m v_\sy(x,h(x)) \cdot P_{m,i-l}\big(\D^\bullet h(x)\big).
  \end{equation*}
  This expression can be made small since
  $\norm{\D_x^l \D_y^m v_\sy} \le C_v$ and each term contains at least
  one factor $\D^j h(x)$ for some $0 \le j \le k$.
\end{proof}

\begin{remark}\label{rem:Ck-small-optimal}
  Note that we cannot improve the result to a $C^k$ size estimate,
  since $\norm{A}_k \le C(\nu)$ may grow with $\nu \to 0$, while
  compensating this by choosing $\Ysize$ smaller would introduce a
  circular dependency in~\ref{eq:small-param-dep}.
\end{remark}

As the next step, we will derive exponential growth estimates for the
perturbed system~\ref{eq:ODE-XY-mod}. More generally, we consider the
horizontal flow $\Phi_y$ and vertical, linear flow $\Psi_x$ generated
by
\index{$\Phi_y$}
\index{$\Psi_x$}
\begin{subequations}\label{eq:ODE-XY-flowgen}
  \begin{align}
    \dot{x} &= \vx(x,y), \label{eq:ODE-X-flowgen}\\
    \dot{y} &= A(x)\,y,  \label{eq:ODE-Y-flowgen}
  \end{align}
\end{subequations}
with specific curves $y\colon I \to Y$ and $x\colon I \to X$
substituted, respectively. The following series of lemmas and
propositions show that these flows are small perturbations of the
flows of~\ref{eq:ODE-XY-restr} and satisfy exponential growth
rates~\ref{eq:exp-est-XY-mod}. We prove the nonlinear case on $X$
and the linear case on $Y$ separately, since we use $C^1$ smoothness
for the nonlinear case, while only continuity can be assumed for the
linear case.

\begin{lemma}[Growth estimates for a perturbed system]
  \label{lem:pert-exp-growth}
  Let $X$ be a Riemannian manifold and let the system
  $\dot{x}= v(t,x)$ with $v,\,\D_x v \in \BC^0$ have flow $\Phi$ with
  exponential growth estimate
  \begin{equation}\label{eq:exp-growth-X}
    \forall\; x_0 \in X,\,t \le t_0 \colon
    \norm*{\D\Phi(t,t_0,x_0)} \le C\,e^{\rho(t-t_0)}.
  \end{equation}

  Let $\tilde{v} = v + r$ be a perturbed system generating a flow\/
  $\tilde{\Phi}$. For each $\tilde{\rho} < \rho$ and $\tilde{C} > C$,
  there exists a $\delta > 0$, such that if
  $\norm{r}_0,\,\norm{\D_x r}_0 < \delta$, then $\tilde{\Phi}$
  satisfies the growth estimate~\ref{eq:exp-growth-X} with
  $\tilde{\rho}$ and $\tilde{C}$ inserted.
\end{lemma}
Note that this lemma is formulated in backward time.

\begin{proof}
  Choose $T > 0$ sufficiently large such that
  $\tilde{C}\,e^{\rho(-T)} \le e^{\tilde{\rho}(-T)}$. By continuous
  dependence of the solutions of differential equations on parameters
  (see Theorem~\ref{thm:unif-smooth-flow} and
  Remark~\ref{rem:smooth-flow-timedep}), a $C^1$ small perturbation
  $r$ results in a $C^1$ small perturbed flow $\tilde{\Phi}$ on
  compact time intervals $-T \le t - t_0 \le 0$. This result is uniform
  in $t_0,t$ when $v,r \in \BC^1$, where differentiation is
  understood with respect to $x$ only. Hence we obtain
  \begin{equation*}
    \sup_{\substack{x_0 \in X\\-T \le t - t_0 \le 0}}\;
      \norm*{\D\tilde{\Phi}(t,t_0,x_0)}\,e^{-\rho(t-t_0)} \le \tilde{C}
  \end{equation*}
  if $\delta$ is chosen sufficiently small. Writing
  $t - t_0 = -(n\,T + \tau)$ with $n \in \N,\,\tau \in \intvCO{0}{T}$, we
  use the group property of the flow to obtain
  \begin{equation*}
        \norm*{\D\tilde{\Phi}(t,t_0,x_0)}
    \le {\left(\tilde{C}\,e^{\rho(-T)}\right)}^n\,\tilde{C}\,e^{\rho(-\tau)}
    \le e^{\tilde{\rho}\,n(-T)}\,\tilde{C}\,e^{\rho(-\tau)}
    \le \tilde{C}\,e^{\tilde{\rho}(t-t_0)}. \qedhere
  \end{equation*}
\end{proof}

\begin{lemma}[Perturbation of linear flow]
  \label{lem:pert-lin}
  Let\/ $Y$ be a Banach space and let $A \in \BC^0\big(\R;\CLin(Y)\big)$
  generate a flow $\Psi(t,t_0)$ with growth estimate
  \begin{equation}\label{eq:exp-growth-Y}
    \forall\; t \ge t_0 \colon
    \norm*{\Psi(t,t_0)} \le C\,e^{\rho(t-t_0)}.
  \end{equation}
  Let $\tilde{\rho} > \rho$ be given and set
  $\delta = \frac{\tilde{\rho} - \rho}{C} > 0$. If\/
  $B \in \BC^0\big(\R;\CLin(Y)\big)$ is globally bounded by $\delta$,
  then the flow $\tilde{\Psi}(t,t_0)$ of $\tilde{A}(t) = A(t) + B(t)$
  satisfies~\ref{eq:exp-growth-Y} with $\tilde{\rho}$ inserted.
\end{lemma}

\begin{proof}
  The variation of constants integral equation for $\tilde{\Psi}$ is
  \begin{equation}\label{eq:var-const-Psi}
    \tilde{\Psi}(t,t_0)
    = \Psi(t,t_0) + \int_{t_0}^t \Psi(t,\tau)\,B(\tau)\,\tilde{\Psi}(\tau,t_0) \d\tau.
  \end{equation}
  We shall prove the estimate for $\tilde{\Psi}$ with an approach
  inspired by Gronwall's lemma. Note that our variation of constants
  formula~\ref{eq:exp-growth-Y} is slightly different from the
  standard context of Gronwall's lemma, since we do not have a
  bound for $A$.

  We denote by $\psi(t,t_0) = C\,e^{\rho(t-t_0)}$ the bound on $\Psi$.
  Now $\tilde{\psi}(t,t_0) = C\,e^{\tilde{\rho}(t-t_0)}$ satisfies the
  integral equation
  \begin{equation}\label{eq:var-const-normPsi}
    \tilde{\psi}(t,t_0)
    = \psi(t,t_0) + \int_{t_0}^t \psi(t,\tau)\,\delta\,\tilde{\psi}(\tau,t_0) \d\tau
  \end{equation}
  when $\delta\,C = \tilde{\rho} - \rho$. We verify this by
  calculating the right-hand side:
  \begin{align*}
&\nop  C\,e^{\rho(t-t_0)} + \int_{t_0}^t C\,e^{\rho(t-\tau)}\,\delta\,
                                         C\,e^{\tilde{\rho}(\tau-t_0)} \d\tau\\
    &= C\,e^{\rho(t-t_0)}\Big[1 +
         \delta\,C\, \int_{t_0}^t e^{(\tilde{\rho}-\rho)(\tau-t_0)} \d\tau \Big]\\
    &= C\,e^{\rho(t-t_0)}\Big[1 +
         \frac{\delta\,C}{\tilde{\rho}-\rho}\,
         \big(e^{(\tilde{\rho}-\rho)(t-t_0)} - 1\big)\Big]\\
    &= C\,e^{\rho(t-t_0)}\,e^{(\tilde{\rho}-\rho)(t-t_0)}\\
    &= C\,e^{\tilde{\rho}(t-t_0)}.
  \end{align*}

  Next, we prove by contradiction that
  \begin{equation*}
    \norm[\big]{\tilde{\Psi}(t,t_0)} \le \tilde{\psi}(t,t_0).
  \end{equation*}
  Thus, let
  \begin{equation*}
    t_1 = \inf\; \set[\big]{t \in \R}{t \ge t_0 \;\text{ and }\;
                    \norm{\tilde{\Psi}(t,t_0)} > \tilde{\psi}(t,t_0)}.
  \end{equation*}
  Note that $\tilde{\Psi}$ is the solution of a differential equation,
  hence continuous. We write
  $\norm[\big]{\tilde{\Psi}(t,t_0)} = \tilde{\psi}(t,t_0) + f(t)$, so
  we may assume that $f(t) \le 0$ for $t \in \intvCC{t_0}{t_1}$, but
  there exist $t \in \intvOC{t_1}{t_2}$ arbitrary close to $t_1$ such that
  $f(t) > 0$. Let $f|_\intvCC{t_1}{t_2}$ attain its supremum at $t$,
  thus we have
  \begin{equation*}
    \sup_\intvCC{t_1}{t} f = f(t) > 0.
  \end{equation*}

  We insert these estimates into the integral
  equality~\ref{eq:var-const-Psi} and obtain
  \begin{align*}
    \norm[\big]{\tilde{\Psi}(t,t_0)} = \tilde{\psi}(t,t_0) + f(t)
    &\le \psi(t,t_0) + \int_{t_0}^t \psi(t,\tau)\,\delta\,\big(\tilde{\psi}(\tau,t_0) + f(\tau)\big) \d\tau\\
    &\le \tilde{\psi}(t,t_0) + \int_{t_1}^t \psi(t,\tau)\,\delta\,f(\tau) \d\tau\\
    &\le \tilde{\psi}(t,t_0)
        +(t_1 - t)\,\delta\sup_{\tau \in \intvCC{t_1}{t}} \psi(t,\tau)\;
         \sup_{\intvCC{t_1}{t}}\, f,
  \end{align*}
  where we used that $\tilde{\psi}(t,t_0)$
  satisfies~\ref{eq:var-const-normPsi} and that
  $f|_\intvCC{t_0}{t_1} \le 0$. Now we choose $t_2$ and therefore $t$
  sufficiently small that
  $(t - t_1)\,\delta\,\sup_{\tau \in \intvCC{t_1}{t}}\, \psi(t,\tau) \le q < 1$,
  which leads to the contradiction
  \begin{equation*}
    f(t) \le q\,\sup_{\intvCC{t_1}{t}}\, f < f(t).
  \end{equation*}
\end{proof}

\begin{proposition}[Perturbation of $X$\ndash flow estimate]
  \label{prop:pert-x}
  If\/ $\delta,\,\Ysize,\,\sigma_1$ are sufficiently small, then the
  flow of~\ref{eq:ODE-X-flowgen} satisfies the modified exponential
  growth estimates~\ref{eq:exp-est-XY-mod} for any
  $y \in \overline{B_\Ysize(\R;Y)}$ inserted.
\end{proposition}
Here $\overline{B_\Ysize(\R;Y)}$ denotes the closed ball of radius
$\Ysize$ in the space of bounded continuous functions $\R \to Y$.

\begin{proof}
  Define the non-autonomous system $v(t,x) = \vx(t,x,y(t))$. This
  system is a $C^1$ small perturbation of $(v_\sx \circ g)(x)$,
  uniformly in $t$:
  \begin{align*}
         \norm{v(t,x) - (v_\sx \circ g)(x)}
    &\le \norm{\vx(x,y(t)) - v_\sx(x,y(t))}
        +\norm{v_\sx(x,(t)) - v_\sx(x,h(x))}\\
    &\le \delta + \epsilon_{\D v}(\Ysize + \sigma_1),\\[8pt]
         \norm{\D_x v(t,x) - \D (v_\sx \circ g)(x)}
    &\le \norm{\D_x \vx(x,y(t)) - \D_x v_\sx(x,y(t))}\\
\con    +\norm{\D_x v_\sx(x,y(t)) - \D_x v_\sx(x,h(x))}\\
    &\le \delta + \epsilon_{\D v}(\Ysize + \sigma_1),
  \end{align*}
  where $\epsilon_{\D v}$ denotes the uniform continuity modulus of
  $v$ and its first derivative, which can be made small by choice of
  $\Ysize,\,\sigma_1$. We apply Lemma~\ref{lem:pert-exp-growth} to
  obtain exponential growth numbers $C_\sx,\,\rho_\sx$
  for~\ref{eq:ODE-X-flowgen} by choosing
  $\delta + \epsilon_{\D v}(\Ysize+\sigma_1)$ sufficiently small.
\end{proof}

The following definition and lemma for flows on $X$ are again
formulated in backward time, similar to
Lemma~\ref{lem:pert-exp-growth}.
\begin{definition}[Approximate solution]
  \label{def:approx-sol}
  \index{approximate solution}
  \index{$\beta$ (approximate solution distance)}
  Let\/ $X$ be a Riemannian manifold and $v(t,x)$ a time-dependent
  vector field on $X$. We call a continuous curve $x\colon \R \to X$ a
  $(\Xsize,T)$\ndash approximate solution of $v$ if for each interval
  $\intvCC{t_2}{t_1} \subset \R$ with $t_1 - t_2 \le T$ and associated
  exact solution curve $\xi$ of $v$ with initial condition
  $\xi(t_1) = x(t_1)$, it holds that
  \begin{equation}\label{eq:approx-sol}
    \sup\limits_{t_2 \le t \le t_1} d(x(t),\xi(t)) < \Xsize.
  \end{equation}
\end{definition}
It would have been easier to define approximate solutions as $C^1$
curves $x$ such that $\norm{\dot{x}(t) - v(t,x(t))} < \Xsize$. We
shall want to work with $C^0$\ndash norms, though, and $C^1$ curves do
not form a complete space under such norms. We use this continuous
curve definition to avoid any complications associated with
non-completeness. We still have the following result, as a discretized
variant on variation by constants estimates.
\begin{lemma}[Growth of approximate solutions]
  \label{lem:approx-growth}
  Let\/ $X$ be a Riemannian manifold, $v(t,x)$ a time-dependent vector
  field on $X$, and $x$ a $(\Xsize,T)$\ndash approximate solution of
  $v$. Assume that $v$ generates a flow $\Phi^{t,t_0}$ that satisfies
  the exponential growth estimate~\ref{eq:exp-growth-X} with
  $C \ge 1,\,\rho < 0$. Let $\xi_0$ denote the exact solution of $v$
  with initial condition $\xi_0(0) = x(0)$.

  Then the distance $d_\rho(x,\xi_0)$ is finite on the interval\/
  $\intvOC{-\infty}{0}$, and explicitly bounded by
  \begin{equation}\label{eq:approx-growth}
    d_\rho(x,\xi_0) \le \Xsize\Big(1 + \frac{C}{1-e^{\rho\,T}}\Big).
  \end{equation}
\end{lemma}

\begin{proof}
  Let $\xi_i$ with $i \in \N$ be the associated exact solutions of $x$ that
  satisfy~\ref{eq:approx-sol} on the interval $\intvCC{-(i+1)T}{-i\,T}$.
  We have $d\big(\xi_i(-(i+1)T),\xi_{i+1}(-(i+1)T)\big) < \Xsize$.
  Hence, $d_\rho(\xi_i,\xi_{i+1}) < \Xsize\,C\,e^{\rho(i+1)T}$
  on the interval $\intvOC{-\infty}{(i+1)T}$ by the exponential growth
  estimate.

  Thus on each interval $\intvCC{-(i+1)T}{-i\,T}$ we can use the
  triangle inequality to estimate
  \begin{align*}
    d_\rho(x,\xi_0)
    &\le d_\rho(x,\xi_i) + \sum_{j=0}^{i-1} d_\rho(\xi_j,\xi_{j+1})\\
    &<   \Xsize\,e^{\rho\,i\,T}
        +\sum_{j=0}^{i-1} \Xsize\,C\,e^{\rho(j+1)T}\\
    &\le \Xsize\Big(1 + \frac{C}{1-e^{\rho\,T}}\Big).
  \end{align*}
  The union of all such intervals is $\intvOC{-\infty}{0}$
  hence~\ref{eq:approx-growth} follows.
\end{proof}

\begin{proposition}[Perturbation of $Y$\ndash flow estimate]
  \label{prop:pert-y}
  Let $x$ be a $(\Xsize,T)$\ndash approximate solution to $v_\sx \circ g$.
  If\/ $T$ is sufficiently large and $\sigma_1,\,\nu,\,\Xsize$ are
  sufficiently small, then the flow $\Psi_x$ of $A(x(t))$
  has exponentially bounded growth as specified
  in~\ref{eq:exp-est-XY-mod}, that is,
  $\norm{\Psi_x(t,t_0)} \le C_\sy\,e^{\rho_\sy(t-t_0)}$ for all
  $t \ge t_0$.
\end{proposition}

\begin{proof}
  Let $\tilde{\rho} = \mfrac{1}{2}(\rho_- + \rho_\sy) < \rho_\sy$ and
  choose $T > 0$ sufficiently large that
  $\tilde{C}_-\,e^{\tilde{\rho}\,T} \le e^{\rho_\sy\,T}$.
  Let $x_i$ be an exact solution to $v_\sx \circ g$ such that
  $\sup_{t \in \intvCC{t_i}{t_i+T}} d(x(t),x_i(t)) \le \Xsize$ per
  Definition~\ref{def:approx-sol}, hence the flow $\hat{\Psi}$ of
  $\hat{A}(x_i(t))$ satisfies~\ref{eq:ODE-Y-est}, that is,
  $\norm{\hat{\Psi}^t} \le \tilde{C}_-\,e^{\rho_-\,t}$.
  We decompose $A(x(t)) = \hat{A}(x_i(t)) + B(t)$ and estimate
  \begin{equation*}
    \begin{aligned}
      \norm{B(t)}
      &=     \norm[\big]{A(x(t)) - \hat{A}(x_i(t))}\\
      &\le   \norm[\big]{A(x(t)) - A(x_i(t))}
            +\norm[\big]{A(x_i(t)) - (\D_y v_\sy \circ g)(x_i(t))}\\
      \con  +\norm[\big]{(\D_y v_\sy \circ g)(x_i(t)) - \hat{A}(x_i(t))}\\
      &\le \norm{\D A}\,d(x(t),x_i(t))
            + \epsilon(\nu) + 4\,C_v\,\sigma_1\,2\,\norm{\pi_{E^-}}.
    \end{aligned}
  \end{equation*}
  Note that $\norm{A}_1$ is bounded close to
  $\norm{\D_y v_\sy}_1 \le C_v$, and~\ref{eq:Ahat-DyVy-est} 
  and~\ref{eq:DxVy-bound} 
  were used to estimate the third
  term. We thus have $\norm{B(t)} \le \delta$ for any $\delta > 0$
  when $\sigma_1,\,\nu,\,\Xsize$ are sufficiently small. Hence by
  Lemma~\ref{lem:pert-lin}, we have
  $\norm{\Psi_x(\tau,\tau_0)} \le \tilde{C}_-\,e^{\tilde{\rho}(\tau-\tau_0)}$
  for any $\tau,\tau_0 \in \intvCC{t_i}{t_i+T}$.

  Now we cover the interval $\intvCC{t_0}{t}$ by intervals
  $\intvCC{t_0+(i-1)T}{t_0+i\,T}$ with corresponding exact solutions
  $x_i$ that approximate $x$. As in the proof of Lemma~\ref{lem:pert-exp-growth}, we write
  $t-t_0 = n\,T + \tau$ and use the group property of the flow to
  obtain
  \begin{equation*}
        \norm*{\Psi_x(t,t_0)}
    \le {\left(\tilde{C}_-\,e^{\tilde{\rho}\,T}\right)}^n\,
               \tilde{C}_-\,e^{\tilde{\rho}\,\tau}
    \le e^{\rho_\sy\,n\,T}\,\tilde{C}_-\,e^{\rho_\sy\,\tau}
    =   \tilde{C}_-\,e^{\rho_\sy(t-t_0)}.
  \end{equation*}
  We note that $C_\sy = \tilde{C}_-$ to complete the proof.
\end{proof}

Using these results, we choose $T$ sufficiently large and
$\delta,\,\Ysize,\,\sigma_1,\,\Xsize,\,\nu$ sufficiently small that
the modified flows $\D\Phi_y,\,\Psi_x$ satisfy exponential growth
rates~\ref{eq:exp-est-XY-mod} when curves
$y \in \overline{B_\Ysize(\R;Y)}$ and $(\Xsize,T)$\ndash approximate
solutions $x \in C^0(\R;X)$ are inserted.

\begin{lemma}[Variation of linear flow]
  \label{lem:var-lin}
  Let\/ $X$ be a metric space and\/ $Y$ a Banach space and let
  $A \in \BUC^\alpha(X;\CLin(Y))$ be a family of linear operators on $Y$ that
  depends uniformly $\alpha$\ndash H\"older continuous on $x \in X$, with H\"older
  coefficient $C_\alpha$ and $0 < \alpha \le 1$. Let\/ $\Psi_x$ denote
  the flow of $A$ under a curve $x \in C(I;X)$, and assume that it
  satisfies the exponential growth condition~\ref{eq:exp-growth-Y}.

  Then the variation of the flow satisfies the H\"older-like
  estimate
  \begin{equation}\label{eq:psi-growth}
    \norm{(\Psi_1 - \Psi_2)(t,\tau)}
    \le \frac{C_\alpha\,C^2}{-\alpha\,\rho}\,e^{\rho(t-\tau)}\,
        d_\rho(x_1,x_2)^\alpha\,e^{\alpha\,\rho\,\tau}
  \end{equation}
  when $d_\rho(x_1,x_2)$ is finite.
\end{lemma}

\begin{proof}
  Let $d_\rho(x_1, x_2)$ be finite, let $\Psi_1, \Psi_2$ be the
  associated flows of $A$ and denote $\Upsilon = \Psi_1 - \Psi_2$. We
  have for $\Upsilon$ the differential equation
  \begin{equation*}
    \der{}{t} \Upsilon(t,\tau)
    = A(x_1(t))\,\Upsilon(t,\tau)
    +\big[A(x_1(t)) - A(x_2(t))\big]\,\Psi_2(t,\tau),
    \qquad \Upsilon(t,t) = 0.
  \end{equation*}
  By variation of constants we obtain
  \begin{align*}
    \norm{\Upsilon(t,\tau)}
    &\le \int_\tau^t \norm{\Psi_1(t,\sigma)}\,
                     C_\alpha\,d(x_1(\sigma),x_2(\sigma))^\alpha\,
                     \norm{\Psi_2(\sigma,\tau)} \d\sigma \\
    &\le C_\alpha\,C^2\,e^{\rho(t-\tau)}
         \int_\tau^t d_\rho(x_1,x_2)^\alpha\,e^{\alpha\,\rho\,\sigma} \d\sigma\\
    &\le \frac{C_\alpha\,C^2}{-\alpha\,\rho}\,e^{\rho(t-\tau)}\,
         d_\rho(x_1,x_2)^\alpha\,e^{\alpha\,\rho\,\tau}. \qedhere
  \end{align*}
\end{proof}

\section{Existence and Lipschitz regularity}\label{sec:NHIM-lipschitz}

We start with proving Lipschitz estimates for two mappings onto curves
in $X, Y$, respectively. These mappings will be combined to a
contraction mapping $T$. Its fixed points parametrized by $x_0 \in X$
will correspond to the unique solution curves of the modified
system~\ref{eq:ODE-XY-mod} that stay bounded.

\index{$B^rY$@$\BY$}
\index{$I$}
Let $B^\rho(I;Y)$ denote the Banach space of exponentially bounded,
continuous curves in $Y$ on the interval $I = \intvOC{-\infty}{0}$ and
recall that in this chapter we always assume $\rho < 0$. Additionally,
we denote by $\BY = B^\rho(I;Y) \cap B_\Ysize(I;Y)$ the subset of
curves $y(t)$ which are moreover globally smaller than $\Ysize$. The
closure of $\BY$ is given by
$\BYc = B^\rho(I;Y)\cap\overline{B_\Ysize(I;Y)}$.
\begin{proposition}
  \label{prop:B_Ysize-closed}
  The space $\BYc$ is a closed subspace of the Banach space
  $B^\rho(I;Y)$, hence a complete metric space.
\end{proposition}

\begin{proof}
  Consider the evaluation mapping
  $\text{ev}_t \colon B^\rho(I;Y) \to Y \colon y \mapsto y(t)$. For
  each fixed $t \in I$ this is a continuous mapping as
  $\norm{y(t)} \le \norm{y}_\rho\,e^{\rho\,t}$ with $e^{\rho\,t}$ a
  finite number.

  Let $R = \overline{B(0;\Ysize)}$ be the closed ball in $Y$, then we
  have
  \begin{equation*}
    \BYc = \bigcap_{t \in I} \; \text{ev}_t^{-1}(R)
  \end{equation*}
  as an intersection of closed preimages under $\text{ev}_t$, hence
  closed.
\end{proof}

\index{$B^rX$@$\BX$}
For curves $x(t)$ in $X$, we cannot construct a similar space
$B^\rho(I;X)$ as $X$ is not a normed linear space. Instead we
construct a (not necessarily complete) metric space. Let
$\mathcal{B}^\rho(I;X) = \big(C^0(I;X),d_\rho\big)$ denote the space
of continuous curves equipped with the metric~\ref{eq:rho-dist} (which
is allowed to take the value $\infty$) and let $\BX$ be the subset of
curves $x \in C^0(I;X)$ that are $(\Xsize,T)$\ndash approximate
solutions to $v_\sx\circ g$ according to
Definition~\ref{def:approx-sol}. We suppress the dependence on $T$ from
the notation (note that $T$ in the equation below is a completely
different object); both $\Xsize$ and $T$ were fixed once and for all to
fulfill the requirements of Proposition~\ref{prop:pert-y}, we keep
just the subscript $\Xsize$ as a reminder and to distinguish from the
space $\mathcal{B}^\rho(I;X)$. Let $\rho < \rho_\sx$. Then exact
solutions of $v_\sx\circ g$ have finite $d_\rho$ distance on $I$, and
by Lemma~\ref{lem:approx-growth} the distance of any two curves
$x_1, x_2 \in \BX$ is finite, too.

\index{$T$}
\oldindex{$T_\sx,\,T_\sy$}
We write $T = T_\sy \circ \big(T_\sx\,,\,\text{pr}_1\big)$ with
\begin{equation}\label{eq:Txy}
  \begin{alignedat}{2}
  T    &\colon \BYc \times X   &&\to \BY,\\
  T_\sx&\colon \BYc \times X   &&\to \BX,\\
  T_\sy&\colon \BX  \times \BY &&\to \BY \subset \BYc,
  \end{alignedat}
\end{equation}
for any $\rho_\sy < \rho < \rho_\sx$. The map $T_\sx$ is defined by
the flow $\Phi_y$ of $\vx(\slot,y(t))$ with initial value $x_0 \in X$,
that is,
\begin{equation}\label{eq:Tx}
  T_\sx(y,x_0)(t) = \Phi_y(t,0,x_0).
\end{equation}
In~\cite{Henry1981:geom-semilin-parab-PDEs}, the $T_\sy$ part of the
\idx{contraction operator} is indirectly defined by another contraction.
Instead, here, we will set up $T_\sy$ as a direct mapping
\begin{equation}\label{eq:Ty}
  T_\sy(x,y)(t) =
      \int_{-\infty}^t \Psi_x(t,\tau)\,\Ynonlin(x(\tau),y(\tau)) \d\tau,
\end{equation}
where $\Psi_x$ is the flow of $A(x(t))$. This should ease proving
smoothness properties of $T$, which will subsequently imply smoothness
of the invariant manifold.

\begin{remark}
  Note that $\BX$ is not a Banach space or even a complete metric
  space. It will only appear as an intermediate space in the
  composition $T = T_\sy \circ \big(T_\sx\,,\,\text{pr}_1\big)$
  though, so this does not affect the Banach fixed point arguments, as
  long as the mappings $T_\sx, T_\sy$ compose to a contraction $T$ on
  $\BYc$ uniformly in the parameter $x_0 \in X$.
\end{remark}

The following two propositions show that the maps $T_\sx,\,T_\sy$ do
indeed map into their specified codomains, when parameters are chosen
sufficiently small.
\begin{proposition}
  \label{prop:bound-y}
  If\/ $\Vsize$ is chosen such that~\ref{eq:Vsize-bound-y} holds and
  $\delta,\,\sigma_1$ are sufficiently small, then $T_\sy$ maps into
  $B_\Ysize(I;Y)$.
\end{proposition}

\begin{proof}
  The conditions of Proposition~\ref{prop:pert-y} are
  satisfied for any $x \in \BX$, so the flow $\Psi_x$ of
  system~\ref{eq:ODE-Y-flowgen} satisfies exponential growth estimates
  with numbers $\rho_\sy,\,C_\sy$.

  Now, $T_\sy$ maps into $B_\Ysize(I;Y)$ since
  \begin{equation}\label{eq:contr-y-bound}
    \begin{aligned}
         \norm*{T_\sy(x,y)(t)}
    &\le \int_{-\infty}^t \norm{\Psi(t,\tau)}\,
                     \big(\norm{\Ynonlin(x(\tau),0)}
                         +\norm{\D_y \Ynonlin}\,\norm{y(\tau)}\big) \d\tau\\
    &\le \frac{C_\sy}{-\rho_\sy}\,
         (C_v\,\sigma_1 + \delta + \Vsize\,\eta)
    \end{aligned}
  \end{equation}
  which can be made smaller than $\Ysize$ by choosing $\Vsize$ such
  that
  \begin{equation}\label{eq:Vsize-bound-y}
    \frac{C_\sy\,\Vsize}{-\rho_\sy} \le \frac{1}{2}
  \end{equation}
  holds, as well as $\delta,\,\sigma_1$ sufficiently small.
\end{proof}
This shows that $T_\sy$ is well-defined in~\ref{eq:Txy} after
choosing $\Vsize,\,\delta,\,\sigma_1$ possibly smaller. Note that the
choice of $\Vsize$ does not depend on any of the other small bounds.
Similarly, we verify that $T_\sx$ maps into $\BX$.
\begin{proposition}
  \label{prop:bound-x}
  If\/ $\delta,\,\sigma_1$, and $\Ysize$ are chosen sufficiently
  small, then $T_\sx$ maps into $\BX$.
\end{proposition}

\begin{proof}
  Let $y \in \BY$ and $x_0 \in X$. The curve $x = T_\sx(y,x_0)$ is
  generated by the vector field $\vx(\slot,y(t))$ which is a small
  perturbation of $v_\sx \circ g$, since
  \begin{equation*}
    \norm{\vx(\slot,y(t)) - v_\sx \circ g}_0 \le \Vsize.
  \end{equation*}
  Let $t_2 - t_1 \le T$ and let $\Phi^t$ denote the flow of
  $v_\sx \circ g$. We apply the nonlinear variation of constants
  estimate~\ref{eq:nonlin-varconst-est} and obtain for
  $t \in \intvCC{t_1}{t_2}$
  \begin{align*}
    d\big(x(t),\Phi^{t-t_2}(x(t_2))\big)
    &\le \int_{t}^{t_2} \norm{\D\Phi^{t-\tau}(x(\tau))}\,\Vsize \d\tau\\
    &\le \int_{t}^{t_2} C_\sx\,e^{\rho_\sx\,(t-\tau)}\,\Vsize \d\tau\\
    &\le \frac{C_\sx\,\Vsize}{-\rho_\sx}\,e^{\rho_\sx\,(t-t_2)}.
  \end{align*}
  Thus, if we choose $\Vsize$ sufficiently small that
  \begin{equation}
    \frac{C_\sx\,\Vsize}{-\rho_\sx}\,e^{\rho_\sx\,T} < \Xsize,
  \end{equation}
  then $x$ is $(\Xsize,T)$\ndash approximated by the exact solution
  $\Phi^{t,t_2}(x(t_2))$ on the interval $\intvCC{t_1}{t_2}$ so
  $x \in \BX$.
\end{proof}

The basic argument for the \idx{Perron method} is encoded in the following
lemma: a \emph{$y$\ndash bounded} solution curve of the
system~\ref{eq:ODE-XY-mod} is equivalent to $y \in B_\Ysize(I;Y)$
being a fixed point of $T$, while this map $T$ will be shown to be a
contraction.
\begin{lemma}[Cotton--Perron]
  \label{lem:perron-equiv}
  Let $x \in C(I;X)$, $y \in B_\Ysize(I;Y)$ bounded, and $x_0 \in X$.
  Then the following statements are equivalent:
  \begin{enumerate}
  \item the pair $(x,y)$ is a solution curve for the modified
    system~\ref{eq:ODE-XY-mod} with partial initial condition
    $x(0) = x_0$; \label{enum:perron-ODE}
  \item $y \in B_\Ysize(I;Y)$ is a fixed point of\/ $T(\slot,x_0)$ and
    $x = T_\sx(y,x_0)$. \label{enum:perron-contr}
  \end{enumerate}
\end{lemma}

\begin{proof}
  The proof goes along the same lines as the classical Perron method
  for hyperbolic fixed points. As an intermediate step, we introduce
  the operator
  \begin{equation}\label{eq:Ty-tilde}
    \hat{T}_\sy(x,y,t_0)(t)
    = \Psi_x(t,t_0)\,y(t_0)
     +\int_{t_0}^t \Psi_x(t,\tau)\,\Ynonlin(x(\tau),y(\tau)) \d\tau
  \end{equation}
  and the following statement that is equivalent to those in the
  lemma:
  \textit{%
    \begin{enumerate}\setcounter{enumi}{2}
    \item the pair $(x,y)$ is a fixed point of\/ $(T_\sx,\,\hat{T}_\sy)$ for
      each $t_0 \in \intvOC{-\infty}{0}$.
      \label{enum:perron-varconst}
    \end{enumerate}
  }%
  Equivalence of~\ref{enum:perron-ODE}
  and~\ref{enum:perron-varconst} (with $t_0 = 0$) is a direct
  consequence of equivalence of differential and integral equations;
  the equation for $y$ has been rewritten as a variation of constants
  integral with respect to the nonlinear term $\Ynonlin$. If $(x,y)$
  is a fixed point of $(T_\sx,\,\hat{T}_\sy)$ for $t_0 = 0$, then this
  holds for any $t_0 \in \intvOC{-\infty}{0}$. Note that the initial
  value for $y$ is left unspecified in both statements.

  We finish by proving the implications
  \ref{enum:perron-varconst}~$\Rightarrow$
  \ref{enum:perron-contr}~$\Rightarrow$ \ref{enum:perron-ODE}. For the
  first we take the limit $t_0 \to -\infty$ in $\hat{T}_\sy(x,y,t_0)$.
  Since $\Psi_x$ decays exponentially and $y$ is bounded, it follows
  that this limit is well-defined:
  \begin{equation*}
    \forall\;t \in I\colon
    \lim_{t_0 \to -\infty} \hat{T}_\sy(x,y,t_0)(t) = T_\sy(x,y)(t).
  \end{equation*}
  Hence, a fixed point of $(T_\sx,\,\hat{T}_\sy)$ is a fixed point of
  $(T_\sx,\,T_\sy)$. The last implication can readily be verified by
  calculating the time derivatives of $x = T_\sx(y,x_0)$ and
  $y = T_\sy(x,y)$ to show that $(x,y)$ is a solution
  of~\ref{eq:ODE-XY-mod} with $x(0) = x_0$.
\end{proof}

Next, we prove that both $T_\sx,\,T_\sy$ are Lipschitz, while the
Lipschitz constant of $T_\sy$ can be made arbitrarily small.
\begin{lemma}
  \label{lem:Ty-contr}
  Let $\rho_\sy < \rho \le \rho_\sx$ and $0 < \contr_\sy < 1$. If\/
  $\Vsize$ is sufficiently small, then $\Lip(T_\sy) \le \contr_\sy$.
\end{lemma}

\begin{proof}
Let $(x_i,y_i),\,i=1,2$ be curves from $\BX \times \BY$. Let $\Psi_i$
be the corresponding flows of $A$ along the curves $x_i$.
Then the application of Lemma~\ref{lem:var-lin} in the Lipschitz case
$\alpha = 1$ leads to a Lipschitz estimate on $T_\sy$ for any
$\rho_\sy < \rho \le \rho_\sx$:
\begin{align*}
  \fst \norm{T_\sy(x_1,y_1)(t) - T_\sy(x_2,y_2)(t)}\\
  &\le \int_{-\infty}^t \norm{\Psi_1(t,\tau)\,\Ynonlin(x_1(\tau),y_1(\tau))
                            - \Psi_2(t,\tau)\,\Ynonlin(x_2(\tau),y_2(\tau))} \d\tau
   \displaybreak[1]\\
  &\le \int_{-\infty}^t \norm{\Psi_1(t,\tau)-\Psi_2(t,\tau)}\,\norm{\Ynonlin(x_1(\tau),y_1(\tau))}\\
  &\hspace{1.2cm}   {} +\norm{\Psi_2(t,\tau)}\,\norm{\Ynonlin(x_1(\tau),y_1(\tau))
                                                    -\Ynonlin(x_2(\tau),y_2(\tau))}\d\tau
   \displaybreak[1]\\
  &\le \int_{-\infty}^t
         \frac{C_v\,C_\sy^2}{-\rho}\,e^{\rho_\sy(t-\tau)}\,
         d_\rho(x_1,x_2)\,e^{\rho\,\tau}\,\norm{\Ynonlin}_0\\
  &\hspace{1.2cm}
    {} + C_\sy\,e^{\rho_\sy(t-\tau)}\norm{\D \Ynonlin}_0\,
         \big(d_\rho(x_1,x_2) + \norm{y_1 - y_2}_\rho\big)\,e^{\rho\,\tau} \d\tau\\
  &\le \Vsize\,C\,\big(d_\rho(x_1,x_2) + \norm{y_1 - y_2}_\rho\big)\,e^{\rho\,t}.
\end{align*}
Here $C < \infty$ depends only on the constants and additional integration
factors ${(\rho - \rho_\sy)}^{-1}$ in the last integral, hence
$\Vsize\,C \le \contr_\sy$ when $\Vsize$ is small enough.
\end{proof}

\begin{lemma}
  \label{lem:Tx-contr}
  Let $\rho_\sy < \rho < \rho_\sx$. If\/ $\Vsize,\,\Ysize$ are
  sufficiently small, then $\Lip(T_\sx) \le \contr_\sx$ for some
  $\contr_\sx > 1$ independent of all small parameters.
\end{lemma}

\begin{proof}
  Let $\xi_1, \xi_2 \in X$ and $y_1, y_2 \in \BY$. For $i=1,2$, define
  $v_i(t,\slot) = \vx(\slot,y_i(t))$ and let $x_i \in C^1(I;X)$ be a
  solution of the system $v_i$ with initial condition $\xi_i$. We
  compare the systems $v_1,\,v_2$:
  \begin{equation*}
        \norm{v_1(t,\slot) - v_2(t,\slot)}
    \le \norm{\D_y \vx}\,\norm{y_1(t) - y_2(t)}
    \le C_v\,\norm{y_1(t) - y_2(t)}.
  \end{equation*}
  The flow $\Phi_{y_1}$ has exponential growth numbers
  $\rho_\sx,\,C_\sx$. We view $v_2$ as a small perturbation of $v_1$
  and apply the nonlinear variation of constants
  estimate~\ref{eq:nonlin-varconst-est} to obtain
  \begin{align*}
         d(x_1(t),x_2(t))
    &\le C_\sx\,e^{\rho_\sx\,t}\,d(\xi_1,\xi_2)
        +\int_t^0 C_\sx\,e^{\rho_\sx(t-\tau)}\,
                  C_v\,\norm{y_1(\tau) - y_2(\tau)} \d\tau\\
    &\le C_\sx\,e^{\rho_\sx\,t}\,d(\xi_1,\xi_2)
        +C_\sx\,C_v\,\norm{y_1 - y_2}_\rho\,
         \int_t^0 e^{\rho_\sx(t-\tau)}\,e^{\rho\,\tau} \d\tau\\
    &\le C_\sx\,e^{\rho_\sx\,t}\,d(x_1,x'_2)
        +\frac{C_\sx\,C_v}{\rho_\sx - \rho}\,\norm{y_1 - y_2}_\rho\,
         e^{\rho_\sx\,t}.
  \end{align*}
  Now
  \begin{equation*}
    \begin{aligned}
         d_\rho\big(T_\sx(y_1,\xi_1),T_\sx(y_2,\xi_2)\big)
    &\le \sup_{t \le 0}\;
         C_\sx\,d(\xi_1,\xi_2)\,e^{(\rho_\sx-\rho)\,t}
       + \frac{C_\sx\,C_v}{\rho_\sx - \rho}\,\norm{y_1 - y_2}_\rho\,
         e^{(\rho_\sx-\rho)\,t}\\
    &\le C_\sx\,d(\xi_1,\xi_2)
       + \frac{C_\sx\,C_v}{\rho_\sx - \rho}\,\norm{y_1 - y_2}_\rho
    \end{aligned}
  \end{equation*}
  exhibits a Lipschitz constant $\contr_\sx$ for $T_\sx$ that does not
  depend on any of the small parameters.
\end{proof}

\begin{proposition}[Extension of solution is bounded in $Y$]
  \label{prop:bound-extend}
  Let $(x,y)(t)$ be a solution of the perturbed
  system~\ref{eq:ODE-XY-mod} satisfying $y \in B_\Ysize(I;Y)$ for
  $t \le 0$. For\/ $\Vsize,\,\delta,\,\sigma_1$ sufficiently small, the
  forward extension to $t \ge 0$ has $y \in B_\Ysize(\R;Y)$.
\end{proposition}

\begin{proof}
  First of all, choose $\Vsize,\,\delta,\,\sigma_1$ sufficiently small
  such that by~\ref{eq:contr-y-bound}, we have
  $y \in B_{\Ysize/2}(I;Y)$. Proceeding by contradiction, let $t_0$ be
  the first time after which $y(t)$ becomes larger than $\Ysize$, thus
  \begin{equation*}
    t_0 = \sup\;\set{t \in \R}{\forall\;\tau \le t\colon \norm{y(\tau)} \le \Ysize}.
  \end{equation*}
  The curve $x(t)$ is a $(\Xsize,T)$\ndash approximate solution to
  $v_\sx \circ g$ on the interval $\intvOC{-\infty}{t_0}$, so from
  Proposition~\ref{prop:bound-y} we conclude that
  $y \in B_{\Ysize/2}\big(\intvOC{-\infty}{t_0};Y\big)$. The
  continuity of $y$ contradicts the assumption that $t_0$ is the
  supremum.
\end{proof}

\subsubsection{Completing the proof of existence and Lipschitz regularity}

We finally put things together and prove that a unique persistent
manifold $\tilde{M}$ exists and that it is Lipschitz.

Since $T_\sx$ satisfies a fixed Lipschitz estimate, we can choose
$\Vsize$ small enough to obtain $\contr_\sx\cdot\contr_\sy < 1$. Thus,
$T$ is a contraction on $\BYc$ for each fixed
$\rho_\sy < \rho < \rho_\sx$; $\Vsize$ will depend on $\rho$ though.
According to Proposition~\ref{prop:B_Ysize-closed}, $\BYc$ is a
complete metric space, so the Banach fixed point theorem shows that
there is a unique $y \in \BYc$ fixed point of $T$; it holds moreover
that $y \in \BY$. This contraction also depends (uniformly) on the
parameter $x_0 \in X$, hence we obtain a fixed point map
\index{$\Theta$}
\begin{equation}\label{eq:FP-map}
  \Theta^\infty \colon X \to \BY,
\end{equation}
satisfying the relation
\begin{equation}\label{eq:FP-rel}
  \forall\;x_0 \in X\colon \Theta^\infty(x_0) = T(\Theta^\infty(x_0),x_0).
\end{equation}
The superscript $\infty$ indicates that this map is obtained as a
limit of applying the uniform contraction $T$. The parameter
dependence in $T$ is Lipschitz, so the map $\Theta^\infty$ will be
Lipschitz as well.

By Proposition~\ref{prop:bound-extend}, the fixed point
$y = \Theta^\infty(x_0)$ is bounded by $\Ysize$ for all time and
$B_\Ysize(I;Y) = \BY$ as sets, so $y$ is the unique $\Ysize$\ndash
bounded solution with partial initial data $x(0) = x_0$. In
combination with the evaluation map $y \mapsto y(0)$, we obtain the
mapping
\begin{equation}\label{eq:persist-graph}
  \tilde{h}\colon X \to B(0;\Ysize) \subset Y
           \colon x_0 \mapsto \Theta^\infty(x_0)(0).
\end{equation}
Its graph $\tilde{M} = \Graph(\tilde{h})$ is the unique invariant
manifold of the modified system~\ref{eq:ODE-XY-mod} and is Lipschitz
as well.

Since both $M$ and $\tilde{M}$ are described by graphs of small
functions $X \to Y$, it follows that they are homeomorphic and
$\norm{\tilde{h}}_0 \le \Ysize$. We can choose another, arbitrarily
small $\Ysize'$ instead. This requires us to choose smaller
$\delta',\,\sigma_1'$ parameters as well. But as can be seen
from~\ref{eq:small-param-dep}, $\Ysize$ does not depend on
$\delta,\,\sigma_1$, so the newly found
$\norm{\tilde{h}'}_0 \le \Ysize'$ will actually be unique in the
original $\Ysize$\ndash sized neighborhood as well.



\section{Smoothness}\label{sec:NHIM-smoothness}
\index{smoothness}

\newcommand{\ux}{{\bar{x}}}
\newcommand{\rx}{{\delta_\sx}}

\newcommand{\dxtilde}{\widetilde{\dx}}

\newcommand{\TFBXc}{\TF \BX\big|_{C^1}}

\newcommand{\BJXb}{\mathcal{B}^\rho_\Xsize(J;X)}
\newcommand{\BJX}{ \mathcal{B}^\rho(J;X)}
\newcommand{\BJY}{B^\rho_\Ysize(J;Y)}

\newcommand{\Th}{\Theta^\infty}
\newcommand{\DTh}{\DF\Theta^\infty}


To study smoothness of $\Theta^\infty$, we can formally differentiate
the fixed point relation~\ref{eq:FP-rel} with respect to $x_0$ to
obtain contractive mappings $T^{(k)}$ for the $k$\th order derivatives
of maps $\Theta$ and then apply the fiber contraction theorem, see
Appendix~\ref{chap:fibercontr}. If we assume that $\Theta$
satisfies~\ref{eq:FP-rel}, then Proposition~\ref{prop:compfunc-deriv}
shows that (at least formally)
\begin{equation}\label{eq:der-fixedpoint}
   \D^k \Theta(x_0)
   = \!\!\sum_{\substack{l,m \ge 0\\l + m \le k\\(l,m)\neq(0,0)}}\!\!
       \D_y^l \D_{x_0}^m T(\Theta(x_0),x_0) \cdot
       P_{l,k-m}\big(\D^\bullet \Theta(x_0)\big),
\end{equation}
which can be rewritten as a fiber contraction map on $\D^k \Theta$
by isolating that term ($l=1,\,m=0$) on the right-hand side as
\begin{equation*}
  \D^k \Theta(x_0)
  = \D_y T(\Theta(x_0),x_0) \cdot \D^k \Theta(x_0) + \ldots
\end{equation*}
All the remaining terms are expressions in the lower order
derivatives $\D^n \Theta(x_0)$ for $n<k$ only; these form the base
space in the fiber contraction theorem.

\index{$\mu$}
The derivatives $\D^k\Theta$ and $\D_y^l \D_{x_0}^m T$
in~\ref{eq:der-fixedpoint} do not exist on the space
$\BY$ as codomain, however. Indeed, if they did, we could have
applied the implicit function theorem right away. Instead, the
derivatives $\D^k\Theta$ are only well-defined on spaces\footnote{%
  Note that the spaces $B^{k\rho+\mu}(I;Y)$ are to be understood as
  the codomains of the maps $\Theta$, hence the $\D^k\Theta$ as
  multilinear operators into these. The spaces $Y_\eta$ play the same
  role in~\cite[Def.~3.10]{Vanderbauwhede1989:centremflds-normal}.%
} $B^{k\rho+\mu}(I;Y)$, where $\mu < 0$ is an arbitrarily small additional
exponential growth rate. Derivatives of the maps $T_\sx,\,T_\sy$ do
not exist at all. By using the fiber contraction theorem and
interpreting the $\D^k T_\sx,\,\D^k T_\sy$ as `formal derivatives' in some
appropriate way, we can still show, though, that the $\D^k\Theta$ are
higher derivatives of $\Theta$ that converge to the derivatives of
$\Theta^\infty$ under iteration of the fiber contraction
maps~\ref{eq:der-fixedpoint}. The gap condition
$\rho_\sy < \fracdiff\,\rho_\sx$ will show up in the requirement
that~\ref{eq:der-fixedpoint} is contractive for $k \le \fracdiff$ (and
finally $k + \alpha \le \fracdiff$ when considering H\"older
continuity). In case of uniform continuity (i.e.\ws when $\alpha = 0$)
we make use of the strict inequality to seize some of the spectral
space left for the terms $\mu < 0$.

The interpretation of $\D T$ and its constituents $\D T_\sx$
and $\D T_\sy$ as true derivatives is obstructed already by the fact
that neither $\BY$ nor $\BX$
are smooth Banach manifolds\footnote{%
  At least, they are not smooth Banach manifolds in a natural way, see
  the discussion in Section~\ref{sec:formal-tangent}.%
}, hence these can never be the (co)domain of differentiable maps.
Thus, the chain rule
\begin{equation*}
  \D\Theta^{n+1} = \D_y T \cdot \D\Theta^n + \D_{x_0} T
\end{equation*}
cannot be used to conclude the existence of $\D\Theta^{n+1}$ from
$\D\Theta^n$ by induction. On the other hand, we can find `formal
tangent bundles' of these spaces on which $\D T_\sx, \D T_\sy$ are
defined as `formal derivatives', and we even have explicit
formulas~\ref{eq:DT-set} for these maps. From here on we shall use
the notation $\DF f$ to indicate a formal derivative and $\D f$ to
indicate that a function $f$ is truly differentiable. We shall not
make precise the notion of `formal', but heuristically these formal
objects can be seen as limits of well-defined real smooth manifolds
and derivatives, see Section~\ref{sec:banach-mflds}.
\oldindex{$\DF$}
\oldindex{$\DF$|seealso{formal derivative}}

First, we outline the procedure of obtaining $\Theta^\infty$ as a
truly differentiable map by careful manipulation of these formal
derivatives. This is followed by the details of working out the
definitions and estimates. Finally, we show how everything generalizes
to higher derivatives. This last step adds more complexity, but
requires no fundamentally new ideas.

Higher derivatives of functions involving variables or values in $X$
need to be treated with some care, as these are not naturally defined.
In such expressions, the derivatives are with respect to normal
coordinates at the base point in domain and range, according to
Definition~\ref{def:deriv-normcoord}. I should point the reader to
Appendix~\ref{chap:exp-growth}: it establishes the essential basic
ingredient for this section on (higher) smoothness, namely how
exponential growth estimates carry over to continuity and higher
derivatives of the flow. Additionally, building on bounded geometry
and Definition~\ref{def:unif-bounded-map}, a framework is set up to
work with these notions on the manifold $X$.
\clearpage

\subsection{A scheme to obtain the first derivative}\label{sec:scheme-deriv}

The map $T$ is not differentiable. Instead, we shall use the scheme
below to obtain differentiability of $\Theta^\infty$. The sequence
$\{\Theta^n\}_{n \ge 0}$ of maps $X \to \BY$ is defined
by $\Theta^{n+1}(x_0) = T(\Theta^n(x_0),x_0)$ and
$\Theta^0 \equiv 0$. We prove the differentiability of the $\Theta^n$
by induction and finally conclude that $\Theta^\infty$ is
differentiable as well.
\begin{enumerate}
\item\label{enum:scheme-formal-deriv}%
  First, we propose candidate formal derivatives~\ref{eq:DT-set}
  of $T_\sx, T_\sy$. These are obtained naturally by standard
  differentiation and variational techniques, postponing for the
  moment the question of which spaces these maps are well-defined on.
  We define $\DF T$ in terms of the formal derivatives $\DF T_\sx$ and
  $\DF T_\sy$.

\item\label{enum:scheme-fibercontr}%
  The pair $(T,\DF T)$ acts as a uniform fiber contraction on
  pairs of maps
  \begin{equation*}
    (\Theta^n,\DF \Theta^n) \colon
    \T X \to \BY \times B^{\rho+\mu}(I;Y)
  \end{equation*}
  when $\rho_\sy < \rho+\mu \le \rho < \rho_\sx$ holds, both in case
  of $\mu = 0$ and $\mu < 0$ small.

\item\label{enum:scheme-formal-tangent}%
  There are appropriate formal tangent bundles of the spaces
  $\BY$, $\BX$ on which these formal derivatives are well-defined.
  Moreover, these formal tangent bundles
  can be endowed with a topology such that $\DF T_\sx$, $\DF T_\sy$,
  and $\DF T$ are uniformly continuous into bundles with slightly
  larger exponential growth rate $\rho+\mu$. Under appropriate
  assumptions (and with $\mu = \alpha\,\rho$) these formal derivatives
  are $\alpha$\ndash H\"older continuous.

\item\label{enum:scheme-appl-FCT}%
  The fiber contraction theorem~\ref{thm:fibercontr} can be applied.
  It follows from~\ref{enum:scheme-fibercontr} that $\DF T$ has a
  unique fixed point $\DF \Theta^\infty\colon \T X \to \TF \BY$, and
  from~\ref{enum:scheme-formal-tangent} that the map
  \begin{equation*}
    \Theta \mapsto \DF T(\Theta)\cdot\DF \Theta^\infty\colon
    C^0\big(X;\BY\big) \to \Gamma_b\big(\CLin\big(\T X;\TF B^{\rho+\mu}_\Ysize(I;Y)\big)\big)
  \end{equation*}
  into bounded sections of the bundle
  $\pi\colon \CLin\big(\T X;\TF B^{\rho+\mu}_\Ysize(I;Y)\big) \to X$
  is continuous. Thus we can conclude that $\DF \Theta^n$
  converges in $\Gamma_b\big(\CLin\big(\T X;B^{\rho+\mu}(I;Y)\big)\big)$
  to the unique fixed point $\DF \Theta^\infty$, simultaneously with
  $\Theta^n \to \Theta^\infty$. See~\ref{eq:fiber-contr-spaces}
  and~\ref{eq:TBY-formal} for precise definitions of these spaces.
  Moreover, $\DF\Theta^\infty$ is uniformly or H\"older continuous.

\item\label{enum:scheme-restr-deriv}%
  There is a family of maps, given by restricting the domain $I$ of
  curves,
  \index{$T^{b,a}$}
  \begin{equation*}
    T^{b,a}\colon B^\rho_\Ysize(\intvCC{a}{0};Y) \times X
              \to B^\rho_\Ysize(\intvCC{b}{0};Y)
  \end{equation*}
  that approximate $T$, and moreover these $T^{b,a}$ are
  differentiable maps between Banach manifolds whose derivatives
  $\D T^{b,a}$ approximate the formal derivative $\DF T$.

\item\label{enum:scheme-induc-deriv}%
  With the continuous embedding
  $B^\rho(I;Y) \embedto B^{\rho+\mu}(I;Y)$ and the previous point, we
  show that if $\Theta^n\colon X \to B^{\rho+\mu}(I;Y)$ is
  differentiable, then
  \begin{equation*}
    \DF \Theta^{n+1} = \DF_y T \cdot \D\Theta^n + \DF_{x_0} T
  \end{equation*}
  is the derivative of $\Theta^{n+1}\colon X \to B^{\rho+\mu}(I;Y)$.

\item\label{enum:scheme-limit-deriv}%
  Finally, we use Theorem~\ref{thm:difflimitfunc} to conclude that
  since the sequence $\Theta^n$ converges to $\Theta^\infty$ and its
  derivatives satisfy $\D\Theta^n \to \DF\Theta^\infty$, it must hold
  that $\D\Theta^\infty = \DF\Theta^\infty$ as a map into
  $B^{\rho+\mu}(I;Y)$.
\end{enumerate}
In the subsequent sections we shall work out the details of this
scheme. With some care, the same ideas generalize to higher
derivatives.

\subsection{Candidate formal derivatives}\label{sec:cand-derivs}

We first explicitly give the candidate mappings for the derivatives of
$T_\sx, T_\sy$. From now on, we will use shorthand notation
$x_y(t) = T_\sx(y,x_0)(t) = \Phi_y(t,0,x_0)$. The spaces that these
maps act on will be made more precise in the following sections;
$\dx,\,\dy$ denote variations of curves $x \in \BX$ and $y \in \BY$,
respectively, and $\dx_0 \in \T_{x_0} X$.
\index{$\dx,\,\dy,\,\dx_0$}
\begin{subequations}\label{eq:DT-set}\jot=5pt
\begin{align}
  \label{eq:DpTx}
  \big(\DF_{x_0} T_\sx(y,x_0)\,\dx_0\big)(t)
  &= \D\Phi_y(t,0,x_0) \cdot \dx_0\\
  \label{eq:DyTx}
  \big(\DF_y T_\sx(y,x_0)\,\dy\big)(t)
  &= \int_t^0 \D\Phi_y(t,\tau,x_y(\tau))\,
              \D_y \vx(x_y(\tau),y(\tau))\,\dy(\tau) \d\tau,\\
  \label{eq:DxTy}
  \big(\DF_x T_\sy(x,y)\,\dx\big)(t)
  &= \int_{-\infty}^t \Psi_x(t,\tau)\,\D_x \Ynonlin(x(\tau),y(\tau))\,\dx(\tau)\\
  &\hspace{1.1cm}    +\big(\DF_x\Psi_x\cdot\dx\big)(t,\tau)\,\Ynonlin(x(\tau),y(\tau))
                    \d\tau,\nonumber\\
  \label{eq:DyTy}
  \big(\DF_y T_\sy(x,y)\,\dy\big)(t)
  &= \int_{-\infty}^t \Psi_x(t,\tau)\,\D_y \Ynonlin(x(\tau),y(\tau))\,\dy(\tau) \d\tau,\\
  \label{eq:DxPsi}
  \big(\DF_x \Psi_x\cdot\dx\big)(t,\tau)
  &= \int_\tau^t\; \Psi_x(t,\sigma)\,\D A(x(\sigma))\,\dx(\sigma)\,
                   \Psi_x(\sigma,\tau) \d\sigma.
\end{align}
\end{subequations}
The correctness of these expressions pointwise in $t$ can be
checked by variation of constants and follows from
Theorem~\ref{thm:diff-flow}. Note also that the expressions above
are linear in the variations $\dx,\,\dy,\,\dx_0$. The
map~\ref{eq:DxPsi} is only included in the list for its occurrence
in~\ref{eq:DxTy}.

\subsection{Uniformly contractive fiber maps}\label{sec:unif-fiber-contr}

We establish uniform boundedness of the formal derivative
maps~\ref{eq:DT-set} as linear operators on $\dx,\,\dy$, and $\dx_0$.
The estimates are straightforward generalizations of those in
Section~\ref{sec:NHIM-lipschitz}. The operator norms are induced by
$\norm{\slot}_\rho$ norms. We have the following list of estimates:
{\jot=5pt
\begin{align*}
  \norm{\DF_{x_0} T_\sx(y,x_0)}
  &=   \sup_{\substack{t \in I\\\norm{\dx_0}=1}}\;
       \norm[\big]{\D\Phi_y(t,0,x_0)\dx_0}\,e^{-\rho\,t}
   \le \sup_{t \in I}\;
       C_\sx\,e^{\rho_\sx\,t}\,e^{-\rho\,t}
   \le C_\sx,\displaybreak[1]\\
  \norm{\DF_y T_\sx(y,x_0)}
  &\le \!\sup_{\substack{t \in I\\\norm{\dy}_\rho=1}}\;
       \int_t^0 \norm{\D\Phi_y(t,\tau,x_y(\tau))\,\D_y \vx(x_y(\tau),y(\tau))\,
                      \dy(\tau)} \d\tau \cdot e^{-\rho\,t}\\
  &\le \sup_{t \in I}\;
       \int_t^0 C_\sx\,e^{\rho_\sx(t-\tau)}\,C_v\,e^{\rho(\tau-t)} \d\tau
   \le \frac{C_\sx\,C_v}{\rho_\sx-\rho},\displaybreak[1]\\
  \norm{\DF_y T_\sy(x,y)}
  &\le \!\sup_{\substack{t \in I\\\norm{\dy}_\rho=1}}\;
       \int_{-\infty}^t \norm{\Psi_x(t,\tau)\,\D_y \Ynonlin(x(\tau),y(\tau))\,
                              \dy(\tau)} \d\tau \cdot e^{-\rho\,t}\\
  &\le \sup_{t \in I}\;
       \int_{-\infty}^t C_\sy\,e^{\rho_\sy(t-\tau)}\,
                        \Vsize\,e^{\rho(\tau-t)} \d\tau
   \le \frac{C_\sy\,\Vsize}{\rho-\rho_\sy},\displaybreak[1]\\
  \norm{\big(\DF_x \Psi_x\cdot\dx\big)(t,\tau)}
  &\le \int_\tau^t \norm{\Psi_x(t,\sigma)\,\D A(x(\sigma))\,\dx(\sigma)\,
                         \Psi_x(\sigma,\tau)} \d\sigma\\
  &\le \int_\tau^t C_\sy\,e^{\rho_\sy(t-\sigma)}\,C_v\,
                   \norm{\dx}_\rho\,e^{\rho\,\sigma}\,
                   C_\sy\,e^{\rho_\sy(\sigma-\tau)} \d\sigma\\
  &\le \frac{C_\sy^2\,C_v}{-\rho}\,\norm{\dx}_\rho\,
       e^{\rho_\sy(t-\tau)}\,e^{\rho\,t},\displaybreak[1]\\
  \norm{\DF_x T_\sy(x,y)}
  &\le \!\sup_{\substack{t \in I\\\norm{\dx}_\rho=1}}\;
       \int_{-\infty}^t \Big[
         \norm{\Psi_x(t,\tau)\,\D_x \Ynonlin(x(\tau),y(\tau))\,\dx(\tau)}\\[-10pt]
  &\hspace{2.3cm}
        +\norm{\big(\DF_x\Psi_x\cdot\dx\big)(t,\tau)\,\Ynonlin(x(\tau),y(\tau))}
       \Big] \d\tau \cdot e^{-\rho\,t}\\[5pt]
  &\le \sup_{t \in I}\;
       \int_{-\infty}^t C_\sy\,e^{\rho_\sy(t-\tau)}\,\Vsize\,e^{\rho\,\tau}
                       +\frac{C_\sy^2\,C_v\,\Vsize}{-\rho}\,e^{\rho_\sy(t-\tau)}\,e^{\rho\,t}
                      \d\tau \cdot e^{-\rho\,t}\\
  &\le \frac{C_\sy\,\Vsize}{\rho-\rho_\sy}
      +\frac{C_\sy^2\,C_v\,\Vsize}{\rho\cdot\rho_\sy}.
\end{align*}}

These estimates show that
\begin{equation}\label{eq:DT-chainrule}
  \begin{aligned}
    \DF_y     T &= \DF_x T_\sy \cdot \DF_y T_\sx  + \DF_y T_\sy,\\
    \DF_{x_0} T &= \DF_x T_\sy \cdot \DF_{x_0} T_\sx
  \end{aligned}
\end{equation}
are bounded linear maps when $\rho_\sy < \rho < \rho_\sx$. Since we
have some spectral elbow room, we can first choose a value for $\rho$ and
then choose $\mu < 0$ sufficiently close to zero, such that this
inequality holds both for $\rho$ and $\rho + \mu$. If $\Vsize$ is
sufficiently small, then
$\norm{\DF_y T} \le \contr < 1$ can be satisfied. This shows that
$(T,\DF T)$ is a uniform fiber contraction on
$\dy \in B^{\rho+\mu}(I;Y)$ over base curves $y \in \BY$, and with
additional parameter $x_0 \in X$. It can also be viewed as a fiber
mapping of maps $\DF \Theta$ over base maps $\Theta$. Let us define
\index{$S$@$\mathcal{S}$}
\begin{equation}\label{eq:fiber-contr-spaces}
  \mathcal{S}_0 = C^0\big(X;\BY\big) \quad\text{and}\quad
  \mathcal{S}_1^\mu = \Gamma_b\big(\CLin\big(\T X; B^{\rho+\mu}(I;Y)\big)\big),
\end{equation}
where $\mathcal{S}_0$ is equipped with the supremum norm and $\mathcal{S}_1^\mu$ is
interpreted as bounded sections of the bounded geometry bundle over
$X$ of linear maps between $\T X$ and the trivial bundle
$\pi\colon X \times B^{\rho+\mu}(I;Y) \to X$, equipped with the
(supremum/operator) norm
\begin{equation*}
  \norm{\DF\Theta} = \sup_{x_0 \in X}\;
  \norm{\DF\Theta(x_0)}_{\CLin(\T_{x_0} X;B^{\rho+\mu}(I;Y))}.
\end{equation*}
Then $(T,\DF T)$ can also be viewed as a fiber mapping
\begin{equation}\label{eq:fiber-map-DT}
  \begin{gathered}
  (T,\DF T)\colon \mathcal{S}_0 \times \mathcal{S}_1^\mu
              \to \mathcal{S}_0 \times \mathcal{S}_1^\mu,\\
  \begin{aligned}
  (\Theta,\DF \Theta) \mapsto
& \Big((x_0,\dx_0) \mapsto \big(T(\Theta(x_0),x_0),\\
&\quad\big[ \DF_y     T(\Theta(x_0),x_0)\cdot\DF \Theta(x_0)
           +\DF_{x_0} T(\Theta(x_0),x_0) \big]\cdot\dx_0 \big)\Big).
  \end{aligned}
  \end{gathered}
\end{equation}
As such, it is again a uniform fiber contraction since the contraction
was uniform in $y$ and $x_0$ to begin with, and the supremum norm does
not affect the contraction factor $\contr < 1$.

\subsection{Formal tangent bundles}\label{sec:formal-tangent}
\index{formal tangent bundle}

Derivatives of the maps $T_\sx,\,T_\sy$ should be defined between
tangent bundles of the spaces $\BY$ and $\BX$. This is problematic for
both spaces: $\BY$ is a subspace of the Banach space
$B^\rho(I;Y)$, but it has empty interior. The restriction to
$(\Xsize,T)$\ndash approximate solutions of $v_\sx \circ g$ creates a
similar problem for $\BX$, but here, construction of the
tangent bundle faces an additional obstruction. There is no clear way
to define local coordinates around a solution curve $x$. The obvious
method would be by constructing a tubular neighborhood of $x$ and
represent nearby curves $\tilde{x}$ in the tubular neighborhood. But
the metric on $\BX$ allows $\tilde{x}$ to
diverge exponentially from $x$ even if $d_\rho(x,\tilde{x})$ is
small. Thus the tubular neighborhood would need to be of infinite size
to contain $\tilde{x}(t)$ for all $t\in I$, which is generally not
possible. Since any finite-size tubular neighborhood does not contain
a full neighborhood of the curve $x$, we cannot use local coordinates
to define tangent spaces.

Instead, we shall construct formal tangent bundles. These are just
convenient spaces to model variations of curves on; they are natural
extensions of true Banach tangent spaces, see
Section~\ref{sec:banach-mflds}. The primary role of
these bundles is to introduce a topology that allows us to show that
$\DF T$ is uniformly or H\"older continuous.

A formal tangent bundle of $\BY$ can be constructed rather easily:
$B^\rho(I;Y)$ is a Banach space, so its tangent bundle is canonically
identified as $\T B^\rho(I;Y) = B^\rho(I;Y) \times B^\rho(I;Y)$. We
then define the formal tangent bundle of $\BY$ by restricting the
base:
\begin{equation}\label{eq:TBY-formal}
  \TF \BY = \T B^\rho(I;Y)|_{\BY} \cong \BY \times B^\rho(I;Y)
\end{equation}
with induced topology and norm.
\oldindex{$T_zBXY$@$\TF \BY,\,\TF \BX$}

To define a formal tangent bundle of $\BX$,
we consider variations $\dx$ of a curve $x$ as sections of a
pullback bundle: $\dx \in \Gamma(x^*(\T X))$. That is,
$\dx \in C(I;\T X)$ is such that $\dx(t) \in \T_{x(t)} X$ for
each $t \in I$. We equip this space with the norm that is natural for
our problem, namely
\begin{equation*}
  \norm{\dx}_\rho = \sup_{t \in I}\; \norm{\dx(t)}\,e^{-\rho\,t},
\end{equation*}
and denote it by
$B^\rho(I;x^*(\T X)) = \big(\Gamma(x^*(\T X)),\norm{\slot}_\rho\big)$.
The curves $\dx \in B^\rho(I;x^*(\T X))$ form the formal tangent space
over one curve $x \in \BX$. The complete formal tangent bundle is then
defined as the coproduct over all curves $x \in \BX$,
\begin{equation}\label{eq:TBX-formal}
  \TF \BX
  \, = \!\!\coprod_{x \in \BX}\! B^\rho(I;x^*(\T X)).
\end{equation}

\index{formal tangent bundle!topology}
A curve $\dx$ lives above a specific base curve $x$, so there is
no direct way of comparing two curves $\dx_1,\,\dx_2$ with
different base curves $x_1,\,x_2$; \ref{eq:TBX-formal} was constructed
as a coproduct without topological structure. We add a topology based
on parallel transport. This requires the base curves $x \in \BX$ to be
differentiable, so we consider the bundle
\begin{equation}\label{eq:TBXc-formal}
  \TFBXc
\end{equation}
restricted to differentiable\footnote{%
  This does not cause problems since $T_\sx$ actually maps into curves
  $x \in C^1$. The fiber contraction theorem only requires that the
  base space has a globally attractive fixed point. Since the fixed
  point is a $C^1$ curve, we can simply restrict to this subset of
  curves.%
} base curves $x \in \BX \cap C^1$. Variational curves
$\dx \in B^\rho(I;x^*(\T X))$ are isometrically mapped onto curves
$\dxtilde \in B^\rho(I;\T_{x(0)} X) = \big(C^0(I;\T_{x(0)} X),\norm{\slot}_\rho\big)$ by
\begin{equation}\label{eq:TBX-triv-partrans}
  \tilde{\partrans}_x\colon \dx \mapsto \dxtilde, \qquad
  \dxtilde(t) = \partrans(x|_0^t)^{-1}\,\dx(t).
\end{equation}
\oldindex{$\Pi$@$\tilde{\partrans}$}%
Let the normal coordinate radius $\rx$ be $X$\ndash small
as in Definition~\ref{def:M-small}. If we now restrict all
base curves $x$ under consideration to a small neighborhood
\begin{equation}\label{eq:TBX-neighborhood}
  U_\ux = \set[\big]{x \in \BX \cap C^1}{x(0) \in B(\ux;\rx)},
\end{equation}
\oldindex{$U_\ux$}%
i.e., the curves $x$ that start in the open ball
$B(\ux;\rx) \subset X$, then there exists a unique shortest
geodesic $\gamma_{\ux,x(0)}$ from $x(0)$ to $\ux$ for each
$x \in U_\ux$. Parallel transport along these geodesics induces a
local trivialization\footnote{%
  We make the specific choice to trivialize $\T B(\ux;\rx)$ by
  parallel transport along geodesics. Any other trivialization with
  uniformly bounded transition maps would also suffice for our
  purposes and induce a trivialization of $\TF \BX|_{U_\ux}$ (see also
  the alternative viewpoint on this trivialization below).
  This explicit choice is somewhat natural in
  this context, though, and it shows that a trivialization with these
  properties does exist.%
} of $\T B(\ux;\rx)$. This in turn induces a local trivialization of
$\TFBXc$:
\begin{equation}\label{eq:TBX-triv}
  \begin{aligned}
    \xymatrix@R=4em@C=0em{
      *+[l]{\TF \BX|_{U_\ux}} \ar[d]_\pi \ar[r]^-{\tau_\ux} &
      *+[r]{U_\ux \times B^\rho(I;\T_\ux X)}  \ar[dl]^{p_1} \\
      *+[l]{\BX|_{C^1} \supset U_\ux}
    }
  \end{aligned}
\end{equation}
The trivialization map is given by
$\tau_\ux(x,\dx)=\big(x,\partrans(\gamma_{\ux,x(0)})\circ\tilde{\partrans}_x(\dx)\big)$.
The transition maps between overlapping local trivializations
$U_{\ux_1} \cap U_{\ux_2} \neq \emptyset$ are induced by transition
functions
\begin{equation*}
  \phi_{2,1}
  \colon \big(B(\ux_2;\rx) \cap B(\ux_1;\rx)\big) \times \T_{\ux_2} X \to \T_{\ux_1} X
  \colon (\xi,\nu) \mapsto \partrans(\gamma_{\ux_2,\xi}\circ\gamma_{\xi,\ux_1}) \cdot \nu
\end{equation*}
between local trivializations of $\T B(\ux_i;\rx)$. The map
$\phi_{2,1}$ is uniformly Lipschitz by
Lemma~\ref{lem:bound-coordtrans} and linear in the fiber. This induces
a Lipschitz continuous transition function
$\tau_{\ux_2}^{}\circ\tau_{\ux_1}^{-1}$ that depends on the base curve
$x \in U_{\ux_1} \cap U_{\ux_2}$ only through $x(0) \in X$; this
dependence is uniform since $X$ has bounded geometry. Thus the bundle
satisfies Definition~\ref{def:BG-bundle}, and the order of bounded
geometry is actually equal to $k-2$ when $X$ has $k$\th order bounded
geometry.

We endow the bundle $\TFBXc$ with the topology induced by
these local trivializations. Note that this topology is induced by a
locally defined distance function, so we can express uniform and
H\"older continuity of maps on $\TFBXc$. That is, if
$(x_1,\dx_1)$ and $(x_2,\dx_2)$ are elements of $\TFBXc$ such
that $d_\rho(x_1,x_2) < \rx$, then the topology is induced by the
locally defined distance function
\begin{equation}\label{eq:TBX-loc-dist}
  d\big((x_1,\dx_1),(x_2,\dx_2)\big) =
  d_\rho(x_2,x_1) +
  \norm{\partrans(\gamma_{x_1(0),x_2(0)})\,
        \tilde{\partrans}_{x_2}(\dx_2)
       -\tilde{\partrans}_{x_1}(\dx_1)}_\rho.
\end{equation}
The transition functions $\tau_{\ux_2}^{}\circ\tau_{\ux_1}^{-1}$ are
uniformly Lipschitz, so they preserve uniform and H\"older continuity
moduli up to a constant. Therefore, overlapping trivializations define
the same topology on their intersection, with compatible local
distances. To summarize, we have
\begin{proposition}\label{prop:formal-bundles}
  The spaces $\TF \BY$ and\, $\TFBXc$ are well-defined normed
  vector bundles of bounded geometry, and they have a (local) distance
  structure.
\end{proposition}

The topologies introduced above
allow us to express uniform and H\"older continuity of the
maps~\ref{eq:DT-set}. The topology on $\TF \BY$ is clear and explicit
from the topology on $B^\rho(I;Y)$. For $\TFBXc$ let
$x \in \BX \cap C^1$ be a curve and $\dx \in B^\rho(I;x^*(\T X))$ a
variational curve at $x$. The topology is induced by the isometric
representation
\begin{equation*}
  \dxtilde = \partrans(\gamma_{\ux,x(0)})\cdot\tilde{\partrans}_x\cdot\dx
  \in B^\rho(I;\T_\ux X)
\end{equation*}
of $\dx$. Uniform continuity of maps~\ref{eq:DT-set} that have
$\TFBXc$ as (co)domain can thus be checked by switching to a
local trivialization, that is, substitute
\begin{equation*}
  \dx(t) = \partrans(x|_0^t)\cdot\partrans(\gamma_{x(0),\ux})\cdot\dxtilde(t)
\end{equation*}
and then use the known topology on $U_\ux \times B^\rho(I;\T_\ux X)$.
In explicit calculations of continuity with respect to the base
$\BX \cap C^1$, we shall thus add parallel transport terms such as
those above to the maps~\ref{eq:DT-set} and let these act on
$\dxtilde \in B^\rho(I;\T_\ux X)$.

\subsubsection{Alternative viewpoints}

Instead of the immediate trivialization~\ref{eq:TBX-triv} of the
bundle $\TFBXc$, we can also introduce an intermediate viewpoint that
corresponds to only applying the parallel transport term
$\tilde{\partrans}_x$, but not $\partrans(\gamma_{\ux,x(0)})$ in the
local neighborhood $B(\ux;\rx)$. We view $\text{ev}_0\colon \BX \to X$
as a bundle; this identifies $\TFBXc$ as a bundle over $X$ as well,
via $\text{ev}_0 \circ \pi$. Let $B^\rho(I;\T X)_X$ denote the space
of (continuous, exponential growth) functions
$\dxtilde\colon I \to \T X$ such that $\pi \circ \dxtilde$ is constant
into $X$, viewed as a bundle over $X$.
\begin{equation}\label{eq:TBX-over-TX}
  \begin{aligned}
    \xymatrix@R=3em@C=-2em{
      *+[l]{\TFBXc} \ar[d]_\pi \ar[r]^-{\tilde{\partrans}} &
      *+[r]{\BX|_{C^1} \times_\sx B^\rho(I;\T X)_\sx} \ar[dl]^{p_1} &
      {\mathstrut\hspace{2.2cm}} \ar[dl]^{\text{ev}_t \circ p_2} \\
      \BX|_{C^1} \ar[d]_{\text{ev}_0} &
      \T X       \ar[dl]^{\pi} \\
      X
    }
  \end{aligned}
\end{equation}
This commutative diagram shows that $\TFBXc$ can be identified via
$\tilde{\partrans}$ with the fiber product bundle
$\BX|_{C^1} \times_\sx B^\rho(I;\T X)_\sx$ over $X$. This
identification is natural in the sense that no local trivialization of
$X$ or $\T X$ is used. The second component $B^\rho(I;\T X)_\sx$ of
this bundle contains the variational curves $\dxtilde$. This is a
(nontrivial) bundle over $X$, but its projection onto the base
$\pi \circ\text{ev}_t\colon B^\rho(I;\T X)_\sx \to X$ factors through
$\T X$. The fact that $\pi\circ\text{ev}_t$ is constant for $t \in I$
simply expresses that each $\dxtilde \in B^\rho(I;\T X)_\sx$ maps
into a fixed tangent space $\T_\xi X$. This shows that a local trivialization
$\sigma\colon \T X|_{B(\ux;\rx)} \to B(\ux;\rx) \times \R^n$ naturally
lifts to a local trivialization
\begin{equation*}
  \tilde{\sigma}\colon B^\rho(I;\T X)|_{B(\ux;\rx)}
                  \to  B(\ux;\rx) \times B^\rho(I;\R^n).
\end{equation*}
We have chosen local trivializations of $\T X$ by parallel transport
along geodesics, i.e.\ws $\partrans(\gamma_{\ux,\xi})$, since this
construction is compatible with the bounded geometry of $X$ in the
sense that trivialization chart transitions are $\BC^k$ maps by
Proposition~\ref{prop:unif-partrans}.

We also introduce a reformulation of the topology on $\TFBXc$ using
frames, as an alternative to
the explicit formulation in terms of parallel transport above. This
allows us to abstract away these ideas into a lighter notation in the
next section and only recall the full details when required.

Let $e_\ux\colon \R^n \to \T_\ux X$ be a choice\footnote{%
  The precise choice does not matter and will drop out in the final,
  relevant equations. The relative choice of frame along curves is
  what matters.%
} of orthonormal frame
at $\ux \in X$. We can extend this to an orthonormal frame $e$ on
$\T B(\ux;\rx)$ by parallel transport of the frame $e_\ux$
\index{parallel transport!of a frame}
along geodesics emanating from $\ux$. As a second step, we further
extend the frame $e$ along any curve $x \in U_\ux$, again by parallel
transport\footnote{%
  Note that $e$ does not define a (global) frame on $\T X$. The choice
  of frame at $x(t)$ depends not just on the point $x(t) \in X$, but
  on the whole curve $x \in \BX|_{C^1}$. Another curve $\tilde{x}$
  with $x(t) = \tilde{x}(t)$ will generally induce a different frame
  in $\T_{x(t)} X$.%
}.

We adopt the notation $v_f = f^{-1}\cdot v$ to express a vector
$v \in \T_x X$ with respect to a frame $f$ at $x$, and use this
notation more generally on the tensor bundle of $X$. Now let $v$ be a
vector field and $\omega$ a one-form on $X$, then the construction of
$e$ above leads to
\begin{equation}\label{eq:TBX-frame}
  \begin{aligned}
  v(x(t))_e      &= e_\ux^{-1}\cdot\partrans(\gamma_{\ux,x(0)})
                              \cdot\partrans(x|_t^0)\cdot v(x(t)),\\[5pt]
  \omega(x(t))_e &= \omega(x(t))\cdot\partrans(x|_0^t)
                                \cdot\partrans(\gamma_{x(0),\ux})\cdot e_\ux,
  \end{aligned}
\end{equation}
and naturally extends to the tensor bundle of $X$.

\subsection{Continuity of the fiber maps}\label{sec:cont-fiber-maps}

We prove the uniform and H\"older continuous dependence on $x,y$, and
$x_0$ of the maps~\ref{eq:DT-set} using a combination of techniques.
One is the variation of constants formula to get expressions for the
variation of flows when changing a parameter. Such variations require
us to compare the variational curves over different base curves; for
this, we use the topologies of the formal tangent bundles in
Section~\ref{sec:formal-tangent}, while we measure the variation of
vector fields with the formulation of continuity via
parallel transport in Proposition~\ref{prop:unif-partrans}. Together
these lead to \idx{holonomy} terms along the base paths (see
Figure~\ref{fig:BX-holonomy}), in addition to the variation of
constants terms that would simply occur in $\R^n$. These holonomy terms can
be estimated with Lemma~\ref{lem:bound-holonomy} and do not
essentially alter the estimates.
\begin{figure}[htb]
  \centering
  \input{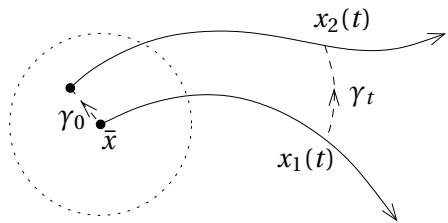}
  \caption{Paths involved in the holonomy term.}
  \label{fig:BX-holonomy}
\end{figure}

We use Nemytskii operator techniques as laid out in
Appendix~\ref{chap:nemytskii} to conclude that functions such as $A$
and $f$ can be interpreted as uniformly continuous maps onto curves
with some $\mu < 0$ exponential growth norm. Instead of uniform continuity, we
can also obtain H\"older continuity if the original maps are H\"older
continuous and if we view the Nemytskii operator as a mapping into a
space with norm $\norm{\slot}_{\alpha\rho}$. In other words, we replace the
uniform continuity modulus by the explicit $\alpha$\ndash H\"older
continuity modulus. H\"older continuity precisely fits the problem, so
in that case there is no need anymore to add a small $\mu < 0$ to the
exponential growth norms.

\subsubsection{One example in full detail}

As an example, let us consider continuity of the map~\ref{eq:DxTy}
with respect to $x \in \BX$, that is,
$x \mapsto \DF_x T_\sy(x,y)$. To be able to explicitly use the topology
on $\TFBXc$, we switch to a local trivialization neighborhood
$U_\ux \ni x_1, x_2$ as in~\ref{eq:TBX-neighborhood}. We choose
$\ux = x_1(0)$ to simplify expressions; any other choice for
$\ux$ can be obtained by a transition of trivialization charts. Let
$\dxtilde \in B^\rho(I;\T_\ux X)$ be the representation of an
arbitrary variational curve in the fiber of this trivialization.

Note that~\ref{eq:DxTy} is defined in terms of~\ref{eq:DxPsi}. We
estimate continuity of the separate components and build towards the
full expression. Let us first focus on the continuity of
$x \mapsto \Psi_x(t,\tau)$, which is a map $\BX \to \CLin(Y)$ for fixed
$t,\,\tau \in I$.
\begin{proposition}
  \label{prop:Psi-cont}
  For any $\mu < 0$, the variation
  \begin{equation}\label{eq:Psi-var}
    \Upsilon^{t,\tau} = \Psi_{x_2}(t,\tau) - \Psi_{x_1}(t,\tau)
  \end{equation}
  of the linear flow $\Psi_x$ on $Y$ satisfies continuity
  estimate~\ref{eq:Psi-var-est}.
\end{proposition}

\begin{proof}
We extend the ideas from the proof of Lemma~\ref{lem:exp-growth-riem}.
The variation $\Upsilon^{t,\tau}$ satisfies the differential equation
\begin{align*}
  \der{}{t}\Upsilon^{t,\tau}
  &= A(x_2(t))\,\Psi_{x_2}(t,\tau)
    -A(x_1(t))\,\Psi_{x_1}(t,\tau)\\
  &= A(x_2(t))\,\Upsilon^{t,\tau}
    +\big[A(x_2(t))-A(x_1(t))\big]\Psi_{x_1}(t,\tau),
\end{align*}
which leads to a variation of constants integral that can be estimated
as
\begin{align*}
  \norm{\Upsilon^{t,\tau}}
  &\le \int_\tau^t \norm{\Psi_{x_2}(t,\sigma)}
                   \norm{A(x_2(\sigma))-A(x_1(\sigma))}
                   \norm{\Psi_{x_1}(\sigma,\tau)} \d\sigma\\
  &\le \int_\tau^t C_\sy\,e^{\rho_\sy(t-\sigma)}\,
                   \epsilon_{\tilde{A}}(d_\rho(x_2,x_1))\,e^{\mu\,\sigma}\,
                   C_\sy\,e^{\rho_\sy(\sigma-\tau)} \d\sigma\\
  &\le C_\sy^2\,e^{\rho_\sy(t-\tau)}\,\epsilon_{\tilde{A}}(d_\rho(x_2,x_1))
       \frac{e^{\mu\,\tau}}{-\mu}.
  \eqnumber\label{eq:Psi-var-est}
\end{align*}
Here we use ideas from Appendix~\ref{chap:nemytskii}; we applied
Corollary~\ref{cor:unifcont-nemytskii} to obtain $A$
as a uniformly continuous fiber mapping $\BX \to B^\mu(I;\CLin(Y))$
with continuity modulus~$\epsilon_{\tilde{A}}$ (that depends on $\mu$).
\end{proof}
Thus, the flow $\Psi_x^{t,\tau}$ depends uniformly continuously on
$x \in \BX$ when viewed as a flow with $\rho_\sy$\ndash exponential
growth \emph{and} measured with an additional exponential factor
$e^{\mu\,\tau}$.

\begin{remark}\label{rem:Psi-var-Holder}
  In the previous proposition, if $A$ is $\alpha$\ndash H\"older
  continuous, then we can replace $\mu$ by $\alpha\,\rho$ to obtain a
  similar, $\alpha$\ndash H\"older continuous result using
  Lemma~\ref{lem:cont-nemytskii}.
\end{remark}

To show that $x \mapsto \DF_x\Psi_x(t,\tau)$ is continuous as well, we
first write down the corresponding variation in the bundle
trivialization chart:
\begin{align*}
\fst\big(\DF_x\Psi_{x_2}\cdot\dxtilde
        -\DF_x\Psi_{x_1}\cdot\dxtilde\big)(t,\tau)\\
\con= \int_\tau^t \Psi_{x_2}(t,\sigma)\,
        \big[\D A(x_2(\sigma))\,\partrans(x_2|_0^\sigma)\,
             \partrans(\gamma_{x_2(0),\ux})\,\dxtilde(\sigma)\big]\,
                 \Psi_{x_2}(t,\sigma) - (2 \rightsquigarrow 1) \d\sigma\\
\con= \int_\tau^t \Psi_{x_2}(t,\sigma)\,
        \big[\D A(x_2(\sigma))_e\,\dxtilde(\sigma)\big]\,
                 \Psi_{x_2}(t,\sigma) - (2 \rightsquigarrow 1) \d\sigma,
  \eqnumber\label{eq:DxPsi-var}
\end{align*}
where the notation $(2 \rightsquigarrow 1)$ means that we take the
first expression and replace all $2$'s by $1$'s (note that
$\gamma_{x_2(0),\ux} = \gamma_0$ in the first term and
$\partrans(\gamma_{x_1(0),\ux}) = \Id$ in the second term). The last
line is just a rewrite in terms of the frame as in~\ref{eq:TBX-frame}
and suppresses all parallel transport terms. We separately estimate
continuity of the three factors in the integrand, and insert the
estimate of Proposition~\ref{prop:Psi-cont} for the variation
$\Psi_\bullet$ in the first and third factor. Note that $\dxtilde$ is
the same over both curves $x_1$ and $x_2$ in this trivialization.

For the middle factor $\D A(x(t))_e$, we again apply
Nemytskii operator techniques from Appendix~\ref{chap:nemytskii}. But
in this case we have to combine these with holonomy terms, due to the
fact that comparison of $\D A$ at nearby points $\xi_2,\,\xi_1 \in X$
only makes sense after identification of the tangent spaces
$\T_{\xi_2} X$ and $\T_{\xi_1} X$.
\begin{proposition}
  \label{prop:DA-cont}
  Let $A \in \BUC^1$ according to
  Definition~\ref{def:unif-bounded-map}. Then for any $\mu < 0$, the
  map
  \begin{equation}\label{eq:DA-nemytskii}
    x \mapsto \big(t \mapsto \D A(x(t))_e\big)\colon
    U_\ux \subset \BX \to B^\mu\big(I;\CLin(\T_\ux X;\CLin(Y))\big)
  \end{equation}
  is uniformly continuous. If moreover $A \in \BUC^{1,\alpha}$, then
  the map~\ref{eq:DA-nemytskii} is $\alpha$\ndash H\"older with $\mu$
  replaced by $\alpha\,\rho$.
\end{proposition}

\begin{proof}
Let $x_1,\,x_2 \in U_\ux$. We introduce another frame $f$ to directly
compare $\D A$ at points $x_1(t),\,x_2(t)$. Let
$f_{x_1(t)} = e_{x_1(t)}\colon \R^n \to \T_{x_1(t)} X$ and define
$f_{x_2(t)} = \partrans(\gamma_t) \cdot f_{x_1(t)}$. Thus, the frames
$e$ and $f$ at $x_2(t)$ are both defined in terms of the frame
$e_\ux$; $e_{x_2(t)}$ by parallel transport along $x_2 \circ \gamma_0$
and $f_{x_2(t)}$ by parallel transport along $\gamma_t \circ x_1$, see
Figure~\ref{fig:BX-holonomy}. Since $f_{x_1(t)} = e_{x_1(t)}$, we can
rewrite the difference of~\ref{eq:DA-nemytskii} at points on these
curves as
\begin{equation*}
         \D A(x_2(t))_e - \D A(x_1(t))_e
  = \big[\D A(x_2(t))_e - \D A(x_2(t))_f\big]
   +\big[\D A(x_2(t))_f - \D A(x_1(t))_f\big].
\end{equation*}
The first term can be estimated by the holonomy defect along the loop
\begin{equation*}
  \gamma_0^{-1} \circ x_2|_t^0 \circ \gamma_t \circ x_1|_0^t
\end{equation*}
using Lemma~\ref{lem:bound-holonomy} and the second term using the
continuity of $\D A$ and Proposition~\ref{prop:unif-partrans}.
Together, this leads to
\begin{align*}
\fst \norm{\D A(x_2(t))_e - \D A(x_1(t))_e}\\
  &\le \norm{\D A}\,
       \norm[\big]{\Id - \partrans\big(\gamma_0^{-1} \circ x_2|_t^0
                                 \circ \gamma_t      \circ x_1|_0^t\big)}
      +\epsilon_{\D A}\big(d(x_2(t),x_1(t))\big)\\
  &\le C_v\,C\,d_\rho(x_2,x_1)\,e^{\rho\,t}
      +\epsilon_{\D A}\big(d(x_2(t),x_1(t))\big).
\end{align*}
If $d_\rho(x_2,x_1)\,e^{\rho\,t} \ge \rx$, then we use the boundedness
estimate $\norm{\Id - \partrans(\gamma)} \le 2$ for any closed loop
$\gamma$ and Remark~\ref{rem:loc-cont-modulus} to effectively extend the local to
a global continuity modulus. We can recover any $\alpha$\ndash
H\"older continuity from the Lipschitz holonomy estimate, again by
using the fact that the holonomy is bounded by $2$ in combination with
Lemma~\ref{lem:holder-lower}.

With the same arguments as in Lemma~\ref{lem:cont-nemytskii}, it
follows that $x \mapsto \big(t \mapsto \D A(x(t))_e\big)$ is uniformly
or $\alpha$\ndash H\"older continuous, and we denote its continuity
modulus by $\epsilon_{\widetilde{\D A}}$. Note that
$\epsilon_{\widetilde{\D A}}$ does not depend on the trivialization
chart since all estimates are uniform with respect to these charts.
\end{proof}

\begin{proposition}
  \label{prop:DxPsi-cont}
  For any $\mu < 0$ and uniformly in $\ux \in X$, the map
  $x \mapsto \DF_x \Psi_x(t,\tau)$ satisfies continuity
  estimate~\ref{eq:DxPsi-var-est} in a trivialization neighborhood\/
  $U_\ux \subset \BX$.
\end{proposition}

\begin{proof}
We combine the estimates from Propositions~\ref{prop:Psi-cont}
and~\ref{prop:DA-cont} and obtain for~\ref{eq:DxPsi-var}
\begin{align*}
\fst \norm*{\big(\DF_x \Psi_{x_2}\cdot\dxtilde
                -\DF_x \Psi_{x_1}\cdot\dxtilde\big)(t,\tau)}\\
  &\le \int_\tau^t \norm[\big]{
         \Psi_{x_2}(t,\sigma)\,
         \big[\D A(x_2(\sigma))_e\,\dxtilde(\sigma)\big]\,
         \Psi_{x_2}(\sigma,\tau)
        -\, (2 \rightsquigarrow 1) } \d\sigma\\
  &\le \int_\tau^t \Big(
                \norm[\big]{\Psi_{x_2}(t,\sigma) - \Psi_{x_1}(t,\sigma)}\,
                \norm{\D A(x_2(\sigma))_e}\,\norm{\Psi_{x_2}(\sigma,\tau)}\\
  &\hspace{1cm}+\norm{\Psi_{x_1}(t,\sigma)}\,
                \norm[\big]{\D A(x_2(\sigma))_e - \D A(x_1(\sigma))_e}\,
                \norm{\Psi_{x_2}(\sigma,\tau)}\\
  &\hspace{1cm}+\norm{\Psi_{x_1}(t,\sigma)}\,\norm{\D A(x_1(\sigma))_e}\,
                \norm[\big]{\Psi_{x_2}(\sigma,\tau) - \Psi_{x_1}(\sigma,\tau)}
                   \Big)\norm{\dxtilde}_\rho\,e^{\rho\,\sigma} \d\sigma\displaybreak[1]\\
  &\le \int_\tau^t \Big(
                C_\sy^2\,e^{\rho_\sy(t-\sigma)}\,\epsilon_A(d_\rho(x_2,x_1))\,
                \frac{e^{\mu\,\sigma}}{-\mu}\,C_v\,C_\sy\,e^{\rho_\sy(\sigma-\tau)}\\
  &\hspace{1cm}+C_\sy\,e^{\rho_\sy(t-\sigma)}\,
                \epsilon_{\widetilde{\D A}}(d_\rho(x_2,x_1))\,
                e^{\mu\,\sigma}\,C_\sy\,e^{\rho_\sy(\sigma-\tau)}\\
  &\hspace{1cm}+C_\sy\,e^{\rho_\sy(t-\sigma)}\,C_v\,
                C_\sy^2\,e^{\rho_\sy(\sigma-\tau)}\,
                \epsilon_A(d_\rho(x_2,x_1))\,\frac{e^{\mu\,\tau}}{-\mu}
                   \Big)\,\norm{\dxtilde}_\rho\,e^{\rho\,\sigma} \d\sigma\displaybreak[1]\\
  &\le C\,e^{\rho_\sy(t-\tau)}\,\epsilon(d_\rho(x_2,x_1))\,
          e^{(\rho+\mu)\tau}\,\norm{\dxtilde}_\rho.
   \eqnumber\label{eq:DxPsi-var-est}
\end{align*}
We absorbed all constants and integration factors such as
$\frac{1}{-\mu}$ into the general constant $C$ and combine the
continuity moduli into one; $C_v$ is a global bound on all vector
fields including $A$ and its derivatives.
\end{proof}

We finally plug estimate~\ref{eq:DxPsi-var-est} into
equation~\ref{eq:DxTy}. We repeat the Nemytskii and holonomy arguments
for $\Ynonlin$ and $\D_x \Ynonlin$ (just as for $A$ and $\D A$) to
obtain a uniform continuity estimate for
\begin{equation*}
  x \mapsto \DF_x T_\sy(x,y)\colon
  U_\ux \to \CLin\big(B^\rho(I;\T_\ux X);B^{\rho+\mu}(I;Y)\big).
\end{equation*}
both for any $\mu < 0$, or with $\mu = \alpha\,\rho$ when
$A,\,f \in \BUC^{1,\alpha}$. That is, $\DF_x \T_\sy(\slot,y)$ is a map
that given a curve $x \in \BX|_{C^1}$, linearly maps a variational
curve $\dx$ over $x$ to a variational curve $\dy$ in the trivial
bundle $\TF \BY$. We can formulate this more abstractly as
\begin{equation*}
  \DF_x \T_\sy(\slot,y) \in
  \Gamma_{b,u}^\alpha\big(\BX|_{C^1};\CLin\big(\TFBXc;B^{\rho+\mu}(I;Y)\big)\big),
\end{equation*}
that is, $\DF_x \T_\sy(\slot,y)$ is a uniformly $\alpha$\ndash
H\"older bounded section of the bounded geometry bundle
\begin{equation*}
  \pi\colon\CLin\big(\TFBXc;B^{\rho+\mu}(I;Y)\big) \to \BX|_{C^1}.
\end{equation*}

\subsubsection{Continuity in the other cases}

We treated the continuity for one of the maps~\ref{eq:DT-set} with
respect to a single variable. The continuity in all other cases can be
shown in a similar fashion. Many arguments can be repeated, but each
of these maps also has its own peculiar details which makes that I
have not been able to find one general, abstract way to prove
continuity of all of these maps at once. In this section we shall
focus on these specific details and not repeat the recurring elements.
Let me reiterate that the uniform continuity results hold for any
$\mu < 0$ sufficiently small, and these can be replaced by
$\alpha$\ndash H\"older continuity when $\mu$ is replaced by
$\alpha\,\rho$ and the spectral gap condition~\ref{eq:spectral-gap} is
satisfied for $r = 1 + \alpha$.

First of all, note that continuity with respect to the combined
variables follows directly from continuity with respect to each
separate variable since we have explicit uniform or H\"older
continuity moduli. If $f(x,y)$ has continuity moduli
$\epsilon_x,\,\epsilon_y$ with respect to $x,\,y$, respectively, then
\begin{equation*}
  \begin{aligned}
       \norm{f(x_2,y_2) - f(x_1,y_1)}
  &\le \norm{f(x_2,y_2) - f(x_1,y_2)}
      +\norm{f(x_1,y_2) - f(x_1,y_1)}\\
  &\le \epsilon_x(d(x_2,x_1)) + \epsilon_y(d(y_2,y_1))\\
  &\le (\epsilon_x+\epsilon_y)\big(d((x_2,y_2),(x_1,y_1))\big)
  \end{aligned}
\end{equation*}
shows that $\epsilon_x + \epsilon_y$ is a continuity modulus for $f$.
We assumed w.l.o.g.\ that $\epsilon_x,\,\epsilon_y$ are
non-decreasing, while all choices of distance on the product space are
equivalent, so we leave it unspecified.

Let us start with the easy cases. Continuity of the map~\ref{eq:DxTy}
as a function of $y$, that is,
\begin{equation*}
  y \mapsto \DF_x T_\sy(x,y)\colon
  \BY \to \CLin\big(B^\rho(I;x^*(\T X));B^{\rho+\mu}(I;Y)\big),
\end{equation*}
requires no additional details: only $f$ and $\D_x f$ depend on
$y \in Y$, and we can reapply the arguments above to show that
these depend continuously on~$y \in \BY$. No holonomy terms are
present since $\TF \BY$ is a trivial bundle. That is, we can directly
compare $\DF_x T_\sy$ at different $y_1,\,y_2 \in \BY$; keeping
$x \in \BX$ fixed means that everything is situated in the fixed fiber
$\TF_x \BX|_{C^1} = B^\rho(I;x^*(\T X))$ and no holonomy terms are
required.

Continuity of the map~\ref{eq:DyTy}, i.e.\ws $\DF_y T_\sy(x,y)$, both
with respect to $x$ and $y$
follows along the same lines. Neither case requires holonomy
arguments; we just apply the Nemytskii technique to $\D_y f$ and reuse
Proposition~\ref{prop:Psi-cont} to show continuity with respect to
$x$.

The formal derivatives~\ref{eq:DpTx} and~\ref{eq:DyTx} of $T_\sx$ map
into $\TFBXc$; here we have to apply holonomy arguments in the
codomain. Let us first focus on
\begin{equation*}
  x_0 \mapsto \DF_{x_0} T_\sx(y,x_0)\colon
  X \to \CLin\big(\T_{x_0} X;B^{\rho+\mu}(I;\T_\ux X)\big)
\end{equation*}
with a local trivialization $U_\ux \times B^{\rho+\mu}(I;\T_\ux X)$
within\footnote{%
  Embeddings $B^\rho \embedto B^{\rho+\mu}$ are continuous, so we can
  view $U_\ux \times B^{\rho+\mu}(I;\T_\ux X)$ as a local
  trivialization of a subset of $\TF \mathcal{B}^{\rho'}_\Xsize(I;X)$
  with $\rho' = \rho + \mu$.%
} the bundle $\TF \mathcal{B}^{\rho+\mu}_\Xsize(I;X)\big|_{C^1}$ with
additional $\mu$ in the exponential growth norm on the fibers. Note
that $\DF_{x_0} T_\sx(y,\slot)$ could actually be considered as a
bundle map on the vector bundle $\T X$ that is linear on each tangent
space $\T_{x_0} X$. We consider a local trivialization of
$\T B(\ux;\rx) \subset \T X$ by parallel transport along
geodesics: this is equivalent to trivialization by a normal coordinate
chart for the purpose of measuring continuity, while it matches the
trivialization of $\TF \BX|_{U_\ux}$. This will lead to a holonomy
term.

For any $x_0 \in B(\ux;\rx)$ we have $T_\sx(y,x_0) \in U_\ux$ by
construction, so let $e$ denote the frame introduced by the
trivialization of $\TF \BX|_{U_\ux}$, i.e.\ws by parallel transport
along solution curves $T_\sx(y,x_0)$. On the other hand, let $f$
denote a frame introduced by local parallel transport. We define
\begin{equation*}
  f_{x_1(t)} = e_{x_1(t)}
  \quad\text{and}\quad
  f_{x_2(t)} = \partrans(\gamma_t) \cdot f_{x_1(t)}.
\end{equation*}
It follows from Lemma~\ref{lem:exp-growth-holder} that
$\DF_{x_0} T_\sx(y,\slot) = \big(t \mapsto \D\Phi_y(t,0,\slot)\big)$
satisfies the correct type of continuity estimates, but with respect
to local charts (or equivalently, with respect to $f$ determined by
local parallel transport) instead of the choice of frame $e$, defined
by the topology of $\TFBXc$. To examine the difference, let
$x_{0,1},\,x_{0,2} \in B(\ux;\rx)$ denote two initial conditions
and $x_i = T_\sx(y,x_{0,i}),\,i=1,2$, their respective solution curves
for a fixed $y \in \BY$. We also fix $x_{0,1} = \ux$ for convenience.
Then we have
\begin{align*}
\fst      \D\Phi_y(t,0,x_{0,2})_e - \D\Phi_y(t,0,x_{0,1})_e\\
  &= \big[\D\Phi_y(t,0,x_{0,2})_e - \D\Phi_y(t,0,x_{0,2})_f\big]
    +\big[\D\Phi_y(t,0,x_{0,2})_f - \D\Phi_y(t,0,x_{0,1})_f\big]\\
  &= \big[\partrans(\gamma_0^{-1} \circ x_2|_t^0)
         -\partrans(x_1|_t^0 \circ \gamma_t^{-1})\big]\cdot
     \D\Phi_y(t,0,x_{0,2}) \cdot \partrans(\gamma_0)\\
\con+\big[\D\Phi_y(t,0,x_{0,2})_f - \D\Phi_y(t,0,x_{0,1})_f\big].
\end{align*}
This shows that uniform and H\"older continuity with respect to the
topology of $\TFBXc$ is equivalent to the same continuity with respect
to normal coordinate charts, since the additional holonomy term can be
estimated in the same way as in Proposition~\ref{prop:DA-cont}.
Continuity of
\begin{equation*}
  y \mapsto \DF_{x_0} T_\sx(y,x_0)\colon
  \BY \to \CLin\big(\T_{x_0} X;B^{\rho+\mu}(I;\T_\ux X)\big)
\end{equation*}
follows in the same way, if we first apply
Corollary~\ref{cor:exp-growth-holder-param} to obtain the continuity
estimates with respect to the frame $f$.

Finally, we consider continuity of the map~\ref{eq:DyTx},
\begin{equation*}
  \DF_y T_\sx(y,x_0) \in \CLin\big(B^\rho(I;Y);B^{\rho+\mu}(I;\T_\ux X)\big)
\end{equation*}
with respect to $y \in \BY$ and $x_0 \in X$. We apply
Corollary~\ref{cor:exp-growth-holder-param} and
Lemma~\ref{lem:Tx-contr} to conclude that $\D\Phi_y(t,\tau,x_y(\tau))$
depends $\alpha$\ndash H\"older or uniformly continuously on $y$.
Lemma~\ref{lem:Tx-contr} in combination with a Nemytskii operator
argument shows that $\D_y \vx$ induces a uniformly continuous map
\begin{equation*}
  y \mapsto \big( t \mapsto \D_y \vx(x_y(t),y(t)) \big)\colon
  \BY \to B^\mu\big(I;\CLin(Y;\T X)\big)
\end{equation*}
with $\mu$ replaced by $\alpha\,\rho$ in the H\"older case. For
continuity with respect to $x_0 \in X$ we need to replace application of
Corollary~\ref{cor:exp-growth-holder-param} by that of
Lemma~\ref{lem:exp-growth-holder} for dependence of
$x_y = T_\sx(y,x_0)$ on $x_0$. Again the continuity estimates obtained
are with respect to the frame $e$ and we use
Lemma~\ref{lem:bound-holonomy} to estimate the additional holonomy
term when switching to the frame $f$.

\subsection{Application of the fiber contraction theorem}\label{sec:appl-FCT}

In Section~\ref{sec:unif-fiber-contr} we already established that the
fiber mapping $(T,\,\DF T)$ in formula~\ref{eq:fiber-map-DT} is
uniformly contractive. With the results
of the previous sections on formal tangent bundles and continuous
formal derivatives, we can now apply the fiber contraction theorem,
see Appendix~\ref{chap:fibercontr}.

\begin{proposition}
  \label{prop:DTh-unif-attract}
  For any $\mu < 0$, the fiber mapping~\ref{eq:fiber-map-DT} has a
  unique, globally attractive fixed point\/
  $(\Theta^\infty,\,\DF\Theta^\infty) \in \mathcal{S}_0 \times \mathcal{S}_1^\mu$, while
  it also holds that $\DF\Theta^\infty \in \mathcal{S}_1^0$.
\end{proposition}

\begin{proof}
In the notation of Theorem~\ref{thm:fibercontr} we take $X = \mathcal{S}_0$ and
$Y = \mathcal{S}_1^\mu$ as in~\ref{eq:fiber-contr-spaces} with $\rho,\,\mu$ such
that $\rho_\sy < \rho + \mu < \rho < \rho_\sx$ holds. The fiber mapping
is $F = (T,\,\DF T)$, as in~\ref{eq:fiber-map-DT}. The first two
conditions of Theorem~\ref{thm:fibercontr} are satisfied due to the arguments in
Section~\ref{sec:unif-fiber-contr}, while the third condition that
$\DF T$ is continuous can be obtained from the results in
Section~\ref{sec:cont-fiber-maps} as follows.

First, note that $(T,\,\DF T)$ is a well-defined, uniformly
contractive fiber mapping both when acting on $B^\rho(I;Y)$ and on
$B^{\rho+\mu}(I;Y)$ variational curves. Thus, for each $n \ge 0$ we have
$\DF\Theta^n \in \mathcal{S}_1^0 \embedto \mathcal{S}_1^\mu$, where the embedding is
continuous. The same conclusion holds for $\DF\Theta^\infty$ by a
simple uniform contraction argument. Next, we view $\DF T$ as a map
\begin{equation}\label{eq:DT-cont-embed}
  \DF T\colon \mathcal{S}_0 \times \mathcal{S}_1^0 \to \mathcal{S}_1^\mu.
\end{equation}
Note that we set $\mu = 0$ in the domain only.
To obtain continuity of~\ref{eq:DT-cont-embed} with respect to the
base variable $\Theta \in \mathcal{S}_0$, it is sufficient to check that the
maps
\begin{equation}\label{eq:DT-pder-cont}
  \begin{alignedat}{2}
    y &\mapsto \DF_y     T(y,x_0)&&\colon \BY \to \CLin\big(B^\rho(I;Y);B^{\rho+\mu}(I;Y)\big),\\
    y &\mapsto \DF_{x_0} T(y,x_0)&&\colon \BY \to \CLin\big(\T_{x_0} X ;B^{\rho+\mu}(I;Y)\big)
  \end{alignedat}
\end{equation}
are uniformly continuous, uniformly in $x_0 \in X$. Continuity
of~\ref{eq:DT-cont-embed} with respect to the base $\mathcal{S}_0$ (with fixed
fiber part $\DF\Theta \in \mathcal{S}_1^0 \embedto \mathcal{S}_1^\mu$) then follows from the interpretation
of~\ref{eq:DT-pder-cont} as acting on maps $(\Theta,\DF\Theta)$ with
the supremum norm on $\mathcal{S}_0$. The maps~\ref{eq:DT-pder-cont} are defined
by the chain rule formula~\ref{eq:DT-chainrule} in terms of the
derivative maps~\ref{eq:DT-set}. A variation of $y \in \BY$ can be
distributed over the product (we only estimate the variation of
$\DF_y T$ with respect to $y$, but the variation of $\DF_{x_0} T$ is
completely analogous),
\begin{equation}\label{eq:DT-cont-est}
\begin{aligned}\jot=8pt
\fst   \norm{\DF_y T(y_2,x_0) - \DF_y T(y_1,x_0)}_{\rho+\mu,\rho}\\
  &\le \norm[\big]{\DF_x T_\sy(x_{y_2},y_2)\cdot
                   \big[\DF_y T_\sx(y_2,x_0) - \DF_y T_\sx(y_1,x_0)\big]}_{\rho+\mu,\rho}\\
\con  +\norm[\big]{\big[\DF_x T_\sy(x_{y_2},y_2)-\DF_x T_\sy(x_{y_1},y_1)\big]
                   \cdot\DF_y T_\sx(y_1,x_0)}_{\rho+\mu,\rho}\\
  &\le \norm{\DF_x T_\sy(x_{y_2},y_2)}_{\rho+\mu,\rho+\mu}\cdot
       \norm{\DF_y T_\sx(y_2,x_0) - \DF_y T_\sx(y_1,x_0)}_{\rho+\mu,\rho}\\
\con  +\norm{\DF_x T_\sy(x_{y_2},y_2)-\DF_x T_\sy(x_{y_1},y_1)}_{\rho+\mu,\rho}\cdot
       \norm{\DF_y T_\sx(y_1,x_0)}_{\rho,\rho}.
\end{aligned}
\end{equation}
\index{${}\norm{\slot}_{\rho_2,\rho_1}$}%
The $\norm{\slot}_{\rho_2,\rho_1}$ denote operator norms on linear
(bundle) maps from $B^{\rho_1}$ to $B^{\rho_2}$ spaces. In the factor
that is not varied we can simply take the operator norm between
functions of either $\rho$ or $\rho+\mu$ exponential growth: in
Section~\ref{sec:unif-fiber-contr} we have seen that the fiber maps
are uniformly bounded linear in both cases. The factor that is
varied satisfies a uniform continuity estimate in
$\norm{\slot}_{\rho+\mu,\rho}$\ndash norm, a result from
Section~\ref{sec:cont-fiber-maps}. Note that we use the topology
defined in Section~\ref{sec:formal-tangent} on the intermediate space
$\TFBXc$, as well as a local trivialization to express the difference
$\DF_y T_\sx(y_2,x_0) - \DF_y T_\sx(y_1,x_0)$.

As a result of the fiber contraction theorem, we conclude that there
is a unique, globally attractive fixed point
$(\Theta^\infty,\,\DF\Theta^\infty)$ of the fiber
mapping~\ref{eq:fiber-map-DT}. Note that $\DF\Theta^\infty$ is already
well-defined as an element of $\mathcal{S}_1^0$, although it is only
proven to be attractive in~$\mathcal{S}_1^\mu$.
\end{proof}

As a next step, we show that the fixed point map $\DTh$ that we found
is actually continuous. This follows from a standard uniform
contraction argument.
\begin{proposition}
  \label{prop:DTh-cont}
  For any $\mu < 0$, the map $\DF\Theta^\infty \in \mathcal{S}_1^\mu$ is uniformly
  continuous. If we set $\mu \le \alpha\,\rho$ and the assumptions
  of\/ Theorem~\ref{thm:persistNHIMtriv} are satisfied with
  $\fracdiff \ge 1 + \alpha$, then it is $\alpha$\ndash H\"older
  continuous.
\end{proposition}

\begin{proof}
First note that it is sufficient to prove the statement for $\mu < 0$
sufficiently small, or $\mu = \alpha\,\rho$ in case of $\alpha$\ndash
H\"older continuity; by continuous embedding of exponential growth
spaces, it then automatically follows for any $\mu$ that is more negative. We use
local trivializations by parallel transport to express continuity
moduli of functions with domain $\T X$.

The assumptions of Theorem~\ref{thm:persistNHIMtriv} imply that the
spectral gap condition $\rho_\sy < \rho + \mu < \rho < \rho_\sx$ is
satisfied. Since $\DF\Theta^\infty$ is (the fiber part of) the fixed
point of the uniform contraction $(T,\,\DF T)$, we have for any two
$x_1,\,x_2 \in B(\ux;\rx) \subset X$ that
\begin{align*}
\fst \norm{\DTh(x_2) - \DTh(x_1)}_{\rho+\mu}\\
  &= \norm[\big]{\DF_y T(\Th(x_2),x_2)\cdot\DTh(x_2)+\DF_{x_0} T(\Th(x_2),x_2)
                -(2 \rightsquigarrow 1)}_{\rho+\mu}\\
 &\le\norm[\big]{\DF_y T(\Th(x_2),x_2)-\DF_y T(\Th(x_1),x_1)}_{\rho+\mu,\rho}
       \cdot\norm{\DTh(x_2)}_\rho\\
\con+\norm[\big]{\DF_y T(\Th(x_1),x_1)}_{\rho+\mu,\rho+\mu}\cdot
     \norm[\big]{\DTh(x_2)-\DTh(x_1)}_{\rho+\mu}\\
\con+\norm[\big]{\DF_{x_0} T(\Th(x_2),x_2) - \DF_{x_0} T(\Th(x_1),x_1)}_{\rho+\mu}\\
  &\le \epsilon_{\DF_y T}\big((L+1)d(x_2,x_1)\big)
      +\contr\,\norm{\DTh(x_2) - \DTh(x_2)}_{\rho+\mu}\\
\con  +\epsilon_{\DF_x T}\big((L+1)d(x_2,x_1)\big).
\end{align*}
Here $L = \Lip(\Th)$ denotes the Lipschitz constant of
$\Theta^\infty \in \mathcal{S}_0$, while $\contr < 1$ is the uniform contraction
factor of $\DF_y T$ on the fibers of $\BY \times B^{\rho+\mu}(I;Y)$.
We saw in Section~\ref{sec:cont-fiber-maps} that the maps
$\DF_y T,\,\DF_x T$ have appropriate continuity moduli into
$B^{\rho+\mu}(I;Y)$. Finally, we move the contraction term to the
left-hand side, divide by $1-\contr$, and obtain
\begin{equation*}
  \norm{\DTh(x_2) - \DTh(x_1)}_{\rho+\mu}
  \le \frac{1}{1-\contr}\Big[
         \epsilon_{\DF_y T}\big((L+1)d(x_2,x_1)\big)
        +\epsilon_{\DF_x T}\big((L+1)d(x_2,x_1)\big)\Big].
\end{equation*}
This shows that $\DTh \in \mathcal{S}_1^\mu$ has the same type of continuity
modulus as $\DF T$.
\end{proof}

\subsection{Derivatives on Banach manifolds}
\label{sec:banach-mflds}

We can recover the maps~\ref{eq:DT-set} as true derivatives on Banach
manifolds if we restrict to bounded time intervals $J \subset I = \R_{\le 0}$. The
maps $T_\sx,\,T_\sy$ naturally restrict to such intervals, either
exactly, or in a well-behaved approximate way. By restricting to
intervals $J = \intvCC{a}{0}$ with $a<0$, the spaces $\BJY$ and
$\BJXb$ become Banach manifolds and the restrictions of
$T_\sx,\,T_\sy$ become continuously differentiable maps on these.
\index{$J$}

\begin{lemma}
  \label{lem:BX-BY-restr-well-def}
  For any $-\infty<a<0$, the spaces $\BJY$ and $\BJXb$ with
  $J = \intvCC{a}{0}$ a bounded interval are well-defined Banach
  manifolds.
\end{lemma}

\begin{proof}
  We first treat the easy case $\BJY$.
  For any $-\infty < a < 0$, the norms $\norm{\slot}_\rho$ and
  $\norm{\slot}_0$ are equivalent on $B^\rho(J;Y)$. The set $\BJY$ is
  an open ball of radius $\Ysize$ in the Banach space $B^0(J;Y)$, so
  it follows that $\BJY$ is a Banach manifold as an open subset of
  $B^\rho(J;Y)$.

  In the same way, the metrics $d_\rho$ and $d_0$ are equivalent on
  $\BJX$, but here we need to do a little more work to show the
  following.
  \begin{proposition}
    \label{prop:BJXb-open}
    The set\/ $\BJXb$ is open in $\BJX$.
  \end{proposition}

\begin{proof}
  Let $x \in \BJXb$, hence by Definition~\ref{def:approx-sol}, $x$ is
  approximated on each interval of length $\abs{\intvCC{t_1}{t_2}} \le T$ by
  $t \mapsto \Phi(t,t_2,x(t_2))$, where $\Phi$ denotes the flow of
  $v_\sx\circ g$, the horizontal part of the unperturbed vector
  field~\ref{eq:ODE-XY-orig}. The map
  \begin{equation*}
    (t,t_2) \mapsto d\big(x(t),\Phi(t,t_2,x(t_2))\big)
  \end{equation*}
  is continuous, and since it is defined on a compact subset of
  $J \times J$, it attains its supremum
  \begin{equation*}
    \eta_1 = \sup_{t_2 \in J}\; \sup_{t \in \intvCC{t_2-T}{t_2}}\;
               d\big(x(t),\Phi(t,t_2,x(t_2))\big),
  \end{equation*}
  so it must hold that $\eta_1 < \Xsize$. Let
  $\tilde{x} \in B(x;\eta_2) \subset \BJX$ with
  \begin{equation*}
    \eta_2 = (\Xsize-\eta_1)\,\frac{e^{-\rho\,a}}{1 + C_\sx\,e^{\rho_\sx\,T}}.
  \end{equation*}
  We apply the triangle inequality and obtain
  \begin{align*}
    d\big(\tilde{x}(t),\Phi(t,t_2,\tilde{x}(t_2))\big)
    &\le d\big(\tilde{x}(t),x(t)\big)
        +d\big(x(t),\Phi(t,t_2,x(t_2))\big)\\
\con    +d\big(\Phi(t,t_2,x(t_2)),\Phi(t,t_2,\tilde{x}(t_2))\big)\\
    &\le e^{\rho\,a}\,\eta_2 + \eta_1 + C_\sx\,e^{\rho_\sx\,T}\,e^{\rho\,a}\,\eta_2\\
    &\le \big(1 + C_\sx\,e^{\rho_\sx\,T}\big)e^{\rho\,a}\,\eta_2 + \eta_1
     < \Xsize.
  \end{align*}
  This shows that all functions in the ball $B(x;\eta_2) \subset \BJX$ are still
  $(\Xsize,T)$\ndash approximate solutions of $v_\sx\circ g$, and thus
  $\BJXb$ is open.
\end{proof} 
  From here on we shall not always precisely distinguish between
  $\BJXb$ and $\BJX$ anymore.

  We introduce a local coordinate chart $\kappa_x$ around a curve
  $x \in \BJX$ using the exponential map (see
  also~\cite[Sect.~2.3]{Klingenberg1995:riemgeom2nd}):
  \begin{equation}\label{eq:BX-chart}
    \kappa_x\colon U_x \subset \BJX \to B^\rho(J;x^*(\T X))
            \colon \xi \mapsto \big(t \mapsto \exp_{x(t)}^{-1}(\xi(t)) \big).
  \end{equation}
  The vector bundle $x^*(\T X)$ is trivial, so the space of sections
  $B^\rho(J;x^*(\T X))$ is isomorphic to $B^\rho(J;\R^n)$. An explicit
  trivialization of $x^*(\T X)$ (and thus isomorphism of sections) can
  be obtained, for example if $x \in C^1$, using parallel transport as
  in~\ref{eq:TBX-triv-partrans} and identification of
  $\T_{x(0)} X \cong \R^n$ by a choice frame, but we refrain from
  making such a choice here; one reason is that curves $x \in \BJX$
  are only assumed continuous. The chart $\kappa_x$ bijectively
  covers a full $\rinj{X}$ neighborhood of $x$ with respect to the
  metric $d_0$, hence a neighborhood of size
  $\rinj{X}\,e^{-\rho\,a} > 0$ with respect to $d_\rho$. Recall that
  $\rx$ is $X$\ndash small as in Definition~\ref{def:M-small}; let us
  choose a radius $\delta_a = \rx\,e^{-\rho\,a}$, such that all
  bounded geometry results also hold true in these induced charts
  $\kappa_x$. Then the coordinate transition map
  \index{coordinate transition map!on a Banach manifold}
  \begin{equation}\label{eq:BX-coordtrans}
    \kappa_{x_2} \circ \kappa_{x_1}^{-1}
    \colon B^\rho(J;x_1^*(\T X)) \to B^\rho(J;x_2^*(\T X))
  \end{equation}
  is a bijection between isomorphic Banach spaces that is as smooth as
  the exponential map of~$X$.
\end{proof}

\begin{remark}
  We could choose isomorphisms
  $\tau_x\colon B^\rho(J;x^*(\T X)) \to B^\rho(J;\R^n)$ to obtain one
  fixed Banach space $B^\rho(J;\R^n)$ as model for the manifold
  $\BJX$. The $\tau_x$ are linear isometries so they
  preserve norms and smoothness, hence there is no need to explicitly
  make this identification. Specifically, note that the construction
  of $B^\rho(J;x^*(\T X))$ as the pullback along a curve $x$ that is
  merely continuous, does not influence the smoothness of coordinate
  transformations on~$\BJX$.
\end{remark}

We shall again call charts in this atlas `normal coordinate charts',
since they are induced by \idx{normal coordinates} on $X$ along the curve
$x \in \BJX$. By construction all bounded geometry results carry over
to these induced charts. In particular, we can measure maps in
terms of their coordinate representations. We will use this fact
without always explicitly mentioning it.

The tangent space of $\BJX$ at a point $x$ can be canonically
identified as
\begin{equation}\label{eq:TBX}
  \T_x \BJX \cong B^\rho(J;x^*(\T X))
\end{equation}
as follows. Let
$s \mapsto x_s\colon \intvOO{-\epsilon}{\epsilon} \subset \R \to \BJX$
be a $C^1$ family of curves such that $x_0 = x$ and let
$\xi_s = \kappa_x(x_s)$ be their representation in the coordinate
chart $B^\rho(J;x^*(\T X))$. The chart $\kappa_x$ is induced by normal
coordinates, so
\begin{equation*}
  d_\rho(x_0,x_s)
  = \sup_{t \in J}\; d(x_0(t),x_s(t))\,e^{\rho\,t}
  = \sup_{t \in J}\; \norm{\exp_{x_0(t)}^{-1}(x_s(t))}\,e^{\rho\,t}
  = \norm{\xi_s}_\rho
\end{equation*}
shows that $\norm{\slot}_\rho$ is the canonical norm on the chart
$B^\rho(J;x^*(\T X))$. Then $v = \der{}{s} \xi_s \big|_{s=0}$
represents a tangent vector in $\T_x \BJX$, while
$v \in B^\rho(J;x^*(\T X))$ by construction.

This completes our exposition of the manifold structure of $\BJX$. We
shall again exclusively make use of induced normal coordinate
charts~\ref{eq:BX-chart}, in order to use results on bounded geometry.

The map $T_\sx$ can be restricted to curves on any subinterval
$J = \intvCC{a}{0} \subset I = \R_{\le 0}$. Let us introduce the
restriction operator on curves
\begin{equation}\label{eq:restrict-map}
  \rho_a\colon C(I;Z) \to C(J;Z) \colon z \mapsto z|_J.
\end{equation}
This operator acts naturally on $\BY$ and $\BX$ and there is a natural
family of restrictions $T_\sx^a$ of $T_\sx$ such that
\index{$\rho_a$ (restriction operator)}
\begin{equation}\label{eq:restrict-Tx}
  T_\sx^a \circ \rho_a = \rho_a \circ T_\sx
  \qquad\text{for any } -\infty < a < 0.
\end{equation}

\begin{proposition}
  \label{prop:Tx-diff}
  Let $J = \intvCC{a}{0}$ with $-\infty < a < 0$. Then
  \begin{equation}
    T_\sx^a\colon \BJY \times X \to \BJXb
  \end{equation}
  is a differentiable map between Banach manifolds with partial
  derivatives given by a natural restriction of the maps~\ref{eq:DpTx}
  and~\ref{eq:DyTx}.
\end{proposition}

\begin{proof}
  Let $s \mapsto y + s\;\dy \in \BJY$ be a one-parameter family of
  curves and let
  \begin{equation*}
    \kappa_x\colon B(x;\delta_a) \subset \BJXb \to B^\rho(J;x^*(\T X))
  \end{equation*}
  be an induced normal coordinate chart centered around the curve
  $x = T_\sx^a(y,x_0)$. The map $T_\sx$ is Lipschitz, so for $s$
  sufficiently small, $T_\sx^a(y + s\;\dy,x_0)$ maps into
  $B(x;\delta_a)$. The vector field $\vx(\slot,(y + s\;\dy)(t))$
  depends smoothly on the parameter $s$ and generates
  $x_s = T_\sx^a(y + s\;\dy,x_0)$. We apply
  Theorem~\ref{thm:diff-flow} with
  \begin{equation*}
    \der{}{s} \Big[\vx\big(x(t),(y + s\;\dy)(t)\big)\Big]_{s=0}
    = \D_y \vx(x(t),y(t))\cdot \dy(t)
  \end{equation*}
  to obtain~\ref{eq:DyTx} as the pointwise derivative of
  $\text{ev}_t\circ T_\sx^a$, for any $t \in J$.

  Now we only need to show that~\ref{eq:DyTx} viewed as derivative
  pointwise in $t$ satisfies linear approximation estimates, uniformly
  for all $t \in J$ with respect to $d_\rho$. We work in the local
  chart $\kappa_x$, so $x_s(t)$ is represented in the normal
  coordinate chart centered at $x(t)$, while the curve
  $s \mapsto y + s\;\dy$ is canonically represented in $B^\rho(I;Y)$
  with derivative $\dy$.

%
%
  Since $s \mapsto \big(\kappa_x \circ T_\sx^a(y+s\;\dy,x_0)\big)(t) \in C^1(\R;\T_{x(t)} X)$,
  we can apply the mean value theorem to estimate
  \begin{equation}\label{eq:DyTx-deriv-est}
    \begin{aligned}
\fst     \norm[\big]{\big[T_\sx^a(y+s\;\dy,x_0) - T_\sx^a(y,x_0)
                         -\DF_y T_\sx(y,x_0)\cdot s\;\dy\big](t)}\\
    &\le \norm[\big]{\big[\big(\DF_y T_\sx(y+\sigma_t\;\dy,x_0)
                              -\DF_y T_\sx(y,x_0)\big)\cdot s\;\dy\big](t)}\\
    &\le \norm[\big]{\DF_y T_\sx(y+\sigma_t\;\dy,x_0)
                    -\DF_y T_\sx(y,x_0)}\,
         \norm{\dy}_\rho\,e^{\rho\,t}\,\abs{s}
    \end{aligned}
  \end{equation}
  for some $\sigma_t \in \intvOO{0}{s}$. Note that $\sigma_t$ will in
  general depend on $t \in J$, so there is (a priori) not one curve
  $y + \sigma\;\dy$ such that~\ref{eq:DyTx-deriv-est} holds for all
  $t \in J$ at once. In Section~\ref{sec:cont-fiber-maps} we showed that
  $\DF_y T_\sx\colon \TF\BY \to \TF\mathcal{B}^{\rho+\mu}(I;X)|_{C^1}$
  is continuous; on the bounded interval $J$ the norms
  $\norm{\slot}_\rho$ and $\norm{\slot}_{\rho+\mu}$ are equivalent, so
  $\DF_y T_\sx$ is continuous into $B^\rho(J;x^*(\T X))$ as well.
  Using this fact, we plug the result above into the definition of
  (directional) derivative and verify
  \begin{align*}
&\nop    \lim_{s \to 0}\; \frac{1}{s}\,
         \norm[\big]{T_\sx^a(y+s\;\dy,x_0) - T_\sx^a(y,x_0)
                    -\DF_y T_\sx(y,x_0)\cdot s\;\dy}_\rho\\
    &\le \lim_{s \to 0}\; \sup_{t \in J}\; \frac{\abs{s}}{s}\,
         \norm[\big]{\DF_y T_\sx(y+\sigma_t\;\dy,x_0)
                    -\DF_y T_\sx(y,x_0)}\,\norm{\dy}_\rho
     = 0.
  \end{align*}
  Therefore, the derivative of $T_\sx^a$ at $(y,x_0)$ in the direction
  of $\dy$ is given by $\DF_y T_\sx(y,x_0)\cdot\dy$ restricted to the
  interval $J$. The limit is uniform on $\norm{\dy}_\rho = 1$ and this
  map is continuous and linear in $\dy$, so $T_\sx^a$ is continuously
  partially differentiable with respect to~$y$.

  If we use a local chart around $x_0 \in X$, then we find in the same
  way that $T_\sx^a$ is continuously partially differentiable with
  respect to $x_0$. Thus, $T_\sx^a$ is (continuously) differentiable.
\end{proof}

The map $T_\sy$ does not have a similarly natural restriction since it
depends on the complete `history' of the curves $x,y$ through the
integral from $-\infty$. The dependence on earlier times is
exponentially suppressed, though. Therefore, we construct a
family of restrictions that approach $T_\sy$ when the amount of
additional history in the input goes to infinity. Let
$-\infty < a \le b < 0$ and define the family $\T_\sy^{b,a}$ of
restrictions as
\begin{equation}\label{eq:restrict-Ty}
  \begin{gathered}
  \T_\sy^{b,a}\colon \mathcal{B}^\rho_\Xsize(\intvCC{a}{0};X) \times
                               B^\rho_\Ysize(\intvCC{a}{0};Y) \to
                               B^\rho_\Ysize(\intvCC{b}{0};Y),\\
  (x,y) \mapsto \Big( t \mapsto
    \int_a^t \Psi_x(t,\tau)\,\Ynonlin(x(\tau),y(\tau)) \d\tau \Big)
    \quad\text{for each } t \in \intvCC{b}{0}.
  \end{gathered}
\end{equation}

\begin{proposition}
  \label{prop:approx-Ty}
  The family $T_\sy^{b,a}$ approximates $T_\sy$ in the sense that for
  any fixed $b \in \intvOC{-\infty}{0}$, we have
  \begin{equation}
    T_\sy^{b,a} \circ \rho_a \to \rho_b \circ T_\sy
  \end{equation}
  when $a \to -\infty$, uniformly in $x,y \in \BX \times \BY$.
\end{proposition}

\begin{proof}
  This follows from straightforward estimates:
  \begin{align*}
    \norm[\big]{T_\sy^{b,a}\circ\rho_a(x,y) - \rho_b \circ T_\sy(x,y)}_\rho
    &\le \sup_{t \in \intvCC{b}{0}}\; e^{-\rho\,t}\,
         \int_{-\infty}^a \norm{\Psi_x(t,\tau)\,\Ynonlin(x(\tau),y(\tau))} \d\tau\\
    &\le \sup_{t \in \intvCC{b}{0}}\; e^{-\rho\,t}\,
         \int_{-\infty}^a C_\sy\,e^{\rho_\sy(t-\tau)}\,\Vsize \d\tau\\
    &\le \frac{C_\sy\,\Vsize}{-\rho_\sy}\,e^{\rho_\sy(b-a) - \rho\,b}.
  \end{align*}
\end{proof}

In a same way we define approximate families for~\ref{eq:DxTy}
and~\ref{eq:DyTy}, denoted by $\DF_x T_\sy^{b,a}$ and
$\DF_y T_\sy^{b,a}$, respectively.
\begin{corollary}
  \label{cor:approx-DTy}
  The families $\DF_x T_\sy^{b,a}$ and\/ $\DF_y T_\sy^{b,a}$
  approximate~\ref{eq:DxTy} and~\ref{eq:DyTy} in the same way as in
  Proposition~\ref{prop:approx-Ty}.
\end{corollary}

\begin{proposition}
  Let $-\infty < a \le b < 0$. Then $T_\sy^{b,a}$ is a differentiable
  map between Banach manifolds.
\end{proposition}

\begin{proof}
  We shall only show that $T_\sy^{b,a}$ is continuously partially
  differentiable respect to $x$. Continuous partial differentiability
  with respect to $y$ follows along the same lines and total
  differentiability then is a direct consequence of these (also in the
  Banach manifold setting, see~\cite[Prop.~3.5]{Lang1995:diff-riem-mflds}).

  Let $x \in \BJXb$ and $y \in \BJY$ with $J = \intvCC{a}{0}$. Let
  $\kappa_x$ be an induced normal coordinate chart around $x$ and let
  \begin{equation*}
    x_s = x + s\;\dx \in B^\rho(J;x^*(\T X))
  \end{equation*}
  be a one-parameter family of curves in $\BJXb$, represented in the
  chart $\kappa_x$ (for $s$ sufficiently small). Then
  $\dx \in B^\rho(J;x^*(\T X))$ is naturally identified as the
  derivative $\der{}{s}x_s\big|_{s=0}$.

  We shall show that the partial derivative $\D_x T_\sy^{b,a}$ is given by
  the formal derivative~\ref{eq:DxTy}, but with $J$ as domain of
  integration and interpreted as a mapping into $B^\rho(\intvCC{b}{0};Y)$.
  Again, we split the full expression into manageable pieces and apply
  the mean value theorem.
  \savenotes
  \begin{align*}
\fst     \norm[\big]{\big[T_\sy^{b,a}(x_s,y) - T_\sy^{b,a}(x,y)
                    -\D_x T_\sy^{b,a}(x,y)\cdot s\;\dx\big](t)}\\
    &\le \int_a^t \big\lVert
       \Psi_{x_s}(t,\tau)\,\Ynonlin(x_s(\tau),y(\tau))
      -\Psi_x    (t,\tau)\,\Ynonlin(x  (\tau),y(\tau))\\
\con \hspace{0.7cm}
      -\Psi_x(t,\tau)\, \D_x \Ynonlin(x    (\tau),y(\tau))\,s\;\dx(\tau)
      -(\DF_x\Psi_x\cdot s\;\dx)(t,\tau)\,\Ynonlin(x(\tau),y(\tau))
                  \big\rVert \d\tau \displaybreak[1]\\
    &\le \int_a^t \begin{aligned}[t]
      & \norm[\big]{\Psi_{x_s}(t,\tau) - \Psi_x(t,\tau)
                   -\big(\DF_x \Psi_x\cdot s\;\dx\big)(t,\tau)}\,
        \norm{\Ynonlin(x(\tau),y(\tau))}\\
      &+\norm{\Psi_x(t,\tau)}\,
        \norm[\big]{\Ynonlin(x_s(\tau),y(\tau)) - \Ynonlin(x(\tau),y(\tau))
                   -\D_x \Ynonlin(x(\tau),y(\tau))\,s\;\dx(\tau)}\\
      &+\norm[\big]{\Psi_{x_s}(t,\tau) - \Psi_x(t,\tau)}\,
        \norm[\big]{\Ynonlin(x_s(\tau),y(\tau)) - \Ynonlin(x(\tau),y(\tau))}
           \d\tau \end{aligned} \displaybreak[1]\\
 \intertext{%
   and application of Theorem~\ref{thm:diff-flow} shows that
   formula~\ref{eq:DxPsi} for $\DF_x \Psi_{x_s}\cdot\dx$ is the
   derivative of $\Psi_{x_s}$. We use this for a mean value theorem
   estimate\footnote{%
     The intermediate point $\sigma$ in the mean value theorem
     implicitly depends on both $t$ and $\tau$ and will be different
     in each term. This does not affect the uniform estimates, so we
     suppress this dependence in the notation.%
   } in the first and third\footnote{%
     We applied the intermediate value theorem to both factors in the
     third term. This is not strictly necessary: we could also have
     applied it to only one of these, and apply a uniform continuity
     estimate to the other term. That would still have yielded a size
     estimate $\epsilon(\abs{s})\,\abs{s} = o(\abs{s})$. When we
     generalize to higher derivatives, we shall make use of this fact:
     at least one of the factors will be differentiable and yield a
     factor $\abs{s}$, while the other term(s) can be estimated by a
     continuity modulus $\epsilon(\abs{s})$.%
   } term to arrive at
 }%
    &\le \int_a^t \begin{aligned}[t]
      & \norm[\big]{\big(\DF_x \Psi_{x_\sigma}\cdot s\;\dx\big)(t,\tau)
                   -\big(\DF_x \Psi_x         \cdot s\;\dx\big)(t,\tau)}\,\Vsize\\
      &+C_\sy\,e^{\rho_\sy(t-\tau)}\,
        \norm[\big]{\D_x \Ynonlin(x_\sigma(\tau),y(\tau)) - \D_x \Ynonlin(x(\tau),y(\tau))}\,
        \abs{s}\,\norm{\dx}_\rho\,e^{\rho\,\tau}\\
      &+\norm[\big]{\big(\DF_x \Psi_{x_\sigma}\cdot s\;\dx\big)(t,\tau)}\,
        \norm{\D_x \Ynonlin(x_\sigma(\tau),y(\tau))}\,\abs{s}\,\norm{\dx}_\rho\,e^{\rho\,\tau}
           \d\tau \end{aligned} \displaybreak[1]\\
    &\le \int_a^t \begin{aligned}[t]
      & C\,e^{\rho_\sy(t-\tau)}\,\epsilon(\abs{\sigma}\,\norm{\dx}_\rho)\,
          e^{(\rho+\mu)\tau}\,\abs{s}\,\norm{\dx}_\rho\,\Vsize\\
      &+C_\sy\,e^{\rho_\sy(t-\tau)}\,
        \epsilon_{\D_x f}\big(\abs{\sigma}\,\norm{\dx}_\rho\,e^{\rho\,\tau}\big)\,
        \abs{s}\,\norm{\dx}_\rho\,e^{\rho\,\tau}\\
      &+\frac{C_\sy^2\,C_v}{-\rho}\,\norm{\dx}_\rho\,e^{\rho_\sy(t-\tau)}\,e^{\rho\,t}\,
        \Vsize\,\abs{s}\,\norm{\dx}_\rho\,e^{\rho\,\tau}
           \d\tau. \end{aligned}
  \end{align*}\spewnotes
  We applied Proposition~\ref{prop:DxPsi-cont} to estimate the
  variation of $\DF_x \Psi$; the induced normal coordinate charts and
  Proposition~\ref{prop:unif-partrans} allow us to freely switch
  between parallel transport and normal coordinates for estimating
  differences. All exponential norms are equivalent on the compact
  interval $J$, so with the usual estimates we see that this
  expression is $o(\abs{s})$, uniformly for all $\norm{\dx}_\rho = 1$.
\end{proof}

We have thus converted the map
$T = \T_\sy \circ (T_\sx,\,\text{pr}_1)$ to a Banach manifold setting
by defining it on curves restricted to compact time intervals. Although all estimates
were already in place, this technicality allows us to draw the
conclusions of the final points~\ref{enum:scheme-induc-deriv}
and~\ref{enum:scheme-limit-deriv} in the scheme in
Section~\ref{sec:scheme-deriv}.
\begin{lemma}[The $\Theta^n$ have true derivatives]
  \label{lem:ind-DTheta}
  Fix $\mu < 0$ and let\/ $\Theta^n\colon X \to \BY$ be
  differentiable into $B^{\rho+\mu}(I;Y)$. Recursively define
  $\Theta^{n+1}(x_0) = T(\Theta^n(x_0),x_0)$. Then $\Theta^{n+1}$ is
  again differentiable into $B^{\rho+\mu}(I;Y)$.
\end{lemma}

\begin{proof}
  We define $\D\Theta^{n+1} \in \mathcal{S}_1^0$ using~\ref{eq:fiber-map-DT} and
  proceed to show that it is the derivative of $\Theta^{n+1}$ as a
  function $\D\Theta^{n+1} \in \mathcal{S}_1^\mu$ by a direct estimate
  \begin{equation*}
    \norm{\Theta^{n+1}(x_0+h) - \Theta^{n+1}(x_0) - \D\Theta^{n+1}(x_0)\cdot h}_{\rho+\mu}
    \le \epsilon\,\norm{h},
  \end{equation*}
  with $x_0,\,x_0+h \in X$ represented in normal coordinate charts.

  First, we use the Nemytskii operator technique to get rid of the
  infinite tail $t \to -\infty$. For any given $\epsilon > 0$, we have
  on $\intvOC{-\infty}{b}$ the crude estimate
  \begin{align*}
    \fst  \sup_{t \le b}\;
         \norm[\big]{\big[\Theta^{n+1}(x_0+h) - \Theta^{n+1}(x_0)
                      - \D\Theta^{n+1}(x_0) \cdot h\big](t)}\,e^{-(\rho+\mu)\,t}\\
    &\le  \Big(\norm{\Theta^{n+1}(x_0+h) - \Theta^{n+1}(x_0)}_\rho
             + \norm{\D\Theta^{n+1}(x_0) \cdot h}_\rho\Big)\,e^{-\mu\,b}\\
    &\le  \big(\Lip(\Theta^{n+1}) + \norm{\D\Theta^{n+1}}\big)\,\norm{h}\,e^{-\mu\,b}
     \le  \epsilon\,\norm{h}
  \end{align*}
  for some $b(\epsilon)$ that is sufficiently negative. We use the
  differentiability of $T^{b,a} = T_\sy^{b,a}\circ(T_\sx^a,\,\text{pr}_1)$
  on the finite interval $\intvCC{b}{0}$ that is left. We define
  $\Theta^{n+1}_{b,a} = T^{b,a} \circ \rho_a \circ \Theta^n$ and
  estimate
  \index{$T^{b,a}$}
  \begin{align*}
    \fst  \sup_{t \in \intvCC{b}{0}}\;
          \norm[\big]{\big[\Theta^{n+1}(x_0+h) - \Theta^{n+1}(x_0)
                       - \D\Theta^{n+1}(x_0) \cdot h\big](t)}\,e^{-(\rho+\mu)\,t}\\
    &\le  \norm{\rho_b\circ\Theta^{n+1}(x_0+h) - \Theta^{n+1}_{b,a}(x_0+h)}_{\rho+\mu}
         +\norm{\rho_b\circ\Theta^{n+1}(x_0)   - \Theta^{n+1}_{b,a}(x_0)  }_{\rho+\mu}\\
    \con +\norm[\big]{\big[\rho_b\circ\D\Theta^{n+1}(x_0)
                                    - \D\Theta^{n+1}_{b,a}(x_0)\big] \cdot h}_{\rho+\mu}\\
    \con +\norm{\Theta^{n+1}_{b,a}(x_0+h) - \Theta^{n+1}_{b,a}(x_0)
            - \D\Theta^{n+1}_{b,a}(x_0) \cdot h}_{\rho+\mu}.
  \end{align*}
  This holds for all $a \le b$ and the first three terms can be made
  arbitrarily small when $a \to -\infty$ due to
  Proposition~\ref{prop:approx-Ty} and Corollary~\ref{cor:approx-DTy},
  while the last term is $o(\norm{h})$ since $\Theta^{n+1}_{b,a}$ is
  differentiable by the chain rule. If the estimate $o(\norm{h})$ is
  independent of $a$, then we can finally, for any $\epsilon > 0$ and
  $\norm{h} \le \delta$ sufficiently small, estimate this by
  $\epsilon\,\norm{h}$.

  That the term $o(\norm{h})$ is independent of $a$ follows from
  another application of the mean value theorem:
  \begin{align*}
\fst    \norm[\big]{\Theta^{n+1}_{b,a}(x_0+h) - \Theta^{n+1}_{b,a}(x_0)
                - \D\Theta^{n+1}_{b,a}(x_0) \cdot h}_{\rho+\mu}\\
   &\le \norm[\big]{\D\Theta^{n+1}_{b,a}(\xi) - \D\Theta^{n+1}_{b,a}(x_0)}_{\rho+\mu}\,
        \norm{h} \hspace{3.2cm}\text{where } d(\xi,x_0) \le \norm{h}\\
   &=   \norm[\big]{\D T^{b,a}(\Theta^n(\xi),\xi) \cdot \rho_a\circ\D\Theta^n(\xi)
                   -\D T^{b,a}(\Theta^n(x_0),x_0) \cdot \rho_a\circ\D\Theta^n(x_0)}_{\rho+\mu}\,
        \norm{h}\\
   &\le \Big(\norm[\big]{\D T^{b,a}(\Theta^n(\xi),\xi)
                        -\D T^{b,a}(\Theta^n(x_0),x_0)}_{\rho+\mu,\rho}\,
             \norm{\D\Theta^n(x_0)}_\rho\\
\con\quad +  \norm[\big]{\D T^{b,a}(\Theta^n(x_0),x_0)}_{\rho+\mu,\rho+\mu}\,
             \norm{\D\Theta^n(\xi) - \D\Theta^n(x_0)}_{\rho+\mu}
        \Big)\,\norm{h}\\
   &\le \epsilon(d(\xi,x_0))\,\norm{h}
  \end{align*}
  since the continuity estimates for the formal derivatives $\DF T$
  directly translate into the same estimates for the true derivative
  counterparts $\D T^{b,a}$ on restricted intervals.
\end{proof}

Thus, we can now conclude by induction, starting at
$\Theta^0 \equiv 0$, that for each $n \ge 0$ the
map $\Theta^n \in \mathcal{S}_0$ is differentiable when viewed as map into
$B^{\rho+\mu}(I;Y)$, while the results in
Section~\ref{sec:cont-fiber-maps} show that we actually have
\begin{equation}\label{eq:Theta-BUC}
  \Theta^n \in \BUC^{1,\alpha}\big(X;B^{\rho+\mu}(I;Y)\big).
\end{equation}
Finally, we have uniformly convergent sequences
\begin{equation}\label{eq:DTheta-conv-sequence}
  \Theta^n \to \Theta^\infty \in \mathcal{S}_0
  \qquad\text{and}\qquad
  \D\Theta^n \to \DF\Theta^\infty \in \mathcal{S}_1^\mu
\end{equation}
by the fiber contraction theorem, so now we apply
Theorem~\ref{thm:difflimitfunc} (taking into account
Remark~\ref{rem:difflimit-mflds}) to conclude that $\DF\Theta^\infty$
is the derivative of $\Theta^\infty$ as a map
$X \to B^{\rho+\mu}(I;Y)$. It was already shown in
Proposition~\ref{prop:DTh-cont} that $\DF\Theta^\infty$ is bounded and
continuous, just as the $\D\Theta^n$ in~\ref{eq:Theta-BUC}.

\begin{remark}[on topologies used]
  \label{rem:Theta-conv-topology}
  \index{comparison!of topologies}
  The convergence in~\ref{eq:DTheta-conv-sequence} is with respect to
  uniform supremum norms as in Definition~\ref{def:unif-bounded-map}.
  These induce a topology that is stronger than the weak Whitney (or
  compact-open) topology, cf.\ws Section~\ref{sec:induced-topo}. The
  convergence in Theorem~\ref{thm:difflimitfunc} is with respect to
  the weak Whitney topology, both the assumption and result. This is
  sufficient, since we are primarily interested in the result that
  $\Theta^\infty$ is differentiable, not in what sense $\Theta^n$ and
  its derivatives converge to $\Theta^\infty$. On the other hand, we
  did already have convergence of
  $\D\Theta^n \to \DF\Theta^\infty = \D\Theta^\infty$ with respect to
  these stronger uniform norms, so clearly
  $\Theta^n \to \Theta^\infty$ in uniform $C^1$\ndash norm as well.
\end{remark}

\subsection{Conclusion for the first derivative}
\label{sec:conclude-deriv}

The evaluation map $\text{ev}_0\colon B^{\rho+\mu}(I;Y) \to Y$ is
bounded linear so the graph~\ref{eq:persist-graph} of the persistent
invariant manifold also satisfies
\begin{equation*}
  \tilde{h} = \text{ev}_0 \circ \Theta^\infty \in \BUC^{1,\alpha}(X;Y).
\end{equation*}
The size of $\D \tilde{h}$ can be estimated using the fixed point equation
for $\DF \Theta^\infty$. This yields
\begin{equation*}
      \norm{\DF\Theta^\infty}_\rho
  \le \frac{\contr}{1 - \contr}\,\norm{\DF_{x_0} T_\sx},
\end{equation*}
and the contraction factor $\contr < 1$ can be made arbitrarily small
by choosing $\Vsize$ small. As indicated in~\ref{eq:small-param-dep},
$\Vsize$ is in turn controlled by $\delta,\,\sigma_1$ from Theorem~\ref{thm:persistNHIMtriv},
and $\nu$ from Lemma~\ref{lem:VB-section-smoothing}, which can be
chosen arbitrarily small. This completes the proof of all statements in
Theorem~\ref{thm:persistNHIMtriv} for $\fracdiff = 1 + \alpha$ with
$\alpha \in \intvCC{0}{1}$. Note that this is the case $k = 1$ as in
Remark~\ref{rem:persistNHIM},~\ref{enum:persist-k-two}.

\subsection{Higher order derivatives}
\label{sec:higher-derivs}
\index{higher order derivative}

To obtain higher order smoothness of the perturbed invariant manifold,
we consider equation~\ref{eq:der-fixedpoint} for $k > 1$. The
principal term governing the contraction is still
$\DF_y T(\Theta(x_0),x_0)$, now acting on multilinear maps
$\DF^k \Theta(x_0) \in \CLin^k\big(\T X;B^{k\rho+\mu_k}(I;Y)\big)$.
The remaining terms only depend on lower order derivatives of
$\Theta(x_0)$, hence they do not influence the contractivity estimate
in the fiber contraction theorem. It~must be verified, though, that
these terms depend continuously on the lower order derivatives as
mappings into $B^{k\rho+\mu_k}(I;Y)$. Note again that we set
$\mu_k = \alpha\,\rho$ in case of $\alpha$\ndash H\"older continuity;
in case of uniform continuity (denoted by $\alpha = 0$) we choose a
sequence $\{\mu_j\}_{1\le j\le k}$ such that the following hold true:
\index{$\mu$}
\begin{enumerate}
\item $\mu_j < 0$ for each $j$;
\item the spectral gap condition
  $\rho_\sy < k\,\rho + \mu_k < \rho < \rho_\sx$ still holds;
\item there exists a $\tilde{\rho} < \rho$ such that
  \begin{equation}\label{eq:mu-conditions}
    k\,\rho + \mu_k < k\,\tilde{\rho}
    \qquad\text{and}\qquad
    j\,\tilde{\rho} \le j\,\rho + \mu_j
    \quad\text{for any }\; j < k.
  \end{equation}
\end{enumerate}
It follows that the sequence $\mu_j$'s is strictly decreasing (i.e.\ws
increasing in absolute value), and that we have continuous embeddings
$B^{k\tilde{\rho}} \embedto B^{k\rho+\mu_k}$ and
$B^{j\rho+\mu_j} \embedto B^{j\tilde{\rho}}$; in the first embedding
we reserved some spectral space to apply
Corollary~\ref{cor:unifcont-nemytskii}. These choices---as well as
more ideas in this section---are inspired
by~\cite[Sec.~3]{Vanderbauwhede1989:centremflds-normal}, which is an
interesting read for comparison in a simpler setting.

We reuse the scheme already defined in Section~\ref{sec:scheme-deriv}
for the first order derivatives~\ref{eq:DT-set}. Let us walk through
these items step by step and indicate the changes that need to be
made.
\begin{enumerate}
\item\label{enum:sch-high-formal-deriv}%
  Candidate functions for the higher order derivatives can be found by
  formal differentiation and application of
  Theorem~\ref{thm:diff-flow}. This is a straightforward procedure,
  although tedious and quite unenlightening to perform. Let us show
  just one example\footnote{%
    Even though it is not obvious from~\ref{eq:DDyTx}, this expression
    is in fact symmetric in $\dy_1,\,\dy_2$. To verify this for the terms
    containing $\D_x \D_y \vx$, one should change the order of
    integration of $\tau,\,\sigma$ and expand the expression
    $\big(\DF_y T_\sx(y,x_0)\dy_2\big)(\tau)$ using~\ref{eq:DyTx}.%
  }:
  {\newcommand{\vv}[1]{\vx(x_y(#1),y(#1))}
  \begin{align*}
&\hspace{-0.9cm} \DF_y^2 T_\sx(y,x_0)\big(\dy_1,\dy_2\big)(t) =\\
&\hspace{-0.5cm}
      \int_t^0 \Bigg(\begin{aligned}[t]
        \D\Phi_y(t,\tau,x_y(\tau))\cdot\Big[
        & \D_y^2 \vv{\tau}\big(\dy_1(\tau),\dy_2(\tau)\big)\\
        &+\D_x\D_y \vv{\tau}\big(\dy_1(\tau),
            \big(\DF_y T_\sx(y,x_0)\dy_2\big)(\tau)\big)\Big]
                     \end{aligned}\\
&\hspace{0.2cm}\begin{aligned}[t]
      +&\Big[\D^2\Phi_y(t,\tau,x_y(\tau))\cdot\big(\DF_y T_\sx(y,x_0)\dy_2\big)(\tau)\\
       &\;\;+\int_t^\tau
               \D\Phi_y(t,\sigma,x_y(\sigma))\cdot
               \big[\D_y\D_x \vv{\sigma}\cdot\dy_2(\sigma)\big]\cdot
               \D\Phi_y(\sigma,\tau,x_y(\tau))
             \d\sigma\Big]
               \end{aligned}\\
&\hspace{0.5cm}\cdot\D_y \vv{\tau}\cdot\dy_1(\tau)
      \Bigg) \d\tau.
    \label{eq:DDyTx}\eqnumber
  \end{align*}}\vspace{-0.3cm}
\item\label{enum:sch-high-fibercontr}%
  For contractivity in the fibers we still only need to consider the
  map $\DF_y T$ as in~\ref{eq:DT-chainrule}, since that is the
  principal term in~\ref{eq:der-fixedpoint}. This map is contractive
  for any $\rho' \in \intvOO{\rho_\sy}{\rho_\sx}$, hence also for
  $\rho' = k\,\rho + \mu$ for any $\mu \le 0$ sufficiently small, when
  $\rho \in \intvOO{\rho_\sy}{\rho_\sx}$ is chosen appropriately. The
  other terms in~\ref{eq:der-fixedpoint} are bounded maps as well, and
  linear in the $\DF^j \Theta(x_0)$. It follows from
  Proposition~\ref{prop:compfunc-deriv} that each of these terms has
  weighted degree
  \begin{equation*}
    \sum_{i=1}^{k-1} i \cdot p_i = k - m
  \end{equation*}
  with respect to the $\DF^j \Theta(x_0)$, while they incur an
  additional exponential factor $e^{m\,\rho\,t}$ from taking $m$
  derivatives with respect to $x_0 \in X$, due to
  Lemma~\ref{lem:exp-growth-basic}. Thus the combined exponential
  growth rates sum to $k\,\rho$, and $\rho_\sy < k\,\rho$ implies
  that the variation of constants integrals still converge, so these
  terms are bounded maps into $B^{k\rho}$ spaces. This still holds if
  we add $\mu_j$'s that satisfy the conditions set out above.

  In the notation of Appendix~\ref{chap:exp-growth} we define spaces
  of higher order derivatives,
  \index{$S$@$\mathcal{S}$}
  \begin{equation}\label{eq:fiber-contr-spaces-higher}
    \mathcal{S}_k^\mu = \Gamma_b\big(\CLin^k\big(\T X; B^{k\rho+\mu}(I;Y)\big)\big)
  \end{equation}
  with norms
  \begin{equation*}
    \norm{\DF^k\Theta} = \sup_{x_0 \in X}\;
    \norm{\DF^k\Theta(x_0)}_{\CLin^k(\T_{x_0} X;B^{k\rho+\mu}(I;Y))}
  \end{equation*}
  extending~\ref{eq:fiber-contr-spaces}. Similarly, we define as
  extensions of~\ref{eq:fiber-map-DT}, higher order fiber mappings
  \begin{equation}\label{eq:fiber-map-higher}
    F^{(k)} = (T,\,\DF^1 T,\,\ldots,\,\DF^k T)
    \quad\text{on}\quad
    \mathcal{S}_0 \times \mathcal{S}_1^{\mu_1} \times \cdots \times \mathcal{S}_k^{\mu_k}.
  \end{equation}
  These are again uniform fiber contractions with respect to the final
  factor $\mathcal{S}_k^{\mu_k}$ as fiber, for any choice of $\mu_j \le 0$
  sufficiently small.

\item\label{enum:sch-high-formal-tangent}%
  Instead of trying to construct higher order formal tangent bundles,
  we represent the higher derivatives on `formal tensor bundles'
  \begin{equation}\label{eq:TBX-tensor}
    \TF \BX^k
    \, = \!\!\coprod_{x \in \BX}\! B^\rho(I;x^*(\T X))^{\otimes k}.
  \end{equation}
  Note that this choice of representation along base curves $x$
  matches our choice to represent higher derivatives as in
  Definition~\ref{def:deriv-normcoord} when the former is evaluated at
  a fixed $t$. The trivializations, then, are defined by tensor
  products of parallel transport terms $\partrans(x|_0^t)^{\otimes k}$,
  again when restricted to curves $x \in C^1$. The resulting holonomy
  terms can be estimated by either the $k$\th power of the single
  holonomy term, or, $k-1$ factors can be bounded by
  $\norm{\partrans(\gamma)} \le 2$ such that the remaining factor
  fulfills the required $\alpha$\ndash H\"older estimate. Then, all
  details in Section~\ref{sec:cont-fiber-maps} can be repeated to
  obtain uniform or $\alpha$\ndash H\"older continuity of the higher
  derivatives of $T_\sx,\,T_\sy$ as maps on these formal tensor bundles.
  Note that we can break each expression into parts such that only one
  factor is varied for the continuity estimate and thus only once adds
  either $\alpha\,\rho$ or $\mu$ to the exponential growth rate.
  Thus, the spectral gap condition is still satisfied.

\item\label{enum:sch-high-appl-FCT}%
  We apply the fiber contraction theorem to~\ref{eq:fiber-map-higher} with base
  $\mathcal{S}_0 \times \mathcal{S}_1^{\mu_1} \times \cdots \times \mathcal{S}_{k-1}^{\mu_{k-1}}$
  and fiber $\mathcal{S}_k^{\mu_k}$. In case of $\alpha = 0$, we again seize
  some of the unused spectral space for the carefully chosen
  $\mu_j$'s, such that condition~\ref{enum:fibercontr-3} of
  Theorem~\ref{thm:fibercontr} holds. Let us assume by induction that
  $F^{(k-1)}$ already is a globally attractive fiber map. The
  conditions~\ref{eq:mu-conditions} imply that if we insert elements
  $\DF^j\Theta \in \mathcal{S}_j^{\mu_j}$ into $F^{(k)}$, then their exponents
  sum at most to $k\,\tilde{\rho}$, so the mapping onto the fiber
  $\mathcal{S}_k^{\mu_k}$ is continuous by application of
  Corollary~\ref{cor:unifcont-nemytskii}. For $\alpha$\ndash H\"older
  continuity we can simply choose $\mu_j = \alpha\,\rho$ for all
  $1 \le j \le k$. Thus, we find a globally attractive fixed point
  \begin{equation*}
    (\Theta^\infty,\,\DF\Theta^\infty,\ldots,\,\DF^k\Theta^\infty)
    \in \mathcal{S}_0 \times \mathcal{S}_1^{\mu_1} \times \cdot \times \mathcal{S}_k^{\mu_k}
    \qquad\text{with}\quad
    \DF^k\Theta^\infty \in \BUC^\alpha.
  \end{equation*}

\item\label{enum:sch-high-restr-deriv}%
  We constructed a manifold structure on $\BJX$ (and a trivial one on
  $\BJY$ as well) with an atlas of charts induced by normal coordinate
  charts of the underlying manifold $X$. We represent
  higher\footnote{%
    It would probably be more natural to consider the higher
    derivatives as maps into Banach manifolds with exponents
    $k\,\rho$, but these norms are equivalent anyways.%
  } derivatives in these induced normal coordinate charts
  $B^\rho(J;x^*(\T X))$. Thus, we have for example
  \begin{equation*}
    \D_x^k T_\sx^a(x,y) \in \CLin^k\big(B^\rho(J;x^*(\T X));B^\rho(J;Y)\big).
  \end{equation*}
  This precisely matches the representation of the formal higher
  derivatives on the tensor space $B^\rho(I;x^*(\T X))^{\otimes k}$ in
  point~\ref{enum:sch-high-formal-tangent}. Higher differentiability
  of the restricted maps $T_\sx^a$ and $T_\sy^{b,a}$ follows as in
  Section~\ref{sec:banach-mflds}.

\enlargethispage{0.3cm}
\item\label{enum:sch-high-induc-deriv}%
  Lemma~\ref{lem:ind-DTheta} can be generalized to prove by induction
  over $n$ that higher derivatives $\D^k \Theta^n$ exist; we define
  $\D^k\Theta^{n+1} \in \mathcal{S}_k^0$ by~\ref{eq:fiber-map-higher}.

\item\label{enum:sch-high-limit-deriv}%
  By induction we may assume that it was already proven that
  \begin{equation*}
    \Theta^n \to \Theta^\infty \in \BUC^{k-1}\big(X;B^{(k-1)\rho+\mu_{k-1}}(I;Y)\big)
    \quad\text{as $n \to \infty$}.
  \end{equation*}
  We apply Corollary~\ref{cor:difflimitfunc} to conclude that
  $\Theta^n \to \Theta^\infty$ as sequence of $C^k$ functions and that
  $\Theta^\infty \in \BUC^k\big(X;B^{k\rho+\mu_k}(I;Y)\big)$. The
  convergence is with respect to our uniform supremum norms, see
  Remark~\ref{rem:Theta-conv-topology}.
\end{enumerate}

Just as in Section~\ref{sec:conclude-deriv}, this finalizes the proof
of all statements in Theorem~\ref{thm:persistNHIMtriv}, but now for
$\fracdiff = k + \alpha$ with $k > 1$ and $\alpha \in \intvCC{0}{1}$.
It follows that $\tilde{h} \in \BUC^{k,\alpha}$, but the last part
that remains to be shown, though, is that $\norm{\tilde{h}}_{k-1}$ can
be made as small as desired.

From the fixed point equation~\ref{eq:der-fixedpoint} it follows that
\begin{equation*}
  \norm{\D^{k-1}\Th(x_0)}
  \le \frac{1}{1-\contr}
      \!\!\sum_{\substack{l,m \ge 0\\l+m \le k-1\\(l,m)\neq(0,0),(1,0)}}\!\!
        \norm[\big]{\D_y^l\D_{x_0}^m T(\Th(x_0),x_0) \cdot
                    P_{l,k-m}\big(\D^\bullet \Th(x_0)\big)}.
\end{equation*}
Each term with $l \ge 1$ contains at least one factor $\D^j \Th(x_0)$
with $j < k-1$; these can be assumed to be small by induction. The one
remaining term with $(l,m) = (0,k-1)$ can be expanded using
Proposition~\ref{prop:compfunc-deriv}. This yields, suppressing
arguments $\Th(x_0)$ and $x_0$,
\begin{equation*}
  \D_{x_0}^{k-1} T
  = \sum_{j=1}^{k-1} \D_x^j T_\sy \cdot P_{j,k-1}(\D_{x_0}^\bullet T_\sx).
\end{equation*}
Since all terms are uniformly bounded in appropriate norms, it
suffices to show that the $\D_x^j T_\sy$ can be made small. Recall
formula~\ref{eq:DxTy} and the fiber contraction estimate for
$\D_x T_\sy$ in Section~\ref{sec:unif-fiber-contr}, where we saw that
$\D_x T_\sy$ could be made small by choosing
$\norm{\Ynonlin},\norm{\D_x \Ynonlin} \le \Vsize$ small. The higher
derivatives $\D_x^j T_\sy$, too, contain a factor $\D_x^i \Ynonlin$
with $0 \le i \le j$ in each term, so by
Proposition~\ref{prop:Ck-small-pert} these can be made small. Hence,
$\D^{k-1}\Theta^\infty$ and consequently $\norm{\tilde{h}}_{k-1}$ can
be made uniformly small. Note that $\norm{\tilde{h}}_k$ cannot be made
small though, see Remark~\ref{rem:Ck-small-optimal}.



\chapter{Extension of results}\label{chap:extension-results}

In this chapter we discuss some ways to extend the main result of
Theorem~\ref{thm:persistNHIMgen} to slightly more general situations.
These extensions are known from the compact and Euclidean settings,
but a bit scattered over the literature. We try to collect a number of
these results here, while extending them to our noncompact setting.

\section{Non-autonomous systems}\label{sec:ext-non-autonomous}
\index{non-autonomous system}

We proved the main theorem for an autonomous system and perturbation.
The Perron method admits without difficulty a time-dependent
formulation; we refrained from including this, since it would only
have cluttered the already detailed proof, while time-dependence is
easily added as an afterthought, as already noted in
Section~\ref{sec:non-autonomous}.

Let us assume that $M$ is an $\fracdiff$\ndash NHIM for the
(time-independent) vector field $v$ on $(Q,g)$ and that all
assumptions of Theorem~\ref{thm:persistNHIMgen} are fulfilled. We can
allow time-dependent perturbations by the standard trick to extend the
phase space of the system by $\R \ni t$. Define
\begin{equation}\label{eq:Q-time-ext}
  \hat{Q} = \R \times Q
  \qquad\text{with metric}\quad
  \hat{g} = {\rm d} t^2 + g.
\end{equation}
Then $(\hat{Q},\hat{g})$ is again of bounded geometry. We trivially
extend the vector field $v$ to
\begin{equation}\label{eq:v-time-ext}
  \hat{v}(t,x) = \big(1,v(x)\big) \in \T_{(t,x)} \hat{Q}
\end{equation}
and set $\hat{M} = \R \times M \subset \hat{Q}$. Then the flow
$\hat{\Phi}$ of $\hat{v}$ has the same hyperbolicity properties as
$\Phi$ since the additional flow along $\dot{t} = 1$ is completely
neutral and decoupled from the original system. It follows that
$\hat{M}$ is again an $\fracdiff$\ndash NHIM for the dynamical system
$(\hat{Q},\hat{\Phi},\R)$. Note that we need a theory for noncompact
NHIMs to perform this extension by the time interval $\R$. Now we can
choose a perturbed vector $\tilde{v}$ that depends explicitly on time,
as long as $\tilde{v} \in \BUC^{k,\alpha}(\hat{Q})$ is close to
$\hat{v}$. This means that the perturbation must be small in
$C^{k,\alpha}$\ndash norm (including derivatives with respect to
time), uniformly for all time. As a result we find that the perturbed
manifold $\tilde{M}$ will depend on time, i.e.\ws it is not exactly of
the form $\tilde{M} = \R \times \mathcal{M}$ for some
$\mathcal{M} \subset Q$. We do find that $\tilde{M}$ is uniformly
close to $\hat{M} = \R \times M$, however, so $\tilde{M}$ is
approximately of this product form.

\begin{remark}
  A direct application of Theorem~\ref{thm:persistNHIMgen} requires
  the perturbed vector field to be $C^{k,\alpha}$ with respect to
  time, too, since $t \in \R$ is added to the phase space variables.
  Note that the result thus depends $C^{k,\alpha}$ smoothly on time as
  well. A closer inspection of the proof shows that this can in fact
  be replaced by the condition that
  $\hat{v}(t,\slot) \in \BUC^{k,\alpha}$, uniformly in $t \in \R$,
  just as in Remark~\ref{rem:smooth-flow-timedep}. In that case the
  resulting manifold $\tilde{M}$ cannot be expected to be
  differentiable with respect to time anymore, but it still satisfies
  all uniform $C^{k,\alpha}$ smoothness and boundedness properties
  with respect to $x \in Q$. In particular, $\tilde{M}$ is still
  uniformly close to $\hat{M}$, uniformly for all $t \in \R$.
\end{remark}

Instead of starting with an autonomous system $v$, we can also
take an initial non-autonomous system $\hat{v}$ and perturb
that. As long as $\hat{v}$ truly describes a non-autonomous system,
that is, it is defined on a space $\R \times Q$ and has component $1$
along $\R$, then normal hyperbolicity is easily tested. The
$\R$\ndash component of the flow is trivially neutral, while the other
$Q$\ndash component must be checked in a context where, for example,
also the invariant splitting~\ref{eq:NHIM-split} may depend on time,
but this introduces no fundamental changes.

\section{Smooth parameter dependence}\label{sec:ext-param-dep}
\index{normally hyperbolic invariant manifold!parameter dependence}

Another interesting question for applications is if the persistent
manifold depends smoothly on the perturbation parameter. This result
can be obtained in a similar way as time-dependence, now adding a
parameter $p \in P$ to the phase space with trivial dynamics
$\dot{p} = 0$. The noncompact theory is not essential here, but it
does allow for a simple proof.

Let again $(Q,g)$ and $v = v(p,x)$ describe the system, where
$p \in P$ denotes the parameter. For simplicity we assume that
$P = \R^n$ and that $p = 0$ corresponds to the unperturbed system for
which we have $M$ as $\fracdiff$\ndash NHIM. We consider again an
extended system $\hat{Q} = P \times Q$ and $\hat{M} = P \times M$. The
extended vector field we choose slightly differently: we use an
external scaling parameter $\alpha \ge 0$ to slowly `turn on' the
parameter dependence. Let
$\chi \in C^\infty(\R_{\ge 0};\intvCC{0}{1})$ be a radial cut-off
function such that $\chi(r) = 1$ for $r \le 1$ and $\chi(r) = 0$ for
$r \ge 2$, and define
\begin{equation}\label{eq:param-ext-vf}
  \hat{v}_\alpha(p,x) = \big(0,v\big(\chi(\norm{p})\,\alpha\,p,x\big)\big)
\end{equation}
as a vector field on $\hat{Q}$. Note that $\hat{M}$ is an
$\fracdiff$\ndash NHIM for $\hat{v}_0$ by trivial extension. One can
verify that $\norm{\hat{v}_\alpha - \hat{v}_0}_\fracdiff$
can be chosen small with $\alpha$. Uniformity with respect
to $p$ follows automatically from $\chi$ having compact support. As a
result of Theorem~\ref{thm:persistNHIMgen} we conclude that there
exists an $\alpha > 0$ such that $\hat{v}_\alpha$ has a $C^\fracdiff$
family of invariant manifolds
\begin{equation}\label{eq:param-NHIM-family}
  \tilde{M} = \coprod_{p' \in P} \tilde{M}_{p'},
\end{equation}
where $\tilde{M}_{p'}$ is the invariant manifold corresponding to the
vector field $v(p,\slot)$ with $p = \chi(\norm{p})\,\alpha\,p'$. This
parametrizes a full neighborhood $B(0;\alpha) \subset P$.

\section{Overflowing invariant manifolds}\label{sec:ext-overflow-invar}
\index{overflow invariance}

Overflowing invariance is a useful tool to study invariant manifolds
whose normal hyperbolicity properties break down beyond a certain
domain, see also Section~\ref{sec:overflow-invar}. We shall indicate
here how our main result can be extended to overflowing invariant
manifolds. We provide conditions for persistence that are slightly
weaker than those in the literature. These might prove useful for some
applications.

The following definition extends that
in~\cite{Fenichel1971:invarmflds} and is equivalent to Definition~2.1
in~\cite{Bates1999:persist-overflow}.
\begin{definition}[Overflowing invariant manifold]
  \label{def:overflow-invar}
  Let\/ $(Q,g)$ be a Riemannian manifold, $M \subset Q$ a $C^1$
  submanifold with boundary $\partial M \in C^1$, and $v \in C^1$ a
  vector field on $Q$ with flow\/ $\Phi$. Let $n$ denote the outward
  normal at\/ $\partial M$. Then $M$ is called overflowing invariant
  under $v$ if the following hold:
  \begin{enumerate}
  \item\label{enum:overflow-orbits}%
    backward orbits stay in $M$, i.e.\ws
    $\forall\, m \in M,\,t<0\colon \Phi^t(m) \in M$;
  \item\label{enum:overflow-outward}%
    the vector field $v$ points uniformly strictly outward at
    $\partial M$, i.e.\ws there exists some $\epsilon > 0$ such that\/
    $\forall\, m \in \partial M\colon g_m(v,n) \ge \epsilon.$
  \end{enumerate}
\end{definition}
Definition~\ref{def:NHIM} of normal hyperbolicity can be adapted to
this setting (only condition~\ref{enum:overflow-orbits} is
necessary): we assume that only stable normal directions are present
and we only require $M$ to be negatively invariant, while the
exponential rate conditions must hold along orbits as long as they
stay inside $M$.

\begin{remark}
  Note that the uniformity in condition~\ref{enum:overflow-outward}
  reduces to the standard `strictly outward' if
  $\bar{M} = M \cup \partial M$ is compact. This is the natural
  generalization for noncompact manifolds, since the condition is used
  to guarantee that under small perturbations and in a small tubular
  neighborhood the vector field is still pointing outward.
\end{remark}

The Perron method uses orbits as fundamental objects and constructs a
contraction operator on these. The essence of
Definition~\ref{def:overflow-invar} is to guarantee
condition~\ref{enum:overflow-orbits} that backward orbits stay inside
$\bar{M}$, even under a small perturbation of the vector field. This
provides an idea to slightly weaken the overflow invariance definition
into an a priori argument. If any orbits considered in the Perron
method proof stay inside $\bar{M}$, then all assumptions throughout
the proof are still valid and we obtain a persistent manifold
$\tilde{M}$. To make this idea explicit, we choose the trivial bundle
setting of Theorem~\ref{thm:persistNHIMtriv} and introduce the
following weakened definition.
\begin{definition}[A priori overflowing invariance]
  \label{def:overflow-invar-priori}
  \index{overflow invariance!a priori}
  Let\/ $(X,g)$ be a Riemannian manifold and let $M \subset X$ an open
  submanifold, i.e.\ws of the same dimension,
  with boundary $\partial M \in C^1$. Let\/ $Y$ be a Banach
  space, and $v \in C^1$ a vector field on $X \times Y$ with flow\/
  $\Phi$. Let $n$ denote the outward normal at\/ $\partial M$.
  Let $\tilde{v}$ be a perturbation of $v$. Then $M$ is called a
  priori overflowing invariant for the pair $(v,\tilde{v})$ if the
  following hold:
  \begin{enumerate}
  \item\label{enum:overflow-priori-orbits}%
    backward orbits of $v$ stay in $M$, i.e.\ws
    $\forall\, m \in M,\,t<0\colon \Phi^t(m) \in M$;
  \item\label{enum:overflow-priori-outward}%
    the vector field $\tilde{v}$ points (non-strictly) outward at a
    tubular neighborhood over $\partial M$, i.e.\ws there exists some
    $\Ysize > 0$ such that
    \begin{equation*}
      \forall\, (m,y) \in \partial M \times Y_{\le \Ysize}\colon
        g\big(\D\pi_\sx\cdot v(m,y),n(m)\big) \ge 0.
  \end{equation*}
  \end{enumerate}
\end{definition}

\begin{remark}\NoEndMark
  Note that Definition~\ref{def:overflow-invar}
  implies~\ref{def:overflow-invar-priori} when $\norm{\tilde{v} - v}_1$
  is small enough and $\tilde{v} \in \BUC^1$.
\end{remark}

\begin{remark}
  A useful generalization of Definition~\ref{def:overflow-invar-priori}
  to the setting of Theorem~\ref{thm:persistNHIMgen} is less trivial.
  There we do not have canonical vertical fibers over $\partial M$ in
  the tubular neighborhood, nor the associated projection of $v$ onto
  $\T X$ at $\partial M$. We cannot simply take a non-vertical fiber;
  the Perron method adapts the curves $x$ and $y$ separately, so it
  may happen that while $x(0) \in \partial M$ is kept fixed, $y(0)$ is
  updated to a new value such that $(x(0),y(0))$ lies outside of the
  tubular neighborhood over $\bar{M}$, and control is lost.
\end{remark}

Let us demonstrate the application of this more general definition with the
following simple example, see also Figure~\ref{fig:overflow-priori}.
\begin{example}[Persistence under a priori overflowing invariance]
  Let $X \times Y = \R \times \R$ and let the unperturbed vector field
  be given by
  \begin{equation*}
    v(x,y) = \big( -(x-1)^2,\, (x^2 - 4)\,y\big).
  \end{equation*}
  Note that $M = \intvOO{-1}{1} \subset X$ is strictly overflowing
  invariant at its left boundary $x = -1$ (we could choose other
  values as well), but non-strictly so at the right boundary $x = 1$,
  which is a degenerate stationary point. The vector field is normally
  attracting over the interval $\intvOO{-2}{2}$ and uniformly so over
  any closed subinterval. Note that there does not exist a subinterval
  of $X$ that is overflowing invariant according to
  Definition~\ref{def:overflow-invar}.

  Let us choose a family $v_\delta$ of perturbations of $v$ such that
  $\norm{v_\delta - v}_1 \le \delta$ and $v = v_\delta$ on a
  neighborhood of $(1,0) \in X \times Y$. Then $M$ satisfies
  Definition~\ref{def:overflow-invar-priori} for this family $v_\delta$
  and application of Theorem~\ref{thm:persist-overflow} below shows
  that for $\delta$ sufficiently small, there exists a unique
  negatively invariant manifold $\tilde{M} = \Graph(\tilde{h})$ for
  the flow of $v_\delta$ such that
  $\tilde{h}\colon \intvCC{-1}{1} \subset X \to
  \intvCC{-\Ysize}{\Ysize} \subset Y$. For any $r \ge 1$ there exists
  a $\delta$ such that $\tilde{h} \in C^r$ holds.
\end{example}
\begin{figure}[htb]
  \centering
  \input{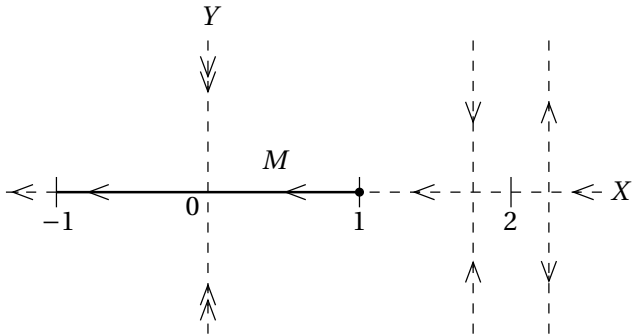}
  \caption{a priori overflowing invariance for the manifold $M$.}
  \label{fig:overflow-priori}
\end{figure}

\begin{theorem}[Persistence under overflowing invariance]
  \label{thm:persist-overflow}
  \index{overflow invariance!persistence}
  Let $k \ge 2$, $\alpha \in \intvCC{0}{1}$ and $\fracdiff =k+\alpha$.
  Let\/ $(X,g)$ be a smooth, complete, connected Riemannian manifold of
  bounded geometry and\/ $Y$ a Banach space. Let
  $v_\delta \in \BUC^{k,\alpha}$ be a family of vector fields defined
  on a uniformly sized neighborhood of the zero-section in
  $X \times Y$ such that $\norm{v_\delta - v_0}_1 \le \delta$. Let $M$
  satisfy Definition~\ref{def:overflow-invar-priori} for the pair
  $(v_0,v_\delta)$ for any $\delta \in \intvOC{0}{\delta_0}$ and let
  $M$ be $\fracdiff$\ndash normally attracting for the flow defined by
  $v_0$, that is, $M$ satisfies the overflowing invariant version of
  Definition~\ref{def:r-NHIM} with $\rank(E^+) = 0$.

  Then for each sufficiently small $\Ysize > 0$ there exist
  $\delta_1 > 0$ such that for any $\delta \in \intvOC{0}{\delta_1}$,
  there is a unique manifold with boundary
  $\tilde{M} = \Graph(\tilde{h})$, $\tilde{h}\colon M \to Y$,
  $\norm{\tilde{h}}_0 \le \Ysize$ such that $\tilde{M}$ is negatively
  invariant under the flow defined by $v_\delta$. Moreover,
  $\tilde{h} \in \BUC^{k,\alpha}$ and $\norm{\tilde{h}}_{k-1}$ can be
  made arbitrary small by choosing $\norm{v_\delta - v_0}_{k-1}$
  sufficiently small. The function $h$ extends continuously to
  $\partial M$.
\end{theorem}

\begin{remark}\NoEndMark
  In this overflowing invariance setting, the condition that
  $\rank(E^+) = 0$ is really necessary and not an artifact of our proof.
  The same results hold for inflowing invariance with no stable normal
  directions present. Definition~\ref{def:overflow-invar-priori} can
  be extended to full normal hyperbolicity with both stable and
  unstable normal directions present. This requires full invariance of
  a tubular neighborhood of $M$ under both the forward and backward
  orbits.
\end{remark}

\begin{wrapfigure}[1]{r}{3.8cm}
  \centering
  \vspace{-2.1cm}
  \input{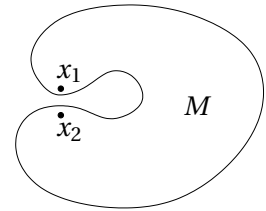}
  \vspace{-0.2cm}
  \caption{a nonconvex subset $M \subset X$.}
  \label{fig:nonconvex-M}
\end{wrapfigure}
\begin{minipage}{1.0\linewidth-0.5em}
\begin{remark}
  \label{rem:overflow-nonconvex}
  We can restrict to a smaller open subset $U$ of $X$ that contains
  $\bar{M}$, so we do not need $v_\delta \in \BUC^{k,\alpha}$ to hold
  on all of $X$. If this subset $U$ is not convex, though, we may run
  into difficulties when applying the mean value theorem, see
  Figure~\ref{fig:nonconvex-M}: an intermediate point
  $\xi \not\in \bar{M}$ on the line between $x_1,\,x_2$ may be
  selected, so we need to make sure that the uniform estimates still
  hold there. Thus the need for $U \supset \bar{M}$ to be convex, see
  also the remark in~\cite[p.~289]{Henry1981:geom-semilin-parab-PDEs}.
\end{remark}
\end{minipage}
\vspace{2ex}

\begin{proof}
  The proof of Theorem~\ref{thm:persistNHIMtriv} requires minimal
  changes. Note that regardless of the modifications and smoothing
  preparations performed in Section~\ref{sec:preparation}, the vector
  field $\vx$ is precisely the horizontal component of the perturbed
  vector field $v_\delta$. In Section~\ref{sec:NHIM-lipschitz} where
  we proved existence and uniqueness of $\tilde{M}$, we take $\Ysize$
  small enough that it satisfies
  condition~\ref{enum:overflow-priori-outward} of
  Definition~\ref{def:overflow-invar-priori}. This guarantees that
  $x = T_\sx(y,x_0)$ is a solution curve such that
  $x\big(\intvOC{-\infty}{0}\big) \subset \bar{M}$ for any $y \in \BY$
  and $x_0 \in \bar{M}$. Hence, the contraction mapping
  $T = T_\sy\circ(T_\sx\,,\,\text{pr}_1)$ is well-defined with
  intermediate space $\mathcal{B}_\beta(I;\bar{M})$ and we find a
  unique Lipschitz continuous fixed point map
  $\Theta^\infty\colon \bar{M} \to \BY$.

  No essential changes are needed with respect to the smoothness proof
  in Section~\ref{sec:NHIM-smoothness}. The formal
  derivatives~\ref{eq:DT-set} are well-defined along all curves $x$
  and $y$ that are considered, since the derivatives of $\vx,\,A,\,f$
  are defined on an open neighborhood of $\bar{M}$. In
  Section~\ref{sec:banach-mflds} we use the mean value theorem to
  prove that the restricted maps $T^{b,a}$ have true derivatives.
  Remark~\ref{rem:overflow-nonconvex} is not problematic here, since
  $T^{b,a}$ is defined on the finite interval $J = \intvCC{a}{0}$ and
  thus we can restrict to arbitrarily small open neighborhoods along
  the curves $x,\,y$ when restricted to $J$. Hence we find that
  $\Theta^\infty \in \BUC^{k,\alpha}$ on $M$.
\end{proof}

\section{Full normal hyperbolicity}\label{sec:ext-full-NHIM}
\enlargethispage*{0.3cm}

We made the assumption in our main theorems that the unstable bundle
$E^+$ was absent, that is, that $M$ was a normally attracting
invariant manifold. As already noted in Remark~\ref{rem:persistNHIM},%
~\ref{enum:persist-gen-NHIM}, it should be possible to generalize this
to the case of full normal hyperbolicity where both stable and
unstable normal directions are present. Let us indicate here how this
more general result can be obtained.

Assume that in Theorem~\ref{thm:persistNHIMgen} we have an invariant
splitting~\ref{eq:NHIM-split} with both stable and unstable bundles
present. The reduction principle in Section~\ref{sec:BG-reduction}
leads to a formulation of Theorem~\ref{thm:persistNHIMtriv} with a
trivial bundle
\begin{equation}\label{eq:triv-split-full}
  \pi\colon X \times \big( Y \times Z \big) \to X,
\end{equation}
where the Banach spaces $Y,\,Z$ are approximate representations of the
stable and unstable bundles $E^\pm$ of $M$. This means that $M$ is
again represented as the graph of an approximate zero section
$h_\sigma\colon X \to Y \times Z$; now, the subbundles $X \times Y$
and $X \times Z$ are approximately invariant under $v_\sigma$. The
deviation from invariance is controlled by $\sigma$, the parameter of
the smoothing approximation of $M$. We find linear operators
$A^\pm(x)$ on $Y$ and $Z$ respectively, that approximate the
linearizations of $v_\sy$ and $v_\sz$, and corresponding flows
$\Psi^\pm$ with approximate growth rates. We add a map\footnote{%
  Note that since $t \le \tau$, we have a reverse flow
  $\Psi^+(t,\tau)$ for the unstable directions, which indeed satisfies
  the growth estimates~\ref{eq:NHIM-rates}.%
}
\begin{equation}\label{eq:Tz}
  T_\sz(x,y,z)(t)
  = \int_t^\infty \Psi^+_x(t,\tau)\,
                  \Ynonlin^+\big(x(\tau),y(\tau),z(\tau)\big) \d\tau
\end{equation}
with $z \in B^\rho_\Ysize(\R_{\ge 0};Z)$ and adapt the other maps to
incorporate $z$ as an argument. We use Lemma~\ref{prop:bound-extend}
and extend all curves in $X,\,Y,\,Z$ to the full real line. This
should yield a contraction
\begin{equation}\label{eq:Txyz}
T = (T_\sy,T_\sz) \circ \big(T_\sx\,,\,\text{pr}_1\,,\,\text{pr}_2\big)
\quad\text{on}\quad
B^\rho_\Ysize(\R_{\le 0};Y) \times B^\rho_\Ysize(\R_{\ge 0};Z),
\end{equation}
again with $x_0 \in X$ as initial value parameter. We obtain a pair
$(\Theta^-,\,\Theta^+)$ of fixed point maps, and after evaluation we
find
\begin{equation}\label{eq:persist-graph-yz}
  (\tilde{h}^-,\tilde{h}^+)\colon X \to Y \times Z,
\end{equation}
which describes the persistent invariant manifold $\tilde{M}$.



\clearpage
\appendix
\pagestyle{appendix}

\renewcommand{\chaptername}{Appendix}

\chapter{Explicit estimates in the implicit function theorem}\label{chap:smoothflows}

In this appendix, we carefully examine the implicit function theorem.
We extend this standard theorem to classes of functions with
additional properties such as boundedness and uniform and H\"older
continuity. The crucial ingredient is the explicit formula~\ref{eq:impl-func-der} for the
derivative of the implicit function, which allows us to transfer
regularity conditions onto the implicit function.

As an application of the implicit function theorem in Banach spaces,
we will establish existence, uniqueness and smooth dependence on
parameters for the flow of a system of ordinary differential
equations. Essentially, these are standard results from differential
calculus, see e.g.\ws Zeidler~\cite[p.~150,165]{Zeidler1986:nonlfunc1}
or~\cite{Robbin1968:exist-thm-ODE,Irwin1972:smooth-comp-map}.
We consider a general setting of ODEs in Banach spaces and show
smooth dependence, both on the initial data, as well as on the vector
field itself. Moreover, our extension of the implicit function theorem
yields boundedness and uniform continuity results.

We start with some results on inversion of linear maps.
\begin{lemma}[Invertibility of linear maps]
  \label{lem:invert-lin}
  Let\/ $X$ be a Banach space and let $A \in \CLin(X)$ be a continuous
  linear operator with continuous inverse. Let\/ $B \in \CLin(X)$ be
  another linear operator such that\/
  $\norm{B} < \frac{1}{\norm{A^{-1}}}$. Then $A+B$ is also a
  continuous linear operator with continuous inverse, given by the
  absolutely convergent series
  \begin{equation}\label{eq:lin-inverse}
    {\big(A+B\big)}^{-1} = \sum_{n \ge 0} {\big(-A^{-1} B\big)}^n\,A^{-1}
                         = \sum_{n \ge 0} A^{-1}\,{\big(-B\,A^{-1}\big)}^n.
  \end{equation}
\end{lemma}

\begin{proof}
  First of all, note that there exists an $M \ge 1$ such that
  $\norm{A}, \norm{A^{-1}} \le M$.
  The base of the geometric series can be estimated in
  operator norm as $\norm{-A^{-1} B} < 1$, so the series is
  absolutely convergent and the limit is a well-defined continuous
  linear operator, whose operator norm can be estimated as
  \begin{equation*}
        \norm{{(A+B)}^{-1}}
    \le \norm{A^{-1}}\sum_{n \ge 0} \norm{-A^{-1} B}^n
    \le \frac{\norm{A^{-1}}}{1-\norm{A^{-1} B}} < \infty.
  \end{equation*}
  That the limit is again a well-defined linear operator follow from
  the fact that $\CLin(X)$ is a Banach space.

  Applying $A+B$ to the left-hand side of~\ref{eq:lin-inverse}, we see
  that the candidate is a right inverse:
  \begin{equation*}
    (A+B)\,\sum_{n \ge 0} A^{-1}\,{\big(-B A^{-1}\big)}^n
    =      \sum_{n \ge 0} {\big(-B A^{-1}\big)}^n
         - \sum_{n \ge 0} {\big(-B A^{-1}\big)}^{n+1} = 1.
  \end{equation*}
  Similarly the candidate can be shown to be a left inverse of $A+B$.
  Now we have that the candidate is continuous and a full inverse and
  furthermore, $A+B$ itself is clearly a continuous operator as the
  sum of two continuous operators, so the proof is completed.
\end{proof}

\begin{corollary}[Linear inversion is analytic]
  \label{cor:anal-inverse}
  Let $I\colon A \mapsto A^{-1}$ be the inversion map defined on
  continuous, linear mappings $A \in \CLin(X)$ with continuous
  inverse, where $X$ is a Banach space. The map $I$ is analytic
  with radius of convergence $\rho(A) \ge 1/\norm{A^{-1}}$. When $X$
  is finite-dimensional, $I$ is a fortiori a rational map.
\end{corollary}

\begin{proof}
  Extending well-known results on analytic functions to Banach spaces
  (see e.g.~\cite{Mujica1986:complanal-banach}), we read off
  from~\ref{eq:lin-inverse} that the inversion map $I$ can be given
  around $A$ by an absolutely convergent power series with
  $\rho(A) \ge 1/\norm{A^{-1}}$ and is thus analytic. When $X$ is
  finite-dimensional, $\det(A) \neq 0$ implies that $A^{-1}$ is a
  rational expression in the matrix coefficients of $A$ according to
  Cramer's rule.
\end{proof}

The inversion map $I$ is locally Lipschitz, like every $C^1$
mapping:
\begin{equation*}
      \norm{{(A+B)}^{-1} - A^{-1}}
  \le \sum_{n \ge 1} \norm{-A^{-1} B}^n\,\norm{A^{-1}}\\
  \le \frac{\norm{A^{-1}}^2}{1 - \norm{A^{-1} B}}\,\norm{B}.
\end{equation*}
However, when we restrict to a domain bounded away from non-invertible
operators $A$, that is, when $\norm{A^{-1}} \le M$, then the Lipschitz
constant is bounded for small $B$. This implies that when $A = A(x)$
depends on a parameter via a certain continuity modulus, then
$A(x)^{-1}$ will have the same continuity modulus up to the Lipschitz
constant, at least in small enough neighborhoods.

The standard implicit function theorem on Banach spaces can be stated
as
\index{implicit function theorem}
\begin{theorem}[Implicit function theorem]
  \label{thm:impl-func}
  Let\/ $X$ be a Banach space, $Y$ a normed linear space, and let
  $f \in C^{k \ge 1}(X \times Y;X)$. Let $(x_0,y_0) \in X \times Y$
  and assume that $f(x_0,y_0) = 0$ and that
  $\D_1 f(x_0,y_0)^{-1} \in \CLin(X)$ exists as a continuous, linear
  operator.

  Then there exist neighborhoods $U \subset X$ of $x_0$ and\/
  $V \subset Y$ of $y_0$, and a unique function $g\colon V \to U$ such
  that $f(g(y),y) = 0$. Furthermore, the map $g$ is $C^k$ and the
  derivative of $g$ is given by the formula
  \begin{equation}\label{eq:impl-func-der}
    \D g(y) = - \D_1 f(g(y),y)^{-1} \cdot \D_2 f(g(y),y).
  \end{equation}
\end{theorem}
See~\cite[p.~150--155]{Zeidler1986:nonlfunc1} for a proof. Note that
we do not need to assume that $Y$ is a complete space, as the
contraction theorem is only applied on $X$. Recall that we use
notation where $\D$ denotes a total derivative, while $\D_i$ with
index $i \in \N$ denotes a partial derivative with respect to the
$i$\th argument.

Formula~\ref{eq:impl-func-der} for the derivative of the implicit
function $g$ will be crucial for the extension of the implicit
function theorem to many classes of regularity, extending $C^k$
smoothness. We use the Lipschitz estimate for the inversion map and
require that the regularity conditions are preserved under composition, addition,
multiplication and localization of functions. By
Proposition~\ref{prop:compfunc-deriv}, the derivatives of $g$ are
expressed in terms of $\D_1 f(g(y),y)^{-1}$ acting on a polynomial
expression of same or lower order derivatives of $f$ and strictly
lower order derivatives of $g$.

As an example, let us take $\BC^{k,\alpha}$ functions. Using
Lemma~\ref{lem:holder-product} and induction over $k$, this
function class is preserved under products. For composition, we check
H\"older continuity,
\begin{equation*}
      \norm{f(g(x_2)) - f(g(x_1))}
  \le C_f\,\big(C_g\,\norm{x_2 - x_1}^\alpha\big)^\alpha
  \le (C_f\,C_g^\alpha)\,\norm{x_2 - x_1}^\alpha
\end{equation*}
for $0 < \alpha \le 1$, when $\norm{x_2 - x_1} \le 1$. In case
$\norm{x_2 - x_1} > 1$ however, we can directly use the boundedness
of~$f$:
\begin{equation*}
      \norm{f(g(x_2)) - f(g(x_1))}
  \le \norm{f(g(x_2))} + \norm{f(g(x_1))}
  \le 2\,\norm{f}_0\,\norm{x_2 - x_1}^\alpha.
\end{equation*}
Thus, H\"older continuity is preserved with some new H\"older
constant, while boundedness is trivially preserved as well. We
conclude that if $f \in \BC^{k,\alpha}$, and $(\D_1 f)^{-1}$ is
globally bounded, then we can read off from
formula~\ref{eq:impl-func-der} that $g \in \BC^{k,\alpha}$.
The same results hold for the class of $\BUC^k$ functions, or any
other class of functions whose properties are preserved when inserted
into~\ref{eq:impl-func-der}. Together, interpreting $\alpha = 0$ as an
empty condition, these lead to
\begin{corollary}
  \label{cor:unif-impl-func}
  Let in the Implicit Function Theorem~\ref{thm:impl-func},
  $f \in \BUC^{k,\alpha}$ with $k \ge 1$ and $0 \le \alpha \le 1$.
  Assume moreover that\/ $\norm{\D_1 f(x,y)^{-1}} \le M$ is bounded on
  $U \times V$ for some constant $M < \infty$. Then
  $g \in \BUC^{k,\alpha}$, and the boundedness and continuity
  estimates depend in an explicit way on those of\/ $f$.
\end{corollary}

\begin{remark}
  Formula~\ref{eq:impl-func-der} only provides control on the
  derivatives of the implicit function, but the size of $g$ itself can
  be controlled by choice of the neighborhood $U$. In our
  applications, this will match up with choosing coordinate charts
  around the origin in $\R^n$.
\end{remark}

Let us now consider an ordinary differential equation
\begin{equation}\label{eq:ODEsystem}
  \dot{x} = f(t,x), \qquad x(t_0) = x_0,
\end{equation}
where $x$ takes values in a Banach space $B$ and
$f \in \BUC^{k,\alpha}(\R \times B;B)$ with
$k \ge 1,\,0 \le \alpha \le 1$. We consider solutions
$x \in X = C^0(I;B)$ equipped with the supremum norm, which turns $X$
into a Banach space\footnote{%
  Note that any actual solution $x$ will be $C^1$ at least, but only
  $x \in C^0$ is required. This makes $X$ a complete space without the
  need to introduce norms more complicated than the supremum norm.%
}. We choose $I$ to be a closed interval
$I = \intvCC{a}{b} \subset \R$. The Picard integral operator
\begin{equation}\label{eq:picard}
  T\colon X \to X
   \colon x(t) \mapsto F(x)(t) = x_0 + \int_{t_0}^t f(\tau,x(\tau)) \;\d\tau
\end{equation}
has exactly the solution curves of~\ref{eq:ODEsystem} as fixed points.
It also implicitly depends on $f \in \BUC^{k,\alpha}(\R \times B;B)$ and
$(t_0,x_0) \in I \times B$. From now on we denote by $\D_x$ a partial
derivative with respect to the argument that is typically described by
the variable~$x$.

This $T$ is a contraction for $\abs{I} = b - a$ small enough:
\begin{align*}
  \norm{T(x_1) - T(x_2)}
  &=   \sup_{t \in I}\; \norm{\int_{t_0}^t f(\tau,x_1(\tau)) -
                                           f(\tau,x_2(\tau)) \;\d\tau}\\
  &\le \sup_{t \in I}\; \int_{t_0}^t \norm{\D_x f(\tau,\xi(\tau))}
                                     \norm{x_1(\tau)-x_2(\tau)} \d\tau\\
  &\le \sup_{t \in I}\; \abs{t-t_0} \norm{\D_x f} \norm{x_1-x_2}\\
  &\le \abs{I}\,\norm{\D_x f}\,\norm{x_1-x_2}.
\end{align*}
We restrict $T$ to a bounded subset of argument functions $f$,
\begin{equation*}
  \mathcal{F} \subset \BUC^{k,\alpha}(\R \times B;B), \qquad
  \sup_{f \in \mathcal{F}}\; \norm{f}_{k,\alpha} \le R.
\end{equation*}
Thus, choosing $\abs{I} \le \frac{1}{2 R}$ turns $T$ into a
$\contr = \frac{1}{2}$ contraction, which shows that there is a unique
$x \in X$ satisfying $T(x) = x$ and therefore~\ref{eq:ODEsystem}.

Next, we consider small perturbations of both $(t_0,x_0)$ and $f$. To
apply the implicit function theorem, we define $F(x) = x - T(x)$.
This function has a unique zero and $\D F(x)$ is invertible, as
\begin{align*}
\fst F(x+\delta x)(t) - F(x)(t)\\
  &= \delta x(t) - \int_{t_0}^t \D_x f\big(\tau,\xi(\tau)\big) \cdot \delta x(\tau) \;\d\tau\\
  &= \delta x(t) - \int_{t_0}^t \D_x f\big(\tau,x(\tau)\big) \cdot \delta x(\tau)
                          +O\big(\norm{\xi(\tau)-x(\tau)}\big)\norm{\delta x(\tau)} \;\d\tau\\
  &= \big(\D F(x)\cdot \delta x\big)(t) + o\big(\norm{\delta x}\big)
\label{eq:Dx-picard}\eqnumber
\end{align*}
The neglected terms are $o\big(\norm{\delta x}\big)$ since $\D_x f$ is
uniformly continuous on $I$, so $\D F(x)$ exists. From the
expression above, we can also easily read off continuity of $\D F(x)$
as a linear operator, by writing
$\D F(x) = \Id + A(x)$ and noticing that
\begin{equation*}
\norm{A(x)} \le \abs{I}\,\norm{\D_x f} < \mfrac{1}{2},
\end{equation*}
thus $\D F(x)$ is a bounded, invertible linear operator such that
$\norm{\D F(x)^{-1}} \le 2$.

By similar estimates, the derivatives of $F$ with respect to the
parameters $t_0, x_0$, and $f$ can be calculated as
\begin{equation}\label{eq:Dparam-picard}
  \begin{aligned}
    \D_{t_0} F(x) &= f(t_0,x(t_0)),\\
    \D_{x_0} F(x) &= -\Id,\\
    \big(\D_f F(x)\cdot \delta f\big)(t)
                  &= -\int_{t_0}^t \delta f(\tau,x(\tau)) \d\tau.
  \end{aligned}
\end{equation}
Note that these are all bounded linear operators; $\D_{t_0} F(x)$ is
because $\norm{f} \le R$. Hence, $F \in \BC^1$ as a function of
$x,t_0,x_0,f$, so by the implicit function theorem, the solution
$x(t;t_0,x_0,f)$ depends $\BC^1$ on $t_0,x_0,f$.

Next, we establish $\BUC^{k,\alpha}$ dependence on the initial
conditions $t_0,x_0$ and $\BC^k$ dependence on $f$ and $t_0,x_0$ together.
Uniform and H\"older dependence on $f$ are lost because the variations
$\delta f \in \BUC^{k,\alpha}$ are not uniformly equicontinuous. The
first derivatives can be differentiated another $k-1$ times with respect to each of
the variables, using similar estimates as in~\ref{eq:Dx-picard}. These
derivatives are continuous as $f$ is uniformly continuous on the
interval $I$. Uniform and H\"older continuity with respect to
$t_0,x_0$ can be read off directly from the
expressions~\ref{eq:Dx-picard},\ref{eq:Dparam-picard} or their higher
order derivatives, as $f \in \BUC^{k,\alpha}$. The implicit function
theorem only gives an explicit formula~\ref{eq:impl-func-der} for the
derivative. Here, this translates into the fact that no boundedness
follows for the $C^0$\ndash norm of the solution curve, only for the
norms on the derivatives.

We have thus shown that the conditions of
Corollary~\ref{cor:unif-impl-func} of the implicit function theorem
have been satisfied, so there exists a neighborhood of $(t_0,x_0,f)$
in $I \times B \times \mathcal{F}$ such that for each $(t_0',x_0',f')$
in that neighborhood there is a unique solution to~\ref{eq:ODEsystem}
and the solutions $x$ depend in a $\BUC^{k,\alpha}$ way on $t_0',x_0'$
and $\BC^k$ on all of $t_0',x_0',f'$. Note that this result is obtained
only on the interval $I$. We can however extend these results to any
bounded interval, by using the composition property of a flow; the
estimates may grow with interval size though. Hence, we have the
following result, see also~\cite[appendix B]{Duistermaat2000:liegroups}.
\index{smoothness!of a flow}
\begin{theorem}[Uniform dependence on parameters of ODE solutions]
  \label{thm:unif-smooth-flow}
  Let an ordinary differential equation~\ref{eq:ODEsystem} be given,
  where $f \in \mathcal{F} \subset \BUC^{k,\alpha}(\R \times B;B)$
  with $k \ge 1,\,0 \le \alpha \le 1$, $B$ a Banach space, and
  $\mathcal{F}$ a bounded subset. Let $I \subset \R$ be a bounded
  interval and $X = C^0(I;B)$ the Banach space of (solution) curves,
  endowed with the supremum norm.

  Then the flow $\Phi$ is a $\BC^k$ mapping
  \begin{equation*}
    \Phi\colon I \times B \times \mathcal{F} \to X
        \colon (t_0, x_0, f) \mapsto \big(t \mapsto x(t) \big).
  \end{equation*}
  The boundedness is understood to hold only for the derivatives.
  Moreover, $\Phi \in \BUC^{k,\alpha}$ holds as a mapping from
  $I \times B$ for fixed $f \in \mathcal{F}$.
\end{theorem}

\begin{remark}\label{rem:smooth-flow-timedep}\NoEndMark
  Differentiable dependence on time can be dropped from this theorem.
  That is, let us instead assume that $f(t,x)$ and its derivatives
  $\D_x^i f(t,x),\,i \le k$ with respect to $x$ are bounded continuous
  with respect to $(t,x)$. Then the flow is a $\BC^k$ mapping
  \begin{equation*}
    \Phi\colon B \times \mathcal{F} \to X
        \colon (x_0, f) \mapsto \big(t \mapsto x(t)\big)
  \end{equation*}
  when $I \subset \R$ is a bounded interval. This result follows
  directly from the proof, since we only used differentiability with
  respect to $t$ for differentiable dependence of $\Phi$ on $t$.
\end{remark}

\begin{remark}\label{rem:smooth-flow-BG}
  Instead of a Banach space $B$, we can also choose the setting of a
  Riemannian manifold $(M,g)$. Solving for the flow of a differential
  equation is defined in terms of local charts, so by standard
  arguments the $C^k$ smoothness result extends to this setting.

  If we assume moreover in the context of Chapter~\ref{chap:boundgeom}
  that $(M,g)$ has bounded geometry and that $f \in \BUC^{k,\alpha}$,
  then we can obtain stronger results close to those of
  Theorem~\ref{thm:unif-smooth-flow}. In any single normal coordinate
  chart the results of Theorem~\ref{thm:unif-smooth-flow} hold. To
  extend the flow beyond one chart, we use the fact that coordinate
  chart transitions are uniformly $C^k$\ndash bounded maps. It follows
  that $\Phi \in \BUC^{k,\alpha}$ on any domain such that all image
  curves are covered by a uniformly bounded number of charts. This
  includes the domain $M \times I$ for any finite interval
  $I \subset R$, since $f$ itself is assumed bounded. The bounds and
  continuity moduli will depend on $\abs{I}$ though.

  Alternatively, uniform (H\"older) continuity estimates independent
  of charts can be obtained by using
  Proposition~\ref{prop:unif-partrans} to express continuity moduli in
  terms of parallel transport. See Lemma~\ref{lem:exp-growth-holder},
  which is proven via a variation of constants method.
\end{remark}


\chapter{The Nemytskii operator}\label{chap:nemytskii}
\index{Nemytskii operator}

The Nemytskii operator creates a mapping on curves from a simple
function between spaces. That is, in its simplest form, if we have a
function $f\colon \R^n \to \R^m$, then the associated Nemytskii
operator
\begin{equation*}
  F\colon C(\R;\R^n) \to C(\R;\R^m), \qquad F(x)(t) = f(x(t)),
\end{equation*}
maps curves $x$ in $\R^n$ to curves $y = F(x) = f \circ x$ in $\R^m$. See
also~\cite[p.~103--109]{Vanderbauwhede1989:centremflds-normal} for a
clear presentation.

We investigate continuity of the Nemytskii operator for certain
classes of curves. The following definition of the Nemytskii operator
in a somewhat more abstract context on bundles over $\R$ allows
e.g.\ws for the map $f$ to be time-dependent.
\begin{definition}[Nemytskii operator]
  \label{def:nemytskii}
  Let $I \subset \R$ and let $X, Y$ be normed vector
  bundles\footnote{%
    For our purposes, a sufficient definition of a normed vector
    bundle $\pi\colon X \to \R$ is that there exist local
    trivializations $\tau\colon \pi^{-1}(U) \to U \times F$ that are
    isometric with respect to the norms on $X$ and the normed linear
    space~$F$. Note that we canonically have such trivializations by
    parallel transport, see~\ref{eq:TBX-triv} and
    Proposition~\ref{prop:DA-cont}.%
  } over $I$. Furthermore, let $f\colon X \to Y$ be a bundle map, i.e.\ws a
  fiberwise mapping that covers the identity on $I$, but which is not
  necessarily linear in the fibers. We define the corresponding
  Nemytskii operator
  \begin{equation}\label{eq:nemytskii}
    F\colon \Gamma(X) \to \Gamma(Y)
     \colon x \mapsto f \circ x,
  \end{equation}
  mapping continuous sections of\/ $X$ to continuous sections of\/ $Y$.
\end{definition}

In the previous definition as well as in the following lemma, we need
not restrict to vector bundles; we shall also require the case that
$X$ is a trivial fiber bundle with a metric space as fiber (e.g.\ws
the bundle $\BX$ in the context of Chapter~\ref{chap:persist}). Recall
that the space of sections $\Gamma(X)$ can be endowed with an
exponential growth distance~\ref{eq:rho-dist} or
norm~\ref{eq:rho-norm}, respectively. This turns $\Gamma(X)$ into a
metric (or normed linear) space denoted by $\Gamma^\rho(X)$ with
exponent $\rho \in \R$. The distance $d_\rho(x_1,x_2)$ may be infinite
for some $x_1,\,x_2 \in \Gamma^\rho(X)$ if $X$ is a trivial metric fiber
bundle. This is not a problem, since it is only used to obtain (local)
continuity estimates for sections such that $d_\rho(x_1,x_2) <\infty$.

\begin{lemma}[Continuity of the Nemytskii operator]
  \label{lem:cont-nemytskii}
  Let $X,\,Y$ be normed vector bundles over $I = \R_{\ge 0}$, or
  alternatively let\/ $X$ be a trivial fiber bundle of a metric space.
  Let $f \in C^0(X;Y)$ be a continuous fiberwise mapping and let
  $F\colon \Gamma(X) \to \Gamma(Y)$ be defined as
  in~\ref{def:nemytskii}. Let $\rho_1, \rho_2 \in \R$ and assume that
  one of the following holds:
  \vspace{-0.1cm}
  \begin{enumerate}
  \item\label{enum:cont-nem-1} $\rho_2 > 0$ and $f$ is bounded into
    the normed vector bundle $Y$;
  \item\label{enum:cont-nem-2} $\rho_2 \ge \alpha\,\rho_1$ and $f$ is
    $\alpha$\ndash H\"older continuous with $0 < \alpha \le 1$,
    uniformly with respect to the fibers.
  \end{enumerate}

  Then $F$ is continuous as a map
  $\Gamma^{\rho_1}(X) \to \Gamma^{\rho_2}(Y)$ and
  under~\ref{enum:cont-nem-2}, $F$ is moreover $\alpha$\ndash H\"older
  continuous again.
\end{lemma}

\begin{proof}
  We first prove the statement under assumption~\ref{enum:cont-nem-1}.
  Fix $x_1 \in \Gamma^{\rho_1}(X)$, let $\epsilon > 0$ be given, and
  let $x_2 \in \Gamma^{\rho_1}(X)$ be arbitrary. As $f$ is bounded and
  $\rho_2 > 0$, we can choose a $T > 0$ such that
  \begin{equation*}
    \forall\; t > T\colon
        \norm{f(x_1(t)) - f(x_2(t))}\,e^{-\rho_2\,t}
    \le 2\,\norm{f}\,e^{-\rho_2\,T}
    \le \epsilon.
  \end{equation*}
  This leaves only the compact interval $\intvCC{0}{T}$ for which we
  still have to show that
  $\norm{f(x_1(t)) - f(x_2(t))}\,e^{-\rho_2\,t} \le \epsilon$. Let us
  denote $g\colon I \to \R\colon t \mapsto e^{-\rho_2\,t}$, then the
  continuity estimate of $f \cdot g\colon X \to Y$ is uniform on the
  compact set $x_1(\intvCC{0}{T})$. Hence, there exists a
  $\delta' > 0$ such that for all $t \in \intvCC{0}{T}$ and
  $\xi_2 \in \pi^{-1}(t) \subset X$,
  \begin{equation*}
    d(x_1(t),\xi_2) \le \delta'
    \; \Longrightarrow \;
    e^{-\rho_2\,t}\,\norm{f(x_1(t)) - f(\xi_2)} \le \epsilon.
  \end{equation*}
  We have that
  $d(x_1(t),x_2(t)) \le e^{\abs{\rho_1\,T}}\,d_{\rho_1}(x_1,x_2)$, so
  choosing $\delta = e^{-\abs{\rho_1\,T}}\,\delta'$ yields the
  required estimate for $d_{\rho_1}(x_1,x_2) \le \delta$. This proves
  that $F$ is continuous at $x_1$.

  Secondly, assume~\ref{enum:cont-nem-2} and let $C_\alpha$ be
  the H\"older coefficient of $f$. Then we can estimate
  \begin{equation*}
    \begin{aligned}
      \norm{F(x_1) - F(x_2)}_{\rho_2}
      &=   \sup_{t \ge 0}\; e^{-\rho_2\,t}\,\norm{f(x_1(t)) - f(x_2(t))}\\
      &\le \sup_{t \ge 0}\; e^{-\rho_2\,t}\,C_\alpha\,
             \big(d_{\rho_1}(x_1,x_2)\,e^{\rho_1\,t}\big)^\alpha
       = C_\alpha\,d_{\rho_1}(x_1,x_2)^\alpha,
    \end{aligned}
  \end{equation*}
  which shows that $F$ is $\alpha$\ndash H\"older continuous again
  with coefficient $C_\alpha$.
\end{proof}

\begin{corollary}
  \label{cor:unifcont-nemytskii}
  Let the assumptions of Lemma~\ref{lem:cont-nemytskii} with
  condition~\ref{enum:cont-nem-1} be satisfied. If $f$ is fiberwise
  uniformly continuous with continuity modulus independent of the
  fiber, then also $F$ is uniformly continuous.
\end{corollary}

\begin{proof}
  This follows easily: in the proof above, the uniform continuity on
  the compact set $x_1(\intvCC{0}{T})$ can be replaced by the uniform
  continuity modulus of $f$ itself. This does not depend on $x_1, x_2$
  anymore, only on their distance, so it leads to a uniform continuity
  modulus of $F$.
\end{proof}

\begin{remark}\label{rem:nemytskii-time-rev}
  The previous results also hold under time inversion. That is, if we
  consider the interval $I = \R_{\le 0}$ and invert the inequalities
  for $\rho_1,\,\rho_2$ in conditions~\ref{enum:cont-nem-1}
  and~\ref{enum:cont-nem-2}, then Lemma~\ref{lem:cont-nemytskii} and
  Corollary~\ref{cor:unifcont-nemytskii} still hold true. We use this
  time inverted version in Chapter~\ref{chap:persist}.
\end{remark}


\chapter{Exponential growth estimates}\label{chap:exp-growth}

\index{exponential growth rate}
\index{higher order derivative}
In this appendix we investigate the growth rate of higher order
derivatives of a general flow on a Riemannian manifold. Basically, if
the growth of the tangent flow is proportional to
$\exp(\rho\,t)$, then the growth of the $\fracdiff$\th order
derivative is of order $\exp(\fracdiff\,\rho\,t)$. This even extends to
`fractional' derivatives, that is, the $C^{k,\alpha}$\ndash norm
(which includes $\alpha$\ndash H\"older continuity bounds) has this
growth behavior for $r = k+\alpha$. These results will be used to
obtain continuity and higher order smoothness of the
persisting NHIM\@. The particular exponential growth behavior
$\exp(\fracdiff\,\rho\,t)$ will precisely prescribe the spectral gap
condition: to construct a contraction on the $\fracdiff$\th
derivative, the normal contraction of order $\exp(\rho_\sy\,t)$ must
dominate the higher order $\exp(\fracdiff\,\rho_\sx\,t)$ along the
invariant manifold, hence $\rho_\sy < \fracdiff\,\rho_\sx$ is
required\footnote{%
  We formulate all statements in this section with respect to
  exponentially bounded flows in the (more natural) forward time
  direction. That is, we work with $t \in \R_{\ge t_0}$ and typical
  exponents $\rho > 0$. In our applications in
  Chapter~\ref{chap:persist} we use the time-reversed statements. See
  also Remark~\ref{rem:sign-rho}.%
}.

These results are based on estimating variation of constants integrals
and similar in spirit to Gronwall's lemma. We work on Riemannian
manifolds, however. This complicates matters with a lot of
technicalities, but the basic ideas are still the same. We do require
uniform bounds and bounded geometry of the manifold, see
Chapter~\ref{chap:boundgeom}. Let us first show
the idea for a flow on $\R^n$ and then introduce some concepts and
notation to finally treat the general case.

\begin{lemma}[Exponential growth estimates for a flow]
  \label{lem:exp-growth-basic}
  Let\/ $\Phi^{t,t_0} \in C^{k \ge 1}$ be the flow of a time-dependent
  vector field $v$ on $\R^m$. Let $v(t,\slot) \in \BC^k(\R^m)$ with
  all derivatives jointly continuous in $(t,x) \in \R \times \R^m$ and
  uniformly bounded by\/ $V < \infty$. Suppose that\/
  $\norm{\D\Phi^{t,t_0}(x)} \le C_1\,e^{\rho(t-t_0)}$ for all
  $x \in \R^m,\, t \ge t_0$ and fixed $C_1 > 0,\,\rho \neq 0$. Then
  for each $n,\,1 \le n \le k$ there exists a bound $C_n > 0$ such that
  \begin{equation}\label{eq:exp-growth-basic}
    \forall\, x \in \R^m,\, t \ge t_0 \colon
    \norm{\D^n \Phi^{t,t_0}(x)} \le \begin{cases}
      C_n\,e^{n\,\rho(t-t_0)} & \text{if } \rho > 0,\\
      C_n\,e^{   \rho(t-t_0)} & \text{if } \rho < 0.
    \end{cases}
  \end{equation}
\end{lemma}

\begin{proof}
Let $\D$ denote the partial derivative with respect to the spatial
variable $x \in \R^m$. We suppress the time dependence in the notation
of $v$ since we have the bound $\norm{\D^n v} \le V$ for all
$1 \le n \le k$, uniformly in space and time.

Since $\Phi^{t,t_0}$ is a flow, we have
\begin{equation}\label{eq:derflow-initial}
  \D  \Phi^{t_0,t_0}(x) = \Id \quad\text{and}\quad
  \D^n\Phi^{t_0,t_0}(x) = 0,  \quad 2 \le n \le k.
\end{equation}

For $1 \le n \le k$ we can write, suppressing arguments $t_0,\,x$,
\begin{equation}\label{eq:ODE-derflow}
    \der{}{t}\D^n\Phi^t
  = \D^n(v\circ\Phi^t)
  = \D v\circ\Phi^t \cdot \D^n\Phi^t +
    \sum_{l=2}^n \D^l v\circ\Phi^t \cdot P_{l,n}\big(\D^1\Phi^t,\ldots,\D^{n-1}\Phi^t\big),
\end{equation}
where the $P_{l,n}$ are homogeneous, weighted polynomials as in
Definition~\ref{def:homog-weight-poly} below.
In the first equality, the switching of partial derivatives is
well-defined, because the spatial derivative in the middle expression
is well-defined and the resulting function continuous. In the
right-hand expression we have already used
Proposition~\ref{prop:compfunc-deriv} and separated the homogeneous
term with $\D^n\Phi^t$ (when $l = 1$). The result is a linear
differential equation for $\D^n\Phi^t$ with the inhomogeneous terms in
the sum consisting of lower order derivatives $\D^i\Phi^t,\, i<n$, only.

For $n = 1$, statement~\ref{eq:exp-growth-basic} is already true by
assumption and in that case we also see that~\ref{eq:ODE-derflow} is a
homogeneous linear differential equation. Denote by $\Psi_x(t,t_0)$ the
solution operator for this system with initial point $x \in \R^m$, then
\begin{equation}\label{eq:derflow}
    \D\Phi^{t,t_0}(x)
  = \Psi_x(t,t_0)\big(\D\Phi^{t_0,t_0}(x)\big)
  = \Psi_x(t,t_0) \cdot \Id
  = \Psi_x(t,t_0).
\end{equation}
This solution operator acts by left-composition on linear maps, so we
read off that $\Psi_x(t,t_0) = \D\Phi^{t,t_0}(x)$ and find the estimate
$\norm{\Psi_x(t,t_0)} \le C_1\,e^{\rho(t-t_0)}$. Now we turn to the
induction step. For $n > 1$, we still have essentially the same
solution operator $\Psi_x(t,t_0)$ for the homogeneous part, only now
acting by composition on multilinear maps
$\D^n\Phi^{t,t_0}(x) \in \CLin^n(\R^m)$: the solution operator is not
influenced by considering multilinear maps, as $\D v$ and $\Psi$ act
by linear composition from the left, essentially on tangent vectors.
Therefore, the same growth estimate for $\Psi_x(t,t_0)$ still holds.

The inhomogeneous terms in~\ref{eq:ODE-derflow} depend only on the
$\D^i\Phi^t,\, i < n$ and by the induction hypothesis we can estimate
$\norm{\D^i\Phi^t} \le C_i\,e^{i\,\rho\,t}$. Using variation
of constants, the solution can now be written as
\begin{equation}\label{eq:derflow-sol-int}
  \D^n\Phi^t(x)
  = \int_{t_0}^t \Psi_x(t,\tau) \cdot
                 \sum_{l=2}^n \D^l v\circ\Phi^\tau \cdot
                 P_{l,n}\big(\D^1\Phi^\tau,\ldots,\D^{n-1}\Phi^\tau\big) \d\tau,
\end{equation}
where the homogeneous part of the solution is zero because
$\D^n\Phi^{t_0,t_0}(x) = 0$ for $n>1$. Given that the weighted degree
of $P_{l,n}$ is $n$, we can directly estimate
\begin{align*}
  \norm{\D^n\Phi^t(x)}
  &\le \int_{t_0}^t \norm{\Psi_x(t,\tau)}
                \sum_{l=2}^n \norm[\big]{\D^l v}
                \norm{P_{l,n}\big(\D^1\Phi^\tau,\ldots,\D^{n-1}\Phi^\tau\big)} \d\tau\\
  &\le \int_{t_0}^t C_1\,e^{   \rho(t-\tau)}\,V\,R\big(\{C_i\}_{i<n}\big)
                       \,e^{n\,\rho(\tau-t_0)} \d\tau\\
  &=   C_1\,V\,R\big(\{C_i\}_{i<n}\big)\,
       \frac{e^{n\,\rho(t-t_0)} - e^{\rho(t-t_0)}}{(n-1)\,\rho}.
  \eqnumber\label{eq:derflow-est}
\end{align*}
The bound $R$ depends on finite sums and products of finite terms, so
is finite again. When $\rho > 0$, the denominator is positive and the
numerator can be estimated by $e^{n\,\rho(t-t_0)}$; when $\rho < 0$,
the numerator can be estimated by $e^{\rho(t-t_0)}$, adding a minus
sign to both parts of the fraction. Thus, in both cases
\ref{eq:exp-growth-basic}~holds. This completes the induction step.
\end{proof}

Before generalizing this lemma to Riemannian manifolds, we first
refine some previous notation. Instead of $\R^m$, we
more generally consider linear spaces $V, W$ and spaces $\CLin^k(V;W)$
of (multi)linear maps for the (higher order) derivatives of maps
$f\colon V \to W$.
\begin{definition}[Homogeneous weighted polynomial]
  \label{def:homog-weight-poly}
  \index{$P_{a,b}$}
  Let\/ $P_{a,b}(y_1,\ldots,y_n)$ be a polynomial in the variables
  $y_1$ to $y_n$. We call $P$ a homogeneous weighted polynomial of
  degree $(a,b)$ if it is a homogeneous polynomial of degree $a$ and
  moreover, each term $y_1^{p_1}\ldots y_n^{p_n}$ has weighted degree
  \begin{equation}\label{eq:weighted-degree}
    \sum_{i=1}^n i \cdot p_i = b.
  \end{equation}
\end{definition}
As a consequence, such a polynomial cannot have factors $y_n$ for
$n > b$ and the factor $y_b$ can only occur as a term on itself when
$a = 1$.


This definition can now be used to denote the higher derivatives of a
composition of two functions $f,g$ on vector spaces.
\index{higher order derivative!chain rule}
\begin{proposition}[Higher order derivatives of compositions of functions]
  \label{prop:compfunc-deriv}
  Let the mapping $x \mapsto f(g(x),x)$ be given with
  $f\colon V \times U \to W$ and $g\colon U \to V$ two sufficiently
  differentiable functions between vector spaces $U,V,W$. Then the
  $k$\th order derivative of this mapping with respect to $x$ is of
  the form
  \begin{equation}\label{eq:compfunc-deriv}
    \Big(\der{}{x}\Big)^k f(g(x),x) =
    \!\!\!\sum_{\substack{l,m \ge 0\\l + m \le k\\(l,m)\neq(0,0)}}\!\!
    \D_1^l\D_2^m f(g(x),x) \cdot P_{l,k-m}\big(\D^1 g(x), \ldots ,\D^{k-m} g(x)\big),
  \end{equation}
  where $P_{l,k-m}$ is a homogeneous weighted polynomial of degree $(l,k \tm m)$ with $l$
  higher order derivatives $\D^i g(x)$ in each term, and weighted
  degree $k \tm m$: the total number of derivatives that either produced
  an additional\/ $\D g(x)$ term or differentiated an existing one.
\end{proposition}

\begin{remark}\NoEndMark
  We will shorten the notation
  $P_{l,k}\big(\D^1 g(x), \ldots ,\D^k g(x)\big) = P_{l,k}\big(\D^\bullet g(x)\big)$.
\end{remark}

\begin{remark}
  Note that $\D_1^l\D_2^m f(g(x),x)$ is actually an element of the
  tensor product space
  $W \otimes \big(V^*\big)^{\otimes l} \otimes \big(U^*\big)^{\otimes m}$
  and $P_{l,k-m}$ an element of the $(l,k \tm m)$\ndash linear maps
  $V^{\otimes l} \otimes \big(U^*\big)^{\otimes k-m}$, or $(l,k \tm m)$
  tensors, so the composition is indeed a mapping in
  $W \otimes \big(U^*\big)^{\otimes k} = \CLin^k(U;W)$, as expected.
\end{remark}

\begin{proof}
  This is easily proven by induction. For $k = 1$ we have
  \begin{equation*}
    \der{}{x} f(g(x),x) = \D_1 f(g(x),x) \cdot \D g(x) + \D_2 f(g(x),x),
  \end{equation*}
  which satisfies~\ref{eq:compfunc-deriv}. For the induction step we
  have
  \begin{align*}
    \Big(\der{}{x}\Big)^{k+1} f(g(x),x)
    &= \der{}{x} \!\!\!\sum_{\substack{l,m \ge 0\\l + m \le k\\(l,m)\neq(0,0)}}\!\!
       \D_1^l\D_2^m f(g(x),x) \cdot P_{l,k-m}\big(\D^\bullet g(x)\big)
       \displaybreak[1]\\
    &= \!\!\!\sum_{\substack{l,m \ge 0\\l + m \le k\\(l,m)\neq(0,0)}}\!\!
       \begin{aligned}[t]
       \Big[&\;   \D_1^{l+1}\D_2^m f(g(x),x) \cdot \D^1 g(x) \cdot
                  P_{l,k-m}\big(\D^\bullet g(x)\big)\\
            &+    \D_1^l\D_2^{m+1} f(g(x),x) \cdot
                  P_{l,k-m}\big(\D^\bullet g(x)\big)\\
            &+    \D_1^l\D_2^m f(g(x),x) \cdot
                  \der{}{x} P_{l,k-m}\big(\D^\bullet g(x)\big)\Big]
       \end{aligned}\displaybreak[1]\\
    &= \!\!\!\sum_{\substack{l,m \ge 0\\l + m \le k\\(l,m)\neq(0,0)}}\!\!
       \begin{aligned}[t]
       \Big[&\;   \D_1^{l+1}\D_2^m f(g(x),x) \cdot
                  P_{l+1,k+1-m}\big(\D^\bullet g(x)\big)\\
            &+    \D_1^l\D_2^{m+1} f(g(x),x) \cdot
                  P_{l,k+m}\big(\D^\bullet g(x)\big)\\
            &+    \D_1^l\D_2^m f(g(x),x) \cdot
                  P_{l,k+1-m}\big(\D^\bullet g(x)\big)\Big]
       \end{aligned}\displaybreak[1]\\
    &= \!\!\!\sum_{\substack{l,m \ge 0\\l + m \le k+1\\(l,m)\neq(0,0)}}\!\!
       \D_1^l\D_2^m f(g(x),x) \cdot P_{l,k+1-m}\big(\D^\bullet g(x)\big).\\
  \end{align*}
  This is again of the form~\ref{eq:compfunc-deriv}:
  $k-m = (k+1) - (m+1)$, so all terms can be absorbed in the new sum
  for $k+1$.
\end{proof}

Let us make a few remarks on the form of~\ref{eq:compfunc-deriv}.
The $P_{0,m}$ for $m < k$ are zero, because then we have too few
derivatives with respect to $x$; we have $P_{0,k} = 1$ though. After
Definition~\ref{def:homog-weight-poly} it was already noted that in
a polynomial of weighted degree $k$, the factor $y_k$ can only occur
as a term on itself, up to a constant factor. More specifically in
this case, $\D^k g(x)$ occurs exactly once, in the term
\begin{equation*}
  \D_1 f(g(x),x) \cdot \D^k g(x).
\end{equation*}
This can easily be seen by direct calculation or induction. Finally,
when the composition mapping is of the form $x \mapsto f(g(x))$, then
we only have terms with $m = 0$ and all polynomials
in~\ref{eq:compfunc-deriv} have weighted degree $k$ in that case.

The next step is to generalize Lemma~\ref{lem:exp-growth-basic} to a
Riemannian manifold $(M,g)$. Here we first need to define what we mean
by higher derivatives of the flow. The tangent flow $\D \Phi^t$ is
well-defined as a mapping on $\T M$, but higher derivatives live on
higher order tangent bundles $\T^k M$. These abstract bundles make doing
explicit estimates as in the proof of Lemma~\ref{lem:exp-growth-basic}
difficult. Instead, we reuse the idea of Definition~\ref{def:unif-bounded-map}
and introduce a different representation of higher
derivatives in terms of normal coordinate charts.
\begin{definition}[Higher derivative on Riemannian manifolds]
  \label{def:deriv-normcoord}
  \index{higher order derivative!on a Riemannian manifold}
  Let $M,\,N$ be Riemannian manifolds and $f\colon M \to N$ a smooth
  map. With the notation $f_x = \exp_{f(x)}^{-1} \circ f \circ \exp_x$
  of $f$ represented in normal coordinate charts, we define for
  $k \ge 1$ and $x \in M$ the higher order derivative
  \begin{equation}\label{eq:deriv-normcoord}
    \D^k f(x) = \D^k f_x(0) = \D^k\big[\exp_{f(x)}^{-1} \circ f \circ \exp_x\big](0)
  \end{equation}
  as an element of\/ $\CLin^k(\T_x M; \T_{f(x)} N)$.
\end{definition}

\begin{remark}\label{rem:jet-bundle-rep}
  Definition~\ref{def:deriv-normcoord} can be viewed as creating a
  more explicit representation of the jet bundle of the trivial fiber
  bundle $\pi\colon M \times N \to M$. A map $f\colon M \to N$ is a
  section of this trivial bundle and the $k$\ndash jet of $f$ at a
  point $x$ is fixed in terms of the derivatives
  in~\ref{eq:deriv-normcoord} up to order $k$, in the normal
  coordinate chart centered at~$x$. We shall see below that this
  representation is still a (global) bundle, while the explicit choice
  of normal coordinate charts introduces a convenient norm to measure
  the jets.
\end{remark}

Let us make a few remarks on this choice of representation of higher
derivatives. First of all, for $k = 1$ this definition coincides with
the ordinary tangent map, as $\D\exp_x(0) = \Id_{\T_x M}$ by the
natural identification $\T_0(\T_x M) \cong \T_x M$. Furthermore, this
representation of derivatives admits operator norms, and all this
behaves nicely under composition of maps by virtue of the property
$(f \circ g)_x = f_{g(x)} \circ g_x$ for local coordinate charts:
\begin{align*}
     \norm[\big]{\D^2 (f \circ g)(x)}
  &= \norm[\big]{\D^2\big[f_{g(x)} \circ g_x\big](0)}\\
  &= \norm[\big]{\D^2 f_{g(x)}(0) \big(\D g_x(0),\D g_x(0)\big)
               + \D   f_{g(x)}(0) \big(\D^2 g_x(0)\big)}\\
  &= \norm[\big]{\D^2 f(g(x)) \cdot \D g(x)^{\otimes 2} + \D f(g(x)) \cdot \D^2 g(x)}\\
  &\le \norm{\D^2 f(g(x))} \cdot \norm{\D g(x)}^2
      +\norm{\D   f(g(x))} \cdot \norm{\D^2 g(x)},
\end{align*}
that is, these operator norms as defined via normal coordinate charts
are truly norms and satisfy the usual product rules for compositions
of (multi)linear maps.

The operator norms are induced by the norms on the tangent spaces of
$\T M$, which in turn are induced by the metric. These norms depend
smoothly on the base point, so they glue together to a smooth function
$\norm{\slot}\colon \T M \to \R_{\ge 0}$ that we will call a
`\idx{bundle norm}' on the tangent bundle\footnote{%
  Note that this is stronger than a Finsler manifold as the
  Finsler structure $F\colon \T M \to \R_{\ge 0}$ is allowed to be
  asymmetric, that is, on each tangent space, $F$\/ need only scale
  linearly for positive scalars. I did not investigate whether it is
  possible to generalize this theory to Finsler manifolds.%
} or sometimes refer to as just a norm on $\T M$. Higher derivatives
can be viewed as partial sections of the vector bundle
\begin{equation}\label{eq:multilin-bundle}
  \CLin^k(\T M;\T N) = \T N \boxtimes (\T M^*)^{\otimes k}.
\end{equation}
That is, we define $\CLin^k(\T M;\T N)$ as a bundle over $M \times N$
with fiber $\CLin^k(\T_x M;\T_y N)$ over the point
$(x,y) \in M \times N$. This is indicated by the operator $\boxtimes$,
which differs from the usual tensor product $\otimes$ in the sense
that the new bundle is constructed on the product of the base spaces
instead of one common base. Now the $k$\th order derivative (as in
Definition~\ref{def:deriv-normcoord}) of a map $f\colon M \to N$ is a
section of the bundle~\ref{eq:multilin-bundle} restricted to the base
submanifold $\Graph(f) \subset M \times N$ and the derivative
$\D^k f(x)$ is the point in the section over $(x,f(x))$. More
generally, we can define vector bundles of $(l,k)$\ndash linear maps
\index{$L^{l,k}$@$\CLin^{l,k}$}
\begin{equation}\label{eq:multilin-bundle-lk}
  \CLin^{l,k}(\T M;\T N) = (\T N)^{\otimes l} \boxtimes (\T M^*)^{\otimes k}
\end{equation}
and the disjoint union of all these bundles. The bundle norms on
$\T M$ and $\T N$ together naturally induce bundle operator norms on
these. From here on, we set $M = N$ and assume that $f = \Phi^t$ is a
flow.

To finally generalize Lemma~\ref{lem:exp-growth-basic} to Riemannian
manifolds, there is still one issue to tackle. When taking the
time-derivative as in~\ref{eq:ODE-derflow}, the target base point
$\Phi^t(x)$ changes. This suggests that a covariant derivative is
required. The $\CLin^{l,k}(\T M;\T M)$ are smooth manifolds in a
natural way, however, so both the tangent vector
$\der{}{t}\D^k\Phi^t(x)$ and the differential of $\norm{\slot}$ are
well defined in this interpretation and independent of a connection,
and certainly their product $\der{}{t}\norm{\D^k\Phi^t(x)}$ is.
The tangent exponential maps $\D \exp$ at $x$ and
$\Phi^t(x)$ together induce a local coordinate chart on
$\CLin^{l,k}(\T M;\T M)$ in a neighborhood of
$\CLin^{l,k}(\T_x M; \T_{\Phi^t(x)} M)$. We will use these local
coordinates for explicit calculations.

The dependence on the base point of the norms and normal coordinate
charts in~\ref{eq:deriv-normcoord} introduces additional terms
when formulating equations~\ref{eq:ODE-derflow}
and~\ref{eq:derflow-sol-int} on a Riemannian manifold. Under the
assumption that $(M,g)$ is of bounded geometry, however, all these
additional terms will be globally bounded. Hence, these will only
contribute to the overall constants $C_n$ in
Lemma~\ref{lem:exp-growth-basic}, but not influence the basic result.

\begin{lemma}[Exponential growth estimates on a Riemannian manifold]
  \label{lem:exp-growth-riem}
  Let\/ $\Phi^{t,t_0} \in C^{k \ge 1}$ be the flow of a
  time-dependent vector field $v$ on a Riemannian manifold
  $(M,g)$ of $(k \tp 3)$\ndash bounded geometry. Let
  $v(t,\slot) \in \vf^k_b(M)$ with all derivatives jointly continuous
  in $(t,x) \in \R \times M$ and uniformly bounded by\/ $V < \infty$
  with respect to Definition~\ref{def:deriv-normcoord}. Suppose that\/
  $\norm{\D\Phi^{t,t_0}(x)} \le C_1\,e^{\rho(t-t_0)}$ for all
  $x \in M,\, t \ge t_0$ and fixed $C_1 > 0,\,\rho \neq 0$. Then for
  each $n,\,1 \le n \le k$ there exists a bound $C_n > 0$ such that
  \begin{equation}\label{eq:exp-growth-riem}
    \forall\, x \in M,\, t \ge t_0 \colon
    \norm{\D^n \Phi^{t,t_0}(x)} \le \begin{cases}
      C_n\,e^{n\,\rho(t-t_0)} & \text{if\/ } \rho > 0,\\
      C_n\,e^{   \rho(t-t_0)} & \text{if\/ } \rho < 0.
    \end{cases}
  \end{equation}
\end{lemma}

\begin{proof}
  The proof is basically the same as the proof of
  Lemma~\ref{lem:exp-growth-basic}, with additional technicalities due
  to $M$ being a manifold. We will focus on these.

  Equation~\ref{eq:ODE-derflow} can be formulated in terms of the
  tangent normal coordinate chart
  \begin{equation*}
    \D\exp_y^{-1}\colon \T B(y;\delta) \subset \T M \to
                        \T (\T_y M) \cong (\T_y M)^2
  \end{equation*}
  with $y = \Phi^t(x)$ fixed. Note that we are finally interested in
  the growth behavior of $t \mapsto \norm{\D^n\Phi^t(x)}$; this is
  defined in a coordinate-free way, so it is not influenced by our
  choice of intermediate coordinates. In these normal coordinates,
  both the metric and its derivatives are bounded due to
  Theorem~\ref{thm:bound-metric}, and the vector field is $C^k$
  bounded by assumption. We have
  \begin{align*}
    \fst \der{}{t}\D\exp_y^{-1}\circ\D^n\Phi^t(x)\\
    &= \der{}{t}\D\exp_y^{-1}\circ\D^n\big[\exp_{\Phi^t(x)}^{-1}\circ\Phi^t\circ\exp_x\big](0)\\
    &= \der{}{t}\D\exp_y^{-1}\circ \sum_{l=1}^n
         \D^l\big[\exp_{\Phi^t(x)}^{-1}\circ\exp_y\big](0)\cdot
         P_{l,n}\big(\D^\bullet \big[\exp_y^{-1}\circ\Phi^t\circ\exp_x\big](0)\big).
\intertext{%
  This splits the dependence on $t$ in the target base point
  $\Phi^t(x)$ from that in the derivatives $\D^n\Phi^t$ itself. Note
  that the sum must be interpreted as a sum of terms in the single
  fiber $\CLin^n(\T_x M;\T_y M)$ over the base point $(x,y)$.
  By using the coordinate map $\D\exp_y^{-1}$, we transferred
  the problem to fixed linear spaces, which allows us to make sense of
  the differentiation with respect to $t$. In other words,
  $\D\exp_y^{-1}$ induces locally trivializing coordinates for
  $\CLin^n(\T_x M;\T M)$ in a neighborhood of $y$ with $x$ fixed. As $\D\exp_y^{-1}$
  is linear on the fibers, we can distribute it over the sum to
  further obtain
}
    &= \sum_{l=1}^n \der{}{t}\Big[\D\exp_y^{-1}\circ
         \D^l\big[\exp_{\Phi^t(x)}^{-1}\circ\exp_y\big](0)\cdot
         P_{l,n}\big(\D^\bullet \big[\exp_y^{-1}\circ\Phi^t\circ\exp_x\big](0)\big)\Big]\\
    &= \der{}{t}\Big[\D\exp_y^{-1}\circ\D\big[\exp_{\Phi^t(x)}^{-1}\circ\exp_y\big](0)\cdot
                     \D^n\big[\exp_y^{-1}\circ\Phi^t\circ\exp_x\big](0)\Big]
      +\sum_{l=2}^n \der{}{t}\big[\ldots\big].
    \label{eq:ODE-derflow-riem}\eqnumber
  \end{align*}
  In the last line, the homogeneous part is separated from the
  non-homogeneous terms as in~\ref{eq:ODE-derflow}.

  Working out the details of the homogeneous part, we obtain\footnote{%
    The time derivative of
    $\D\exp_y^{-1}\circ\D\big[\exp_{\Phi^t(x)}^{-1}\circ\exp_y\big](0)$
    actually turns out to be zero in local coordinates. This follows
    from an analysis of the exponential map as the time-one geodesic
    flow in normal coordinates around~$y$. This result is not
    relevant for us, so we leave out this tedious calculation.%
  }
  \begin{align*}
    \fst  \der{}{t}\Big[\D\exp_y^{-1}\circ\D\big[\exp_{\Phi^t(x)}^{-1}\circ\exp_y\big](0)\cdot
                        \D^n\big[\exp_y^{-1}\circ\Phi^t\circ\exp_x\big](0)\Big]\\
    &= \der{}{t}\Big[\D\exp_y^{-1}\circ\D\big[\exp_{\Phi^t(x)}^{-1}\circ\exp_y\big](0)\Big]\cdot
       \D^n \Phi^t(x)\\
    \con +\D\big[\D\exp_y^{-1}\big]\cdot\D^n\der{}{t}\big[\exp_y^{-1}\circ\Phi^t\circ\exp_x\big](0)\\
    &= \der{}{t}\Big[\D\exp_y^{-1}\circ\D\big[\exp_{\Phi^t(x)}^{-1}\circ\exp_y\big](0)\Big]\cdot
       \D^n \Phi^t(x)\\
    \con +\D\big[\D\exp_y^{-1}\big]\cdot\sum_{l=1}^n
          \D^l\big[\D\exp_y^{-1}\circ v\circ\exp_y\big](0)\cdot P_{l,n}(\D^\bullet\Phi^t(x)).
  \end{align*}
  Note that again all terms $l \ge 2$ in the sum are inhomogeneous
  terms that we will add to those already present
  in~\ref{eq:ODE-derflow-riem}. The homogeneous term is some linear
  vector field acting (from the left) on $\D^n\Phi^t(x)$ and it is
  precisely the vector field generating $\D\Phi^t(x)$, which is
  the original case $n = 1$. Hence, we can again define the
  operator $\Psi_x^{t,t_0}$ as post-composition with $\D\Phi^t(x)$ and
  write the flow of $\D^n\Phi^t(x)$ using a variation of constants
  integral with all the non-homogeneous terms. These terms again
  contain only lower order derivative flows $\D^l\Phi^t(x),\; l < n$.

  We can now take the operator norm of this expression. In principle
  we should be careful that this bundle norm depends on the changing
  target point $\Phi^t(x)$. The normal coordinates were chosen around
  $y = \Phi^t(x)$, however, and in these coordinates the derivative of
  the metric at the origin (corresponding to $y$) is zero, hence the
  norm has zero derivative. We can thus
  simply apply the operator norm to the variation of constants
  integral and obtain estimates as in~\ref{eq:derflow-est}. The
  additional factors introduced by differentiation of normal
  coordinate transition maps are bounded by
  Lemma~\ref{lem:bound-coordtrans} under the assumption that $(M,g)$
  is of $(k \tp 3)$\ndash bounded geometry. The inhomogeneous terms
  still contain at least one factor $\D^l\Phi^t(x)$, so the result in
  case $\rho < 0$ holds as well.
\end{proof}

These exponential growth results can be extended further to uniform and H\"older
continuity in the highest derivatives. The H\"older continuity then is
with respect to the growth rate $(k+\alpha)\rho$, where $k$ is the
order of the derivative and $0 < \alpha \le 1$ the H\"older constant.
Thus, $\alpha$\ndash H\"older continuity can be viewed as a fractional
derivative; Lipschitz continuity (when $\alpha = 1$) can indeed be
viewed as almost differentiability to one higher order. The case
$\alpha = 0$ we shall identify with uniform continuity. Here we have
no explicit modulus of continuity, which requires an arbitrarily small
additional $\mu > 0$ in the exponent $k\,\rho + \mu$ to compensate.

\begin{remark}[On using a global continuity modulus]
  \label{rem:glob-cont-mod}
  \index{continuity modulus!in bounded geometry}
  In the next lemma, as well as in
  Corollary~\ref{cor:exp-growth-holder-param} below, we shall make
  abuse of notation in writing expressions such as
  $\norm{s(x_2) - s(x_1)}$, where $s$ is a section of a vector bundle,
  cf.~\ref{eq:exp-growth-unifcont}, that is, we compare objects that
  live in different fibers of a vector bundle\footnote{%
    Note that the higher derivatives $\D^k \Phi^t(x)$ of a flow are
    actually interpreted as elements of a bundle of
    type~\ref{eq:multilin-bundle}. These bundles are still naturally
    induced by the tangent bundles of underlying manifolds, so all
    bounded geometry techniques, such as uniformity of normal
    coordinate charts, unique local trivializations by parallel
    transport, are induced on these bundles as well.%
  }. This notation should be interpreted according to
  Remark~\ref{rem:loc-cont-modulus}. That is, if $x_1,\,x_2$ are
  $M$\ndash close in the spirit of Definition~\ref{def:M-small}, then
  this is well-defined in terms of local charts, and for continuity
  estimates this is equivalent to an estimate by identification of the
  vector bundle over $x_1,\,x_2$ via parallel transport, cf.\ws
  Proposition~\ref{prop:unif-partrans}. If $x_1$ and $x_2$ are not
  close, then we can use any choice of isometric identification of the
  vector bundle over these points, such as the construction of
  parallel transport along solutions curves in
  Section~\ref{sec:formal-tangent}. In this case the notation can
  effectively be interpreted as an estimation by the sum of the norms
  of the separate terms with the triangle inequality. When applying
  this lemma, we shall always have such an isometric identification at
  hand, hence these arguments can be made rigorous, and the notation
  provides a sensible heuristic then.
\end{remark}

\begin{lemma}[Exponential growth estimates with H\"older continuity]
  \label{lem:exp-growth-holder}
  Let\/ $\Phi^{t,t_0} \in C^{k \ge 1}$ be the flow of a time-dependent
  vector field $v$ on a Riemannian manifold $(M,g)$ of $(k \tp 3)$\ndash
  bounded geometry. Let\/ $\D$ denote the partial derivative with
  respect to the spatial variable $x \in M$ as in
  Definition~\ref{def:deriv-normcoord} and let
  $v(t,\slot) \in \vf^{k,\alpha}_{b,u}(M),\,0 < \alpha \le 1$ with all
  derivatives jointly continuous in $(t,x) \in \R \times M$.  Suppose
  that\/ $\norm{\D\Phi^{t,t_0}(x)} \le C_1\,e^{\rho(t-t_0)}$ for all
  $x \in M,\, t \ge t_0$ and fixed $C_1 > 0,\,\rho > 0$.

  Then in addition to the results of Lemma~\ref{lem:exp-growth-riem},
  there exists a bound $C_{k,\alpha} > 0$ such that
  \begin{equation}\label{eq:exp-growth-holder}
    \forall\; t \ge t_0\colon
    \norm{\D^k \Phi^{t,t_0}}_\alpha \le C_{k,\alpha}\,e^{(k+\alpha)\,\rho(t-t_0)}.
  \end{equation}
  If instead $v(t,\slot) \in \vf^k_{b,u}(M)$, i.e.\ws the special case
  $\alpha = 0$, then for each $\mu > 0$ there exists a continuity
  modulus $\epsilon_{k,\mu}$ such that
  \begin{equation}\label{eq:exp-growth-unifcont}
    \forall\; t \ge t_0\colon
    \norm{\D^k \Phi^{t,t_0}(x_2)-\D^k \Phi^{t,t_0}(x_1)}
    \le \epsilon_{k,\mu}(d(x_1,x_2))\,e^{(k\,\rho+\mu)(t-t_0)},
  \end{equation}
  that is, $x \mapsto \big(t \mapsto \D^k \Phi^{t,t_0}(x)\big)$ is
  uniformly continuous in $x$, in $\norm{\slot}_{k\,\rho + \mu}$\ndash
  norm.
\end{lemma}

\begin{remark}
  We restricted this lemma to the case $\rho > 0$ only. A result
  similar to that in~\ref{lem:exp-growth-riem} for $\rho < 0$ could be
  obtained for completeness sake, but it clutters the already detailed
  proof, while we do not need the result.
\end{remark}

\begin{proof}
  The idea of the proof is essentially the same as that of
  Lemma~\ref{lem:exp-growth-riem}. The additional difficulty is that
  (H\"older) continuity requires finite, non-differential estimates
  when comparing any two flows starting from different initial points
  $x_1,\,x_2 \in M$.

  Let $d(x_1,x_2) < \delta_M$ where $\delta_M$ is $M$\ndash small as
  in Definition~\ref{def:M-small}. We drop $t_0$ from the
  notation and define $\xi_i(t) = \Phi^t(x_i),\, i=1,2$ as the
  solution curves with $x_i$ as initial conditions. We want to study
  the growth behavior of
  \begin{equation}\label{eq:ODE-varflow}
    t \mapsto \D^k\Phi^t(x_2) - \D^k\Phi^t(x_1).
  \end{equation}
  Note that this difference is defined with respect to coordinate
  charts at source and target that contain $x_1,\,x_2$ and
  $\xi_1(t),\,\xi_2(t)$, respectively, but not in general.

  We denote by $\gamma_t$ the unique shortest geodesic that connects $\xi_1(t)$
  to $\xi_2(t)$ when $d(\xi_1(t),\xi_2(t)) < \delta_M$. Next, we set
  \begin{equation}\label{eq:ODE-varflow-partrans}
    \Upsilon^t = \D^k\Phi^t(x_2) \cdot \partrans(\gamma_0)^{\otimes k}
                -\partrans(\gamma_t) \cdot \D^k\Phi^t(x_1)
    \in \CLin^k\big(\T_{x_1} M; \T_{\xi_2(t)} M\big)
  \end{equation}
  to be the difference of the respective $k$\th order derivative
  flows, parallel transported to matching spaces at their source and
  target. It is easily verified that $\Upsilon^t$ satisfies initial
  conditions $\Upsilon^{t_0,t_0} = 0$ for any $k \ge 1$.

  Due to Proposition~\ref{prop:unif-partrans}, the formulation
  in~\ref{eq:ODE-varflow-partrans} with parallel transport to measure
  variation of the flows is equivalent to
  measuring~\ref{eq:ODE-varflow} in normal coordinate charts. Hence,
  if we study $\norm{\Upsilon^t}$ in charts, we may drop\footnote{%
    We could include the parallel transport terms, repeat similar
    arguments as in the proof of Lemma~\ref{lem:exp-growth-riem} and
    express everything in (induced) normal coordinate charts, but this
    would clutter the proof here even more. These terms would all be
    bounded and Lipschitz continuous by bounded geometry, hence not
    essentially alter the result.%
  } the parallel transport terms at the cost of an (unimportant)
  global factor in the estimates. We assume that
  $d(\xi_1(t),\xi_2(t)) < \delta_M$ and study $\Upsilon^t$ in a normal
  coordinate chart covering both points. Taking the difference
  of~\ref{eq:ODE-derflow} with $x_1,\,x_2$ inserted, we see that
  $\Upsilon^t$ satisfies the differential equation
  \begin{equation*}
    \begin{aligned}
      \der{}{t}\Upsilon^t
   &= \D v \circ \Phi^t(x_2) \cdot \Upsilon^t
     +\big[\D v \circ \Phi^t(x_2) - \D v \circ \Phi^t(x_1)\big] \cdot \D^k \Phi^t(x_1)\\
\con +\sum_{l=2}^k \D^l v\circ\Phi^t(x_2) \cdot P_{l,k}\big(\D^\bullet\Phi^t(x_2)\big)
     -(x_2 \rightsquigarrow x_1).
    \end{aligned}
  \end{equation*}
  This equation provides a variation of constants integral for
  $\Upsilon^t$ based on the flow $\Psi_{x_2}(t,t_0)$:
  \begin{equation}\label{eq:holder-varconst-int}
    \begin{aligned}
      \Upsilon^t &= \int_{t_0}^t
        \Psi_{x_2}(t,\tau) \cdot \big[\D v\circ\Phi^\tau(x_2) - \D v\circ\Phi^\tau(x_1)\big]
                           \cdot \D^k \Phi^\tau(x_1)\\
      &\hspace{0.8cm}
       +\Psi_{x_2}(t,\tau) \cdot
        \bigg[\sum_{l=2}^k \D^l v\circ\Phi^\tau(x_2) \cdot
                           P_{l,k}\big(\D^\bullet\Phi^\tau(x_2)\big)
             -(x_2 \rightsquigarrow x_1)
        \bigg] \d\tau.
    \end{aligned}
  \end{equation}
  We proceed by induction over $k$. For $k = 1$ we only have the first
  term of the integrand. Using that $\D v$ is uniformly $\alpha$\ndash H\"older,
  we have
  \begin{align*}
    \norm{\Upsilon^t}
    &\le \int_{t_0}^t \norm{\Psi_{x_2}(t,\tau)}\,
                      \norm{\D v\circ\Phi^\tau(x_2) - \D v\circ\Phi^\tau(x_1)}\,
                      \norm{\D\Phi^\tau(x_1)} \d\tau\\
    &\le \int_{t_0}^t C_1\,e^{\rho(t-\tau)}\,
                      \norm{\D v}_\alpha\,\norm{\Phi^\tau(x_2) - \Phi^\tau(x_1)}^\alpha\,
                      C_1\,e^{\rho(\tau-t_0)} \d\tau\\
    &\le C_1^2\,\norm{\D v}_\alpha\,e^{\rho(t-\tau)}\,
         \int_{t_0}^t {\big(C_1\,e^{\rho(\tau-t_0)}\,\norm{x_2 - x_1}\big)}^\alpha \d\tau\\
    &\le C_1^{2+\alpha}\,\norm{\D v}_\alpha\,e^{\rho(t-\tau)}\,\norm{x_2 - x_1}^\alpha\,
         \frac{e^{\alpha\,\rho(t-t_0)}}{\alpha\,\rho}.
  \end{align*}
  Next, in the induction step for $k > 1$, we get the additional terms
  from~\ref{eq:holder-varconst-int} in the integrand. These are (up to
  constants) a product of the flow $\Psi$, $\D^l v\circ\Phi^\tau(x)$
  and $\D^i\Phi(x)$'s with weighted degree $k$. The terms
  $\D^l v\circ\Phi^\tau(x)$ are uniformly $\alpha$\ndash H\"older
  continuous in $x$ analogous to the case $k = 1$ above. Each of the
  $\D^i\Phi(x)$'s satisfies the H\"older estimate of this lemma by the
  induction hypothesis and the growth estimates of
  Lemma~\ref{lem:exp-growth-basic}. Hence, for each term in the
  integrand, we obtain H\"older continuity with respect to $x$ with
  growth behavior at most
  $e^{\rho(t-\tau)}\,e^{(k+\alpha)\rho(\tau-t_0)}$. Integration then
  yields the stated result.

  Finally, the uniformly continuous case is an extension along the
  same lines as Corollary~\ref{cor:unifcont-nemytskii}. The map
  $x \mapsto \D^l v\circ\Phi^\tau(x)$ is uniformly continuous when
  measured in $\norm{\slot}_\mu$\ndash norm and by induction the
  alternative result~\ref{eq:exp-growth-unifcont} follows.
\end{proof}

Finally, we extend the H\"older continuous growth estimates to a
parameter dependent version. This is formulated to exactly fit the
context of derivatives of $T_\sx$ with respect to $y \in \BY$, such as
in~\ref{eq:DyTx}. Note that Remark~\ref{rem:glob-cont-mod} applies
again.
\begin{corollary}[Exponential growth with H\"older continuity and a parameter]
  \label{cor:exp-growth-holder-param}
  Assume the setting of Lemma~\ref{lem:exp-growth-holder}. Let the
  vector field $v$ furthermore depend on a third variable $y \in Y$
  such that\/ $\D_x^l v(t,x,\slot) \in \BUC^\alpha$ for all
  $0 \le l \le k$, uniformly in $t,x$ and that all original bounds are
  uniform in $y$ as well. Let $\eta \in B^\rho(\R;Y)$ denote a curve
  in $Y$ and\/ $\Phi^{t,t_0}_\eta$ the flow of $v(t,\slot,\eta(t))$.
  Assume that $\eta \mapsto \big(t \mapsto \Phi^{t,t_0}_\eta(x) \big)$
  is uniformly Lipschitz with respect to the distance function
  $d_\rho$ on curves $C(\R_{\ge t_0};X)$.

  Then the map $\eta \mapsto \D^k\Phi_\eta$ is H\"older continuous in
  the sense that there exists a bound $C_{k,\alpha,\sy} > 0$ such that
  \begin{equation}\label{eq:exp-growth-holder-param}
    \forall\; t \ge t_0,\, x \in X\colon
    \norm{\eta \mapsto \D^k \Phi^{t,t_0}_\eta(x)}_\alpha
    \le C_{\sy,\alpha}\,e^{(k+\alpha)\,\rho(t-t_0)}.
  \end{equation}
  In case of uniform continuity (i.e.\ws $\alpha = 0$), then for each
  $\mu > 0$ there exists a continuity modulus $\epsilon_{k,\mu,\sy}$
  such that
  \begin{equation}\label{eq:exp-growth-unifcont-param}
    \forall\; t \ge t_0,\, x \in X\colon
    \norm{\D^k \Phi^{t,t_0}_{\eta_2}(x)-\D^k \Phi^{t,t_0}_{\eta_1}(x)}
    \le \epsilon_{k,\mu,\sy}(d(\eta_1,\eta_2))\,e^{(k\,\rho+\mu)(t-t_0)}.
  \end{equation}
  In both cases we interpret the continuity moduli as globally defined
  using Remark~\ref{rem:loc-cont-modulus}.
\end{corollary}

\begin{proof}
  The proof closely follows that of Lemma~\ref{lem:exp-growth-holder};
  let us indicate the differences.

  We define the variation
  \begin{equation}\label{eq:ODE-varflow-param}
    \Upsilon^t = \D^k\Phi^t_{\eta_2}(x) \cdot \partrans(\gamma_0)^{\otimes k}
                -\partrans(\gamma_t) \cdot \D^k\Phi^t_{\eta_1}(x)
  \end{equation}
  and study it by a variation of constants integral in local charts,
  similar to~\ref{eq:holder-varconst-int}. In this case we obtain
  \begin{align*}
      \Upsilon^t &= \int_{t_0}^t
        \Psi_{\eta_2}(t,\tau) \cdot \Bigg(
        \Big[\D v\big(\tau,\Phi^\tau_{\eta_2}(x),\eta_2(\tau)\big)
           - \D v\big(\tau,\Phi^\tau_{\eta_1}(x),\eta_1(\tau)\big)\Big] \cdot
        \D^k \Phi^\tau_{\eta_2}(x)\\
      &\hspace{1.8cm}
      +\bigg[\sum_{l=2}^k \D^l v\big(\tau,\Phi^\tau_{\eta_2}(x),\eta_2(\tau)\big) \cdot
                          P_{l,k}\big(\D^\bullet\Phi^\tau_{\eta_2}(x)\big)
             -(2 \rightsquigarrow 1)
       \bigg]\Bigg) \d\tau.
    \eqnumber\label{eq:holder-param-varconst-int}
  \end{align*}
  As in Lemma~\ref{lem:exp-growth-holder}, the factors
  $\Psi_{\eta_2}(t,\tau)$ and $\D^l \Phi^\tau_{\eta_2}(x)$ satisfy
  appropriate exponential growth conditions. By induction over $l < k$
  the maps $\eta \mapsto \D^l \Phi^t_\eta(x)$ are $\alpha$\ndash
  H\"older continuous, while all $\D^l v$ are uniformly $\alpha$\ndash
  H\"older in $x,\,y$, and $\eta \mapsto \D^l \Phi^t_\eta(x)$ is
  uniformly Lipschitz by assumption, so
  $\eta \mapsto \D^l v\big(t,\Phi^t_\eta(x),\eta(t)\big)$ is also
  $\alpha$\ndash H\"older when measured in
  $\norm{\slot}_{\alpha\,\rho}$\ndash norm (or in
  $\norm{\slot}_\mu$\ndash norm in case of uniform continuity, see
  Appendix~\ref{chap:nemytskii}).

  In each term of the integrand, we can estimate the variation with
  respect to $\eta$ as a sum of the variations with respect to each
  factor (a product rule). The factor that is being varied adds
  $e^{\alpha\,\rho(\tau-t_0)}$ (or $e^{\mu(\tau-t_0)}$ in case $\alpha=0$) to the overall
  growth estimate. The proof is completed by inserting all these
  estimates into~\ref{eq:holder-param-varconst-int} and again using the
  fact that we have a finite number of globally bounded terms.
\end{proof}


\chapter{The fiber contraction theorem}\label{chap:fibercontr}
\index{fiber contraction theorem}

In this appendix, we give a proof of the fiber contraction theorem.
This result is originally due to Hirsch and
Pugh~\cite{Hirsch1970:stabmflds-hyperb}; the proof presented here is
taken from Vanderbauwhede~\cite[p.~105]{Vanderbauwhede1989:centremflds-normal}.
The fiber contraction theorem is a convenient general tool to obtain
convergence of functions in $C^k$\ndash norm when a direct contraction
in $C^k$\ndash norm is not available. Instead, one inductively
constructs contractions for the $k$\th derivative with all lower order
derivatives assumed fixed. If this contraction depends continuously on
the lower order derivatives, then the fiber contraction theorem can be
applied to conclude that the sequence of the function together with
its derivatives converges to a fixed point. With the additional
theorem on the differentiability of limit functions, it can then be
concluded that the sequence converges in $C^k$\ndash norm.

\enlargethispage*{0.1cm}
\begin{theorem}[Fiber contraction theorem]
  \label{thm:fibercontr}
  Let\/ $X$ be a topological space, $(Y,d)$ a complete metric space and
  let $F: X \times Y \to X \times Y$ be a fiber mapping, that is,
  $F(x,y) = \big(F_1(x), F_2(x,y)\big)$, with the following properties:
  \begin{enumerate}
  \item\label{enum:fibercontr-1}%
    $F_1$ has a unique, globally attracting fixed point
    $x^\star \in X$, that is,
    \begin{equation*}
      \forall\, x \in X\colon \lim_{n \to \infty} F_1^n(x) = x^\star;
    \end{equation*}
  \item\label{enum:fibercontr-2}%
    there is a neighborhood $U \subset X$ of $x^\star$, such that
    $F_2: U \times Y \to Y$ is a uniform contraction on $Y$ with
    contraction factor $\contr < 1$; let $y^\star \in Y$ denote the
    unique fixed point of $F_2(x^\star,\slot): Y \to Y$, as given by
    the Banach fixed point theorem;
  \item\label{enum:fibercontr-3}%
    the mapping $F_2(\slot,y^\star)\colon X \to Y$ is continuous.
  \end{enumerate}

  If only properties~\ref{enum:fibercontr-1}
  and~\ref{enum:fibercontr-2} are assumed, then $(x^\star,y^\star)$ is
  the unique fixed point of\/ $F$. If moreover~\ref{enum:fibercontr-3}
  holds, then this fixed point is globally attractive.
\end{theorem}

\begin{proof}
  The point $(x^\star,y^\star)$ is clearly the unique fixed point of
  $F$, where property~\ref{enum:fibercontr-1} implies
  uniqueness of $x^\star$ as fixed point of $F_1$
  and~\ref{enum:fibercontr-2} uniqueness of $y^\star$ under
  $F_2(x^\star,\slot)$.

  The point $x^\star$ is by assumption attractive under $F_1$, thus
  for the final conclusion of global attractivity, it remains to show
  that $y \to y^\star$ under $F$.

  Let $(x,y) \in X \times Y$ be arbitrary and consider the sequence
  $(x_n,y_n) = F^n(x,y)$ for $n \ge 0$. Since $x_n \to x^* \in U$,
  there exists an $N \in \N$ such that $x_n \in U$ for all $n \ge N$.
  By shifting the sequence $(x_n,y_n)$, we can assume without loss of
  generality that $x_n \in U$ for all $n \ge 0$ and use
  property~\ref{enum:fibercontr-2} to estimate
  \begin{align*}
         d(y_{n+1},y^\star)
    &=   d(F_2(x_n,y_n)    ,F_2(x^\star,y^\star))\\
    &\le d(F_2(x_n,y_n)    ,F_2(x_n    ,y^\star))
       + d(F_2(x_n,y^\star),F_2(x^\star,y^\star))\\
    &\le \contr\,d(y_n,y^\star) + \alpha_n.
    \label{eq:fibcontr-y}\eqnumber
  \end{align*}
  On the other hand,
  $\alpha_n = d(F_2(x_n,y^\star),F_2(x^\star,y^\star)) \to 0$ as
  $n \to \infty$ from properties~\ref{enum:fibercontr-1}
  and~\ref{enum:fibercontr-3}. Let
  $\bar{\alpha}_k = \sup_{n \ge k} \alpha_n$, then we also have
  $\bar{\alpha}_k \to 0$.

  For each $k \in \N$, let $\delta_{k,k} = d(y_k,y^\star)$ and
  recursively define
  $\delta_{n+1,k} = \contr\,\delta_{n,k} + \bar{\alpha}_k$.
  From~\ref{eq:fibcontr-y} we see that
  $d(y_n,y^\star) \le \delta_{n,k}$ when $n \ge k$. Now the map
  $f\colon \delta \mapsto \contr\,\delta + \bar{\alpha}$ is a
  contraction for any $\bar{\alpha} \in \R$, so it has a unique,
  attractive fixed point $\delta^\star(\bar{\alpha})$ and solving the
  equation $f(\delta^\star) = \delta^\star$ yields
  \begin{equation*}
    \delta^\star = \frac{\bar\alpha}{1-\contr}.
  \end{equation*}
  Let $\epsilon > 0$ be given and choose $k$ large enough that
  $\bar{\alpha}_k < \mfrac{1}{2}(1-\contr)\epsilon$. As
  $\lim_{n \to \infty} \delta_{n,k} = \delta_k^\star$ we see that
  there exists some $N$ such that
  \begin{equation*}
    \forall\; n \ge N\colon\quad
    \delta_{n,k}
    < 2 \delta_k^\star = \frac{2\,\bar{\alpha}_k}{1-\contr}
    < \epsilon.
  \end{equation*}
  From this we conclude that $d(y_n,y^\star) < \epsilon$ for all $n \ge N$.
\end{proof}

The following theorem is quite standard. We shall extend it to smooth
manifolds and higher derivatives, though.
\begin{theorem}[Differentiability of limit functions]
  \label{thm:difflimitfunc}
  Let\/ $Y$ be a Banach space and let $C^k(\R^n;Y)$ denote the space
  of\/ $C^k$ functions $\R^n \to Y$ equipped with the weak Whitney
  topology. Let ${\{f_n\}}_{n \ge 0}$ be a sequence in $C^1(\R^n;Y)$
  that converges to $f \in C^0(\R^n;Y)$ with respect to the $C^0$
  topology, and assume that there is a function
  $g \in C^0(\R^n;\CLin(\R^n;Y))$ such that\/ $\D f_n \to g$.

  Then $\D f = g$, or in other words, $f_n \to f$ in $C^1(\R^n;Y)$
  with respect to the weak Whitney topology.
\end{theorem}

\begin{proof}
  By the fundamental theorem of calculus we have
  \begin{equation*}
    f_n(x+t\,h) = f_n(x) + \int_0^t \der{}{\tau} f_n(x+\tau\,h) \d\tau
                = f_n(x) + \int_0^t \D f_n(x+\tau\,h)\cdot h \d\tau.
  \end{equation*}
  Uniform convergence of $\D f_n \to g$ on the compact set
  $\set{x + \tau h}{\tau \in \intvCC{0}{t}}$ allows us to take the
  limit $n \to \infty$ inside the integral to obtain
  \begin{equation*}
    f(x+t\,h) = f(x) + \int_0^t g(x+\tau\,h)\cdot h \d\tau,
  \end{equation*}
  and by differentiation with respect to $t$ we conclude that
  $g(x)\cdot h$ is the directional derivative of $f$ at $x$ along~$h$.

  Note that $g(x)\colon \R^n \to Y$ is a bounded linear operator by
  assumption, so let us verify that it is the total derivative,
  $\D f(x) = g(x)$, that is,
  \begin{equation*}
    \lim_{h \to 0} \frac{\norm{f(x+h) - f(x) - g(x) \cdot h}}{\norm{h}} = 0.
  \end{equation*}
  Using the mean value theorem, we have
  \begin{equation*}
    \norm{f(x+h) - f(x) - g(x) \cdot h}
    \le \sup_{\xi \in \intvCC{0}{1}} \norm[\big]{g(x+\xi\,h) - g(x)}\norm{h}
  \end{equation*}
  and $g$ is continuous, so indeed differentiability holds
  and $\D f(x) = g(x)$.
\end{proof}

\begin{remark}\label{rem:difflimit-mflds}
  The statement that $f$ is differentiable at $x$ is local, so this
  result immediately translates to maps $C^1(X;Y)$ with $X$ a smooth
  manifold by considering a local coordinate chart around $x \in X$.

  This theorem could probably be generalized even further such that
  $X,\,Y$ are allowed to be Banach manifolds. The fact that $g$ is
  continuous linear by assumption mitigates possible convergence
  problems when having to consider infinitely many independent partial
  derivatives. We should be careful though, since the weak Whitney (or
  compact-open) topology is not clearly defined anymore when $X$ is
  infinite-dimensional.
\end{remark}

\begin{corollary}
  \label{cor:difflimitfunc}
  Assume the setting of\/ Theorem~\ref{thm:difflimitfunc}. Let
  ${\{f_n\}}_{n \ge 0}$ be a sequence in $C^{k \ge 2}(\R^n;Y)$ that
  converges to $f$ in $C^{k-1}(\R^n;Y)$ and let\/
  $\D^k f_n \to g$ converge in $C^0(\R^n;\CLin^k(\R^n;Y))$. Then
  $f_n \to f$ converges in $C^k(\R^n;Y)$.
\end{corollary}

This is a trivial extension of Theorem~\ref{thm:difflimitfunc} when
using the natural identification
$\CLin(\R^n;\CLin^{k-1}(\R^n;Y)) \cong \CLin^k(\R^n;Y)$.


\chapter{Nonlinear variation of flows}\label{chap:nonlin-varconst}

In this appendix we collect two results on variation of nonlinear
flows. The first is a generalization of Lagrange's variation of
constants formula and the second is an application of it to calculate
the derivative of a flow with respect to parameters. Both results are
formulated for fully nonlinear flows.

The classical variation of constants integral due to Lagrange is well
known. Although Lagrange applied this method to the nonlinear problem
of orbital mechanics, a less known result of
Alekseev~\cite{Alekseev1961:pert-ODE} (see
also~\cite[p.~78]{Lakshmikantham1969:diffintineq}) generalizes the
variation of constants integral to the full nonlinear case.

\begin{theorem}[Nonlinear variation of constants]
  \label{thm:nonlin-varconst}
  Let\/ $X$ be a smooth manifold and let\/ $\Phi^{t,t_0}(x)$ be
  the flow generated by the time-dependent vector field $v(t,x)$,
  locally Lipschitz in $x$. Let $r(t,x)$ be an arbitrary (not
  necessarily small) perturbation, locally Lipschitz in $x$ as well.
  Then $\Phi_r^{t,t_0}(x)$ is the flow generated by $v + r$ if and
  only if it satisfies the nonlinear variation of constants formula
  \begin{equation}\label{eq:nonlin-varconst}
    \Phi_r^{t,t_0}(x) =
    \Phi  ^{t,t_0}(x) + \int_{t_0}^t \D\Phi(t,\tau,\Phi_r^{\tau,t_0}(x))\,
                                            r(\tau,\Phi_r^{\tau,t_0}(x))  \d\tau.
  \end{equation}
\end{theorem}

\begin{proof}
Using uniqueness of solutions, it is sufficient to show
that~\ref{eq:nonlin-varconst} satisfies the differential equation and
initial conditions $\Phi_r^{t,t}(x) = x$. The latter follows
automatically from $\Phi^{t,t}(x) = x$. For the first part, we
differentiate
\begin{align*}
  \der{}{\tau} \Big[\Phi^{t,\tau}\circ\Phi_r^{\tau,t_0}(x)\Big]
  &=  \pder{}{\tau}\,\Phi^{t,\tau}(y)\Big|_{y = \Phi_r^{\tau,t_0}(x)}
    + \D\Phi^{t,\tau}(\Phi_r^{\tau,t_0}(x))\cdot\der{}{\tau}\Phi_r^{\tau,t_0}(x)\\
  &= -\D\Phi^{t,\tau}(\Phi_r^{\tau,t_0}(x))\cdot v   (\tau,\Phi_r^{\tau,t_0}(x))\\
  \con
    + \D\Phi^{t,\tau}(\Phi_r^{\tau,t_0}(x))\cdot(v+r)(\tau,\Phi_r^{\tau,t_0}(x))\\
  &=  \D\Phi^{t,\tau}(\Phi_r^{\tau,t_0}(x))\cdot r   (\tau,\Phi_r^{\tau,t_0}(x)).
\end{align*}
This expression yields~\ref{eq:nonlin-varconst} when integrated from
$t_0$ to $t$.
\end{proof}

Notice that~\ref{eq:nonlin-varconst} looks ill-defined on a manifold,
but should be read as integration from the point $x$ along the vector
field defined by the integrand, which is indeed, for each
$\tau \in \intvCC{t_0}{t}$, exactly defined to be the tangent vector to
the curve $\tau \mapsto \Phi^{t,\tau}\circ\Phi_r^{\tau,t_0}(x)$,
making the equation self-consistent. If $(X,g)$ is a Riemannian
manifold, then this formula yields the distance estimate
\begin{equation}\label{eq:nonlin-varconst-est}
  d\big(\Phi_r^{t,t_0}(x),\Phi^{t,t_0}(x)\big) \le
  \int_{t_0}^t \norm*{\D\Phi(t,\tau,\Phi_r^{\tau,t_0}(x))\,
                             r(\tau,\Phi_r^{\tau,t_0}(x))} \d\tau.
\end{equation}

As a differential variant of the previous result, we state the
following.
\begin{theorem}[Differentiation of a flow]
  \label{thm:diff-flow}
  Let\/ $\Phi_s^{t,t_0}(x_0)$ be a flow on a manifold $X$, defined by a
  vector field $v_s(t,x)$ that also depends on time and an external
  parameter $s \in \R$. Let $(s,x) \mapsto v_s(t,x) \in \BC^1$ with
  derivative jointly continuous in $(s,t,x)$. Then the derivative of
  the flow with respect to $s$ is given by
  \begin{equation}\label{eq:diff-flow}
    \der{}{s} \Phi_s^{t,t_0}(x_0) =
    \int_{t_0}^t \D\Phi_s^{t,\tau}(x(\tau))\,
                 \der{}{s} v_s(\tau,x(\tau)) \d\tau,
  \end{equation}
  for any fixed $t,t_0$, and where $x(\tau) = \Phi_s^{\tau,t_0}(x_0)$.
\end{theorem}

See~\cite[Thm~B.3]{Duistermaat2000:liegroups} for a proof of the
formula for differentiation of a flow with respect to a parameter.
This is a slightly modified case where the vector field is
time-dependent. Theorem~\ref{thm:unif-smooth-flow} and
Remark~\ref{rem:smooth-flow-timedep} show that the result can be
generalized to the non-autonomous case and differentiable
time-dependence of $v$ is not required.


\chapter{Riemannian geometry}\label{chap:riemgeom}

In this appendix we recall standard facts from Riemannian geometry and
establish some notational conventions. This appendix is targeted at
the reader who has basic knowledge of Riemannian manifolds, but wants
to have a quick refresh. For more detailed expositions see for
example~\cite{Jost2008:riem-geom-anal,Gallot2004:riemgeom},
or~\cite{Lang1995:diff-riem-mflds} for a more abstract presentation in
the context of Banach manifolds. We shall not try to be exhaustive or
as general as possible in this overview.

A \idx{Riemannian manifold} $(M,g)$ is a pair of a smooth (or at least
$C^1$, respectively $C^2$ for defining curvature) manifold together
with a metric $g$: a family of positive-definite bilinear forms $g_x$
on each tangent space $\T_x M$. The metric is a generalization of the
Euclidean inner product on $\R^n$ and $g_x$ depends in a smooth way on
the point $x \in M$ in the manifold. The metric can be used to measure
angles and lengths of tangent vectors, so we can define the length of
a piecewise $C^1$ curve $\gamma\colon \intvCC{a}{b} \to M$ as
\begin{equation*}
  l(\gamma) = \int_a^b \sqrt{g_{\gamma(t)}(\gamma'(t),\gamma'(t))} \d t.
\end{equation*}
This length functional induces the distance function
\begin{equation}\label{eq:riem-dist}
  d(x,y) = \inf_\gamma\; l(\gamma)
\end{equation}
on $M$, where the infimum is taken over all piecewise $C^1$ curves
$\gamma$ connecting the points $x$ and~$y$. This turns $M$ into a
metric space.

Simple examples of Riemannian manifolds are $\R^n$ with the standard
Euclidean inner product and the sphere $S^{n-1} \subset \R^n$ with the
induced metric on its tangent bundle. Due to the Nash embedding
theorem, any $C^{k \ge 3}$ Riemannian manifold can actually be
realized as a submanifold of $\R^n$ equipped with the induced metric.

Each Riemannian manifold $(M,g)$ has an associated linear connection,
or, covariant derivative $\nabla$ on the tangent bundle $\T M$. This
so-called Levi-Civita connection is uniquely defined by the requirements
that it is torsion-free and compatible with the metric, i.e.\ws
\begin{equation*}
  \nabla_X Y - \nabla_Y X = [X,Y]
  \qquad\text{and}\qquad
  X\,g(Y,Z) = g(\nabla_X Y,Z) + g(Y,\nabla_X Z)
\end{equation*}
for all smooth vector fields $X,Y,Z$ on $M$. The connection is given
in local coordinates $x^i$ by the \idx{Christoffel symbols} $\Christoffel^i_{jk}$,
\index{$\Gamma$!Christoffel symbols}
\begin{equation*}
  \nabla_{\partial_j} \partial_k = \Christoffel^i_{jk}\,\partial_i,
\end{equation*}
where we used the Einstein summation convention for the repeated
index~$i$. The connection can be extended to the tensor bundle of $M$
so that it satisfies the Leibniz rule.

A connection, more generally on a vector bundle $\pi\colon E \to M$,
can also be viewed as a choice of a horizontal subbundle in $\T E$.
There is a naturally defined vertical subbundle
$\Ver(E) \subset \T E$ where $\Ver(E)_\xi = \T_\xi E_x$
for $\xi \in E_x = \pi^{-1}(x)$. A horizontal bundle $\Hor(E)$
is any subbundle complementary to the vertical bundle, so
\index{$Hor$@$\Hor,\,\Ver$}
\begin{equation*}
  \T E = \Hor(E) \oplus \Ver(E).
\end{equation*}
This definition of a connection is related to the definition via the
covariant derivative. The horizontal bundle precisely corresponds to
the tangent plane to a section $s$ of $E$ that is flat at a given
point $x \in M$:
\begin{equation*}
  \Hor(E)_{s(x)} = \text{Im}\big(\D s(x)\big)
  \qquad\Longleftrightarrow\qquad
  (\nabla_{\scriptscriptstyle \bullet}\,s)(x) = 0.
\end{equation*}

The Levi-Civita connection induces two important concepts: the
geodesic flow and parallel transport. Intuitively, the geodesic flow
says how to follow a straight line from an initial point along a given
direction, while parallel transport defines how to keep a tangent
vector fixed while carrying it along a path\footnote{%
  If the path is a geodesic, then parallel transport carries the
  initial velocity vector to the velocity vector along the entire
  path.}. %
Both maps are defined in local coordinates as solutions of (subtly
different) differential equations involving the Christoffel symbols.

The \idx{geodesic flow} $\geodflow^t$ is a flow on the tangent bundle $\T M$
and defined in local coordinates $x^i$ by
\begin{equation}\label{eq:geodflow-local-RG}
  \begin{aligned}
    \dot{x}^i &= v^i,\\
    \dot{v}^i &= -\Christoffel^i_{jk}(x)\,v^j\,v^k.
  \end{aligned}
\end{equation}
Here, the $v^i$ denote the induced additional coordinates on the tangent
bundle. This geodesic flow need not be complete, that is, defined for
all times. However, by the Hopf--Rinow theorem, the geodesic flow is
complete if and only if $M$ is complete as a metric space with respect
to~\ref{eq:riem-dist}. In the following we shall assume that $M$ is
complete to simplify the exposition.

If we restrict the geodesic flow map to the tangent space $\T_p M$ at a
fixed point $p \in M$ and to time $t = 1$, and finally project onto
$M$, then we obtain the exponential map
\index{$exp$@$\exp$}
\begin{equation*}
  \exp_p = \pi\circ\geodflow^1|_{\T_p M} \colon \T_p M \to M.
\end{equation*}
We have $\D \exp_p(0_p) = \Id_{\T_p M}$, so by the inverse function
theorem, $\exp_p$ is a local diffeomorphism at $0_p$. The local
inverse $\phi_x = \exp_x^{-1}$ of the exponential map can be viewed as
a coordinate chart since $\T_p M \cong \R^n$ isometrically. An
explicit identification would require a choice of orthonormal basis in
$\T_p M$, which we shall refrain from.

Such coordinates are called \idx{normal coordinates},
and locally around the point $p$ these coordinates make $M$ resemble
$\R^n$ as close as possible, in the sense that the metric at $p$ in
these coordinates is equal to the Euclidean metric and the Christoffel
symbols are zero. The exponential map is only a local diffeomorphism,
and generally there is a maximum radius $r > 0$ such that
$\exp_p \colon B(0;r) \subset \T_p M \to M$ is a diffeomorphism onto its image. This is
called the injectivity radius $\rinj{p}$ of $M$ at the point $p$. The
global \idx{injectivity radius} of $M$ is then defined as
\index{$r_\textrm{inj}$}
\begin{equation*}
  \rinj{M} = \inf_{p \in M} \rinj{p}.
\end{equation*}
If $M$ is noncompact then this global injectivity radius need not be
positive. The shortest path from $p \in M$ to any point $x$ within
distance $\rinj{p}$ is uniquely realized by one geodesic curve. In
normal coordinates these curves are rays emanating from the origin.
That is, let $v = \exp_p^{-1}(x)$ and $\gamma(t) = \exp_p(t\,v)$ with
$t \in \intvCC{0}{1}$, then $d(p,x) = l(\gamma) = \norm{v}$.

Let $\gamma\colon \intvCC{a}{b} \to M$ be a $C^1$ curve, then parallel
transport is a linear isometry (i.e. it preserves the metric~$g$)
\index{$\Pi$}
\begin{equation}\label{eq:partrans-RG}
  \partrans(\gamma)\colon \T_{\gamma(a)} M \to \T_{\gamma(b)} M
\end{equation}
between the tangent spaces at the endpoints. We use the notation
$\partrans(\gamma|_a^t)$ for parallel transport along a part of the
curve. Parallel transport is defined in local coordinates $x^i$
by the differential equation
\index{parallel transport}
\index{parallel transport!differential equation}
\begin{equation}\label{eq:partrans-local-RG}
  \der{}{t} \partrans(\gamma|_a^t)^i
  = -\Christoffel^i_{jk}(\gamma(t))\;\gamma'(t)^j\,\partrans(\gamma|_a^t)^k
  \qquad\text{with}\qquad
  \partrans(\gamma|_a^a) = \Id.
\end{equation}
In~\ref{eq:partrans-local-RG} the $x^i$ are local coordinates around
the point $\gamma(t)$ with additional induced coordinates $\partial_i$
on the tangent bundle. The representation $\partrans(\gamma|_a^t)^i$ is
defined by $\partrans(\gamma|_a^t) = \partrans(\gamma|_a^t)^i\,\partial_i$.
Put more abstractly, parallel transport defines a horizontal extension of a
vector $v \in \T_{\gamma(a)} M$ to a section of the pullback bundle
$\gamma^*(\T M)$, that is, a vector field $v(t)$ defined along
$\gamma(t)$, which has covariant derivative zero.

On a Riemannian manifold there is the concept of curvature. A manifold
is flat, i.e.\ws it has zero curvature, if it is (locally) isometric
to $\R^n$. The Riemann \idx{curvature} $R$ measures non-flatness on an
infinitesimal level. It is given by
\begin{equation*}
  R(X,Y)\,Z = \nabla_X \nabla_Y Z - \nabla_Y \nabla_X Z - \nabla_{[X,Y]} Z,
\end{equation*}
which measures how much the direction of a vector $Z$ changes when
parallel transporting it around an infinitesimal loop spanned by the
directions~$X,Y$. There is a relation between the curvature and
parallel transport that is important to us. If we consider holonomy,
that is, parallel transport along a closed loop $\gamma$, then the
deficit $\partrans(\gamma) - \Id$ is (heuristically put) equal to the
curvature form $R$ integrated over any surface enclosed by $\gamma$.
This relation can be seen as an
application of Stokes' theorem and the differential statement is that
the curvature $R$ is the generator of the infinitesimal holonomy
group~\cite{Ambrose1953:holonomy,Reckziegel2006:curv-gen-hol}.


\backmatter

\cleardoublepage
\phantomsection
\addcontentsline{toc}{chapter}{\protect\numberline{}\bibname}
\begingroup
\sloppy
\bibliographystyle{amsalpha}
\bibliography{thesis}
\endgroup

\cleardoublepage
\phantomsection
\addcontentsline{toc}{chapter}{\protect\numberline{}\indexname}
\setindexprenote{}
\printindex

\end{document}
